\documentclass[final,oneside,notitlepage,10pt]{article}

\usepackage{amsthm}
\usepackage{thmtools}
\usepackage{graphicx}
\usepackage[margin=2cm,font=normalsize,labelfont=bf]{caption}
\usepackage{verbatim}
\usepackage{enumerate}
\usepackage{fancyhdr}
\usepackage[]{geometry}
\usepackage{xstring}
\usepackage{needspace}
\usepackage{color}
\usepackage{amstext}
\usepackage{hyperref}
\usepackage{mathdots}
\usepackage{soul}

\makeatletter
\newcounter{denseversion}
\newcounter{comments}
\newcounter{authorcounter}
\newcounter{adresscounter}

\def\title#1{\gdef\@title{#1}}
\def\@title{}
\def\subtitle#1{\gdef\@subtitle{#1}}
\def\@subtitle{}

\def\authortagsused{0}
\def\adresstag#1{\if!#1!\else$^{\;#1\;}$\fi}

\renewcommand{\author}[2][]{
  \stepcounter{authorcounter}
  \if!#1!\else\gdef\authortagsused{1}\fi
  \ifnum\value{authorcounter}=1
    \def\@authorstringa{#2\adresstag{#1}}
    \def\@authorstringb{#2}
    \def\@authorstringc{#2\adresstag{#1}}
  \else
    \g@addto@macro\@authorstringa{\ and #2\adresstag{#1}}
    \g@addto@macro\@authorstringb{\ and #2}
    \g@addto@macro\@authorstringc{\\#2\adresstag{#1}}
  \fi}
\def\@author{\ifnum\value{denseversion}=0\@authorstringa\else\@authorstringb\fi}

\def\@adressstringa{}
\def\@adressstringb{}
\newcommand{\adress}[2][]{
  \stepcounter{adresscounter}
  \ifnum\value{adresscounter}=1
    \g@addto@macro\@adressstringa{\ifnum\authortagsused=0\def\br{\\}\else\def\br{, }\fi\adresstag{#1}#2}
    \g@addto@macro\@adressstringb{\def\br{\\}\adresstag{#1}\parbox[t]{14cm}{#2}}
  \else
    \g@addto@macro\@adressstringa{\\[\bigskipamount]\adresstag{#1}#2}
    \g@addto@macro\@adressstringb{\\[\medskipamount]\adresstag{#1}\parbox[t]{14cm}{#2}}
  \fi}

\def\preprint#1{\gdef\@preprint{#1}}
\def\@preprint{}
\def\keywords#1{\gdef\@keywords{#1}}
\def\@keywords{}
\def\msc#1{\gdef\@msc{#1}}
\def\@msc{}
\def\email#1{
   \gdef\@email{#1}
   \g@addto@macro\@authorstringc{ {\it (#1)}}}
\def\@email{}
\def\dedication#1{\gdef\@dedication{#1}}
\def\@dedication{}

\def\mybaselinestretch#1{\gdef\@mybaselinestretch{#1}}
\def\@mybaselinestretch{}

\def\myparskip#1{\gdef\@myparskip{#1}}
\def\@myparskip{}

\binoppenalty=\maxdimen
\relpenalty=\maxdimen

\mybaselinestretch{1.2}
\myparskip{1ex}
\renewcommand{\baselinestretch}{\@mybaselinestretch}
\def\denseversion{
  \setcounter{denseversion}{1}
  \newgeometry{left=3cm,right=3cm,top=3cm}
  \mybaselinestretch{1.1}
  \myparskip{0.5ex}
  \renewcommand{\baselinestretch}{\@mybaselinestretch}
  \setlength{\parskip}{\@myparskip}
  \normalfont
  \def\possiblelinebreak{}
  \fancyfoot[C]{\itshape{--$\,\,$\thepage$\,\,$--}}}

\def\possiblelinebreak{\\}

\renewcommand{\emph}[1]{\def\reserved@a{it}\ifx\f@shape\reserved@a\ul{#1}\else\textit{#1}\fi}

\newcommand{\mytableofcontents}{
   \ifnum\value{denseversion}=0
     \tableofcontents
   \else
     \renewcommand{\baselinestretch}{1.1}
     \setlength{\parskip}{0ex}
     \normalfont
     \begingroup
     \def\addvspace##1{\vskip0.4em}
     \tableofcontents
     \endgroup
     \renewcommand{\baselinestretch}{\@mybaselinestretch}
     \setlength{\parskip}{\@myparskip}
     \normalfont
   \fi}

\newlength{\zeilenlaenge}
\def\putindent#1{
  \settowidth{\zeilenlaenge}{#1}
  \ifnum\zeilenlaenge>\textwidth
    #1
  \else
    \noindent #1
  \fi
}

\def\href#1#2{#2}

\def\pdfdaten{
  \hypersetup{
    linktocpage = true,
    pdftitle = {\@title},
    pdfauthor = {\@author},
    pdfkeywords = {\@keywords},    
    bookmarksopen = true,
    bookmarksopenlevel = 1
  }}  

\def\showkeywords{\begin{flushleft}\footnotesize\textbf{Keywords}: \@keywords\end{flushleft}}
\def\showmsc{\begin{flushleft}\footnotesize\textbf{MSC 2010}: \@msc\end{flushleft}}

\def\tocsection#1{\section*{#1}\addcontentsline{toc}{section}{#1}}

\def\mytitle{}
\def\zmptitle{
  \begin{tabular}{cc}
    \begin{minipage}[c]{0.4\textwidth}
      \begin{flushleft}
        \includegraphics[width=110pt]{../../tex/zmp}
      \end{flushleft}  
    \end{minipage}&
    \begin{minipage}[c]{0.55\textwidth}
      \begin{flushright}
      {\small\sf\@preprint}
      \end{flushright}
    \end{minipage}
  \end{tabular}
  \vskip 2cm}

\def\maketitle{
  \pdfdaten
  \newpage
  \noindent
  \mytitle
  \begin{center}
    \LARGE\@title\\
    \if!\@subtitle!\else\smallskip\LARGE\@subtitle\\\fi
    \bigskip
    \if!\@author!\else\bigskip\large\@author\\\fi
    \ifnum\value{denseversion}=0
      \if!\@adressstringa!\else\bigskip\normalsize\@adressstringa\\\fi
      \if!\@email!\else\ifnum\value{authorcounter}=1\bigskip\normalsize\textit{\@email}\\\else\fi\fi
    \else
    \fi
    \if!\@dedication!\else\bigskip\normalsize{\@dedication}\\\fi
  \end{center}
  \ifnum\value{denseversion}=0\vskip 1.5cm\else\vskip0.5cm\fi
  \if!\@draft!\else\thispagestyle{empty}\fi}

\def\kobib#1{
  \begin{raggedright}
  \ifnum\value{denseversion}=0\else\small\fi
  \Oldbibliography{#1/kobib}
  \bibliographystyle{#1/kobib}
  \end{raggedright}
  \ifnum\value{denseversion}=0\else
      \noindent
      \if!\@authorstringc!\else
        \ifnum\authortagsused=0\ifnum\value{authorcounter}>1\normalsize\@authorstringc\\[\medskipamount]\else\fi\else\normalsize\@authorstringc\\[\medskipamount]\fi
      \fi
      \if!\@adressstringb!\else\normalsize\@adressstringb\\{}\fi
      \ifnum\authortagsused=0
        \ifnum\value{authorcounter}=1
          \if!\@email!\else\linebreak\normalsize\textit{\@email}\\{}\fi
        \else
        \fi
      \else
      \fi
  \fi}

\let\Oldbibliography\bibliography
\def\bibliography#1{
  \begin{raggedright}
  \ifnum\value{denseversion}=0\else\small\fi
  \Oldbibliography{#1}
  \end{raggedright}
  \ifnum\value{denseversion}=0\else
      \medskip
      \noindent
      \if!\@authorstringc!\else
        \ifnum\authortagsused=0\ifnum\value{authorcounter}>1\normalsize\@authorstringc\\[\medskipamount]\else\fi\else\normalsize\@authorstringc\\[\medskipamount]\fi
      \fi
      \if!\@adressstringb!\else\normalsize\@adressstringb\\{}\fi
      \ifnum\authortagsused=0
        \ifnum\value{authorcounter}=1
          \if!\@email!\else\linebreak\normalsize\textit{\@email}\\{}\fi
        \else
        \fi
      \else
      \fi
  \fi
}

\newenvironment{commentfigure}{\begin{comment}}{\end{comment}}
\newenvironment{sidewayscommentfigure}{\begin{minipage}}{\end{minipage}}

\def\comments{
  \setcounter{comments}{1}
  \renewenvironment{comment}{\begin{list}{}{\rightmargin=1cm\leftmargin=1cm}\item\sf\begin{small}}{\end{small}\end{list}}
  
  }

\def\draftstamp#1#2#3{
  \ifnum\value{comments}=0
    \gdef\@draft{DRAFT - Version #1 - Last edited on #2 by #3 - Comments are not displayed}
  \else
      \gdef\@draft{DRAFT - Version #1 - Last edited on #2 by #3 - Comments are displayed}
  \fi
  \fancyhead[C]{\footnotesize\tt\textcolor{red}{\@draft}}}

\def\draft#1#2#3#4{
  \ifnum#4=1\comments\else\fi
  \draftstamp}
\def\@draft{}

\makeatother

\input{kobib}

\usepackage{amssymb}
\usepackage{amsmath}
\usepackage{amstext}
\usepackage{mathrsfs}
\usepackage{euscript}
\usepackage[all]{xy}
\usepackage{xstring}
\usepackage{slashed}
\usepackage{tensor}

\newcounter{marke}
\stepcounter{marke}

\def\N {\mathbb{N}}

\def\R {\mathbb{R}}

\def\id{\mathrm{id}}

\def\subset{\subseteq}

\def\sep{\;|\;}

\renewcommand{\varepsilon}{\epsilon}
\renewcommand{\to}{\!\xymatrix@R=0cm@C=1.4em{\ar[r] &}}
\renewcommand{\mapsto}{\!\xymatrix@R=0cm@C=1.4em{\ar@{|->}[r] &}\!}
\renewcommand{\Rightarrow}{\!\xymatrix@R=0cm@C=1.4em{\ar@{=>}[r] &}\!}
\renewcommand{\Leftarrow}{\!\xymatrix@R=0cm@C=1.4em{\ar@{<=}[r] &}\!}
\newcommand{\incl}{\!\xymatrix@R=0cm@C=1.4em{\ar@{^(->}[r] &}\!}

\renewcommand\Leftrightarrow{\!\xymatrix@R=0cm@C=1.4em{\ar@{<=>}[r] &}\!}

\makeatletter
\renewenvironment{proof}[1][\nameProof]
  {\par\pushQED{\qed}%
   \normalfont \topsep6\p@\@plus6\p@\relax
   \trivlist
   \item[\hskip\labelsep
         \itshape
         #1\@addpunct{.}]
  \leavevmode}
  {\popQED\endtrivlist\@endpefalse}
\makeatother

\def\notebox#1#2{\begin{minipage}[b]{#1}\sloppy\renewcommand{\baselinestretch}{0.8}\footnotesize \begin{center}#2\end{center}\end{minipage}}

\def\mquad{\hspace{-2em}}

\def\erf#1{(\ref{#1})}

\def\stackref#1#2{\stackrel{\text{\ref{#1}}}{#2}}
\def\eqref#1{\stackref{#1}{=}}

\newlength{\myeqt} 
\newlength{\myeqs} 
\newlength{\myeqm} 
\newlength{\myeqn} 
\newcommand\eqtext[2][\myeqn]{\symtext[#1]{#2}{=}}
\newcommand\symtext[3][\myeqn]{
  \settowidth{\myeqt}{#2}
  \settowidth{\myeqs}{$#3$}
  \addtolength{\myeqs}{\the\myeqm}
  \ifdim\myeqt>\myeqs
    \stackrel{\hspace{-#1}\notebox{#1}{\medskip #2 \\ $\downarrow$\smallskip}\hspace{-#1}}{#3}
  \else
    \stackrel{\text{#2}}{#3}
  \fi}
\newcommand\eqcref[2][\myeqn]{\symcref[#1]{#2}{=}}
\newcommand\symcref[3][\myeqn]{\symtext[#1]{\cref{#2}}{#3}}

\def\brackets#1{\IfStrEq{#1}{-}{}{(#1)}}
\def\subindex#1{\IfStrEq{#1}{-}{}{_{#1}}}

\newcommand{\alxydim}[2]{\begin{aligned}\xymatrix#1{#2}\end{aligned}}



\newlength{\myl}

\def\ddt#1#2#3{\left.\frac{\mathrm{d}^{\IfStrEq{#1}{1}{}{#1}}}{\mathrm{d}#2}\IfStrEq{#2}{#3}{\right.}{\right|_{#3}}}

\def\lie{\mathcal{L}\!i\!e}

\def\grpd{\mathcal{G}\!rpd}

\def\liegrpd{\lie\grpd}

\def\fun{\mathcal{F}un}
\def\two{2\text{-}}

\def\hom{\mathcal{H}\!om}

\def\act#1#2{#1/\!\!/#2}
\def\idmorph#1{#1_{dis}}

\def\pr{{\mathrm{pr}}}

\newlength{\widthtmp}
\def\length#1{\settowidth{\widthtmp}{#1}\the\widthtmp}

\def\ttimes#1#2{\hspace{-0.15em}\tensor[_{#1}]{\times}{_{#2}}}

\def\buntech#1#2{\mathcal{B}\hspace{-0.01em}un_{\hspace{0.05em}#1}^{#2}}

\def\bun#1#2{\buntech{#1}{}\brackets{#2}}
\def\buncon#1#2{\buntech{#1}{\nabla}\hspace{-0.05em}\brackets{#2}}
\def\bunconmod#1#2#3{\buntech{#1}{\nabla_{\!#3}}\hspace{-0.05em}\brackets{#2}}

\def\zweibun#1#2{\two\bun{#1}{#2}}
\def\zweibuncon#1#2{\two\buncon{#1}{#2}}

\usepackage[latin1]{inputenc}
\usepackage[english]{babel}
\usepackage[english]{cleveref}

\def\refname{References}

\def\quot#1{``#1''}

\def\quand{\quad\text{ and }\quad}
\def\quomma{\quad\text{, }\quad}

\def\quith{\quad\text{ with }\quad}

\def\nameTheorem{Theorem}
\def\nameTheorems{Theorems}
\def\nameDefinition{Definition}
\def\nameDefinitions{Definitions}
\def\nameCorollary{Corollary}
\def\nameCorollaries{Corollaries}
\def\nameProposition{Proposition}
\def\namePropositions{Propositions}
\def\nameConstruction{Construction}
\def\nameConstructions{Constructions}
\def\nameLemma{Lemma}
\def\nameLemmas{Lemmas}
\def\nameRemark{Remark}
\def\nameRemarks{Remarks}
\def\nameProblem{Problem}
\def\nameProblems{Problems}
\def\nameExample{Example}
\def\nameExamples{Examples}
\def\nameExercise{Exercise}
\def\nameExercises{Exercises}

\def\nameProof{Proof}
\def\nameEquation{Eq.}
\def\nameEquations{Eqs.}
\def\nameSection{Section}
\def\nameSections{Sections}
\def\nameAppendix{Appendix}
\def\nameAppendixs{Appendices}

\input{theorems}

\def\ff{f\!f}

\def\con#1#2{\mathcal{C}\!on_{#1}\brackets{#2}}

\def\conff#1#2{\mathcal{C}\!on^{{f\!f}}_{#1}\brackets{#2}}

\def\PE{poe}
\def\SE{soe}

\def\inf#1{\EuScript{#1}}

\def\xyst{3em}

\def\zweibunconff#1#2{\two\bunconmod{#1}{#2}{f\!f}}

\def\eqendofproof{\vspace*{-1.9em}}

\def\fa#1{{#1}^{a}}
\def\fb#1{{#1}^{b}}
\def\fc#1{{#1}^{c}}
\def\tor#1{#1\text{-}\mathcal{T}\!\!or}

\def\ob#1{\mathrm{Obj}(#1)}
\def\mor#1{\mathrm{Mor}(#1)}
\def\1mor#1{1\text{-}\mathrm{Mor}(#1)}
\def\2mor#1{2\text{-}\mathrm{Mor}(#1)}

\def\refsmoothfunctor{\cite[Remarks 2.3.3 (a) \& 2.4.2 (b)]{Waldorf2016}}
\def\refdefanafunctor{\cite[Definition 2.4.1 (b)]{Waldorf2016}}
\def\refsmoothnattrans{\cite[Remark 2.4.2 (b)]{Waldorf2016}}
\def\reftransfunctoranafunctor{\cite[Remarks 2.3.3 (c) \& 2.4.2 (b)]{Waldorf2016}}
\def\reflietwoalgebravaluedforms{\cite[Section 4]{Waldorf2016}}
\def\refpullback{\cite[Sections 4.3 \& 5.2]{Waldorf2016}}
\def\refpullbackconnection{\cite[Proposition 5.2.12]{Waldorf2016}}
\def\refclassification{\cite[Theorem 5.3.4]{Waldorf2016}}
\def\refcanonicalpullback{\cite[Remark 5.2.10 (a) -- (c)]{Waldorf2016}}
\def\refcanonicalpullbackshift{\cite[Remark 5.2.10 (e) -- (g)]{Waldorf2016}}
\def\refpullbackoncomposition{\cite[Lemma 4.3.5 (a)]{Waldorf2016}}
\def\reftransitionspan{\cite[Lemma 3.1.6]{Waldorf2016}}
\def\refcompositionanafunctors{\cite[Remark 2.3.2 (a)]{Waldorf2016}}
\def\refmorgammaactiononcomposition{\cite[Remark 2.4.2 (a)]{Waldorf2016}}
\def\refactionfullyfaithful{\cite[Lemma 3.1.4]{Waldorf2016}}
\def\reftransitionspananafunctor{\cite[Lemma 3.1.9]{Waldorf2016}}

\def\refactiontwobundle{\cite[Example 5.1.11]{Waldorf2016}}
\def\refactiontwobundlereduction{\cite[Corollary 5.3.6]{Waldorf2016}}

\def\reftrivialbundle{\cite[Section 5.4]{Waldorf2016}}
\def\reffunctorL{\cite[Proposition 5.4.4]{Waldorf2016}}
\def\reffunctorLproperties{\cite[Propositions 5.4.6 \& 5.4.9]{Waldorf2016}}
\def\refpullbackform{\cite[Lemma 4.3.3]{Waldorf2016}} \def\refshiftedpullback{\cite[Remark 5.2.10 (e)]{Waldorf2016}}
\def\refdefshiftedpullbacktrivialbundle{\cite[Eq. 5.4.4]{Waldorf2016}}

\title{Parallel transport in principal 2-bundles}
\author{Konrad Waldorf}
\email{konrad.waldorf@uni-greifswald.de}

\adress{Ernst-Moritz-Arndt-Universität Greifswald\\
Institut für Mathematik und Informatik\\
Walther-Rathenau-Str. 47\\
D-17487 Greifswald}

\msc{}

\denseversion
\usepackage{amstext}
\usepackage{amssymb}
\usepackage{amsmath}

\begin{document}


\maketitle 

\begin{abstract}
A nice differential-geometric framework for (non-abelian) higher gauge theory is provided by principal 2-bundles, i.e. categorified principal bundles. Their total spaces are Lie groupoids,  local trivializations are  kinds of   Morita equivalences, and  connections are Lie-2-algebra-valued 1-forms.
  In this article, we construct explicitly the parallel transport of a connection on a principal 2-bundle. Parallel transport along a path is a Morita equivalence between the fibres over the end points, and  pa\-ral\-lel transport along a surface is an intertwiner between  Morita equivalences. We prove that our constructions fit into the general axiomatic framework for categorified parallel transport and surface holonomy.

\end{abstract}

\setcounter{tocdepth}{2}
\mytableofcontents


\setsecnumdepth{2}

\section{Introduction}

Many different concrete models for 2-bundles (sometimes called categorified bundles or gerbes) have been developed so far. For most of them, there exists a notion of a connection. For some of them, it is proved that there exists a corresponding parallel transport along paths and surfaces. However, to my best knowledge, in none of these models the relation between the connection and the parallel transport is concretely realized. The aim of the present paper is to fill this gap by constructing the parallel transport in one of these models: principal 2-bundles.

Let me try to clarify some of  above statements. First of all, our categorified bundles live over a smooth manifold $M$, and their structure group is a strict Lie 2-group.  Familiar models of  2-bundles with connection are (non-abelian) bundle gerbes \cite{murray,aschieri},   $G$-gerbes \cite{Laurent-Gengoux}, (non-abelian) differential cocycles \cite{breen1}, and  principal 2-bundles \cite{wockel1,pries2,Waldorf2016}.

In joint work with Urs Schreiber \cite{schreiber2}, based on earlier work of Baez-Schreiber \cite{baez2}, we have developed a model-independent, axiomatic framework for the parallel transport of connections in categorified bundles, called \quot{transport 2-functors}. Such a transport 2-functor is a 2-functor
\begin{equation*}
\mathrm{tra}: \mathcal{P}_2(M) \to \mathcal{C}\text{,}
\end{equation*} 
where $\mathcal{P}_2(M)$ is the path-2-groupoid of $M$ and $\mathcal{C}$ is some bicategory that depends on the model. The basic idea is that the objects of $\mathcal{P}_2(M)$  are the points $x \in\ M$, the morphisms are all smooth paths $\gamma$ in $M$, and the 2-morphisms are fixed-ends homotopies $\Sigma$ between paths (\quot{bigons}). Then, $\mathrm{tra}(x)$ is the \quot{fibre over $x$}, $\mathrm{tra}(\gamma)$ is the parallel transport along the path $\gamma$, and $\mathrm{tra}(\Sigma)$ is the parallel transport along the surface $\Sigma$. The axioms of a 2-functor describe how parallel transport behaves under gluing and cutting of paths and surfaces. The most difficult aspect of this framework is to axiomatically characterize smoothness conditions for the transport 2-functor. This has been worked out in  \cite{schreiber2}. It was proved there  that -- after picking particular bicategories $\mathcal{C}$  -- the bicategory of transport 2-functors is equivalent to several bicategories of above-mentioned models.

In all cases discussed in \cite{schreiber2}, these equivalences are given by spans of 2-functors which are in general not canonically invertible. This means, for instance, that not even for an abelian bundle gerbe with connection one associate in a canonical way a transport 2-functor. In particular, there is no clear answer to the question, what the parallel transport of such a bundle gerbe along a path is.  This is unsatisfying, in particular regarding the applications to higher gauge theory, where the parallel transport along a path constitutes the coupling to gauge fields.

In the present paper we consider the model of principal 2-bundles and provide a solution to this problem. Principal 2-bundles have been introduced by Wockel \cite{wockel1} and further worked out by Schommer-Pries \cite{pries2}. A principal 2-bundle is the most direct categorification of an ordinary principal bundle: its total space is a Lie groupoid on which a Lie 2-group $\Gamma$ acts in a certain way making it fibre-wise principal. Morphisms between principal 2-bundles -- in particular, local trivializations -- are not smooth functors between Lie groupoids but a generalization called \emph{anafunctor}, a kind of directed Morita equivalence.
Connections on principal 2-bundles have  recently been introduced in \cite{Waldorf2016}. We recall the central definitions in \cref{sec:2bundles}. The main part of this article is to construct the parallel transport of these connections:
\begin{enumerate}[(1)]

\item 
If $\gamma$ is a smooth path in $M$ starting at $x$ and ending at $y$, then we construct a $\Gamma$-equivariant anafunctor
$F_{\gamma}: \inf P_x \to \inf P_y$
between the fibres of $\inf P$ over these points. This  is the content of  \cref{sec:ptpaths}.    

\item
Suppose the connection is fake-flat. If $\Sigma$ is a smooth fixed-ends homotopy between paths $\gamma_1$ and $\gamma_2$, then we construct a $\Gamma$-equivariant transformation
$\varphi_{\Sigma}: F_{\gamma_1}\Rightarrow F_{\gamma_2}$
between the anafunctors associated to the two paths.
This is the content of \cref{sec:ptbigons}. 
\end{enumerate}
In principle, constructions (1) and (2) are performed in a  very similar way as for ordinary principal bundles. The basic idea is to  lift paths and homotopies \quot{horizontally} to the objects of the total space Lie groupoid. There are two main differences compared to ordinary principal bundles: horizontal lifts (i)  exist only locally and (ii) are not unique. Local existence requires to compensate  differences with structure on the morphisms of the total space Lie groupoid; this makes the whole construction more complex. Non-uniqueness requires to consider all possible horizontal lifts at one time; this forces us to consider anafunctors instead of ordinary functors. All these issues are carefully discussed and resolved in \cref{sec:ptpaths,sec:ptbigons}. The following is the main result of this article.

\begin{maintheorem}
\label{th:main}
Our constructions (1) and (2) of the parallel transport of a principal $\Gamma$-2-bundle fit into the axiomatic framework of transport 2-functors. This means:
\begin{enumerate}[(a)]

\item 
For every principal $\Gamma$-2-bundle $\inf P$ with fake-flat connection the assignments
$x \mapsto \inf P_x$, $\gamma\mapsto F_{\gamma}$, and $\Sigma \mapsto \varphi_{\Sigma}$
form a transport 2-functor
\begin{equation*}
\mathrm{tra}_{\inf P}:\mathcal{P}_2(M) \to \tor\Gamma
\end{equation*}
with target the bicategory of $\Gamma$-torsors.

\item
The assignment $\inf P \mapsto \mathrm{tra}_{\inf P}$ is  compatible with the bicategorical structure of principal $\Gamma$-2-bundles in the sense that it extends to a 2-functor
\begin{equation*}
\zweibunconff\Gamma M \to \mathrm{Trans}_{\Gamma}(M,\tor\Gamma)
\end{equation*}
between the bicategories of principal $\Gamma$-2-bundles with fake-flat connections and the bicategory of transport 2-functors.

\end{enumerate}
\end{maintheorem}

\cref{th:main} is proved in \cref{sec:pt2functor}  as \cref{th:2funct,th:pt2fun}.
We will show in a forthcoming paper that the 2-functor in (b) is actually an equivalence of bicategories. 
This means that  the model of principal $\Gamma$-2-bundles with connections comprises all aspects of categorified parallel transport.

This article is organized in a straightforward way. In \cref{sec:2bundles} we offer a short review about principal 2-bundles and connections. This review covers all material sufficient to understand the statements and constructions of this article. By intention, understanding all details of  the proofs  might require to consult \cite{Waldorf2016}. Therefore, all definitions, notations, and most symbols used in this article coincide with the corresponding ones in \cite{Waldorf2016}. 
  \cref{sec:ptpaths,sec:ptbigons} contain the constructions of the parallel transport, in \cref{sec:backcomp} we reduce these constructions to two important subclasses of principal 2-bundles, and  \cref{sec:pt2functor}  contains the proof of our main result. In an appendix we summarize and slightly extend results of \cite{schreiber5} about path-ordered and surface-ordered exponentials, which provide the \quot{local} foundations for parallel transport.

Admittedly, some constructions and proofs we perform in this article are quite  laborious. However, we believe that
our results -- once established --  provide a rather complete and convenient \quot{calculus} for categorified parallel transport in the well-established context of Lie groupoids.

\paragraph{Acknowledgements.} This work was supported by the German Research Foundation under project code WA 3300/1-1.

\setsecnumdepth{1}

\section{Principal 2-bundles}

\label{sec:2bundles}

We give a very short introduction to principal 2-bundles and connections. A comprehensive treatment is given in \cite{Waldorf2016}.
There is a bicategory $\liegrpd$ whose objects are \emph{Lie groupoids}, whose 1-morphisms are called \emph{anafunctors} (a.k.a. bibundles, Hilsum-Skandalis maps, Morita equivalences,...), and whose 2-morphisms are called \emph{transformations} (bibundle maps, intertwiners,...). Ordinary (smooth) functors form a proper subset among all anafunctors. Ordinary (smooth) natural transformations correspond to all transformations between functors.
The purpose of enlarging the set of 1-morphisms from functors to anafunctors is to invert certain functors (called \emph{weak equivalences}). One effect of this enlargement is that $\liegrpd$ is equivalent to the bicategory of differential stacks \cite{pronk}.

In this paper, a \emph{Lie 2-group} is a Lie groupoid whose objects and morphisms are equipped with Lie group structures, so that the structure maps are Lie group homomorphisms. Lie 2-groups are in one-to-one correspondence with \emph{crossed modules} of Lie groups. Often this version of a Lie 2-group is called \quot{strict}. 
A \emph{smooth right action} of a Lie 2-group $\Gamma$ on a Lie groupoid $\mathcal{X}$ is a smooth functor $R:\mathcal{X} \times \Gamma \to \mathcal{X}$ satisfying strictly the axioms of an action. 
Now, there is a new bicategory, whose objects are Lie groupoids equipped with smooth right $\Gamma$-actions, whose morphisms are $\Gamma$-equivariant anafunctors, and whose 2-morphisms are $\Gamma$-equivariant transformations. 

Finally, we fix the following conventions. If $X$ is a smooth manifold, we denote by $\idmorph{X}$ the Lie groupoid with objects $X$ and only identity morphisms. A smooth functor $\phi: \mathcal{X} \to \mathcal{Y}$ is called surjective/submersive, if it is so on the level on objects.

\begin{definition}
\label{def:zwoabun}
Let $M$ be a smooth manifold.
\begin{enumerate}[(a)]
\item 
A \emph{principal $\Gamma$-2-bundle over $M$} is a Lie groupoid $\inf P$, a smooth, surjective and submersive  functor $\pi: \inf P \to \idmorph M$,  and a smooth right action $R: \Gamma \times \inf P \to \inf P$ such that $\pi \circ R = \pi \circ \pr_1$ and the smooth functor
$(\mathrm{pr}_1, R) : \inf P \times \Gamma \to \inf P \times_M \inf P$
is a weak equivalence.

\item
A \emph{1-morphism} between principal $\Gamma$-2-bundles is a  $\Gamma$-equivariant anafunctor 
$J: \inf P_1 \to \inf P_2$
such that $\pi_2 \circ F =\pi_1$.

\item
A \emph{2-morphism} between 1-morphisms is a $\Gamma$-equivariant transformation.

\end{enumerate}
\end{definition}

Principal $\Gamma$-2-bundles over $M$ form a bigroupoid that we denote by $\zweibun\Gamma M$. Moreover, the assignment
$M \mapsto \zweibun \Gamma M$
is a stack  over the site of smooth  manifolds \cite[Theorem 6.2.1]{Nikolaus}.

\begin{remark}
We describe some notation and technical features related to our 1-morphisms, which will be used later in the paper. Let $\inf P_1$ and $\inf P_2$ be principal $\Gamma$-2-bundles over $M$.

\begin{enumerate}[(a)]

\item 
\label{rem:anafunctor:a}
The anafunctor underlying a 1-morphism $J:\inf P_1 \to \inf P_2$ consists of a \emph{total space} $J$, \emph{anchor maps} $\alpha_l: J \to \inf P_1$ and $\alpha_r: J \to \inf P_2$, and commuting smooth groupoid actions $\rho_l:\mor{\inf P_1} \ttimes{s}{\alpha_l} J \to J$ and $\rho_r:J \ttimes{\alpha_r}{t} \mor{\inf P_2} \to J$, which we will often denote by $\rho \circ j$ and $j \circ \rho$, respectively. Its $\Gamma$-equivariance consists of a  smooth right action $\rho:J \times \mor{\Gamma} \to J$, usually denoted by $j \cdot \gamma$, that is compatible with the groupoid actions in the sense that 
\begin{equation*}
R(\rho_1,\gamma_1) \circ (j \cdot \gamma) \circ R(\rho_2,\gamma_2) = (\rho_1 \circ j \circ \rho_2) \cdot (\gamma_1\circ \gamma\circ \gamma_2)
\end{equation*}
whenever all compositions are defined, see \refdefanafunctor.

\item
\label{rem:anafunctor:b}
If $\phi: \inf P_1 \to \inf P_2$ is a smooth functor that preserves the fibres and strictly commutes with the $\Gamma$-actions, then it induces a 1-morphism with total space  $J_{\phi} := \ob{\inf P_1} \ttimes{\phi}{t} \mor{\inf P_2}$, anchors  $\alpha_l(p,\rho) := p$ and $\alpha_r(p,\rho) := s(\rho)$, groupoid  actions  $\eta \circ (p,\rho) := (t(\eta),\phi(\eta) \circ \rho)$ and $(p,\rho) \circ \eta := (p, \rho \circ \eta)$, and $\mor{\Gamma}$-action $(p,\rho) \cdot\gamma :=  (R(p,t(\gamma)),R(\rho,\gamma))$, see \refsmoothfunctor.  
 
\item
\label{rem:anafunctor:c}
A smooth natural transformation $\eta:\phi\Rightarrow \phi'$ induces a transformation $f_{\eta}:J_{\phi} \Rightarrow J_{\phi'}$ by $f_{\eta}(p,\rho) := (x,\eta(p) \circ \rho)$. If $\eta$  is $\Gamma$-equivariant then $f_{\eta}$ is also $\Gamma$-equivariant, hence a 2-morphism, see \refsmoothnattrans.

\item
\label{rem:indtrans}
Let $\phi: \inf P_1 \to \inf P_2$ be a smooth, fibre-preserving, $\Gamma$-equivariant functor, and let $J:\inf P_1 \to \inf P_2$ be a 1-morphism. For a smooth map $\tilde f: \ob{\inf P_1} \to J$ we consider three conditions:
\begin{enumerate}[(T1)]

\item 
\label{T1}
$\alpha_l(\tilde f_{\gamma}(p)) = p$ and $\alpha_r (\tilde f_{\gamma}(p)) =  \phi(p)$

\item
\label{T2}
$\alpha\circ \tilde f_{\gamma}(p)\circ \beta=\tilde f_{\gamma}(t(\alpha))\circ \phi(\alpha)\circ \beta$

\item
\label{T3}
$\tilde f_{\gamma}(R(p,g))=\tilde f_{\gamma}(p)\cdot \id_g$.

\end{enumerate}
There is a bijection between smooth maps $\tilde f$ satisfying \cref{T1*}, \cref{T2*} and \cref{T3*} and 2-morphisms $f: J_{\phi} \Rightarrow J$. This bijection is established by the relation
$\tilde f(p)=f(p,\phi(\id_p))$, see \reftransfunctoranafunctor.

\end{enumerate}
\end{remark}

Next we come to connections. 
If $\mathcal{X}$ is a Lie groupoid and $\gamma$ is a Lie 2-algebra, then there is a differential graded-commutative Lie algebra $\Omega^{*}(\mathcal{X},\gamma)$ of \emph{$\gamma$-valued differential forms on $\mathcal{X}$} \reflietwoalgebravaluedforms. If $\phi:\mathcal{X} \to \mathcal{Y}$ is a smooth functor, then there is a \quot{pullback} Lie algebra homomorphism $\phi^{*}:\Omega^{*}(\mathcal{Y},\gamma) \to \Omega^{*}(\mathcal{X},\gamma)$. If $\gamma$ is the Lie 2-algebra of a Lie 2-group $\Gamma$, then there is an \emph{adjoint action} of $\Gamma$ on $\Omega^{*}(\mathcal{X},\gamma)$.
Further, $\Gamma$ carries a \quot{Maurer-Cartan}-form $\Theta\in\Omega^1(\Gamma,\gamma)$.

\begin{definition}
\label{def:connection}
A \emph{connection} on a principal $\Gamma$-2-bundle $\inf P$ is a $\gamma$-valued 1-form $\Omega \in \Omega^1(\inf P,\gamma)$ such that
\begin{equation*}
R^{*}\Omega = \mathrm{Ad}_{\pr_\Gamma}^{-1}(\pr_\inf P^{*}\Omega) + \pr_\Gamma^{*}\Theta
\end{equation*}
over $\inf P \times \Gamma$, where $\pr_{\inf P}$ and $\pr_{\Gamma}$ are the projections to the two factors.
\end{definition}

Let us spell out explicitly all structure and conditions that are packed into \cref{def:connection}. For this purpose, we assume that the Lie 2-group $\Gamma$ is given as a crossed module $(G,H,t,\alpha)$, where $t:H \to G$ is the Lie group homomorphism, and $\alpha:G \times H \to H$ is the action of $G$ on $H$. We will   denote by  $\alpha_g \in \mathrm{Aut}(H)$  the action of a fixed $g\in G$ on $H$, and for $h \in H$ we denote by $\tilde \alpha_h: G \to H$ the map defined by $\tilde \alpha_h(g) := h^{-1}\alpha(g,h)$.  The correspondence between $\Gamma$ and $(G,H,t,\alpha)$ is $\ob{\Gamma}=G$ and $\mor{\Gamma}=H \ltimes_{\alpha} G$, with $s(h,g)=g$ and $t(h,g)=t(h)g$. The associated Lie 2-algebra is the crossed module $(\mathfrak{g},\mathfrak{h},t_{*},\alpha_{*})$, where $\mathfrak{g}$ and $\mathfrak{h}$ are the Lie algebras of $G$ and $H$, respectively, and $t_{*}$ and $\alpha_{*}$ are the differentials of $t$ and $\alpha$. Throughout the whole paper we will work in exactly this setting of crossed modules. We  point to a formulary for calculations collected in \cite[Appendix A]{Waldorf2016}, which we will eventually use without further mentioning.

Now, a connection $\Omega$  on a principal $\Gamma$-2-bundle $\inf P$ consists of three components $\Omega=(\fa\Omega,\fb\Omega,\fc\Omega)$, which are ordinary differential forms:
\begin{equation*}
\fa\Omega\in\Omega^1(\ob{\inf P},\mathfrak{g})
\quomma 
\fb\Omega\in\Omega^1(\mor{\inf P},\mathfrak{h})
\quand
\fc\Omega\in\Omega^2(\ob{\inf P},\mathfrak{h})\text{.}
\end{equation*}
These satisfy the following conditions:
\begin{align}
R^{*}\fa\Omega &= \mathrm{Ad}_{g}^{-1}(p^{*}\fa\Omega) + g^{*}\theta &&\text{over }\ob{\inf P}\times \ob{\Gamma} 
\label{eq:conform:a}
\\
R^{*}\fb\Omega 
&=(\alpha_{g^{-1}})_{*}\left (\mathrm{Ad}_{h}^{-1}(p^{*}\fb\Omega)+(\tilde \alpha_{h})_{*}(p^{*}s^{*}\fa\Omega) +h^{*}\theta \right)\hspace{-1em} &&\text{over }\mor{\inf P}\times \mor{\Gamma}
\label{eq:conform:b}
\\
\label{eq:conform:c}
R^{*}\fc\Omega &= (\alpha_{g^{-1}})_{*}(p^{*}\fc\Omega)
&&\text{over }\ob{\inf P}\times \ob{\Gamma}
\text{.}\end{align}
Here, $p$, $g$ and $h$ denote the projections to either $\ob{\inf P}$ or $\mor{\inf P}$, $G$ and $H$, respectively.

The 2-form
$\mathrm{curv}(\Omega) := \mathrm{D}\Omega + \frac{1}{2}[\Omega\wedge \Omega] \in \Omega^2(\inf P,\gamma)$
is called the \emph{curvature} of $\Omega$. The connection $\Omega$ is called \emph{flat} if $\mathrm{curv}(\Omega)=0$.
Between general connections and flat connections  are \emph{fake-flat} connections: these  satisfy the conditions (with $\Delta := t^{*}-s^{*}$)
\begin{align*}
\mathrm{d}\fa\Omega +\frac{1}{2}[\fa\Omega \wedge \fa\Omega] +t_{*}(\fc\Omega)&=0
\quand
\Delta\fc\Omega +\mathrm{d}\fb\Omega+\frac{1}{2}[\fb\Omega\wedge \fb\Omega]+ \alpha_{*}(s^{*}\fa\Omega \wedge \fb\Omega)=0\text{.}
\end{align*}

If $J:\inf P_1 \to \inf P_2$ is a 1-morphism, then pulling back a connection $\Omega_2$ on $\inf P_2$ to $\inf P_1$ requires the following additional structure on $J$, as explained in \refpullback.

\begin{definition}
\label{def:pullback:a}
An \emph{$\Omega_2$-pullback} on a 1-morphism $J:\inf P_1 \to \inf P_2$ is a pair $\nu=(\nu_0,\nu_1)$ of differential forms $\nu_0\in\Omega^1(J,\mathfrak{h})$ and $\nu_1\in \Omega^2(J,\mathfrak{h})$ which are compatible with the $\inf P_2$-action $\rho_r$ in the sense that
\begin{align*}
\rho_r^{*}\nu_0= \pr_J^{*}\nu_0+ \pr_{\mor{\inf P_2}}^{*}\fb\Omega_2 
\quand
\rho_r^{*}\nu_1 = \pr_J^{*}\nu_1 + \pr_{\mor{\inf P_2}}^{*}\Delta\fc\Omega_2 
\end{align*}
over $J \ttimes{\alpha_r}{t} \mor{\inf P_2}$. An $\Omega_2$-pullback is called:
\begin{enumerate}[(a)]

\item
\label{def:pullback:b}
\emph{connective}, if it is compatible with the $\mor{\Gamma}$-action $\rho$ in the sense that
\begin{align*}
\rho^{*}\nu_0 &= (\alpha_{g^{-1}})_{*}\left (\mathrm{Ad}_{h}^{-1}(\pr_J^{*}\nu_0)+(\tilde\alpha_{h})_{*}(\pr_J^{*}\alpha_r^{*}\fa{\Omega}_2) +h^{*}\theta \right) 
\\
\rho^{*}\nu_1 &=  (\alpha_{g^{-1}})_{*} \big ( \mathrm{Ad}_h^{-1} (\pr_J^{*}\nu_1)+(\tilde \alpha_h)_{*}(t_{*} (\pr_J^{*}\alpha_r^{*}\fc\Omega_2)) \big )
\end{align*}
over $J \times \mor{\Gamma}$, where $g$ and $h$ are the projections to the factors of $\mor{\Gamma}=H \ltimes G$.

\item
 \emph{fake-flat}, if $\mathrm{d}\nu_0+ \frac{1}{2}[\nu_0\wedge \nu_0]+ \alpha_{*}(\alpha_r^{*}\fa\Omega_2 \wedge \nu_0)+\nu_1=0$. 
\end{enumerate}
\end{definition}

Given an $\Omega_2$-pullback $\nu$ on $\inf P_2$, one can define a 1-form $J_{\nu}^{*}\Omega_2$ on $\inf P_1$ that depends on the choice of $\nu$. If $\nu$ is connective, then $J_{\nu}^{*}\Omega_2$ is a connection on $\inf P_1$, and if $\Omega_2$ and $\nu$ are fake-flat, then $J_{\nu}^{*}\Omega_2$ is fake-flat (\refpullbackconnection). If a connection $\Omega_1$ on $\inf P_1$ is given, then we say that $\nu$ is \emph{connection-preserving} if $\Omega_1=J_{\nu}^{*}\Omega_2$.

A 2-morphism $f: J \Rightarrow J'$ between 1-morphisms $J,J':\inf P_1 \to \inf P_2$ equipped with $\Omega_2$-pullbacks $\nu$ and $\nu$, respectively, is called \emph{connection-preserving} if $\nu=f^{*}\nu'$. We form two bicategories of principal $\Gamma$-2-bundles with connection:
\begin{itemize}

\item 
A bicategory $\zweibuncon\Gamma M$ consisting of principal $\Gamma$-2-bundles with connections,  1-morphisms  with  connective, connection-preserving pullbacks, and connection-preserving 2-morphisms. 

\item
A bicategory $\zweibunconff\Gamma M$ consisting of principal $\Gamma$-2-bundles with fake-flat connections,  1-morphisms with fake-flat,  connective, connection-preserving pullbacks, and connection-preserving 2-morphisms. 

\end{itemize}
There is a classification result showing that these bicategories correspond to  non-abelian differential cohomology \refclassification. Moreover, it is straightforward to see that they form presheaves of bicategories over the category of smooth manifolds, i.e., there are consistent pullback 2-functors along smooth maps.

\begin{remark}
\label{rem:smoothfunctorconnection}
We describe how smooth functors can be turned into 1-morphisms in the setting with connections.  
Suppose $\phi: \inf P_1 \to \inf P_2$ is a fibre-preserving, $\Gamma$-equivariant smooth functor between principal $\Gamma$-2-bundles equipped with connections $\Omega_1$ and $\Omega_2$, respectively.  Let $J_{\phi}=\ob{\inf P_1} \ttimes{\phi}{t} \mor{\inf P_2}$ be the associated anafunctor (\cref{rem:anafunctor:b}).
\begin{enumerate}[(a)]

\item 
\label{rem:smoothfunctorconnection:a}
A \quot{canonical} $\Omega_2$-pullback on  $J_{\phi}$ is defined by   $\nu_0 := \pr_{2}^{*}\fb\Omega_2$ and $\nu_1 :=- \pr_{2}^{*}s^{*}\fc\Omega_2 + \pr_{1}^{*}\phi^{*}\fc\Omega_2$. It is always connective,  fake-flat if $\Omega_2$ is fake-flat, and connection-preserving if $\Omega_1=\phi^{*}\Omega_2$. See \refcanonicalpullback.

\item
\label{rem:smoothfunctorconnection:b}
The canonical $\Omega_2$-pullback $\nu$ on $J_{\phi}$ can be shifted by a pair $\kappa=(\kappa_0,\kappa_1)$ of differential forms $\kappa_0\in \Omega^1(\ob{\inf P_1},\mathfrak{h})$ and $\kappa_1\in \Omega^2(\ob{\inf P_1},\mathfrak{h})$, and the shifted pullback is again connective provided that these forms are $G$-equivariant in the sense that    $R^{*}\kappa_i = (\alpha_{\pr_2^{-1}})_{*}(\pr_{1}^{*}\kappa_i)$ over $\ob{\inf P_1} \times G$. See \refcanonicalpullbackshift.  

\end{enumerate}
\end{remark}

\setsecnumdepth 2

\section{Parallel transport along paths}

\label{sec:ptpaths}

Let $\inf P$ be a principal $\Gamma$-bundle with a connection $\Omega$. For $x\in M$ we denote by $\inf P_x := \pi^{-1}(\{x\})$ the fibre of $\inf P$ over $x$, which is a Lie groupoid with  smooth right $\Gamma$-action. In this section we define for each path $\gamma:[0,1] \to M$  a $\Gamma$-equivariant anafunctor 
\begin{equation*}
F_{\gamma}: \inf P_{\gamma(0)} \to \inf P_{\gamma(1)}\text{,}
\end{equation*}  
which we regard as the parallel transport along $\gamma$.
For this purpose, we first introduce and study in \cref{sec:horizontality} the notion of a horizontal path in the total space of $\inf P$. In \cref{sec:defpartranspaths} we give a complete definition of the anafunctor $F_{\gamma}$. In \cref{sec:compcomppaths,sec:natbundlemorph,sec:natpullbackpaths} we derive several properties of $F_{\gamma}$ with respect to path composition, 1-morphisms between principal 2-bundles, and pullback.

\subsection{Horizontal paths}

\label{sec:horizontality}

We start with some basic terminology and notation. By a \emph{path} in a smooth manifold  $X$ we understand a smooth map $\gamma:[a,b] \to X$, where $a,b\in \R$ with $a<b$. If $x:=\gamma(a)$ and $y:=\gamma(b)$, we use the notation $\gamma:x \to y$. If no interval is specified, then the unit interval $[0,1]$ is assumed. The tangent vector at $t\in [a,b]$ is denoted by $\dot\gamma(t)$ or $\partial_t\gamma(t)$. The constant path at a point $x\in X$ will be denoted by $x$ or $\id_x$. If $f:X \to  Y$ is a smooth map, we  write $f(\gamma)$ for the path $f\circ \gamma$. Further, if $R:\mathcal{X} \times \Gamma \to \mathcal{X}$ is a right action, we will write $R(\rho,g)$ instead of $R(\rho,\id_{g})$, for $\rho\in \mor{\inf X}$ and $g\in G$. For instance, if $\beta$ is a path in $\ob{\inf X}$ and $g$ is a path in $G$, then $R(\beta,g)$ stands for the path $t \mapsto R(\beta(t),\id_{g(t)})$. 

First we discuss horizontality for paths in the objects a principal 2-bundle $\inf P$ with connection $\Omega$.
A path $\beta:[a,b] \to \ob{\inf P}$ is  \emph{horizontal}, if 
$\fa\Omega(\dot\beta(t))=0$ for all $t\in [a,b]$. 

\begin{proposition}
Let $\inf P$ be a principal $\Gamma$-2-bundle with connection $\Omega$.
\begin{enumerate}[(a)]

\item 
\label{lem:obhorexists}
Suppose $\beta:[a,b] \to \ob{\inf P}$ is a path. Then, there exists a unique path $g:[a,b] \to G$ with $g(a)=1$ such that  $\beta^{hor} := R(\beta,g)$ is horizontal. \begin{comment}
In particular, $\beta^{hor}(a)=\beta(a)$ and $\pi(\beta(t))=\pi(\beta^{hor}(t))$.
\end{comment}

\item
\label{lem:obhor}
Suppose $\beta:[a,b]\to\ob{\inf P}$ is a horizontal path and $g\in G$. Then, $R(\beta,g)$ is horizontal.

\end{enumerate}
\end{proposition}

\begin{proof}
For \cref{lem:obhorexists*} we claim that the following statements are equivalent:
\begin{enumerate}[(1)]

\item 
$g$ is a solution of the differential equation
$\dot g(\tau) = - \fa\Omega(\dot\beta(\tau))g(\tau)$.

\item
$\beta^{hor}=R(\beta,g)$ is horizontal.

\end{enumerate}
Equivalence is proved by following the calculation using  \cref{eq:conform:a}: 
\begin{equation*}
\fa\Omega(\dot\beta^{hor})=\fa\Omega(\partial_t R(\beta,g) )=R^{*}\fa\Omega(\dot\beta,\dot g)=\mathrm{Ad}_{g}^{-1}(\fa\Omega(\dot\beta)) + g^{-1}\dot g\text{.}
\end{equation*}
Now, existence and uniqueness of $g$ follow from existence and uniqueness of solutions of  linear initial value problems.
\cref{lem:obhor*} follows immediately from the transformation behaviour of $\fa\Omega$, see \cref{eq:conform:a}. 
\end{proof}

Next we turn to paths in the morphisms of $\inf P$, and collect various statements that we will use throughout this article.
A path $\rho:[a,b] \to \mor{\inf P}$ is  \emph{horizontal}, if 
$\fb\Omega(\dot\rho(t))=0$ 
for all $t\in [a,b]$.

\begin{proposition}
Let $\inf P$ be a principal $\Gamma$-2-bundle with connection $\Omega$.
\begin{enumerate}[(a)]

\item 
\label{lem:makemorhor}
Suppose $\rho:[a,b] \to \mor{\inf P}$ is a path. Then, there exists a unique path $h:[a,b] \to\ H$ with $h(a)=1$ such $\rho^{hor} := R(\rho,(h,1))$ is horizontal. 
\begin{comment}
In particular, we have $s(\rho^{hor})=s(\rho)$ and $t(\rho^{hor})=R(t(\rho),t(h))$.
\end{comment}

\item 
\label{lem:hormor:a}
Suppose $\beta:[a,b] \to \ob{\inf P}$ is a path. Then, the path $\id_{\beta}$ in $\mor{\inf P}$ is horizontal.

\item
\label{lem:hormor:b}
Suppose $\rho:[a,b] \to \mor{\inf P}$ is a path. Then, $\rho$ is horizontal if and only if  its pointwise groupoid inversion $\rho^{-1}$ is horizontal.

\item
\label{lem:hormor:d}
Suppose $\rho_1,\rho_2:[a,b] \to \mor{\inf P}$ are horizontal paths with $s(\rho_2)=t(\rho_1)$. Then, their pointwise composition $\rho_2 \circ \rho_1$ is horizontal.

\item
\label{lem:hormor:h}
Suppose $\rho:[a,b] \to \mor{\inf P}$ is a horizontal path and $\gamma \in \mor{\Gamma}$. Then, $R(\rho,\gamma)$ is horizontal.

\item
\label{lem:hormor:c}
Suppose  $\rho:[a,b] \to \mor{\inf P}$ is a  horizontal path and $g:[a,b] \to G$ is a path. Then, $R(\rho,g)$ is horizontal. 

\item
\label{lem:hormor:e}
Suppose $\rho:[a,b] \to \mor{\inf P}$ is horizontal, and of the paths $s(\rho)$ and $t(\rho)$ in $\ob{\inf P}$ one is horizontal. Then, the other is horizontal, too. 

\item
\label{lem:hormor:f}
Suppose $\beta_1,\beta_2:[a,b] \to \ob{\inf P}$ are horizontal, and $\pi \circ \beta_1=\pi \circ \beta_2$. Then, there exists a horizontal path $\rho:[a,b] \to \mor{\inf P}$ and $g\in G$ such that $\beta_1=R(s(\rho),g^{-1})$ and $\beta_2=t(\rho)$.
Moreover, if there is $\rho_0\in \mor{\inf P}$ such that $s(\rho_0)=\beta_1(a)$ and $t(\rho_0)=\beta_2(a)$, then one can choose $\rho$ and $g$ such that $\rho(a)=\rho_0$ and $g=1$.

\item
\label{lem:hormor:g}
Suppose $\rho,\rho':[a,b] \to \mor{\inf P}$ are horizontal paths such that $s(\rho)=s(\rho')$ is horizontal and $t(\rho)=t(\rho')$. Then, there exists a unique $h\in H$ with $t(h)=1$ and $\rho'=R(\rho,(h,1))$.

\end{enumerate}
\end{proposition}

\begin{proof}
For \cref{lem:makemorhor*} we claim that the following statements are equivalent:
\begin{enumerate}[(1)]

\item 
$h$ is a solution of the differential equation
$\dot h(\tau) = -\fb\Omega(\dot \rho(\tau))h(\tau)-( \alpha_{h(\tau)})_{*}(\fa\Omega(s_{*}(\dot \rho(\tau))))$.

\item
$\rho^{hor}$ is horizontal.

\end{enumerate}
Equivalence is proved by the following equation obtained  using \cref{eq:conform:b},
\begin{equation*}
\fb\Omega(\dot\rho^{hor})= \mathrm{Ad}_{h}^{-1}(\fb\Omega(\dot \rho))+(\tilde \alpha_{h})_{*}(\fa\Omega(s(\dot \rho))) +h^{-1}\dot h \text{.}
\end{equation*}
\begin{comment}
Indeed,
\begin{multline*}
\fb\Omega(\dot\rho^{hor}(\tau))=\fb\Omega(\partial_t  R(\rho(t),(h(t),1)) )=R^{*}\fb\Omega(\dot\rho(\tau),(\dot h(\tau),0))\\= \mathrm{Ad}_{h(\tau)}^{-1}(\fb\Omega(\dot \rho(\tau)))+(\tilde \alpha_{h(\tau)})_{*}(\fa\Omega(s_{*}(\dot \rho(\tau)))) +h(\tau)^{-1}\dot h(\tau) 
\end{multline*}
\end{comment}
Now, existence and uniqueness follow like in the proof of \cref{lem:obhorexists}.
For \cref{lem:hormor:a*} we have
$\fb\Omega(\partial_t\id_{\beta})=\id^{*}\fb\Omega(\dot\beta)=0$ 
since $\id^{*}\fb\Omega=0$. For \cref{lem:hormor:b*} we use $inv^{*}\fb\Omega=-\fb\Omega$, and for \cref{lem:hormor:d*}  $c^{*}\fb\Omega = \pr_1^{*}\fb\Omega + \pr_2^{*}\fb\Omega$.  \cref{lem:hormor:h*} is trivial. 
For \cref{lem:hormor:c*} we check 
\begin{equation*}
\fb\Omega(\partial_t  R(\rho,(1,g)) )=R^{*}\fb\Omega(\dot\rho,(0,\dot g))=(\alpha_{g})_{*}(\fb\Omega(\dot \rho)) =0\text{.}
\end{equation*}
For \cref{lem:hormor:e*} we use that $t^{*}\fa\Omega - s^{*}\fa\Omega = t_{*}(\fb\Omega)$. Since $\rho$ is horizontal, we have
\begin{equation*}
0=t_{*}(\fb\Omega(\dot\rho))=\fa\Omega(\partial_t  s(\rho))-\fa\Omega(\partial_t  t(\rho))\text{.}
\end{equation*}
For \cref{lem:hormor:f*} we note that $(\beta_1,\beta_2)$ is a path in $\ob{\inf P} \times_M \ob{\inf P}$. By \reftransitionspan\  there exists a transition span $\rho:[a,b] \to \mor{\inf P}$ along $(\beta_1,\beta_2)$ with transition function $g:[a,b] \to G$, i.e.  $\beta_1=R(s(\rho),g^{-1})$ and $\beta_2=t(\rho)$. If there is $\rho_0\in \mor{\inf P}$ such that $s(\rho_0)=\beta_1(a)$ and $t(\rho_0)=\beta_2(a)$, then  by \refactionfullyfaithful\ there exists $h_0\in H$ such that $R(\rho(a),(h_0,g(a)^{-1}))=\rho_0$ and $t(h_0)=g(a)$. Then we use $\tilde\rho:=R(\rho,(h_0,t(h_0)^{-1}))$ and $\tilde g := gt(h_0)^{-1}$, satisfying $R(s(\tilde\rho),\tilde g^{-1})=R(s(\rho),g^{-1})=\beta_1$ and $t(\tilde\rho)=t(\rho)=\beta_2$, as well as $\tilde\rho(a)=\rho_0$ and $\tilde g(a)=1$.
By \cref{lem:makemorhor*} there exists $h:[a,b] \to H$ with $h(a)=1$ such that $R(\rho,(h,1))$ is horizontal. Then, by \cref{lem:hormor:c*} also $\rho' := R(\rho,(h,t(h)^{-1}))$ is horizontal, and $t(\rho')=t(\rho)=\beta_2$. We set $g':= g t(h)^{-1}$. Then, $R(s(\rho'),g'^{-1})=R(s(\rho),t(h)^{-1}g'^{-1})=R(s(\rho),g^{-1})=\beta_1$. Now, by \cref{lem:hormor:e*}, it follows that $s(\rho')$ is horizontal. A short calculation shows that $0=\fa\Omega(\dot\beta_1)=-\dot g'g'^{-1}$; hence $g'$ is constant.
\begin{comment}
We have
\begin{align*}
0=\fa\Omega(\dot\beta_1(t))&=R^{*}\fa\Omega(\partial_t  s(\rho'(\tau)),\partial_t  g'(\tau)^{-1})
\\&=\mathrm{Ad}_{g'(t)}(\fa\Omega(\partial_t  s(\rho'(\tau))))-\theta(\partial_t  g'(\tau)^{-1})
\\&=\theta(\partial_t g'(\tau)^{-1})
\end{align*}
\end{comment}
For
\cref{lem:hormor:g*} we obtain by \refactionfullyfaithful\ a smooth map $h:[a,b] \to H$ with $t(h)=1$ and $\rho' = R(\rho,(h,1))$. Again, a short calculation shows $0=\fb\Omega(\dot\rho')=h^{-1}\dot h$; hence $h$ is constant.
\begin{comment}
We have
\begin{align*}
0=\fb\Omega(\dot\rho'(\tau)) &=\fb\Omega(\partial_t R(\rho(t),(h(t),1)) )
\\&=R^{*}\fb\Omega(\dot\rho(\tau),(\dot h(\tau),1))
\\&=\mathrm{Ad}_{h(\tau)}^{-1}(\fb\Omega(\dot \rho(\tau)))+(\tilde \alpha_{h(\tau)})_{*}(\fa\Omega(s_{*}(\dot \rho(\tau)))) +h(\tau)^{-1}\dot h(\tau) =0
\\&=h(\tau)^{-1}\dot h(\tau) =0\text{.}
\end{align*}
\end{comment} 
\end{proof}

Finally, we consider a 1-morphism $J:\inf P_1 \to \inf P_2$ in $\zweibuncon\Gamma M$ between principal $\Gamma$-2-bundles over $M$, connections $\Omega_1$ and $\Omega_2$ on $\inf P_1$ and $\inf P_2$, respectively, and a   connective, connection-preserving $\Omega_2$-pullback  $\nu=(\nu_0,\nu_1)$  on $J$. 
A path $\lambda:[a,b] \to J$ is  \emph{horizontal}, if $\nu_0(\dot\lambda(t))=0$ for all $t\in [a,b]$.

\begin{remark}
\label{rem:hor:functor}
Suppose $\phi:\inf P_1 \to \inf P_2$ is a smooth functor, $J_{\phi}=\ob{\inf P_1} \ttimes{\phi}{t} \mor{\inf P_2}$ is the associated anafunctor (\cref{rem:anafunctor:b}), and $\nu$ is the canonical $\Omega_2$-pullback on $J_{\phi}$ (\cref{rem:smoothfunctorconnection}). Then, a path $\lambda=(\gamma,\rho)$ in $J_{\phi}$ is horizontal if and only if  $\rho$ is horizontal in $\mor{\inf P_2}$.
\begin{comment}
Indeed,
\begin{equation*}
0=\nu_0(\dot\lambda(t))=\pr_2^{*}\fb\Omega_2(\dot\gamma(t),\dot\rho(t))=\fb\Omega_2(\dot\rho(t))
\end{equation*}
\end{comment}
If $\kappa=(\kappa_0,\kappa_1)$ shifts the canonical $\Omega_2$-pullback, then $\lambda$ is horizontal if and only if $\fb\Omega_2(\dot\rho)+\kappa_0(\dot\gamma)=0$.
\begin{comment}
Indeed,
\begin{equation*}
0=\nu_0^{\kappa}(\dot\lambda(t))=(\pr_1^{*}\kappa_0+\pr_2^{*}\fb\Omega_2)(\dot\gamma(t),\dot\rho(t))=\kappa_0(\dot\gamma(t))+\fb\Omega_2(\dot\rho(t))
\end{equation*}
\end{comment}
\end{remark}

\begin{proposition}
Let  $J:\inf P_1 \to \inf P_2$ be a 1-morphism in $\zweibuncon\Gamma M$. 
\begin{enumerate}[(a)]

\item
\label{lem:horF:b}
Suppose $\lambda:[a,b] \to J$ is horizontal, and $\gamma:[a,b] \to G$ is a path. Then, $\lambda \cdot \id_{\gamma}$ is horizontal.

\item 
\label{lem:horF:a}
Suppose $\lambda:[a,b] \to J$ is horizontal, and of the paths $\alpha_l(\lambda)$ and $\alpha_r(\lambda)$ one is horizontal. Then, the other is horizontal, too. 

\item
\label{lem:horF:c}
Suppose $\lambda:[a,b] \to J$ and $\rho:[a,b] \to \mor{\inf P_2}$ are paths with $\alpha_r(\lambda)=t(\rho)$. If of the three paths $\lambda$, $\rho$, and  $\lambda \circ  \rho$ two are horizontal, then the third is horizontal, too.

\item
\label{lem:horF:d}
Suppose $\lambda:[a,b] \to J$ and $\rho:[a,b] \to \mor{\inf P_1}$ are paths with $s(\rho) =\alpha_l(\lambda)$. If of the three paths $\lambda$, $\rho$, and  $\rho \circ\lambda$ two are horizontal, then the third is horizontal, too.

\item
\label{lem:horF}
Suppose $\lambda:[a,b] \to J$ is a path. Then, there exists a unique path $h:[a,b] \to H$ with $h(a)=1$ such that $\lambda_1(t) := \lambda(t) \cdot (h(t),1)$ and $\lambda_2(t) := \lambda(t) \cdot (h(t),t(h)^{-1}) $  are horizontal.

\end{enumerate}
\end{proposition}

\begin{proof}
\cref{lem:horF:b*} follows since $\nu$ is connective.
\cref{lem:horF:a*} is exactly as \cref{lem:hormor:e}, using that $\nu$ is connection-preserving, which implies $t_{*}(\nu_0)=\alpha_l^{*}\fa\Omega_1 - \alpha_r^{*}\fa\Omega_2$. For \cref{lem:horF:c*} we check that $\nu_0(\partial_t\rho_r (\lambda, \rho))=\nu_0(\dot\lambda)+ \fb\Omega_2(\dot\rho)$, and   \cref{lem:horF:d*} is analogous. 
\begin{comment}
Indeed,
\begin{equation*}
\nu_0(\partial_t\rho_r (\lambda(t), \rho(t)))=\rho_r^{*}\nu_0(\dot\lambda(t),\dot\rho(t))=\nu_0(\dot\lambda(t))+ \fb\Omega_2(\dot\rho(t))
\end{equation*} 
and
\begin{equation*}
\nu_0(\partial_t\rho_l (\rho(t),\lambda(t)))=\rho_l^{*}\nu_0(\dot\rho(t),\dot\lambda(t))=\fb\Omega_1(\dot\rho(t))+\nu_0(\dot\lambda(t))\text{.}
\end{equation*} 
\end{comment}
For \cref{lem:horF*} we claim that
the following three statements are equivalent:
\begin{enumerate}[(1)]
\item 
$\lambda_2$ is horizontal.

\item 
$\lambda_1$ is horizontal.

\item
$h$ solves the differential equation $\dot h(t)=-\nu_0(\dot\lambda(t))h(t)-(\alpha_{h(t)})_{*}(\alpha_r^{*}\fa\Omega_2(\dot\lambda(t) )))$.

\end{enumerate}
Equivalence between (1) and (2) is \cref{lem:horF:b*}.
Equivalence between (2) and (3) is proved by the following calculation, using connectivity:
\begin{equation*}
\nu_0(\dot\lambda_1) = \rho^{*}\nu_0(\dot\lambda,(\dot h,0))=\mathrm{Ad}_{h}^{-1}(\nu_0(\dot\lambda))+(\tilde\alpha_{h})_{*}(\alpha_r^{*}\fa\Omega_2(\dot\lambda))+h^{-1}\dot h\text{.}
\eqendofproof
\end{equation*}
\end{proof}

\subsection{Definition of  parallel transport along paths}

\label{sec:defpartranspaths}

Let $\inf P$ be a principal $\Gamma$-2-bundle over $M$ with connection $\Omega$.
Let $\gamma:[0,1] \to M$ be a path in $M$. In this section we define the anafunctor $F_{\gamma}: \inf P_{\gamma(0)} \to \inf P_{\gamma(1)}$. We proceed in the following four steps: Step 1 is to define a set $F_{\gamma}(t)$ with respect to a fixed subdivision $t$ of $[0,1]$. Step 2 is to define  anchors and actions for $F_{\gamma}(t)$ with the required algebraic properties. Step 3 is to get rid of the subdivision via a direct limit construction, resulting in a set $F_{\gamma}$. Step 4 is  to equip $F_{\gamma}$ with the structure of a smooth manifold, and to show that anchors and actions are smooth. 

\subsubsection{Step 1: Total space with respect to a fixed subdivision}

We consider for $0<n\in \N$ the set $T_n :=\{ (t_i)_{i=0}^{n} \sep 0=t_0<t_1<...<t_n=1 \}$ of possible $n$-fold subdivisions of the interval $[0,1]$,
and define for $t\in T_n$ the set
\begin{multline}
\label{eq:defF}
F_{\gamma}(t) := \{ (\{\rho_i\}_{i=0}^{n},\{\gamma_i\}_{i=1}^{n}) \sep \rho_i\in \mor{\inf P},\gamma_i:[t_{i-1},t_i]\to \ob{\inf P}\text{ horizontal paths},\\\pi \circ \gamma_i=\gamma|_{[t_{i-1},t_i]}\text{, }t(\rho_i)=\gamma_{i+1}(t_i)\text{ and }s(\rho_i)=\gamma_i(t_i) \}/\sim
\end{multline}
where $\sim$ is an equivalence relation defined below. In words, $F_{\gamma}(t)$ consists of locally defined horizontal lifts $\gamma_i$  of the pieces $\gamma|_{[t_{i-1},t_i]}$, together with morphisms $\rho_i$ between the endpoint of each lift to the initial point of the next one. We think about the elements of $F_{\gamma}(t)$ as  \quot{formal} compositions of paths in $\ob{\inf P}$ and and morphisms of $\inf P$, and we will use the notation $\xi=\rho_n \ast \gamma_n \ast ... \ast \rho_2 \ast \gamma_1 \ast \rho_0$ for a representative $\xi$ of an element in $F_{\gamma}(t)$.

The equivalence relation in \cref{eq:defF} is generated by  relations $\{\sim_j\}_{1\leq j \leq n}$ defined as follows:
we define a relation
\begin{equation}
\label{eq:equivreldef}
\rho_n \ast \gamma_{n}\ast ... \ast \gamma_1\ast \rho_0\;\; \sim_j \;\; \rho_n' \ast \gamma_{n}'\ast ... \ast \gamma_1'\ast \rho_0'
\end{equation}
if there exist a horizontal path  $\tilde\rho:[t_{j-1},t_j] \to \mor{\inf P}$ such that 
\begin{equation}
\label{eq:er:condgamma}
\gamma_j=s(\tilde\rho)
\quand
\gamma_j'=t(\tilde\rho)\text{,}
\end{equation}
and $\gamma_i'=\gamma_i$ for all $1\leq i \leq n$, $i \neq j$, as well as
\begin{align}
\label{eq:er:condrho:all}
\rho_{j-1}'=\tilde\rho(t_{j-1})\circ \rho_{j-1}
\quand
\rho'_j=\rho_j \circ  \tilde\rho(t_j)^{-1}
\end{align}
and $\rho_i'=\rho_i$  for all $0\leq i \leq n$, $i\neq j,j-1$. We will use the terminology  that the relation \cref{eq:equivreldef} is \emph{via} $\tilde\rho$. It is straightforward to check using \cref{lem:hormor:e} that given one representative $\xi=\rho_n \ast \gamma_{n}\ast ... \ast \gamma_1\ast \rho_0$  and a horizontal path $\tilde\rho$ with $\gamma_j=s(\tilde\rho)$, then one can turn \cref{eq:er:condgamma,eq:er:condrho:all} into definitions, producing another element $\rho_n' \ast \gamma_{n}'\ast ... \ast \gamma_1'\ast \rho_0'$, related to $\xi$ via $\tilde\rho$.
\begin{comment} 
If $\rho_n \ast \gamma_{n}\ast ... \ast \gamma_1\ast \rho_0 $ and $j,\tilde\rho$ are given such that $\gamma_j=s(\tilde\rho)$ then one can use \cref{eq:er:condgamma,eq:er:condrho:1,eq:er:condrho:2,eq:er:condrho:3,eq:er:condrho:4} as definitions of $\gamma_i'$ and $\rho_i'$, provided that $\tilde\rho$ is horizontal.
All compositions are well-defined:
first we check the compositions:
\begin{equation*}
t(\rho_{j-1})=\gamma_j(t_{j-1})=s(\tilde\rho(t_{j-1}))
\end{equation*} 
and
\begin{equation*}
t(\tilde\rho(t_j)^{-1}) =s(\tilde\rho(t_j))=\gamma_j(t_j)=s(\rho_j)
\end{equation*}
and then the conditions between jumps and paths:
\begin{align*}
s(\rho_{i}') &= \gamma_{i}'(t_{i}) &&\forall 0\leq i<j-1 
\\
s(\rho_{j-1}') &=s(\rho_{j-1})=\gamma_{j-1}(t_{j-1})=\gamma_{j-1}'(t_{j-1})
\\
s(\rho_{j}') &=t(\tilde\rho(t_j))=\gamma_{j}'(t_{j})
\\
s(\rho_{i}') &= s(\rho_i)=\gamma_{i}(t_{i})= \gamma_{i}'(t_{i}) &&\forall j+1<i\leq n
\end{align*}
and
\begin{align*}
t(\rho_{i}') &= \gamma_{i+1}'(t_{i}) &&\forall 0\leq i<j-1 
\\
t(\rho_{j-1}') &=t(\tilde\rho(t_{j-1}))=t(\tilde\rho(t_{j-1}))=\gamma_{j}'(t_{j-1})
\\
t(\rho_{j}') &=t(\rho_j)=\gamma_{j+1}(t_{j})=\gamma_{j+1}'(t_{j})
\\
t(\rho_{i}') &= t(\rho_i)=\gamma_{i+1}(t_{i})= \gamma_{i+1}'(t_{i}) &&\forall j+1<i\leq n
\end{align*}
\end{comment}

\begin{lemma}
\label{lem:equiv}
\begin{enumerate}[(a)]

\item 
\clabel{lem:equiv:a}
For each $1\leq j \le n$, $\sim_j$ is an equivalence relation.

\item
\clabel{lem:equiv:b}
For $1\leq i<j\leq n$ and $\xi^1 \sim_i \xi' \sim_j \xi^2$, there exists $\xi''$ such that $\xi^1\sim_j \xi''\sim_i \xi^2$.

\end{enumerate}
\end{lemma}

\begin{proof}
To \cref{lem:equiv:a*}: For reflexivity put $\tilde\rho(t):=\id_{\gamma(t)}$, which is horizontal by \cref{lem:hormor:a}.
\begin{comment}
\cref{eq:er:condgamma,eq:er:condrho:1,eq:er:condrho:2,eq:er:condrho:3,eq:er:condrho:4} result in $\gamma_i'=\gamma_i$  and $\rho_i'=\rho_i$. \end{comment}
For symmetry assume that $\xi\sim_{j} \xi'$ via $\tilde\rho$. Then,  $\tilde\rho' := \tilde\rho^{-1}$ is horizontal by \cref{lem:hormor:b}, and we have $\xi' \sim_j \xi$ via $\tilde\rho'$. 
\begin{comment}
Starting with:
\begin{equation*}
\rho_n \ast \gamma_{n}\ast ... \ast \gamma_1\ast \rho_0 \sim_j \rho_n' \ast \gamma_{n}'\ast ... \ast \gamma_1'\ast \rho_0'
\end{equation*}
Then we have 
\begin{equation*}
\gamma_j'=t(\tilde\rho)=s(\tilde\rho^{-1})=s(\tilde\rho')\text{.} \end{equation*}
Applying $\tilde\rho'$ to $\rho_n' \ast \gamma_{n}'\ast ... \ast \gamma_1'\ast \rho_0$ we obtain $\rho_n \ast \gamma_{n}\ast ... \ast \gamma_1\ast \rho_0$. We find
\begin{equation*}
\forall 1\leq i <j:\gamma_i=\gamma_i'
\quomma 
\gamma_j=t(\tilde\rho')
\quand
\forall j<i\leq n:\gamma_i=\gamma_i'
\end{equation*}
as well as
\begin{align*}
\forall0\leq i <j-1:&&\rho_i&=\rho_i'
\\
&&\rho_{j-1}&=\tilde\rho(t_{j-1})^{-1}\circ \rho_{j-1}=\tilde\rho'(t_{j-1})^{-1}\circ \rho_{j-1}'=\tilde\rho'(t_{j-1})\circ \rho_{j-1}'
\\
\forall j+1<i\leq n:&&\rho_i&=\rho_i'
\end{align*}
For the computation of $\rho_j$ we first get from our assumption and \cref{eq:er:condrho:3}: 
\begin{equation*}
\rho'_j\circ  \tilde\rho(t_j)=\rho_j
\end{equation*}
so that
\begin{align*}
\rho_j &= \rho'_j\circ  \tilde\rho(t_j) 
\\&=\rho_j' \circ  \tilde\rho'(t_j)^{-1}
\end{align*}
This is the correct result.
\end{comment}
Transitivity goes analogously using  \cref{lem:hormor:d}.
\begin{comment}
Indeed,
we assume that
\begin{equation*}
\rho_n \ast  ... \ast \gamma_1\ast \rho_0 \sim_j \rho_n' \ast  ... \ast \gamma_1'\ast \rho_0'
\quand
\rho_n' \ast  ... \ast \gamma_1'\ast \rho_0' \sim_j \rho_n'' \ast  ... \ast \gamma_1''\ast \rho_0''
\end{equation*}
with the first relation via $\tilde\rho_2$ and the second via $\tilde\rho_1$. In particular, we have $\gamma_j=s(\tilde\rho_2)$ and $\gamma_j'=t(\tilde\rho_2)=s(\tilde\rho_1)$ and $t(\tilde\rho_1)=\gamma_j''$. 
We consider  $\tilde\rho:= \tilde\rho_{1} \circ \tilde \rho_{2}$, this is horizontal by \cref{lem:hormor:d}; then
\begin{equation*}
\rho_n \ast  ... \ast \gamma_1\ast \rho_0 \sim_j \rho_n'' \ast  ... \ast \gamma_1''\ast \rho_0
\end{equation*}  
via $\tilde\rho$.  
We check all identities: we have $\gamma_j=s(\tilde\rho)$ and 
\begin{multline*}
\label{eq:er:condgamma}
\forall 1\leq i <j:\gamma_i''=\gamma_i'=\gamma_i
\quomma 
\gamma_j''=t(\tilde\rho)\\
\quand
\forall j<i\leq n:\gamma_i''=\gamma_i'=\gamma_i
\end{multline*}
as well as
\begin{align*}
\forall0\leq i <j-1:&&\rho_i''&=\rho_i'=\rho_i
\\
&&\rho_{j-1}''&=
\tilde\rho_1(t_{j-1}) \circ \rho'_{j-1}
\\&&&=\tilde\rho_1(t_{j-1}) \circ \tilde \rho_{2}(t_{j-1})\circ \rho_{j-1}
\\&&&=\tilde\rho(t_{j-1})\circ \rho_{j-1}
\\
\forall j<i\leq n:&&\rho_i''&=\rho_i'=\rho_i
\end{align*}
and
\begin{align*}
&&\rho''_j&= \rho_j' \circ  \tilde\rho_1(t_j)^{-1}
\\&&&= \rho_j \circ  \tilde\rho_2(t_j)^{-1}\circ  \tilde\rho_1(t_j)^{-1}
\\&&&= \rho_j \circ  \tilde\rho(t_j)^{-1}
\end{align*}
This is the correct result.
\end{comment} 
To \cref{lem:equiv:b*}: 
We let  $\xi^1 \sim_i \xi'$ be via $\tilde\rho^1$ and $\xi' \sim_j \xi^2$ be via $\tilde\rho^2$.  
\begin{comment}
Thus, we  $\gamma_i^1=s(\tilde\rho^1)$ and 
\begin{equation}
\forall 1\leq k <i:\gamma_k'=\gamma_k^1
\quomma 
\gamma_i'=t(\tilde\rho^1)
\quand
\forall i<k\leq n:\gamma_k'=\gamma_k^1
\end{equation}
as well as
\begin{align}
\forall0\leq k <i-1:&&\rho_k'&=\rho_k^1
\\
&&\rho_{i-1}'&=\tilde\rho^1(t_{i-1})\circ \rho_{i-1}^1
\\
&&\rho'_i&=\rho_i^1 \circ  \tilde\rho^1(t_i)^{-1}
\\
\forall i<k\leq n:&&\rho_k'&=\rho_k^1
\end{align}
We further have $\gamma_j'=s(\tilde\rho^2)$ and 
\begin{equation*}
\forall 1\leq k <j:\gamma_k^2=\gamma_k'
\quomma 
\gamma_j^2=t(\tilde\rho^2)
\quand
\forall j<k\leq n:\gamma_k^2=\gamma_k'
\end{equation*}
as well as
\begin{align*}
\forall0\leq k <j-1:&&\rho_k^2&=\rho_k'
\\
&&\rho_{j-1}^2&=\tilde\rho^2(t_{j-1})\circ \rho_{j-1}'
\\
&&\rho^2_j&=\rho_j' \circ  \tilde\rho^2(t_j)^{-1}
\\
\forall j<k\leq n:&&\rho_k^2&=\rho_k'
\end{align*}
\end{comment}
We define $\tilde\rho^3 :=\tilde\rho^2$. We define $\xi''$ such that $\xi^1 \sim_j \xi''$ via  $\tilde\rho^3$. \begin{comment}
This works since
\begin{equation*}
\gamma_j^1=\gamma_j'=  s(\tilde\rho^2)=  s(\tilde\rho^3)\text{.}
\end{equation*}
We have
\begin{equation*}
\forall 1\leq k <j:\gamma_k''=\gamma_k^1
\quomma 
\gamma_j''=t(\tilde\rho^3)
\quand
\forall j<k\leq n:\gamma_k''=\gamma_k^1
\end{equation*}
as well as
\begin{align*}
\forall0\leq k <j-1:&&\rho_k''&=\rho_k^1
\\
&&\rho_{j-1}''&=\tilde\rho^3(t_{j-1})\circ \rho_{j-1}^1
\\
&&\rho''_j&=\rho_j^1 \circ  \tilde\rho^3(t_j)^{-1}
\\
\forall j<k\leq n:&&\rho_k''&=\rho_k^1
\end{align*}
\end{comment}
Now one  can check  that  $\xi''\sim_i \xi^2$ via $\tilde\rho^1$. 
\begin{comment}
We first check that
\begin{equation*}
 s(\tilde\rho^1)=\gamma_i^1= \gamma''_i
\text{.}\end{equation*}
Further we have
\begin{align*}
\forall 1\leq k <i:&&\gamma_k^2&=\gamma_k'=\gamma_k^1=\gamma_k''
\\
&&\gamma_i^2&=\gamma_i'=t(\tilde\rho^1)
\quomma
\\
\forall i<k< j:&&\gamma_k^2&=\gamma_k'=\gamma_k^1=\gamma_k''
\\
&&\gamma_j^2 &= t(\tilde\rho^2)=t(\tilde\rho^3)=\gamma_j''
\\
\forall j<k\leq n:&&\gamma_j^2 &= \gamma_k'=\gamma_k^1=\gamma_k''
\end{align*}
Finally we have
\begin{align*}
\forall 1\leq k <i-1:&&\rho_k^2&=\rho_k'=\rho_k^1=\rho_k''
\\&& \rho^2_{i-1}&=\rho'_{i-1}= \tilde\rho^1(t_{i-1})\circ \rho_{i-1}^1=\tilde\rho^1(t_{i-1})\circ \rho_{i-1}''
\\
\text{if $i=j-1$}:&&\rho_i^2&=\tilde\rho^2(t_{i})\circ \rho_{i}'
\\&&&= \tilde\rho^2(t_i) \circ \rho_i^1 \circ  \tilde\rho^1(t_i)^{-1} 
\\&&&= \tilde\rho^2(t_i) \circ \tilde\rho^3(t_{i})^{-1}\circ \rho_i'' \circ  \tilde\rho^1(t_i)^{-1} 
\\&&&= \tilde\rho^2(t_i) \circ \tilde\rho^2(t_i)^{-1}\circ \rho_i'' \circ  \tilde\rho^1(t_i)^{-1} 
\\&&&= \rho_i'' \circ  \tilde\rho^1(t_i)^{-1}
\\ \text{if }i<j-1:  &&\rho_i^2 &=\rho'_i
 \\&&&=\rho_i^1 \circ  \tilde\rho^1(t_i)^{-1}
\\&&&=\rho_i'' \circ \tilde\rho^1(t_i)^{-1}
\\
\forall i<k< j-1:&&\rho_k^2&=\rho_k'=\rho_k''
\\
&&\rho_j^2 &= \rho_j' \circ  \tilde\rho^2(t_j)^{-1}
\\&&&=\rho_j^1 \circ  \tilde\rho^2(t_j)^{-1}
\\&&&=\rho_j''
\\
\forall j<k\leq n:&&\rho_k^2 &= \rho_k'=\rho_k^1=\rho_k''=\rho_k''
\end{align*}
\end{comment}
\end{proof}

\subsubsection{Step 2: Anchors and actions}

Next we define a left $\inf P_{x}$-action and a right $\inf P_{y}$-action on the set $F_{\gamma}(t)$; their anchors are
\begin{align*}
\alpha_l &:F_{\gamma}(t) \to \inf P_{x}: \rho_n \ast \gamma_{n}\ast ... \ast \gamma_1\ast \rho_0 \mapsto s(\rho_0)
\\
\alpha_r &:F_{\gamma}(t) \to \inf P_{y}: \rho_n \ast \gamma_{n}\ast ... \ast \gamma_1\ast \rho_0 \mapsto t(\rho_n)\text{.}
\end{align*}
These maps are obviously well-defined under the equivalence relation.

\begin{lemma}
\label{lem:laction}
The map  $\mor{\inf P_{x}} \ttimes s{\alpha_l} F_{\gamma}(t) \to F_{\gamma}(t)$ defined by 
\begin{equation*}
\rho \circ (\rho_n \ast \gamma_{n}\ast ... \ast \gamma_1\ast \rho_0) :=  \rho_n \ast \gamma_{n}\ast ... \ast \gamma_1\ast (\rho_0 \circ \rho^{-1})
\end{equation*}
is a well-defined left action of $\inf P_{x}$ on $F_{\gamma}(t)$ with anchor $\alpha_l$, and keeps $\alpha_r$ invariant.  \end{lemma}

\begin{proof}
For the well-definedness, only $\sim_1$ has to be checked, which is done via \cref{eq:er:condrho:all}. The other statements are obvious.
\begin{comment}
First of all, the composition is well-defined: 
\begin{equation*}
t(\rho^{-1}) = s(\rho) = \alpha_l(\xi) = s(\xi)\text{.}
\end{equation*}
The left anchor is preserved:
\begin{equation*}
\alpha_l(\rho \circ \xi) =\alpha_l(\xi \ast \rho^{-1})=s(\rho^{-1}) = t(\rho)\text{.}
\end{equation*}
It is obviously an action. 
\end{comment} 
\end{proof}

\begin{lemma}
\label{lem:action}
The map  $F_{\gamma}(t) \ttimes{\alpha_r}t\mor{\inf P_{y}} \to F_{\gamma}(t)$ defined by 
\begin{equation*}
(\rho_n \ast \gamma_{n}\ast ... \ast \gamma_1\ast \rho_0)\circ \rho := \rho^{-1} \circ \rho_n \ast \gamma_{n}\ast ... \ast \gamma_1\ast \rho_0
\end{equation*}
is a well-defined right action of $\inf P_{y}$ on $F_{\gamma}(t)$ with anchor $\alpha_r$,  it keeps $\alpha_l$ invariant, and it commutes with the left action of \cref{lem:laction}. Moreover, if $\xi,\xi'\in F_{\gamma}$ with $\alpha_l(\xi)=\alpha_l(\xi')$, then there exists a unique $\rho\in \mor{\inf P_{y}}$ such that $\xi\circ \rho=\xi'$.
\end{lemma}

\begin{proof}
Well-definedness and the properties of an action are straightforward to check.
\begin{comment}
The composition makes sense, since
\begin{equation*}
t(\rho_n) =t(\rho)=s(\rho^{-1})\text{.}
\end{equation*}
It is well-defined: suppose
\begin{equation*}
\rho_n \ast \gamma_{n}\ast ... \ast \gamma_1\ast \rho_0 \sim \rho_n' \ast \gamma_{n}'\ast ... \ast \gamma_1'\ast \rho_0'
\end{equation*}
via $j,\tilde\rho$, in particular, we have  from \cref{eq:er:condrho:4} $\rho_n'=\rho_n$.
Then we have
\begin{equation*}
\rho^{-1} \circ \rho_n \ast \gamma_{n}\ast ... \ast \gamma_1\ast \rho_0 \sim \rho^{-1} \circ\rho_n' \ast \gamma_{n}'\ast ... \ast \gamma_1'\ast \rho_0'\text{.}
\end{equation*}
via $j,\tilde\rho$: the only condition we have to check is the new \cref{eq:er:condrho:4} for $i=n$:
\begin{align*}
\rho^{-1} \circ\rho_n' 
&\eqcref{eq:er:gtrans} \rho^{-1}\circ \rho_n
\end{align*}
\end{comment}
More difficult is to prove existence and uniqueness of $\rho$; this is exactly the point where our equivalence relation becomes relevant. For existence, suppose
\begin{equation*}
\xi = \rho_n \ast \gamma_{n}\ast ... \ast \gamma_1\ast \rho_0 
\quand
\xi' = \rho_n' \ast \gamma_{n}'\ast ... \ast \gamma_1'\ast \rho_0' 
\end{equation*}
satisfy  $\alpha_l(\xi)=\alpha_l(\xi')$. This implies that $\rho_0'\circ \rho_0^{-1}\in \mor{\inf P}$ satisfies $s(\rho_0'\circ \rho_0^{-1})=t(\rho_0)=\gamma_1(t_0)$ and $t(\rho_0'\circ \rho_0^{-1})=t(\rho_0')=\gamma_1'(t_0)$.  By  \cref{lem:hormor:f} there exists a horizontal path $\tilde\rho: [t_0,t_1] \to \mor{\inf P}$ with $\tilde\rho(t_0)=\rho_0'\circ \rho_0^{-1}$, $s(\tilde\rho)=\gamma_1$ and $t(\tilde\rho)=\gamma_1'$. 
Via $\tilde\rho$ we obtain a relation
\begin{equation*}
\rho_n\\ \\  \ast \gamma_{n}\ast ... \ast \gamma_1\ast \rho_0 \sim_1 \rho_n^{(1)} \ast \gamma_{n}^{(1)}\ast ... \ast \rho_1^{(1)} \ast  \gamma_1^{(1)} \ast \rho_0^{(1)}\text{,}
\end{equation*}
with $\gamma_1^{(1)} = t(\tilde\rho)=\gamma_1'$ and $\rho_0^{(1)}=\tilde\rho(t_{0})\circ \rho_{0} = \rho_0'$. 
Now we are in the situation that $s(\rho_1^{(1)})=s(\rho_1')$, and can use $\sim_2$ in the same manner as $\sim_1$ before. After $n$ steps, we arrive at a relation
\begin{equation*}
\rho_n \ast \gamma_{n}\ast ... \ast \gamma_1\ast \rho_0 \sim \rho_n^{(n)} \ast \gamma_{n}'\ast ... \ast \gamma_1'\ast \rho_0'\text{.}
\end{equation*} 
Now we define $\rho := \rho_n^{(n)} \circ \rho_n'^{-1}$; this definition yields $\xi\circ \rho=\xi'$. \begin{comment}
This is well-defined because $s(\rho_n'')=\gamma_{n}'(t_n)=s(\rho_n')$. 
We have
\begin{multline*}
\xi \circ \rho \sim (\rho_n'' \ast \gamma_{n}'\ast ... \ast \gamma_1'\ast \rho_0') \circ \rho\\= \rho^{-1}\circ\rho_n''\ast \gamma_{n}'\ast ... \ast \gamma_1'\ast \rho_0' = \rho_n' \ast \gamma_{n}'\ast ... \ast \gamma_1'\ast \rho_0' =\xi'\text{.}
\end{multline*}
\end{comment}

Next we show that $\rho$ is unique, i.e. we prove that a relation $\xi\circ \rho\sim\xi$ implies $\rho=\id$. 
\begin{comment}
We assume that $\xi \circ \rho_1\sim\xi'\sim\xi\circ \rho_2$, with the consequence that $\xi\circ \rho\sim\xi$ for $\rho := \rho_1\circ \rho_2^{-1}$. \end{comment}
Putting $\xi=\rho_n \ast \gamma_{n}\ast ... \ast \gamma_1\ast \rho_0$, the assumption is
\begin{equation*}
\rho_n\\ \\  \ast \gamma_{n}\ast ... \ast \gamma_1\ast \rho_0 \sim (\rho^{-1} \circ \rho_n) \ast \gamma_{n}\ast ... \ast \gamma_1\ast \rho_0\text{,}
\end{equation*}
where  $\sim$ is a finite chain composed of the relations $\sim_1$, ..., $\sim _n$.
By \cref{lem:equiv} we can assume that this chain is ordered with descending $i$ and each $i$ appears at most once.  If the chain $\sim$ is empty, we must have $\rho_n = \rho^{-1} \circ \rho_n$; this proves $\rho=\id$. If it is non-empty, we proceed by induction over the minimum $j$ of  occurring relations $\sim_j$. We write the chain $\sim$ as
 \begin{equation*}
\rho_n\\ \\  \ast \gamma_{n}\ast ... \ast \gamma_1\ast \rho_0 \sim' \rho_n'\\ \\  \ast \gamma_{n}'\ast ...\ast \rho_1' \ast \gamma_1'\ast \rho_0' \sim_j (\rho^{-1} \circ \rho_n) \ast \gamma_{n}\ast ... \ast \gamma_1\ast \rho_0\text{,}
\end{equation*}
with $\sim'$ a chain composed  only of the relations $\sim_{j+1,}...,\sim_n$. Since $\sim'$  does not affect the parts before $\rho_{j}$, we have  $\gamma_i'=\gamma_i$ for $1\leq i\leq\ j$ and $\rho_i'=\rho_i$ for $0\leq i <j$. We claim that $\sim_j$ implies coincidence of the remaining parts: (A)  $\gamma_i'=\gamma_i$ for  $j< i \leq n$, (B) $\rho_j'=\rho_j$ if $j<n$, (C) $\rho_i'=\rho_i$ for all $j< i < n$, and (D)
$\rho_n'=(\rho^{-1} \circ \rho_n)$. Given the claim, we obtain
\begin{equation*}
\rho_n\\ \\  \ast \gamma_{n}\ast ... \ast \gamma_1\ast \rho_0 \sim' (\rho^{-1} \circ \rho_n) \ast \gamma_{n}\ast ... \ast \gamma_1\ast \rho_0\text{,}
\end{equation*}
and have hence shifted the induction parameter from $j$ to $j+1$. At the end the minimum is shifted to $n+1$, meaning that the chain of relations becomes empty.

In order to prove the claim, we assume that $\sim_j$ is via $\tilde\rho$. Since $\gamma_j'=\gamma_j$ and $\rho_{j-1}'=\rho_{j-1}$, we have   $t(\tilde\rho)=s(\tilde\rho)$ and $\tilde\rho(t_{0})=\id$. 
This shows (A) and (C).  If $n=j$, then we have by
\cref{eq:er:condrho:all} 
$\rho_j'=\rho^{-1} \circ \rho_j\circ  \tilde\rho(t_j)^{-1}$.
If $j<n$, then we have by \cref{eq:er:condrho:all}
$\rho_j'=\rho_j\circ  \tilde\rho(t_j)^{-1}$
and $\rho_n'=\rho^{-1} \circ \rho_n$. 
We show that $\tilde\rho(t_j)^{-1}=\id$,
\begin{comment}
For references, we write again
\begin{equation}
\label{eq:324324}
\tilde\rho(t_j)^{-1}=\id
\end{equation}
\end{comment}
which proves the remaining  claims (B) and (C). Indeed, we observe that $\tilde\rho$ and $\id_{\gamma_j}$ satisfy the assumptions of \cref{lem:hormor:g}, and since  $\tilde\rho(t_{j-1})=\id_{\gamma_j(t_{j-1})}$, we have $\tilde\rho=\id_{\gamma_j}$.
\begin{comment}
Now \cref{eq:324324} follows. 
\end{comment}
\end{proof}

Next we define the $\mor{\Gamma}$-action on $F_{\gamma}(t)$, which at the end constitutes the $\Gamma$-equivariance of the anafunctor $F_{\gamma}$.

\begin{lemma}
\label{lem:gammaaction}
The map $F_{\gamma}(t)\times \mor{\Gamma} \to F_{\gamma}(t)$ defined by 
\begin{equation*}
(\rho_n \ast \gamma_n \ast ... \ast \gamma_1\ast \rho_0) \cdot (h,g)
 := R(\rho_n,g) \ast R(\gamma_n,g) \ast ... \ast R(\gamma_1,g)\ast R(\rho_0,(h^{-1},t(h)g)) \end{equation*}
is a well-defined action and compatible with the left $\inf P_{x}$-action and the right $\inf P_{y}$-action in the sense of \cref{rem:anafunctor:a}.
\end{lemma}

\begin{proof}
The axioms of an action are straightforward to check on the level of representatives.
\begin{comment}
For the action:
\begin{align*}
&\mquad(\rho_n \ast \gamma_n \ast ... \ast \gamma_1\ast \rho_0) \cdot (h,g)\cdot (h',g')
 \\&= (R(\rho_n,g) \ast R(\gamma_n,g) \ast ... \ast R(\gamma_1,g)\ast R(\rho_0,(h^{-1},t(h)g)))\cdot (h',g')
 \\&= R(\rho_n,gg') \ast R(\gamma_n,gg') \ast ... \ast R(\gamma_1,gg')\ast R(\rho_0,(h^{-1},t(h)g)\cdot (h'^{-1},t(h')g'))
 \\&= R(\rho_n,gg') \ast R(\gamma_n,gg') \ast ... \ast R(\gamma_1,gg')\ast R(\rho_0,(h^{-1},t(h)g)\cdot (h'^{-1},t(h')g'))
 \\&= R(\rho_n,gg') \ast R(\gamma_n,gg') \ast ... \ast R(\gamma_1,gg')\ast R(\rho_0,(\alpha(g,h'^{-1})h^{-1},t(h\alpha(g,h'))gg'))\text{.}
\end{align*}
\end{comment} 
In order to check  well-definedness, we can then write $(h,g)=(h,1)\cdot (1,g)$ and check separately. For well-definedness with respect to elements $(1,g)\in \mor{\Gamma}$, it is straightforward to see that  if
\begin{equation*}
\rho_n \ast  ... \ast \gamma_1\ast \rho_0 \sim_j \rho_n' \ast  ... \ast \gamma_1'\ast \rho_0'
\end{equation*}
via $\tilde\rho$, then  
\begin{equation*}
(\rho_n \ast  ... \ast \gamma_1\ast \rho_0) \cdot (1,g) \sim_j (\rho_n' \ast  ... \ast \gamma_1'\ast \rho_0')\cdot (1,g)
\end{equation*}
via $R(\tilde\rho, g)$, which is horizontal by \cref{lem:hormor:c}.
\begin{comment}
Indeed, our assumptions mean that $\gamma_j=s(\tilde\rho)$ and 
\begin{equation*}
\forall 1\leq i <j:\gamma_i'=\gamma_i
\quomma 
\gamma_j'=t(\tilde\rho)
\quand
\forall j<i\leq n:\gamma_i'=\gamma_i
\end{equation*}
as well as
\begin{align*}
\forall0\leq i <j-1:&&\rho_i'&=\rho_i
\\
&&\rho_{j-1}'&=\tilde\rho(t_{j-1})\circ \rho_{j-1}
\\
&&\rho'_j&=\rho_j \circ  \tilde\rho(t_j)^{-1}
\\
\forall j<i\leq n:&&\rho_i'&=\rho_i
\end{align*}
We have
\begin{equation*}
s(R(\tilde\rho ,g)) = R(s(\tilde\rho), g)=R(\gamma_j,g)\text{.}
\end{equation*}
Applying the relation to $(\rho_n \ast  ... \ast \gamma_1\ast \rho_0) \cdot (1,g)$, we obtain 
and 
\begin{multline*}
\forall 1\leq i <j:\gamma_i''=R(\gamma_i,g)=R(\gamma_i',g)
\quomma 
\gamma_j''=t(R(\tilde\rho, g))=R(\gamma_j',g)
\\\quand
\forall j<i\leq n:\gamma_i'':=R(\gamma_i,g)=R(\gamma_i,g)=R(\gamma_i',g)\text{.}
\end{multline*}
as well as
\begin{align*}
\forall0\leq i <j-1:&&\rho_i''&:=R(\rho_i,g)=R(\rho_i',g)
\\
&&\rho_{j-1}''&:=R(\tilde\rho(t_{j-1}), g)\circ R(\rho_{j-1},g)=R(\tilde\rho(t_{j-1})\circ \rho_{j-1},g)=R(\rho_{j-1}',g)
\\
&&\rho''_j&:=R(\rho_j,g) \circ  R(\tilde\rho(t_j)^{-1},g)= R(\rho_j',g)
\\
\forall j<i\leq n:&&\rho_i''&:=R(\rho_i,g)=R(\rho_i',g)\text{.}
\end{align*}
\end{comment}
For well-definedness with respect to  elements of the form $(h,1)$, it suffices to consider $\sim_1$, which is easy. 
\begin{comment}
If $\rho_{0}'=\tilde\rho(t_{0})\circ \rho_{0}$ then 
\begin{equation*}
R(\rho'_0,(h^{-1},t(h)))=R(\tilde\rho(t_{0})\circ \rho_{0},(1,1)\circ (h^{-1},t(h)))=\tilde\rho(t_0) \circ R(\rho_0,(h^{-1},t(h)))\text{.}
\end{equation*}
\end{comment}
The compatibility with the anchors and the  left $\inf P_x$-action hold on the level of representatives and are  straightforward to check.
\begin{comment}
For the left anchor:
\begin{equation*}
\alpha_l(\xi \cdot (h,g)) =s(\xi\cdot (h,g))=s(R(\id_{s(\xi)},(h^{-1},t(h)g)))= R(s(\xi),t(h)g)=R(\alpha_l(\xi),t(h)g)\text{.}
\end{equation*}
and the right anchor:
\begin{equation*}
\alpha_r(\xi\cdot (h,g)) =t(\xi\cdot (h,g))=t(\xi \cdot g)=R(t(\xi),g)=R(\alpha_r(\xi),g)\text{.}
\end{equation*}
The condition for left composition  is
\begin{eqnarray*}
&&\hspace{-2cm}(\rho_l \circ \xi ) \cdot ((h_l,g_l)\circ (h,g) )
\\&=& ( \xi \ast \rho_l^{-1}) \cdot (h_lh,g)
\\&=&  (\xi \cdot g) \ast (\rho_l^{-1} \cdot g) \ast R(\id_{s( \xi \ast \rho_l^{-1})},(h^{-1}h_l^{-1},t(h_lh)g))
\\&=&  (\xi \cdot g) \ast( R(\rho_l^{-1}, (1,g)) \circ R(\id_{t( \rho_l)},(h^{-1}h_l^{-1},t(h_lh)g)))
\\&=&  (\xi \cdot g) \ast R(\rho_l^{-1}, (h^{-1}h_l^{-1},t(h_lh)g))
\\&=&    (\xi \cdot g) \ast (R(\rho_l^{-1},(h^{-1}h_l^{-1},t(h_l)t(h)g)) \\&=&    (\xi \cdot g) \ast (R(\id_{s(\xi)},(h^{-1},t(h)g))\circ R(\rho_l^{-1},(h_l^{-1},t(h_l)t(h)g))) \\&=&    (\xi \cdot g) \ast R(\id_{s(\xi)},(h^{-1},t(h)g))\ast R(\rho_l,(h_l,t(h)g))^{-1} \\&=& R(\rho_l,(h_l,g_l)) \circ (\xi\cdot (h,g))
\end{eqnarray*}
\end{comment}
Compatibility with the right $\inf P_y$-action, however, only holds on the level of equivalence classes: we claim that
\begin{equation}
\label{eq:actioncomp:left}
(\xi \circ \rho)\cdot ((h,g)\circ (h_r,g_r))= R(\rho^{-1},g_r) \ast R(\xi,g_r)\ast R(\id_{s(\xi)},(h_r^{-1}h^{-1},t(h)g))
\end{equation}
and
\begin{equation}
\label{eq:actioncomp:right}
(\xi\cdot (h,g))\circ R(\rho,(h_r,g_r))=R(\rho^{-1},(h_r^{-1},g))\ast R(\xi,g)\ast  R(\id_{s(\xi)},(h^{-1},t(h)g)) 
\end{equation}
are equivalent.
\begin{comment}

More explicitly, this is
\begin{align*}
&\mquad(\xi \circ \rho)\cdot ((h,g)\circ (h_r,g_r))
\\&= (\rho^{-1} \ast \xi)\cdot (hh_r,g_r) 
\\&= R(\rho^{-1},g_r) \ast R(\xi,g_r)\ast R(\id_{s(\xi)},(h_r^{-1}h^{-1},t(h)g))
\\&\sim R(\rho^{-1},(h_r^{-1},g))\ast R(\xi,g)\ast  R(\id_{s(\xi)},(h^{-1},t(h)g)) 
\\&= (R(\xi,g)\ast  R(\id_{s(\xi)},(h^{-1},t(h)g)))\circ R(\rho,(h_r,g_r))
\\&= (\xi\cdot (h,g))\circ R(\rho,(h_r,g_r))
\end{align*}
Here, the following conditions have to be assumed:
\begin{align*}
g&=t(h_r)g_r
\\
t(\xi)&=t(\rho)\text{.}
\end{align*}
\end{comment}
This relation will be proved by induction over the length of $\xi$. We start with the case $n=0$, i.e. $\xi =\rho_0$. Then, \cref{eq:actioncomp:left,eq:actioncomp:right} are, respectively,
\begin{align*}
\rho_0' &:= R(\rho^{-1},g_r) \circ R(\rho_0,g_r)\circ R(\id_{s(\rho_0)},(h_r^{-1}h^{-1},t(h)g))
\\
\rho_0'' &:=R(\rho^{-1},(h_r^{-1},g))\circ R(\rho_0,g)\circ  R(\id_{s(\rho_0)},(h^{-1},t(h)g))\text{.}
\end{align*}
Both expressions are in fact  equal; in particular, $\rho_0'\sim\rho_0''$.
\begin{comment}
We have
\begin{align*}
\rho_0'' &=R(\rho^{-1},(h_r^{-1},g))\circ R(\rho_0,(h^{-1},t(h)g))
\\&=R(\rho^{-1} \circ \rho_0,(h_r^{-1}h^{-1},t(h)g))
\\&=R(\rho^{-1},g_r) \circ R(\rho_0,(h_r^{-1}h^{-1},t(h)g))
\\&=\rho_0' \text{,}
\end{align*}
\end{comment}
Now let $n>1$ and  $\xi=\rho_n \ast ... \ast \rho_0$. Then, \cref{eq:actioncomp:left} is an element $\xi'$ consisting of
\begin{align*}
\rho_n' &:=R(\rho^{-1},g_r) \circ R(\rho_n,g_r)
&
\rho_i' &:=  R(\rho_i,g_r)\text{ for $1\leq i < n$}
\\
\rho_0' &:= R(\rho_0,g_r)\circ R(\id_{s(\xi)},(h_r^{-1}h^{-1},t(h)g))
&
\gamma_i' &:= R(\gamma_i,g_r)\text{ for $1\leq i \leq n$}\text{.}
\end{align*}
We use $\sim_1$ with  $\tilde\rho := R(\id_{\gamma_1},(h_r,g_r))$; this is horizontal by \cref{lem:hormor:h,lem:hormor:c}, and satisfies $s(\tilde\rho)=R(\gamma_1,g_r)=\gamma_1'$. Then, we obtain an equivalent representative $\xi''$ with $\xi' \sim_1 \xi''$, which consists of the components $\gamma_1'':=t(\tilde\rho)$,
\begin{align*}
\rho''_1 &:= \rho_1' \circ  \tilde\rho(t_1)^{-1}
= R(\rho_1,g_r) \circ  R(\id_{\gamma_1(t_1)},(h_r^{-1},g))\text{,}
\\
\rho_{0}''&:=\tilde\rho(0)\circ \rho'_{0} 
\\&=R(\id_{\gamma_1(0)},(h_r,g_r))\circ R(\rho_0,g_r)\circ R(\id_{s(\xi)},(h_r^{-1}h^{-1},t(h)g))
\\&= R(\rho_0,(h^{-1},t(h)g))
\\&= R(\rho_0,g)\circ R(\id_{s(\xi)},(h^{-1},t(h)g))\text{,}
\end{align*}
as well as $\gamma_i'':=\gamma_i'$ and $\rho_i'' := \rho_i'$ for all other indices $i$. 
\begin{comment}
More precisely, we have
\begin{equation*}
\gamma_1'':=t(\tilde\rho)=R(\gamma_1,g)
\quand
\forall 1<i\leq n:\gamma_i'':=\gamma_i'=R(\gamma_i,g_r)
\end{equation*}
as well as
\begin{align*}
&&\rho_{0}''&:=\tilde\rho(0)\circ \rho'_{0} \\&&&=R(\id_{\gamma_1(0)},(h_r,g_r))\circ R(\rho_0,g_r)\circ R(\id_{s(\xi)},(h_r^{-1}h^{-1},t(h)g))
\\&&&= R(\rho_0,(h_r,g_r)\circ (h_r^{-1}h^{-1},t(h)g))
\\&&&= R(\rho_0,(h^{-1},t(h)g))
\\&&&= R(\rho_0,g)\circ R(\id_{s(\xi)},(h^{-1},t(h)g))
\\
&&\rho''_1&:=\rho_1' \circ  \tilde\rho(t_1)^{-1}
\\&&&= R(\rho_1,g_r) \circ  R(\id_{\gamma_1(t_1)},(h_r^{-1},g))
\\
\forall j<i\leq n-1:&&\rho_i''&:=\rho_i'=R(\rho_i,g_r)
\\&&\rho_n'' &:= \rho_n'
\\&&&= R(\rho^{-1},g_r) \circ R(\rho_n,g_r )
\end{align*}
\end{comment}
We define $\xi^{-} := \rho_n \ast ... \ast \rho_1 $, $g_r^{-} := g_r$, $\rho^{-} := \rho$, $h_r^{-}:= h_r$ and $h^{-} := 1$, and $g^{-}:=g$. 
\begin{comment}
Then we have 
\begin{align*}
t(h_r^{-})g_r^{-}&=t(h_r)g_r=g=g^{-}
\\
t(\xi^{-})&=t(\rho_n'')=...=t(\rho)=t(\rho^{-})
\end{align*}
We also have
\begin{equation*}
 h_r^{-1} =(h_r^{-})^{-1}(h^{-})^{-1}
\quand
g = t(h^{-})g^{-}
\end{equation*}
\end{comment}
With this new notation we can rewrite $\xi''$ as
\begin{align*}
\xi'' &=\rho_n''\ast \gamma_n'' \ast ... \ast \gamma_2'' \ast \rho_1'' \ast \gamma_1''\ast \rho_0''
\\&=R(\rho^{-1},g_r) \ast R(\rho_n \ast ... \ast \rho_1,g_r) \ast  R(\id_{\gamma_1(t_1)},(h_r^{-1},g)) \ast R(\gamma_1,g)\ast R(\rho_0,g)\circ R(\id_{s(\xi)},(h^{-1},t(h)g))
\\&= R((\rho^{-})^{-1},g_r^{-}) \ast R(\xi^{-},g_r^{-}) \ast  R(\id_{s(\xi^{-})},((h_r^{-})^{-1}(h^{-})^{-1},t(h^{-})g^{-}))\\&\qquad\ast R(\gamma_1,g)\ast R(\rho_0,g)\circ R(\id_{s(\xi)},(h^{-1},t(h)g))\text{.}
\end{align*}
Now the first line of the last result is precisely in the form where we can apply the induction hypothesis; thus, we have
\begin{align*}
\xi'' &\sim R((\rho^{-})^{-1},((h_r^{-})^{-1},g^{-}))\ast R(\xi^{-},g^{-})\ast  R(\id_{s(\xi^{-})},((h^{-})^{-1},t(h^{-})g^{-})) \\&\qquad\ast R(\gamma_1,g)\ast R(\rho_0,g)\circ R(\id_{s(\xi)},(h^{-1},t(h)g))
\\ &\sim R( \rho^{-1},(h_r^{-1},g))\ast R(\rho_n \ast ... \ast \rho_1,g)
\\&\qquad\ast  R(\id_{\gamma_1(1)},(1,g)) \ast R(\gamma_1,g)\ast R(\rho_0,g)\circ R(\id_{s(\xi)},(h^{-1},t(h)g))
\\ &\sim R( \rho^{-1},(h_r^{-1},g))\ast R(\xi,g) \ast R(\id_{s(\xi)},(h^{-1},t(h)g))\text{.}
\end{align*}
This is \cref{eq:actioncomp:right}.
\end{proof}

\subsubsection{Step 3: Direct limit}

So far we have worked relative to the fixed subdivision $t\in T_n$ of the interval; the next step is to get rid of this parameter.  
The set $T := \bigsqcup_{n\in \N} T_n$ is a directed set, where $t \leq t'$ if $\{t_i\} \subset \{t_i'\}$ as subsets of $\R$. For $t\leq t'$ we have a map
\begin{equation*}
f_{t,t'}: F_{\gamma}(t) \to F_{\gamma}(t')
\end{equation*}
defined by adding identities $\rho_i=\id$ and splitting  $\gamma_i$ in two parts, at all additional points. These maps give $\{F_{\gamma}(t)\}_{t\in T}$  the structure of  a direct system of sets. Its direct limit is denoted by $F_{\gamma}$. It is straightforward to see that the anchors $\alpha_r$ and $\alpha_l$, the actions of \cref{lem:laction,lem:action}, and the $\mor\Gamma$-action of \cref{lem:gammaaction}  descent to $F_{\gamma}$.

\subsubsection{Step 4: Smooth structure}

Next we equip $F_{\gamma}$ with the structure of a smooth manifold. 
Consider a point $p_0\in \inf P_x$ and choose an element $\xi_0\in F_{\gamma}$ with $\alpha_l(\xi_0)=p_0$. Such an element $\xi_0$ exists: choose $t\in T$ such that there exist sections $\sigma_i: U_i \to \ob{\inf P}$ defined on open sets $U_i$ with $\gamma([t_{i-1},t_i]) \subset U_i$. By \reftransitionspan\ there exist $\rho\in \mor{\inf P}$ and $g\in G$ such that $t(\rho) = \sigma_1(x)$ and $R(s(\rho),g^{-1})=p_0$. We set $\rho_0 := R(\rho,g^{-1})$ and $\gamma_1' := R(\sigma_1(\gamma|_{[t_0,t_1]}),g^{-1})$. \begin{comment}
We have $s(\rho_0)=R(s(\rho),g^{-1})=p_0$ and $t(\rho_0)=R(t(\rho),g^{-1})=R(\sigma_1(x),g^{-1})=\gamma_1'(t_0)$.
\end{comment}
By \cref{lem:obhorexists} there exists $\gamma_1$ such that $\pi(\gamma_1)=\pi(\gamma_1')$ and $\gamma_1(t_0)=\gamma_1'(t_0)$. 
We set $p_1 := \gamma_1(t_1)$ and proceed in the same way until $i=n$. We end up with an element $\xi_0 = \rho_n \ast ... \ast \gamma_1 \ast \rho_0$ with $\alpha_l(\xi_0)=s(\rho_0)=p_0$.

Given $\xi_0$ we construct an open neighborhood $U$ of $p_0$ with a local section against $\alpha_l$.
Choose an open neighborhood $\tilde U \subset \ob{\inf P} \times_M \ob{\inf P}$ of $(p_0,p_0)$ together with a transition span $\rho$ and a transition function $g$ (\reftransitionspan). 
\begin{comment}
That is, $t(\rho(p_0,p))=p_0$ and $R(s(\rho(p_0,p)),g(p_0,p)^{-1})=p$.
\end{comment}
We consider  the smooth map $i_{p_0}: \ob{\inf P_x} \to \ob{\inf P} \times_M \ob{\inf P}$ with $i_{p_0}(p):=(p_0,p)$, and define  $U := i_{p_0}^{-1}(\tilde U) \subset \ob{\inf P_x}$. We define a map $\sigma_{\xi_0,\rho,g}: U \to F_{\gamma}$ by setting
\begin{equation*}
\sigma_{\xi_0,\rho,g}(p):= (\rho(p_0,p)^{-1}\circ \xi_0)\cdot (1,g(p_0,p)^{-1})\text{.}
\end{equation*}
\begin{comment}
Explicitly, for $\xi_0 = \rho_n \ast ... \ast \gamma_1 \ast \rho_0$ this is an element $\sigma(p)=\rho_n' \ast ... \ast \gamma_1' \ast \rho_0'$  with $\rho_0' = R(\rho_0 \circ \rho(p_0,p),g(p_0,p)^{-1})$ and $\rho_i':=R(\rho_i,g(p_0,p)^{-1})$ and $\gamma_i':=R(\gamma_i,g(p_0,p)^{-1})$ for $1\leq i \leq n$. 
\end{comment}
By \cref{lem:laction,lem:gammaaction} we have $\alpha_l(\sigma_{\xi_0,\rho,g}(p))=p$, i.e. $\sigma_{\xi_0, \rho,g}$ is a section against $\alpha_l$. It determines a map 
\begin{equation*}
\phi_{\xi_0,\rho,g}:U \ttimes{\alpha_r\circ \sigma_{\xi_0,\rho,g}}{t} \mor{\inf P_y} \to \alpha_l^{-1}(U):(p,\tilde \rho)\mapsto \sigma_{\xi_0,\rho,g}(p) \circ \tilde \rho\text{,}
\end{equation*}
which is a bijection due to \cref{lem:action}.

\begin{lemma}
\label{lem:smoothstructure}
There exists a unique smooth manifold structure on $F_{\gamma}$  such  that  all bijections $\phi_{\xi_0,\rho,g}$ are smooth.
\end{lemma}

\begin{proof}
The bijections $\phi_{\xi_0,\rho,g}$ induce a topology on $F_{\gamma}$, which is  Hausdorff and second countable. It remains to prove that the \quot{transition functions} are smooth.

We write $\sigma := \sigma_{\xi_0,\rho,g}$ and 
consider another section $\sigma':=\sigma_{\xi_0',\rho',g'}: U' \to F_{\gamma}$ constructed in the same way around $p_0' = \alpha_l(\xi_0')$, such that   $W:=U \cap U'$ is non-empty. 
The transition function is the unique map $\tilde \rho: W \to \mor{\inf P}$ with 
\begin{equation}
\label{eq:rhox}
\sigma_{\xi_0',\rho',g'}(p)=\sigma_{\xi_0,\rho,g}(p) \circ \tilde \rho(p)
\end{equation}
for all $p\in W$.
\begin{comment}
Existence: suppose $[\xi_1],[\xi_2]\in F^{p}_{\gamma}$, with $\xi_1\in F_{\gamma}(t_1)$ and $\xi_2 \in F_{\gamma}(t_2)$ and $\alpha_l(\xi_1)=\alpha_l(\xi_2)$. Choose $t$ such that $t_1\leq t$ and $t_2\leq t$. By \cref{lem:action} there exists $\rho$ such that $f_{t_1,t}(\xi_1)=f_{t_2,t}(\xi_2) \circ \rho$. Thus, $[\xi_1]=[\xi_2]\circ \rho$. 

Uniqueness: suppose $[\xi]\in F_{\gamma}$, with $\xi\in F^{p}_{\gamma}(t)$, and suppose $[\xi] \circ \rho= [\xi]$. Thus, there exists $t'$ such that $f_{t,t'}(\xi) \circ \rho= f_{t,t'}(\xi)$ in $F^{p}_{\gamma}(t')$. By \cref{lem:action} we have $\rho=\id$. 
\end{comment}
In order to show  that $\tilde \rho$ is smooth, we compute it explicitly. We fix $q\in W$. By \cref{lem:action} there exists a unique $\tilde\rho_q\in \mor{\inf P}$ such that $\sigma(q)=\sigma'(q)\circ \tilde \rho_q$.  With \refactionfullyfaithful\  there exists a smooth map $h:W \to H$ such that
\begin{comment}
Indeed, consider the smooth maps
\begin{align*}
\rho_1&: W \to \mor{\inf P}: p \mapsto \rho(p_0,p)
\\
\rho_2&: W \to \mor{\inf P}: p \mapsto \rho(p_0,q)\circ R(\rho(q,p),g(p_0,q))
\\
g_1&: W \to G: p \mapsto g(p_0,p)^{-1}g(q,p)g(p_0,q)
\end{align*}
and $g_2 := 1$. By \cite[Lemma 3.1.4]{Waldorf2016} there exists a smooth map $h:W \to H$ such that  $R(\rho_1,(h,g_1))=\rho_2$ and $t(h)g_1=1$
\end{comment}
\begin{align}
\label{eq:smooth:1}
R(\rho(p_0,p),(h(p),g(p_0,p)^{-1}g(q,p)g(p_0,q)))&=\rho(p_0,q)\circ R(\rho(q,p),g(p_0,q))
\\
\label{eq:smooth:2}
t(h(p))g(p_0,p)^{-1}g(q,p)g(p_0,q)&=1
\end{align}
for all $p\in W$.
The definition of $\sigma$ and \cref{eq:smooth:1} imply that
\begin{equation}
\label{eq:smooth:5}
\sigma(p) =(R(\rho(q,p)^{-1},(\alpha(g(p_0,q),h(p)),1)) \circ \sigma'(q)\circ \tilde \rho_q )\cdot (1,g(p_0,q)g(p_0,p)^{-1})\text{.}
\end{equation}
\begin{comment}
Indeed, we have
\begin{align*}
\sigma(p) &= (\rho(p_0,p)^{-1}\circ \xi_0)\cdot (1,g(p_0,p)^{-1})  
\\&=(R(\rho(q,p)^{-1},(\alpha(g(p_0,q),h(p)),g(p_0,q))) \circ \rho(p_0,q)^{-1}\circ \xi_0)\cdot (1,g(p_0,p)^{-1})
\\&=(R(\rho(q,p)^{-1},(\alpha(g(p_0,q),h(p)),g(p_0,q))) \circ \rho(p_0,q)^{-1}\circ \xi_0)\cdot (1,g(p_0,q)^{-1}) \cdot (1,g(p_0,q)g(p_0,p)^{-1})
\\&=(R(\rho(q,p)^{-1},(\alpha(g(p_0,q),h(p)),g(p_0,q))\cdot (1,g(p_0,q)^{-1})) \circ \sigma(q) )\cdot (1,g(p_0,q)g(p_0,p)^{-1})
\\&=(R(\rho(q,p)^{-1},(\alpha(g(p_0,q),h(p)),1)) \circ \sigma(q) )\cdot (1,g(p_0,q)g(p_0,p)^{-1})
\end{align*}
Here we have used that \cref{eq:smooth:1} implies
\begin{align*}
\rho(p_0,p)&=R(\rho(p_0,q)\circ R(\rho(q,p),g(p_0,q)),(\alpha(g_1^{-1},h(p)^{-1}),g_1^{-1}))
\\
&=R(\rho(p_0,q)\circ R(\rho(q,p),g(p_0,q)),(h(p)^{-1},t(h(p))))\text{.}
\end{align*}
and thus
\begin{align*}
\rho(p_0,p)^{-1}&=R(R(\rho(q,p)^{-1},g(p_0,q)) \circ \rho(p_0,q)^{-1},(h(p),1))
\\&=R(R(\rho(q,p)^{-1},(1,g(p_0,q))) \circ \rho(p_0,q)^{-1},(h(p),1)\circ (1,1))
\\&=R(\rho(q,p)^{-1},(1,g(p_0,q))\cdot (h(p),1)) \circ \rho(p_0,q)^{-1}
\\&=R(\rho(q,p)^{-1},(\alpha(g(p_0,q),h(p)),g(p_0,q))) \circ \rho(p_0,q)^{-1}
\end{align*}

\end{comment}
Analogously, for the primed quantities, there exists  a smooth map $h': W \to H$ with
\begin{comment}
\begin{equation}
\label{eq:smooth:4}
t(h'(p))g'(p_0',p)^{-1}g'(q,p)g'(p_0',q)=1
\end{equation}
and
\end{comment}
\begin{equation}
\label{eq:smooth:6}
\sigma'(p)=(R(\rho'(q,p)^{-1},(\alpha(g'(p_0',q),h'(p)),1)) \circ \sigma'(q) )\cdot (1,g'(p_0',q)g'(p_0',p)^{-1})\text{.}
\end{equation}
\begin{comment}
\begin{align*}
R(\rho'(p_0',p),(h'(p),g'(p_0',p)^{-1}g'(q,p)g'(p_0',q)))&=\rho'(p_0',q)\circ R(\rho'(q,p),g'(p_0',q))
\\
t(h'(p))g'(p_0',p)^{-1}g'(q,p)g'(p_0',q)&=1\text{.}
\end{align*}
\end{comment}
Again by \refactionfullyfaithful\ there exists another smooth map $\eta:W \to H$ such that  
\begin{comment}
Next we set $\rho_1(p) := \rho(q,p)$, $\rho_2(p) := \rho'(q,p)$, as well as $g_2 := 1$ and $g_1(p) := g(q,p)^{-1}g'(q,p)$. By \cite[Lemma 3.1.4]{Waldorf2016} there exists a smooth map $\eta:W \to H$ such that  $R(\rho_1,(\eta,g_1))=\rho_2$ and $t(\eta)g_1=1$.
\end{comment}
\begin{equation*}
R(\rho(q,p),(\eta(p),g(q,p)^{-1}g'(q,p)))=\rho'(q,p)\text{.}
\end{equation*}
\begin{comment}
and
\begin{equation*}
t(\eta(p))g(q,p)^{-1}g'(q,p)=1\text{.}
\end{equation*}
\end{comment}
Solving for $\rho(q,p)^{-1}$
\begin{comment}
which gives
\begin{align*}
\rho(q,p)^{-1} &= R(\rho'(q,p),(\eta(p)^{-1},t(\eta)))^{-1} =R(\rho'(q,p)^{-1},(\eta(p),1))\text{.}
\end{align*}
\end{comment}
and substituting in \cref{eq:smooth:5} gives
\begin{align*}
\sigma(p) 
= (R(\rho'(q,p)^{-1},(\eta(p)\alpha(g(p_0,q),h(p)),1)) \circ \sigma'(q)\circ \tilde \rho_q  )\cdot (1,g(p_0,q)g(p_0,p)^{-1})\text{.}
\end{align*}
Forcing \cref{eq:smooth:6} into this expression yields
\begin{equation*}
\sigma(p)=(\sigma'(p)\circ R(\tilde \rho_q ,g'(p_0',q)g'(p_0',p)^{-1}))\cdot (h(p),t(h(p))^{-1})\text{,}
\end{equation*}
where $h (p):=\alpha(g'(p_0',p)g'(p_0',q)^{-1},\alpha(g'(p_0',q),h'(p)^{-1})\eta(p)\alpha(g(p_0,q),h(p)))$.
\begin{comment}
We obtain
\begin{align*}
\sigma(p)&= (R(\rho'(q,p)^{-1},(\alpha(g'(p_0',q),h'(p)),1)\cdot (\alpha(g'(p_0',q),h'(p)^{-1})\eta(p)\alpha(g(p_0,q),h(p)),1))\\&\qquad \circ \sigma'(q)\circ \tilde \rho_q )\cdot (1,g(p_0,q)g(p_0,p)^{-1})  
\\&= (R(\rho'(q,p)^{-1},(\alpha(g'(p_0',q),h'(p)),1)) \circ \sigma'(q)\circ \tilde \rho_q )
\\&\qquad\cdot (\alpha(g'(p_0',q),h'(p)^{-1})\eta(p)\alpha(g(p_0,q),h(p)),g(p_0,q)g(p_0,p)^{-1}) \\&= (R(\rho'(q,p)^{-1},(\alpha(g'(p_0',q),h'(p)),1)) \circ \sigma'(q)\circ \tilde \rho_q )
 \cdot (1,g'(p_0',q)g'(p_0',p)^{-1}) \\&\qquad \cdot (1,g'(p_0',p)g'(p_0',q)^{-1}) \cdot (\alpha(g'(p_0',q),h'(p)^{-1})\eta(p)\alpha(g(p_0,q),h(p)),g(p_0,q)g(p_0,p)^{-1}) \\&= (\sigma'(p)\circ R(\tilde \rho_q ,g'(p_0',q)g'(p_0',p)^{-1})) \\&\qquad \cdot (\alpha(g'(p_0',p)g'(p_0',q)^{-1},\alpha(g'(p_0',q),h'(p)^{-1})\eta(p)\alpha(g(p_0,q),h(p))),g'(p_0',p)g'(p_0',q)^{-1}g(p_0,q)g(p_0,p)^{-1})
\end{align*} 
We claim that the target of the acting element is $1\in G$:
\begin{align*}
&\mquad t(\alpha(g'(p_0',p)g'(p_0',q)^{-1},\alpha(g'(p_0',q),h'(p)^{-1})\eta(p)\alpha(g(p_0,q),h(p))))
\\&= g'(p_0',p)g'(p_0',q)^{-1}t(\alpha(g'(p_0',q),h'(p)^{-1})\eta(p)\alpha(g(p_0,q),h(p)))g'(p_0',q)g'(p_0',p)^{-1} \\&= g'(p_0',p)t(h'(p)^{-1})g'(p_0',q)^{-1}t(\eta(p))g(p_0,q)t(h(p))g(p_0,q)^{-1}g'(p_0',q)g'(p_0',p)^{-1} \\&= g(p_0,p)g(p_0,q)^{-1}g'(p_0',q)g'(p_0',p)^{-1}  
\end{align*} 
\end{comment}
Using the compatibility of the $\Gamma$-action with the right $\inf P_y$-action, we can write
\begin{align*}
\sigma(p)&= \sigma'(p)\circ (R(\tilde \rho_q ,(1,g'(p_0',q)g'(p_0',p)^{-1}) \cdot (h(p),t(h(p))^{-1}))\text{.}
\end{align*}
\begin{comment}
We obtain
\begin{equation*}
\tilde \rho(p):=R(\tilde \rho_q ,(\alpha(g'(p_0',q),h'(p)^{-1})\eta(p)\alpha(g(p_0,q),h(p)),g'(p_0',p)g'(p_0',q)^{-1}g(p_0,q)g(p_0,p)^{-1}))
\end{equation*}
which depends smoothly on $p$.
\end{comment}
This is an explicit expression for the transition function $\tilde\rho$, and it  depends smoothly on $p$. 
\end{proof}

\begin{proposition}
\label{prop:smoothpt}
The smooth manifold 
$F_{\gamma}$  together with the anchor maps $\alpha_l$ and $\alpha_r$, the actions of  \cref{lem:laction,lem:action}, and the $\mor{\Gamma}$-action of \cref{lem:gammaaction}, define a $\Gamma$-equivariant anafunctor $F_{\gamma}: \inf P_{x} \to \inf P_{y}$,
\end{proposition}

\begin{proof}
In a chart $\phi$, we have  $\alpha_l \circ \phi = \pr_1$ and $\alpha_r \circ \phi=s \circ \pr_2$ hence $\alpha_l$ and $\alpha_r$ are smooth and submersions. The right $\inf P_y$-action is
in a chart
$\phi(p,\tilde \rho)\circ \rho' =\sigma(p)\circ\tilde  \rho\circ \rho'=\phi(p,\tilde \rho\circ \rho')$,
hence it is smooth. We consider the smooth bijection $F_{\gamma} \ttimes{\alpha_r}t \mor{\inf P_y} \to F_{\gamma} \ttimes{\alpha_l}{\alpha_l} F_{\gamma}$; that its inverse is smooth is -- in charts -- precisely the smoothness of the transition function $\tilde\rho$ of \cref{lem:smoothstructure}. Thus, $F_{\gamma}$ is a principal $\inf P_y$-bundle over $\inf P_x$.

For the left $\inf P_x$-action, consider a section $\sigma_{\xi_0,\rho,g}$ defined in an open neighborhood $U$ around $p_0$, and a morphism $\rho_0\in \mor{\inf P_x}$ such that $s(\rho_0) = p_0$. 
We set $p_0' := t(\rho_0)$ and $\xi_0' := \rho_0 \circ \xi_0$. Choose a transition span $\rho':\tilde U' \to \mor{\inf P_x}$ with transition function $g'$ defined in an open neighborhood $\tilde U' \subset \ob{\inf P_x} \times_M \ob{\inf P_x}$ of $(p_0,p_0)$. 
\begin{comment}
That is, $t(\rho'(p_0',p'))=p_0'$ and $R(s(\rho'(p_0',p')),g'(p_0',p')^{-1})=p'$.
\end{comment}
This makes up another section $\sigma_{\xi_0',\rho',g'}$ defined in an open neighborhood $U'$ of $p_0'$.
Let $V := s^{-1}(U) \cap t^{-1}(U') \subset \mor{\inf P_x}$. Using \refactionfullyfaithful\  there exists a unique smooth map $h:V \to H$ such that
\begin{align}
\label{eq:smoothness:h1}
R(\rho'(p_0',t(\eta))^{-1} \circ \rho_0,(h(\eta),g(p_0,s(\eta))^{-1}))&= \eta \circ  R(\rho(p_0,s(\eta))^{-1},(1,g(p_0,s(\eta))^{-1}))
\\
\label{eq:smoothness:h2}
t(h(\eta))g(p_0,s(\eta))^{-1}&=g'(p_0',t(\eta))^{-1}\text{,}
\end{align}
for all $\eta\in V$.
\begin{comment}
Indeed, for  $\eta\in V$ we define
\begin{equation*}
\rho_1 :=\rho'(p_0',t(\eta))^{-1} \circ \rho_0 
\quand
\rho_2:= \eta \circ  R(\rho(p_0,s(\eta))^{-1},(1,g(p_0,s(\eta))^{-1}))
\end{equation*}
\begin{comment}
These compositions are well-defined:
\begin{align*}
t(\rho_0) &= p_0' =t(\rho'(p_0',t(\eta)))=s(\rho'(p_0',t(\eta))^{-1})
\\
s(\eta)&=R(s(\rho(p_0,s(\eta)),g(p_0,s(\eta))^{-1})=t(R(\rho(p_0,s(\eta))^{-1},g(p_0,s(\eta))^{-1}))
\end{align*}
\end{comment}
Then we have
\begin{align*}
s(\rho_2) &=R(t(\rho(p_0,s(\eta))),g(p_0,s(\eta))^{-1})=R(p_0,g(p_0,s(\eta))^{-1})=R(s(\rho_1),g(p_0,s(\eta))^{-1})
\\
t(\rho_2) &=t(\eta)=R(s(\rho'(p_0',t(\eta))),g'(p_0',t(\eta))^{-1})=R(t(\rho_1),g'(p_0',t(\eta))^{-1})\text{.}
\end{align*}
By \refactionfullyfaithful\
 and there exists a unique $h_{\eta}\in H$ such that
\begin{align*}
R(\rho_1,(h_{\eta},g(p_0,s(\eta))^{-1}))=\rho_2
\quand
t(h_{\eta})g(p_0,s(\eta))^{-1}=g'(p_0',t(\eta))^{-1}\text{,}
\end{align*}
and the map $h: V \to H: \eta \mapsto h_{\eta}$ is smooth. 
Using these formulas, one can check that the left action is given in charts by
$\eta \circ \phi_{\xi_0,\rho,g}(p,\tilde\rho) =\phi_{\xi_0',\rho',g'}(t(\eta),R(\tilde\rho,  (h(\eta),g(p_0,p)^{-1})))$, 
and is hence smooth.  
\begin{comment}
Indeed, the map
\begin{equation*}
\alxydim{@C=0.5cm}{V \ttimes{s}{\id} U \ttimes{\alpha_r\circ \sigma_{\xi_0,\rho,g}}{t} \mor{\inf P_y} \ar[rr]^-{\id \times \phi_{\xi_0,\rho,g}} && V \ttimes{s}{\alpha_l} \alpha_l^{-1}(U) \ar[r]^-{\circ} & \alpha_l^{-1}(U') \ar[r]^-{\phi^{-1}_{\xi_0',\rho',g'}} & U' \ttimes{\alpha_r\circ \sigma_{\xi_0',\rho',g'}}{t} \mor{\inf P_y}}
\end{equation*}
is
\begin{align*}
(\eta,p,\tilde\rho) &\mapsto (\eta,\sigma_{\xi_0,\rho,g}(p) \circ \tilde \rho) 
\\&\mapsto \eta \circ\sigma_{\xi_0,\rho,g}(p) \circ \tilde \rho 
\\&\qquad=\eta \circ((\rho(p_0,p)^{-1}\circ \xi_0)\cdot (1,g(p_0,p)^{-1})) \circ \tilde \rho 
\\&\qquad=\eta \circ(\rho(p_0,p)^{-1}\cdot (1,g(p_0,p)^{-1}))\circ (\xi_0\cdot (1,g(p_0,p)^{-1})) \circ \tilde \rho 
\\&\qquad\eqcref{eq:smoothness:h1} R(\rho'(p_0',p')^{-1} \circ \rho_0,(h(\eta),g(p_0,p)^{-1}))\circ (\xi_0\cdot (1,g(p_0,p)^{-1})) \circ \tilde \rho
\\&\qquad=((\rho'(p_0',p')^{-1} \circ \rho_0 \circ \xi_0) \cdot (h(\eta),g(p_0,p)^{-1})) \circ \tilde \rho
\\&\qquad \eqcref{eq:smoothness:h2} ((\rho'(p_0',p')^{-1}\circ \rho_0 \circ \xi_0)\cdot (1,g'(p_0',p')^{-1}))\circ R(\tilde\rho,  (h(\eta),g(p_0,p)^{-1})) \\&\qquad= ((\rho'(p_0',p')^{-1}\circ \xi_0')\cdot (1,g'(p_0',p')^{-1}))\circ \tilde \rho' 
\\&\qquad= \sigma_{\xi_0',\rho',g'}(p') \circ \tilde \rho' 
\\&\mapsto (p',\tilde\rho')
\end{align*}
with $p' := t(\eta)$ and $\tilde\rho' := R(\tilde\rho,  (h(\eta),g(p_0,p)^{-1}))$.  In this calculation, we had $p=s(\eta)$. 
\end{comment}

It remains to verify the smoothness of the $\mor{\Gamma}$-action. Consider again a section $\sigma_{\xi_0,\rho,g}$ defined in an open neighborhood $U$ of a point $p_0$,  and a morphism $(h,g) \in \mor{\Gamma}$. We recall that $\rho$ is a transition span defined in an open set $\tilde U \subset \ob{\inf P_x} \times_M \ob{\inf P_x}$, and that $U = i_{p_0}^{-1}(\tilde U)$. Let $g_0 := t(h) g$ and $p_0' := R(p_0,g_0)$. Choose open neighborhoods $V \subset  \ob{\Gamma}$ of $g_0$ and $\tilde U' \subset \ob{\inf P_x} \times_M \ob{\inf P_x}$ of $(p_0',p_0')$ such that $(R(x',\tilde g^{-1}),R(y',\tilde g^{-1})) \in \tilde U$ and $(p_0,R(p,\tilde g g_0^{-1}))\in \tilde U$ for all $(p_0,p)\in \tilde U'$, $(x',y')\in  \tilde U'$ and $\tilde g\in V$. 
\begin{comment}
This makes sense, since
\begin{equation*}
(R(p_0',g_0^{-1}),R(p_0',g_0^{-1}))=(p_0,p_0)
\quand
(p_0,R(p,g_0 g_0^{-1}))=(p_0,p)
\text{.}
\end{equation*}
\end{comment}
Using \refactionfullyfaithful\ one can construct a  smooth map $h:U \times V \to H$ such that parameterizes the dependence in the second argument of $\rho$ under the action be group elements of the form $\tilde gg_0^{-1}$ for $\tilde g\in V$, in the sense that
\begin{align}
\label{eq:smoothness:h3}
R( \rho(p_0,p),(h(p,\tilde g),g(p_0,p)^{-1}\tilde g g_0^{-1}g(p_0,R(p,\tilde g g_0^{-1}))))&= \rho(p_0,R(p,\tilde gg_0^{-1}))
\\
\label{eq:smoothness:h4}
t(h(p,\tilde g))g(p_0,p)^{-1}\tilde g g_0^{-1}g(p_0,R(p,\tilde g g_0^{-1}))&=1\text{.}
\end{align}
\begin{comment}
Indeed, we set
\begin{equation*}
\rho_2 := \rho(p_0,R(p,\tilde gg_0^{-1}))
\quand
\rho_1 := \rho(p_0,p)\text{.}
\end{equation*}
We have
\begin{align*}
t(\rho_2) &=p_0=t(\rho_1)
\\
s(\rho_2) &= R(p,\tilde g g_0^{-1}g(p_0,R(p,\tilde g g_0^{-1}))) = R(s(\rho_1),g(p_0,p)^{-1}\tilde g g_0^{-1}g(p_0,R(p,\tilde g g_0^{-1})))
\end{align*}
Thus we set $g_1 := g(p_0,p)^{-1}\tilde g g_0^{-1}g(p_0,R(p,\tilde g g_0^{-1}))$. Then, by the lemma there exists a unique $h \in H$  such that
\begin{align*}
R(\rho_1,(h,g_1))=\rho_2
\\
t(h)g_1=1\text{.}
\end{align*}

\end{comment}
Next we translate the transition span $\rho$ along $g_0$, and obtain  another transition span $\rho': \tilde U' \to \mor{\inf P_x}$ with  transition function $g': \tilde U' \to G$  by setting 
\begin{align*}
\rho'(x',y') &:= R(\rho(R(x',g_0^{-1}),R(y',g_0^{-1})),(h,g))\text{.}
\\
g'(x',y') &:=g_0^{-1}g(R(x',g_0^{-1}),R(y',g_0^{-1}))g
\end{align*}
\begin{comment}
Indeed,
\begin{align*}
t(\rho'(x',y')) &= R(t(\rho(R(x',g_0^{-1}),R(y',g_0^{-1}))),g_0)=x'
\\
s(\rho'(x',y')) &=R(s(\rho(R(x',g_0^{-1}),R(y',g_0^{-1}))),g)
\\&=R(R(R(y',g_0^{-1}),g(R(x',g_0^{-1}),R(y',g_0^{-1}))),g)
\\&=R(y',g_0^{-1}g(R(x',g_0^{-1}),R(y',g_0^{-1})) g)
\\&=R(y',g'(x',y'))
\end{align*}
\end{comment}
We set $\xi_0' := \xi_0 \cdot (h,g)$; this satisfies $\alpha_l(\xi_0')=R(\alpha_l(\xi_0),g_0)=p_0'$. Now we have defined another section $\sigma_{\xi_0',\rho',g'}$  in a neighborhood $U' :=i_{p_0'}^{-1}(\tilde U')$ of $p_0'$. We let $\tilde V := t^{-1}(V) \subset \mor{\Gamma}$. Now, the action in charts of $(\tilde h,\tilde g) \in \tilde V$ is given by
$\phi_{\xi_0,\rho,g}(p,\tilde\rho) \cdot (\tilde h,\tilde g) = \phi_{\xi_0',\rho',g'}(p',\tilde\rho')$
with $p' :=R(p,\tilde g_0)$, and
\begin{equation*}
\tilde \rho' :=R( \tilde \rho, (\tilde h\alpha(\tilde gg_0^{-1}g(p_0,R(p,\tilde g_0g_0^{-1}))^{-1},h(p,\tilde g_0)),\tilde gg_0^{-1}g(p_0,R(p,\tilde g_0g_0^{-1}))^{-1}))\text{,}
\end{equation*}
where $\tilde g_0 := t(\tilde h)\tilde g$. 
Both expressions depend smoothly on $p$, $\tilde\rho$, $\tilde h$, and $\tilde g$; this shows the smoothness.
\end{proof}

\subsection{Compatibility with path composition}

\label{sec:compcomppaths}

In this section we describe the compatibility of the parallel transport along paths with the composition of paths. In the transport 2-functor formalism described in \cref{sec:orga} they constitute the functoriality of the 2-functor on the level of 1-morphisms.

Before we come to path composition, we look at the constant path $\id_x$ at $x\in M$. We define a $\Gamma$-equivariant transformation \begin{equation*}
u_x:\id_{\inf P_x} \Rightarrow F_{\id_x}\text{,} \end{equation*}
which expresses the fact that the parallel transport along a constant path $\id_x$ is canonically 2-isomorphic to the identity 2-functor on the fibre $\inf P_x$.
We define $u_x$
using \cref{rem:indtrans}. The underlying smooth map  $\tilde u_x: \ob{\inf P_x} \to F_{\id_x}$ is defined by $\tilde u_x(p) := \id_p$, where $\id_p$ denotes the constant path at $p$. Verifying \cref{T1*,T2*,T3*} are straightforward calculations; alone in \cref{T2*} one has to use once the equivalence relation on $F_{\id_x}$.
\begin{comment}
\cref{T1*} is obvious. For \cref{T2*}:
for well-definedness, we have $s(\rho)=p=t(\rho')$. Then, 
\begin{equation*}
\rho \circ \tilde u_x(p) \circ \rho' =\rho'^{-1}\ast \id_p\ast \rho^{-1} \sim\rho'^{-1} \circ \rho^{-1} \ast \id_{t(\rho)} =\tilde u_x(t(\rho))\circ \rho\circ \rho'
\end{equation*}
where the equivalence relation was applied to the constant path $t \mapsto \rho$.
For \cref{T3*}:
\begin{equation*}
\tilde u_x(R(p,g))=\id_{R(p,g)}=\id_p\cdot \id_g=\tilde u_x(p)\cdot \id_g
\end{equation*}
\end{comment}

Two paths $\gamma_1,\gamma_2:[0,1] \to M$ are \emph{composable}, if $\gamma_1(1)=\gamma_2(0)$ and the usual path concatenation $\gamma_2 \ast \gamma_1$ is smooth.  In the following, we will often assume composability without explicitly mentioning it; at no place we use piecewise smoothness or any other regularity.    
We construct a transformation
\begin{equation*}
c_{\gamma_1,\gamma_2}: F_{\gamma_2} \circ F_{\gamma_1} \Rightarrow F_{\gamma_2\ast \gamma_1}\text{,}
\end{equation*}
which expresses the fact that the parallel transport along a composite path is canonically 2-isomorphic to the composition of the separate parallel transports. 
In order to define $c_{\gamma_1,\gamma_2}$ we consider $\xi_1\in F_{\gamma_1}$ and $\xi_2\in F_{\gamma_2}$ such that $\alpha_r(\xi_1)=\alpha_l(\xi_2)$, i.e. $(\xi_1,\xi_2)\in F_{\gamma_2} \circ F_{\gamma_1}$. Its image under $c_{\gamma_1,\gamma_2}$ is the element  $\xi_2 \ast \xi_1 \in F_{\gamma_2\ast \gamma_1}$ obtained by reparameterizing all path pieces, and composing the last morphism of $\xi_1$ with the first of $\xi_2$. 
\begin{comment}
The reparameterization is done in the canonical way, i.e. the ones of $\xi_1$ are transformed by $t \mapsto \frac{1}{2}t$, and the ones of $\xi_2$ are transformed by $t\mapsto \frac{1}{2}(1+t)$.
\end{comment}
This is obviously anchor-preserving and action-preserving, and preservation of the $\mor \Gamma$-action can easily be checked.
\begin{comment}
We check
\begin{align*}
c_{\gamma_1,\gamma_2}((\xi_1,\xi_2)\cdot (h,g)) &\eqtext{\cite[Remark 2.4.2]{Waldorf2016}}c_{\gamma_1,\gamma_2}(\xi_1\cdot (h,g),\xi_2\cdot (1,g))
\\&= c_{\gamma_1,\gamma_2}(\xi_1\cdot (1,g)\ast R(\rho_0,(h^{-1},t(h)g)),\xi_2\cdot (1,g))
\\&= (\xi_2\cdot (1,g)) \ast(\xi_1\cdot (1,g)\ast R(\rho_0,(h^{-1},t(h)g)))
\\&= (\xi_2 \ast\xi_1)\cdot (1,g)\ast R(\rho_0,(h^{-1},t(h)g))) \\&= (\xi_2 \ast\xi_1)\cdot (h,g) 
\\&=c_{\gamma_1,\gamma_2}(\xi_1,\xi_2)\cdot(h,g)
\end{align*}
\end{comment}

If $\gamma_1$, $\gamma_2$, and $\gamma_3$ are  paths, and  $\gamma_3 \ast (\gamma_2 \ast \gamma_1)$ and $(\gamma_3 \ast \gamma_2)\ast \gamma_1$ are again paths (i.e. smooth), then the canonical reparameterization between $\gamma_3 \ast (\gamma_2 \ast \gamma_1)$ and $(\gamma_3 \ast \gamma_2)\ast \gamma_1$ induces an obvious transformation $\alpha_{\gamma_1,\gamma_2,\gamma_3}: F_{\gamma_3  \ast (\gamma_2  \ast \gamma_1)} \Rightarrow F_{(\gamma_3  \ast \gamma_2) \ast \gamma_1}$.
The following coherence property follows then directly from the definition of $c_{\gamma_1,\gamma_2}$. 

\begin{proposition}
\label{lem:F3}
Let $\inf P$ be a principal $\Gamma$-2-bundle with connection.
Then, the following diagram commutes:
\begin{equation*}
\alxydim{@C=-1cm@R=\xyst}{ && F_{\gamma_3} \circ F_{\gamma_2} \circ F_{\gamma_1} \ar@{=>}[drr]^{c_{\gamma_2,\gamma_3}\circ \id} \ar@{=>}[dll]_{\id\circ c_{\gamma_1,\gamma_2}} && \\ F_{\gamma_3} \circ F_{\gamma_2 \ast \gamma_1} \ar@{=>}[dr]_{c_{\gamma_3,\gamma_2 \ast\gamma_1}} &&&& F_{\gamma_3 \ast \gamma_2} \circ F_{\gamma_1} \ar@{=>}[dl]^{c_{\gamma_3 \ast\gamma_2,\gamma_1}} \\ & F_{\gamma_3 \ast (\gamma_2  \ast\gamma_1)} \ar@{=>}[rr]_-{a_{\gamma_1,\gamma_2,\gamma_3}} & \hspace{3.2cm} & F_{(\gamma_3 \ast\gamma_2) \ast\gamma_1}}
\end{equation*}
\end{proposition}

Next we observe a compatibility condition between the transformations $c_{\gamma_1,\gamma_2}$ and $u_x$. For this purpose, we first  identify transformations $l_{\gamma}: F_{\gamma} \Rightarrow F_{\gamma\ast \id_x}$ and $r_{\gamma}:  F_{\gamma} \Rightarrow F_{\id_y\ast \gamma}$ associated to a path $\gamma:x \to y$.
Indeed, given $\xi\in F_{\gamma}$, we reparameterize all path pieces by $t \mapsto \frac{1}{2}+\frac{1}{2}t$, and add (using the formal composition $\ast$) the constant path $[0,\frac{1}{2}]\ni t\mapsto \alpha_l(\xi)$ at the beginning; this gives an element of $F_{\gamma\ast \id_x}$, and defines the transformation $l_{\gamma}$. The transformation $r_{\gamma}$ is defined analogously. The following result follows directly from the definitions.

\begin{proposition}
\label{lem:F4}
Let $\inf P$ be a principal $\Gamma$-2-bundle with connection.
For every path $\gamma:x \to y$ the following diagrams
commute:
\begin{equation*}
\alxydim{@=\xyst}{F_{\gamma}\circ \id_{\inf P_x} \ar@{=>}[d]_{\id\circ u_x} \ar@{=}[r]^-{} & F_{\gamma} \ar@{=>}[d]^{l_{\gamma}} \\ F_{\gamma}\circ F_{\id_x} \ar@{=>}[r]_{c_{\id_x,\gamma}}& F_{\gamma \ast \id_x}}
\quand
\alxydim{@=\xyst}{\id_{\inf P_y} \circ F_{\gamma} \ar@{=>}[d]_{u_y\circ \id} \ar@{=}[r]^-{} & F_{\gamma} \ar@{=>}[d]^{r_{\gamma}} \\ F_{\id_y}\circ F_{\gamma} \ar@{=>}[r]_{c_{\gamma,\id_y}}& F_{\id_y \ast \gamma}}
\end{equation*}
\end{proposition}

\begin{comment}
\begin{proof}
For the diagram on the left, consider $\xi\in F_{\gamma}$. Then, $(\id \circ u_x)(\xi)=(\id_{\alpha_l(\xi)},\xi)$ and then $c_{\id_x,\gamma}(\id_{\alpha_l(\xi)},\xi)=\xi\ast \id_{\alpha_l(\xi)}=l_{\gamma}(\xi)$ .
Thus, the diagram is commutative. The one on the right commutes analogously.
\end{proof}
\end{comment}

\subsection{Naturality with respect to  bundle morphisms}

\label{sec:natbundlemorph}

Suppose $J:\inf P \to \inf P'$ is a 1-morphism between principal $\Gamma$-2-bundles with connections $\Omega$ and $\Omega'$, respectively, equipped with a connective, connection-preserving $\Omega'$-pullback $\nu$. Let $\gamma: [0,1] \to M$ be a path with $x:=\gamma(0)$ and $y:=\gamma(1)$. We construct a $\Gamma$-equivariant transformation
\begin{equation*}
\alxydim{@=\xyst}{\inf P_x \ar[r]^{F_{\gamma}} \ar[d]_{J_x} & \inf P_y \ar@{=>}[dl]|*+{J_{\gamma}} \ar[d]^{J_y} \\ \inf P_x' \ar[r]_{F_{\gamma}'} & \inf P_y'}
\end{equation*}
relating the parallel transport $F_{\gamma}$ in $\inf P$ with the parallel transport $F'_{\gamma}$ in $\inf P'$. Involving the definition of composition of anafunctors (\refcompositionanafunctors)  $J_{\gamma}$ is induced by a map
\begin{equation*}
J_{\gamma} : F_{\gamma} \ttimes{\alpha_r}{\alpha_l} J_y \to J_x \ttimes{\alpha_r}{\alpha_l}F'_\gamma\text{,}
\end{equation*}
which we define in the following. We start with the following terminology:
a \emph{horizontal lift of $\xi\in F_{\gamma}$ to $J$}
is a collection $\tilde\xi =(n,t,\{\rho_i\}_{i=0}^{n},\{\tilde\gamma_i\}_{i=1}^{n})$ consisting of $n\in \N$, a subdivision $t\in T_n$, morphisms $\rho_i\in \mor{\inf P_x}$ and horizontal paths $\tilde\gamma_i:[t_{i-1},t_i]\to J$ such that
$\alpha_l(\tilde\gamma_i(t_i))=s(\rho_i)$  and $\alpha_l(\tilde\gamma_i(t_{i-1}))=t(\rho_{i-1})$ for $1\leq i \leq n$, and
\begin{equation*}
\xi\sim\rho_n\ast \alpha_l(\tilde\gamma_n)\ast \rho_{n-1} \ast ... \ast \alpha_l(\tilde\gamma_1)\ast \rho_0\text{.}
\end{equation*}
This means, in particular, that the paths $\alpha_l(\tilde\gamma_1)$ are horizontal.
It is easy to see, e.g. using \cref{lem:horF:b,lem:horF}, that every $\xi\in F_{\gamma}$  admits horizontal lifts. 
\begin{comment}
We can choose the representative $\xi=\rho_n \ast \gamma_n \ast ... \ast \rho_0$ such that, for $1\leq i \leq n$, each $\gamma_i:[t_{i-1},t_i] \to \ob{\inf P}$ maps into an open set $U_i \subset \ob{\inf P}$ that supports  sections $\sigma_i: U \to J$ against $\alpha_l: J \to \ob{\inf P}$. By \cref{lem:horF} there exist unique smooth maps $h_i: [t_{i-1},t_i] \to H$ with $h_i(t_{i-1})=1$ and $\tilde\gamma_i(t):= \sigma_i(\gamma_i(t)) \cdot (h_i(t),t(h_i(t))^{-1})$ horizontal in $J$ (using \cref{lem:horF:b}). \end{comment}

If $j\in J$ such that  $\alpha_r(\xi)=\alpha_l(j)$, then the \emph{$j$-target} of a horizontal lift $\tilde\xi$  is an element $(j',\xi')\in J_x \ttimes{\alpha_r}{\alpha_l}F'_\gamma$ defined in the following way. We set $\gamma_i' := \alpha_r(\tilde\gamma_i)$. Since $\tilde\gamma_i$ and $\alpha_l(\tilde\gamma_i)$ are horizontal, $\gamma_i'$ is horizontal by \cref{lem:horF:a}.  We proceed for an index $1\leq i <n$, and note that $\alpha_l(\rho_i \circ \tilde\gamma_i(t_i))=\alpha_l(\tilde\gamma_{i+1}(t_i))$.
\begin{comment}
These compositions are well-defined:
\begin{equation*}
s(\rho_i)=\gamma_i(t_i)=\alpha_l(\tilde\gamma_i(t_i) )
\end{equation*}
Indeed, for $1\leq i <n$ we have
\begin{equation*}
\alpha_l(\rho_i \circ \tilde\gamma_i(t_i)) = t(\rho_i)=\gamma_{i+1}(t_{i})=\alpha_l(\tilde\gamma_{i+1}(t_i))
\end{equation*}
\end{comment}
Since $\alpha_l:J \to \ob{\inf P}$ is a principal $\inf P'$-bundle, there exists a unique $\rho_i' \in \mor{\inf P'}$ such that $\rho_i \circ \tilde\gamma_i(t_i) = \tilde\gamma_{i+1}(t_i)\circ \rho_i'$, and we get $t(\rho_i')=\gamma'_{i+1}(t_i)$ and $s(\rho_i')=\gamma_i'(t_i)$. 
\begin{comment}
It satisfies $t(\rho_i')=\alpha_r(\tilde\gamma_{i+1}(t_i))=\gamma_{i+1}'(t_i)$ and $s(\rho_i')=\alpha_r(\rho_i \circ \tilde\gamma_i(t_i))=\alpha_r(\tilde\gamma_i(t_i))=\gamma_i'(t_i)$, and we have $\rho_0'= \id_{\tilde\gamma_1(0)}$. \end{comment}
The case $i=n$ is treated separately involving the element  $j$. We have $\alpha_l(\rho_n \circ \tilde\gamma_n(1))=\alpha_l(j)$.
\begin{comment}
Indeed,
\begin{equation*}
\alpha_l(\rho_n \circ \tilde\gamma_n(1))=t(\rho_n)=\alpha_r(\xi)=\alpha_l(j)\text{.}
\end{equation*}
\end{comment}
Hence, there exists a unique $\rho'_n \in \mor{\inf P'}$ such that $\rho_n \circ \tilde\gamma_n(1) = j\circ \rho_n'$, satisfying $s(\rho_n')=\gamma_n'(t_n)$. 
\begin{comment}  
Indeed,
\begin{equation*}
s(\rho_n')=\alpha_r(\tilde\rho_n)=\alpha_r(\tilde\gamma_n(t_n))=\gamma_n'(t_n)
\end{equation*}
\end{comment}
The relations we have collected assert that we can combine the morphisms $\rho_i'$ and paths $\gamma_i'$ and set $\xi' := \rho_n' \ast \gamma_n' \ast ... \ast \gamma_1'\ast \id \in F_{\gamma'}$. Finally, we put $j':=\rho_0^{-1}\circ \tilde\gamma_1(0) \in J_x$. The pair $(j',\xi')$ is by definition the $j$-target of $\tilde\xi$. 
\begin{comment}
The composition is well-defined: $\alpha_l(\tilde\gamma_1(0))=\gamma_1(0)=t(\rho_0)=s(\rho_0^{-1})$.
We have
\begin{equation*}
\alpha_r(j')=\alpha_r(\tilde\gamma_1(0))=\gamma_1'(0)=\alpha_l(\xi')\text{.}
\end{equation*}
\end{comment}

\begin{lemma}
\label{lem:compmorph:def}
The $j$-target is independent of the horizontal lift: if $\tilde\xi_1$ and $\tilde\xi_2$ are horizontal lifts of $\xi$ and $j\in J$ with $\alpha_r(\xi)=\alpha_l(j)$, then the $j$-targets of $\tilde\xi_1$ and $\tilde\xi_2$ coincide.
\end{lemma}

\begin{proof}
If the lifts are $\tilde\xi_1=(\{\rho_{1,i}\},\{\tilde\gamma_{1,i}\})$ and $\tilde\xi_2=(\{\rho_{2,i}\},\{\tilde\gamma_{2,i}\})$, then we have 
\begin{equation*}
\rho_{1,n}\ast \alpha_l(\tilde\gamma_{1,n})\ast \rho_{1,n-1} \ast ... \ast \alpha_l(\tilde\gamma_{1,1})\ast \rho_{1,0} \sim \rho_{2,n}\ast \alpha_l(\tilde\gamma_{2,n})\ast \rho_{2,n-1} \ast ... \ast \alpha_l(\tilde\gamma_{2,1})\ast \rho_{2,0}\text{.}
\end{equation*}
Thus, there exist horizontal paths $\phi_i$ in $\mor{\inf P_x}$ with $s(\phi_i)= \alpha_l(\tilde\gamma_{1,i})$ and $t(\phi_i)= \alpha_l(\tilde\gamma_{2,i})$, as well as 
\begin{equation}
\label{eq:lem:horliftJ}
\phi_1(0)\circ \rho_{1,0}=\rho_{2,0}
\quomma
\rho_{1,n}=\rho_{2,n} \circ \phi_n(1)
\quand
\phi_{i+1}(t_i)\circ \rho_{1,i} = \rho_{2,i} \circ \phi_{i}(t_i)\text{.} 
\end{equation}
We have $\alpha_l(\phi_i(t) \circ \tilde\gamma_{1,i}(t))=\alpha_l(\tilde\gamma_{2,i}(t))$, so that there exist unique paths $\phi_i'$ in $\mor{\inf P_y}$ with $t(\phi_i')=\alpha_r(\tilde\gamma_{2,i})=\gamma'_{2,i}$ and $s(\phi_i')=\alpha_r(\tilde\gamma_{1,i})=\gamma'_{1,i}$, such that 
\begin{equation}
\label{eq:lem:horliftJ2}
\phi_i(t) \circ \tilde\gamma_{1,i}(t)= \tilde \gamma_{2,i}(t) \circ \phi_i'(t)
\end{equation}
By \cref{lem:horF:c,lem:horF:d*} it follows that $\phi_i'$ is horizontal. Next we collect the necessary identities \cref{eq:lem:horliftJ3,eq:lem:horliftJ4,eq:lem:horliftJ5}  that prove that the paths $\phi_i'$ constitute an equivalence between $\xi_1'$ and $\phi_1'(0)^{-1}\circ\xi_2'$.
We consider for $1\leq i < n$  the defining relations
\begin{equation*}
\rho_{1,i} \circ \tilde\gamma_{1,i}(t_i) = \tilde\gamma_{1,i+1}(t_i)\circ \rho_{1,i}'
\quand
\rho_{2,i} \circ \tilde\gamma_{2,i}(t_i) = \tilde\gamma_{2,i+1}(t_i)\circ \rho_{2,i}'
\end{equation*}
for $\rho'_{1,i}$ and $\rho'_{2,i}$.  Combining  with \cref{eq:lem:horliftJ,eq:lem:horliftJ2} we get
\begin{equation}
\label{eq:lem:horliftJ3}
 \rho_{2,i}' \circ \phi_i'(t_i) = \phi_{i+1}'(t_i)\circ \rho_{1,i}'\text{.}
\end{equation}
At $i=0$ we have $\rho_{1,0}'=\rho_{2,0}'=\id$. But $\phi_1'(0)^{-1}\circ\xi_2'$ has as its first morphism not $\rho'_{2,0}$ but $\rho''_{2,0}=\rho'_{2,0}\circ \phi_1'(0)$, so that we have 
\begin{equation}
\label{eq:lem:horliftJ4}
\rho'_{1,0}\circ \phi_1'(0)=\rho_{2,0}''
\end{equation}   
At $i=n$ the defining relations are
$\rho_{1,n} \circ \tilde\gamma_{1,n}(1) = j\circ \rho_{1,n}'$ and $\rho_{2,n} \circ \tilde\gamma_{2,n}(1) = j\circ \rho_{2,n}'$. 
Combining with with \cref{eq:lem:horliftJ,eq:lem:horliftJ2} gives 
\begin{equation}
\label{eq:lem:horliftJ5}
\rho_{2,n}' \circ \phi_n'(1)=  \rho_{1,n}'\text{.}
\end{equation} 
Finally, we have $\rho_{1,0}\circ j'_1= \tilde\gamma_{1,1}(0)$ and $\rho_{2,0} \circ j'_2= \tilde\gamma_{2,1}(0)$. Combining  with \cref{eq:lem:horliftJ,eq:lem:horliftJ2} we get $j'_2\circ \phi_1'(0)=  j'_1$. Thus, we have 
\begin{equation*}
(j_1',\xi_1')=(j'_2\circ \phi_1'(0),\phi_1'(0)^{-1}\circ \xi_2')=(j_2',\xi_2')\text{;}
\end{equation*}
this shows the claim.
\end{proof}

By \cref{lem:compmorph:def} we have a well-defined map
\begin{equation*}
J_{\gamma} : F_{\gamma} \ttimes{\alpha_r}{\alpha_l} J_y \to J_x \ttimes{\alpha_r}{\alpha_l}F'_\gamma :(\xi,j) \mapsto (j',\xi')\text{.}
\end{equation*}

\begin{lemma}
The map $J_{\gamma}$ induces a transformation $J_y \circ F_{\gamma} \Rightarrow F_{\gamma}'\circ J_x$.
\end{lemma}

\begin{proof}
Again consulting  \refcompositionanafunctors\ we have to check first that
\begin{equation*}
J_{\gamma}(\xi\circ \rho, \rho^{-1}\circ j)=J_{\gamma}(\xi,j)\text{.}
\end{equation*}
Let $\tilde\xi = (\{\rho_i\},\{\tilde\gamma_i\})$ be a lift $\xi$ with $j$-target $(j',\xi')$. Then, a lift $\tilde\zeta$ of $\xi \circ \rho$ is obtained from $\tilde\xi$ by only changing $\rho_n$ to $\tilde \rho_n:=\rho^{-1}\circ \rho_n$. We compute the $j$-target $(j',\xi')$ of $\tilde\xi$ and the $(\rho^{-1}\circ j)$-target $(j'',\zeta')$ of $\tilde\zeta$. The only difference is at their last morphisms $\rho_n'$ and $\tilde\rho_n'$, whose defining identities are $\rho_{n} \circ \tilde\gamma_{n}(1) = j\circ \rho_{n}'$ and $\tilde\rho_{n} \circ \tilde\gamma_{n}(1) = \rho^{-1}\circ j\circ \tilde\rho_{n}'$, showing that $\tilde\rho_n'=\rho_n'$ and thus $\xi'=\zeta'$. The equality $j'=j''$ is obvious from their defining identities; this shows the claim.
Now we have a well-defined  map  $J_y \circ F_{\gamma} \Rightarrow F_{\gamma}'\circ J_x$ and have to check that it is a $\Gamma$-equivariant transformation.

That $J_{\gamma}$ is anchor-preserving is straightforward to see.
\begin{comment}
It is anchor-preserving:
\begin{equation*}
\alpha_l(\xi) =s(\rho_0)=t(\rho_0^{-1})=\alpha_l(j')
\quand
\alpha_r(j)=t(\rho_n')=\alpha_r(\xi')\text{.}
\end{equation*}
\end{comment}
It also respects the actions:
for the right action we have to show that
\begin{equation*}
J_{\gamma}(\xi,j \circ \rho)=(j',\xi'\circ \rho)\text{,}
\end{equation*}
where $(j',\xi')$ is the $j$-target of a horizontal lift $\tilde\xi=(\{\rho_i\},\{\tilde\gamma_i\})$ of $\xi$. The same $\tilde\xi$ is also a horizontal  lift of $\xi$, and the only difference between its $j$-target and its $(j \circ \rho)$-target is at the last morphisms, where we get instead of $\rho_n'$ the morphism $\rho_n''=\rho^{-1} \circ \rho_n'$.  Thus, its $(j \circ \rho)$-target is $(j',\xi'\circ \rho)$, as claimed.
For the left action we have to show
\begin{equation*}
J_{\gamma}(\rho \circ \xi,j )=(\rho \circ j',\xi')\text{.}
\end{equation*}
We choose for $\rho\circ \xi$ the horizontal  lift $\tilde\xi$ with only $\rho_0$ changed to $\tilde \rho_0=\rho_0 \circ \rho^{-1}$. Correspondingly, $j'$ changes to $\tilde j' =\rho \circ j'$, while $\xi'$ remains unchanged. This shows the claim.

We check that the $\mor{\Gamma}$-action is preserved, which is equivalent to the identity
\begin{equation*}
J_{\gamma}(\xi\cdot(h,g),j\cdot\id_g)=(j'\cdot(h,g),\xi'\cdot g)
\end{equation*}
see \refmorgammaactiononcomposition. 
\begin{comment}
Indeed,
\begin{multline*}
J_{\gamma}((\xi,j)\cdot (h,g))=J_{\gamma}(\xi\cdot(h,g),j\cdot\id_g)
\\=(j'\cdot(h,g),\xi'\cdot \id_g)=(j',\xi')\cdot(h,g)=J_{\gamma}(\xi,j)\cdot(h,g)
\end{multline*}
\end{comment}
Here we have again fixed a choice $\tilde\xi=(\{\rho_i\},\{\tilde\gamma_i\})$ of a horizontal  lift of $\xi$, and defined $(j',\xi')$ as the $j$-target of $\tilde\xi$. For 
\begin{equation*}
\xi\cdot(h,g)=R(\rho_n,g) \ast R(\gamma_n,g) \ast ... \ast R(\gamma_1,g)\ast R(\rho_0,(h^{-1},t(h)g))
\end{equation*}
we choose the horizontal  lift $(\{\tilde\rho_i\},\{\tilde\gamma_i \cdot \id_g\})$ with $\tilde\rho_i:= R(\rho_i ,g)$ for $1\leq i \leq n$ and $\tilde\rho_0 := R(\rho_0,(h^{-1},t(h)g))$.
We compute its $(j\cdot\id_g)$-target. We obtain the paths $R(\gamma_i',g)$ and the morphisms $R(\rho_i',g)$, and hence $\xi'\cdot \id_g$. The change in $\tilde\rho_0$ only enters the defining identity for $j'$ from $\rho_{0}\circ j'= \tilde\gamma_{1}(0)$ to $R(\rho_0,(h^{-1},t(h)g))\circ \tilde j'= \tilde\gamma_{1}(0)\cdot \id_g$, ending up with $\tilde j'= j'\cdot(h,g)$; this shows the claim.
\begin{comment}
Indeed,
\begin{align*}
(\rho_0 \circ (\tilde j' \cdot (h,g)^{-1}))\cdot\id_g &=R(\rho_0,(1,t(h)g)) \circ \tilde j' )\cdot (h,g)^{-1}\cdot\id_g 
\\&= R(\rho_0,(1,t(h)g)\cdot (h,g)^{-1}) \circ \tilde j'  
\\&=R(\rho_0,(h^{-1},t(h)g))\circ \tilde j'
\\&= \tilde\gamma_{1}(0)\cdot \id_g
\\&= (\rho_{0}\circ j')\cdot \id_g
\end{align*}
\end{comment}

Finally, we check that our map $J_{\gamma}$  is smooth. We consider a chart of $F_{\gamma} \ttimes{\alpha_r}{\alpha_l} J_y$,
\begin{equation*}
U \ttimes{\alpha_r\circ \sigma_{\xi_0,\rho,g}}{t} \mor{\inf P_y}\ttimes{s}{\alpha_l} J_y \to \alpha_l^{-1}(U)\ttimes{\alpha_r}{\alpha_l} J_y:(p,\tilde \rho,j)\mapsto (\sigma_{\xi_0,\rho,g}(p) \circ \tilde \rho,j)\text{,}
\end{equation*}
where $\sigma_{\xi_0,\rho,g}$ is the smooth section defined in \cref{sec:defpartranspaths}. 
\begin{comment}
Explicitly,
\begin{equation*}
\sigma_{\xi_0,\rho,g}(p) =(\rho(p_0,p)^{-1}\circ \xi_0)\cdot (1,g(p_0,p)^{-1})\text{.}
\end{equation*}
\end{comment}
Using the  approved compatibility with the various  actions, we obtain
\begin{equation}
\label{eq:lem:Jsmooth}
J_{\gamma}(\sigma_{\xi_0,\rho,g}(p) \circ \tilde \rho,j)=(\rho(p_0,p)^{-1}\circ J_{\gamma}( \xi_0,(\tilde\rho\circ j)\cdot (1,g(p_0,p))))\cdot (1,g(p_0,p)^{-1})\text{.}
\end{equation}
\begin{comment}
Indeed,
\begin{align*}
J_{\gamma}(\sigma_{\xi_0,\rho,g}(p) \circ \tilde \rho,j) &=J_{\gamma}((\rho(p_0,p)^{-1}\circ \xi_0)\cdot (1,g(p_0,p)^{-1}),\tilde\rho\circ j) 
\\&=J_{\gamma}((\rho(p_0,p)^{-1}\circ \xi_0,(\tilde\rho\circ j)\cdot (1,g(p_0,p)))\cdot (1,g(p_0,p)^{-1})) \\&=J_{\gamma}((\rho(p_0,p)^{-1}\circ \xi_0,(\tilde\rho\circ j)\cdot (1,g(p_0,p)))\cdot (1,g(p_0,p)^{-1}) \\&=(\rho(p_0,p)^{-1}\circ J_{\gamma}( \xi_0,(\tilde\rho\circ j)\cdot (1,g(p_0,p))))\cdot (1,g(p_0,p)^{-1})\text{.} \end{align*}
\end{comment}
Now let $\tilde\xi$ be a horizontal lift for $\xi_0$. Let $(j',\xi')$ be its $j$-target (note that $j=(\tilde\rho\circ j)\cdot (1,g(p_0,p))$ for $\tilde\rho=\id$ and $p=p_0$). Now we compute its $((\tilde\rho\circ j)\cdot (1,g(p_0,p)))$-target for general $\tilde\rho$ and $p$.
The change does not affect $j'$, and we denote it by $(j',\xi'(p,\tilde\rho,j))$, with $\xi'(p_0,\id,j)=\xi'$. In fact, the change only affects the last morphism $\rho_n'(p,\tilde\rho,j)$ of $\xi'(p,\tilde\rho,j)$. Its defining identity is
\begin{equation*}
\rho_n \circ \tilde\gamma_n(1) = (\tilde\rho\circ j)\cdot (1,g(p_0,p))\circ \rho_n'(p,\tilde\rho,j)\text{.}
\end{equation*}
Since $\alpha_l:J \to \ob{\inf P_x}$ is a principal $\inf P_y$-bundle, this shows that $\rho_n'(p,\tilde\rho,j)$ depends smoothly on all parameters. We can write $\xi'(p,\tilde\rho,j)=\xi'\circ (\rho_n' \circ \rho_n'(p,\tilde\rho,j)^{-1} )$. Thus,
\begin{equation*}
J_{\gamma}( \xi_0,(\tilde\rho\circ j)\cdot (1,g(p_0,p)))=(j',\xi'\circ (\rho_n' \circ \rho_n'(p,\tilde\rho,j)^{-1}))=(j',\xi')\circ (\rho_n' \circ \rho_n'(p,\tilde\rho,j)^{-1} )\text{.}
\end{equation*}
Inserting into \cref{eq:lem:Jsmooth} gives our final result for the map $J_{\gamma}$ in above chart:
\begin{equation*}
J_{\gamma}(\sigma_{\xi_0,\rho,g}(p) \circ \tilde \rho,j)=(\rho(p_0,p)^{-1}\circ (j',\xi')\circ (\rho_n' \circ \rho_n'(p,\tilde\rho,j)^{-1} ))\cdot (1,g(p_0,p)^{-1})\text{.}
\end{equation*}
The right hand side is an expression of smooth functions in $(p,\tilde\rho,j)$ using operations that are smooth by \cref{prop:smoothpt}; thus, it is smooth.
\end{proof}

\begin{remark}
\label{re:Jfunctor}
Suppose $\phi:\inf P \to \inf P'$ is a fibre-preserving, smooth, $\Gamma$-equivariant functor such that $\Omega = \phi^{*}\Omega'$. Then, there is a well-defined map $\phi: F_{\gamma} \to F_{\gamma}'$ defined by associating to $\xi=\rho_n\ast \gamma_n\ast ... \ast \rho_0\in F_{\gamma}$ the element $\phi(\xi) := \phi(\rho_n)\ast \phi(\gamma_n) \ast ... \ast \phi(\rho_0)\in F_{\gamma}'$. Now let $J=\ob{\inf P} \ttimes{\phi}{t} \mor{\inf P'}$ be the associated anafunctor equipped with its canonical $\Omega'$-pullback $\nu$, see \cref{rem:smoothfunctorconnection}. Then, the transformation
\begin{equation*}
J_{\gamma} : F_{\gamma} \ttimes{\alpha_r}{\alpha_l} J_y \to J_x \ttimes{\alpha_r}{\alpha_l}F'_\gamma
\end{equation*}
can be expressed in terms of the functor $\phi$  by the formula
\begin{equation*}
J_{\gamma}(\xi,(p,\rho'))= ((p',\id_{\phi(p')}), \phi(\xi) \circ \rho')\text{,}
\end{equation*}
where $p=\alpha_r(\xi)$ and $p'=\alpha_l(\xi)$.
Indeed, a horizontal  lift $\tilde\xi=(\{\rho_i\},\{\tilde\gamma_i\})$ of $\xi$ is obtained by choosing a representative  $\xi=\rho_n\ast \gamma_n\ast ... \ast \rho_0$  and setting $\tilde\gamma_i:=(\gamma_i,\id_{\phi( \gamma_i)})$. By \cref{rem:hor:functor,lem:hormor:a}, $\tilde\gamma_i$ is horizontal with respect to $\nu$. We compute the $(p,\rho')$-target $(j',\xi')$ of $\tilde\xi$: first we obtain $\gamma_i' =\phi(\gamma_i)$.
\begin{comment}
Indeed,
\begin{equation*}
\gamma_i' = \alpha_r (\tilde\gamma_i)=t(\id_{\phi( \gamma_i)})=\phi(\gamma_i)
\end{equation*}
\end{comment}
It is then straightforward to check that $\rho_i'=\phi(\rho_i)$ (for $1\leq i < n$) as well as  $\rho_n' = \rho'^{-1}\circ \phi(\rho_n)$, so that 
\begin{equation*}
\xi'=(\rho'^{-1}\circ \phi(\rho_n)) \ast \phi(\gamma_n)\ast ... \ast \phi(\gamma_1)\ast \id= \phi(\rho_0)\circ \phi(\xi) \circ \rho'\text{.}
\end{equation*} 
\begin{comment}
Indeed, the defining identity for $1\leq i < n$ is $\rho_i \circ \tilde\gamma_i(t_i) = \tilde\gamma_{i+1}(t_i)\circ \rho_i'$, which gives here
\begin{equation*}
\rho_i \circ (\gamma_i(t_i),\id_{\phi( \gamma_i(t_i))}) = (\gamma_{i+1}(t_i),\id_{\phi( \gamma_{i+1}(t_i))})\circ \rho_i'
\end{equation*}
which simplifies to
\begin{equation*}
 (t(\rho_i),\phi(\rho_i)) = (\gamma_{i+1}(t_i),\rho_i')
\end{equation*}
and thus gives the claim. For $i=n$ the defining identity is $\rho_n \circ \tilde\gamma_n(1) = j\circ \rho_n'$, which is here
\begin{equation*}
\rho_n \circ (\gamma_n(1),\id_{\phi( \gamma_n(1))})= (p,\rho')\circ \rho_n'
\end{equation*}
which simplifies to
\begin{equation*}
 (t(\rho_n),\phi(\rho_n))= (p,\rho'\circ \rho_n')\text{.}
\end{equation*}
\end{comment}
Finally, we obtain $j'=(\alpha_l(\xi),\phi(\rho_0)^{-1})$.
\begin{comment}
Indeed.
\begin{equation*}
j'=\rho_0^{-1}\circ \tilde\gamma_1(0)=\rho_0^{-1}\circ (\gamma_i(0),\id_{\phi( \gamma_i(0))})=(s(\rho_0),\phi(\rho_0)^{-1})=(\alpha_l(\xi),\phi(\rho_0)^{-1})\text{.}
\end{equation*} 
\end{comment}
The result is 
\begin{equation*}
(j',\xi')=((\alpha_l(\xi),\phi(\rho_0)^{-1}),\phi(\rho_0)\circ \phi(\xi) \circ \rho')= ((\alpha_l(\xi),\id), \phi(\xi) \circ \rho')\text{,}
\end{equation*}
as claimed.

Now we suppose that  $\kappa=(\kappa_0,\kappa_1)$ shifts the canonical $\Omega'$-pullback  $\nu$ to another connection-preserving and connective  pullback $\nu^{\kappa}$.
The connections $\Omega$ and $\Omega'$ are then related by the formulas of \refshiftedpullback. Our lifts $\tilde\gamma_i=(\gamma_i,\id_{\phi( \gamma_i)})$ are no longer horizontal with respect to $\nu^{\kappa}$. By \cref{lem:horF}
there exist  unique paths $h_i:[t_{i-1},t_i] \to H$ with $h_i(t_{i-1})=1$ such that $\tilde\gamma_i^{\kappa} := \tilde\gamma_i \cdot (h_i,t(h_i)^{-1})= (\gamma_i,R(\id_{\phi( \gamma_i)},(h_i,t(h_i)^{-1}))) $ is horizontal. From \cref{rem:hor:functor} one can conclude that $h_i$ solves the initial value problem
\begin{equation}
\label{eq:functorJ:h}
\dot h_i(t) = -h_i(t) \kappa_{0}(\dot\gamma_i(t)) \quand h_i(t_{i-1})=1\text{.}
\end{equation}
\begin{comment}
Note that $\nu^{\kappa}$ is connection-preserving if and only if $\fa\Omega = \phi^{*}\fa{\Omega'} + t_{*}(\kappa_0)$ (\cite[Remark 5.2.10 (e)]{Waldorf2016}).
Indeed, the condition of \cref{rem:hor:functor} is for $\rho_i:=R(\id_{\phi( \gamma_i)},(h_i,t(h_i)^{-1}))$
\begin{align*}
0&=\fb{\Omega'}(\dot\rho_i(t))+\kappa_0(\dot\gamma_i(t))
\\&= R^{*}\fb{\Omega'}(\id_{\phi( \gamma_i)},(h_i,t(h_i)^{-1}))+\kappa_0(\dot\gamma_i(t))
\\&= (\alpha_{t(h_i)})_{*}\left (\mathrm{Ad}_{h_i}^{-1}(\phi^{*}\id^{*}\fb{\Omega'}(\dot\gamma_i))+(\tilde \alpha_{h_i})_{*}(\phi^{*}\fa{\Omega'}(\dot\gamma_i)) +h_i^{-1}\dot h_i \right)+\kappa_0(\dot\gamma_i(t))
\\&= \mathrm{Ad}_{h_i}\left ((\tilde \alpha_{h_i})_{*}(\phi^{*}\fa{\Omega'}(\dot\gamma_i)) +h_i^{-1}\dot h_i \right)+\kappa_0(\dot\gamma_i(t))
\\&= \mathrm{Ad}_{h_i}\left ((\tilde \alpha_{h_i})_{*}(\underbrace{\fa{\Omega}(\dot\gamma_i)}_{=0}-t_{*}(\kappa_0(\dot\gamma_i))) +h_i^{-1}\dot h_i \right)+\kappa_0(\dot\gamma_i(t))
\\&= \mathrm{Ad}_{h_i}\left (-\mathrm{Ad}_{h_i}^{-1}(\kappa_0(\dot\gamma_i))+\kappa_0(\dot\gamma_i) +h_i^{-1}\dot h_i \right)+\kappa_0(\dot\gamma_i(t))
\\&= \mathrm{Ad}_{h_i}\left (\kappa_0(\dot\gamma_i) +h_i^{-1}\dot h_i \right)
\end{align*}
This is equivalent to
\begin{equation*}
h_i^{-1}\dot h_i=-\kappa_0(\dot\gamma_i)\text{.} 
\end{equation*}
\end{comment}
Since $\alpha_l(\tilde\gamma_i^{\kappa})=\gamma_i$ it is clear that $\tilde\xi^{\kappa}=(\{\rho_i\},\{\tilde\gamma_i^{\kappa}\})$ is a horizontal lift of $\xi$. We compute the $(p,\rho')$-target $(j',\xi')$ of $\tilde\xi^{\kappa}$. We get $\gamma_i' =R(\phi(\gamma_i),t(h_i)^{-1})$, and for the morphisms we get $\rho_i'=  R(\phi(\rho_i),(h_i(t_i),t(h_i(t_i))^{-1}))$ for $1\leq i < n$,  $ \rho_n' = \rho'^{-1}\circ R( \phi(\rho_n) ,(h_n(1),t(h_n(1))^{-1}))$, and $j'=(p,\phi(\rho_0)^{-1})$. 
\begin{comment}
The defining identity for $1\leq i < n$ is $\rho_i \circ \tilde\gamma_i^{\kappa}(t_i) = \tilde\gamma_{i+1}^{\kappa}(t_i)\circ \rho_i'$, which gives 
\begin{equation*}
\rho_i \circ(\gamma_i(t_i),R(\id_{\phi( \gamma_i(t_i))},(h_i(t_i),t(h_i(t_i))^{-1})))=(\gamma_{i+1}(t_i),\id_{\phi( \gamma_{i+1}(t_i))})\circ \rho_i'
\end{equation*}
This implies
\begin{equation*}
\rho_i'=\phi(\rho_i) \circ R(\id_{\phi( \gamma_i(t_i))},(h_i(t_i),t(h_i(t_i))^{-1}))=R(\phi(\rho_i),(h_i(t_i),t(h_i(t_i))^{-1}))
\end{equation*}
The defining identity for $\rho_n$ is $\rho_n \circ \tilde\gamma_n^{\kappa}(1) = j\circ \rho_n'$, which gives
\begin{equation*}
\rho_n \circ (\gamma_n(1),R(\id_{\phi( \gamma_n(1))},(h_n(1),t(h_n(1))^{-1}))) = (p,\rho')\circ \rho_n'
\end{equation*}
This implies
\begin{equation*}
 \rho_n' = \rho'^{-1}\circ R( \phi(\rho_n) ,(h_n(1),t(h_n(1))^{-1}))\text{.} 
\end{equation*}
Finally,
\begin{equation*}
j'=\rho_0^{-1}\circ \tilde\gamma_1(0)=\rho_0^{-1}\circ (\gamma_1(0),\id_{\phi( \gamma_1(0))})
\end{equation*}
\end{comment} 
Summarizing, we have 
\begin{equation*}
J_{\gamma}(\xi,(p,\rho'))=((p',\id_{\phi(p')}), \phi^{\kappa}(\xi)\circ \rho' )\text{,}
\end{equation*}
where $p':=\alpha_l(\xi)$  and we now define
\begin{equation*}
\phi^{\kappa}(\xi) := \rho_n'' \ast \gamma_n'\ast \rho_{n-1}''\ast ... \ast \gamma_1'\ast \rho_0'' \end{equation*}
from the following components:
 $\gamma_i'$ is given by above formulae, with $h_i$ determined by \cref{eq:functorJ:h},  $\rho_i'':=R(\phi(\rho_i),(h_i(t_i),t(h_i(t_i))^{-1}))$ for $1\leq i \leq n$, and $\rho_0'':=\phi(\rho_0)$. \begin{comment}
Indeed,
\begin{align*}
J_{\gamma}(\xi,(p,\rho'))&=((p,\phi(\rho_0)^{-1}),\rho_n' \ast \gamma_n'\ast ... \ast \gamma_1'\ast \id )\\&=((p,\phi(\rho_0)^{-1}),\rho_0'\circ \phi^{\kappa}(\xi)\circ \rho' )\\&=((p,\id), \phi^{\kappa}(\xi)\circ \rho'  )
\end{align*}
\end{comment}
\end{remark}

The following result describes in which way the transformation $J_{\gamma}$ is compatible with path composition. In the 2-functor formalism described in \cref{sec:orga} it is one of the axioms of a pseudonatural transformations between the transport 2-functors associated to the two principal $\Gamma$-2-bundles.

\begin{proposition}
\label{lem:T1}
Let $J:\inf P_1 \to \inf P_2$ be a 1-morphism in $\zweibuncon\Gamma M$. The associated transformation $J_{\gamma}$ is compatible with path composition in the sense that the following diagram commutes for each pair of composable paths $\gamma_1:x \to y$ and $\gamma_2:y \to z$:
\begin{equation*}
\alxydim{@C=-1.2cm@R=\xyst}{ && J_z \circ F_{\gamma_2}\circ F_{\gamma_1} \ar@{=>}[drr]^{\id \circ c_{\gamma_1,\gamma_2}} \ar@{=>}[dll]_{J_{\gamma_2}\circ \id} && \\ F'_{\gamma_2} \circ J_y \circ F_{\gamma_1} \ar@{=>}[dr]_{\id \circ J_{\gamma_1}} &&&& \qquad J_z \circ F_{\gamma_2\ast \gamma_1} \quad \ar@{=>}[dl]^{J_{\gamma_2\ast\gamma_1}} \\ & F'_{\gamma_2}\circ F'_{\gamma_1}\circ J_x \ar@{=>}[rr]_-{c'_{\gamma_1,\gamma_2}\circ \id} & \hspace{3.5cm} & F_{\gamma_2\ast \gamma_1} \circ J_x}
\end{equation*}
\end{proposition}

\begin{proof}
We consider an element $(\xi_1,\xi_2,j)\in F_{\gamma_1}\ttimes{\alpha_r}{\alpha_l}F_{\gamma_2} \ttimes{\alpha_r}{\alpha_l} J_z$. We choose separately horizontal lifts $\tilde\xi_1$ and $\tilde\xi_2$ of $\xi_1$ and $\xi_2$ to $J$. Let $(j',\xi_2')$ be the $j$-target of $\tilde\xi_2$, and let $(j'',\xi_1')$ be the $j'$-target of $\tilde\xi_1$. Then, going counter-clockwise, we have
\begin{equation*}
(\xi_1,\xi_2,j) \mapsto (\xi_1,j',\xi_2')\mapsto (j'',\xi_1',\xi_2')\mapsto (j'',\xi_2'\ast\xi_1')\text{.}
\end{equation*}
Under the obvious reparameterization and renumbering one can combine the horizontal lifts $\tilde\xi_1$ and $\tilde\xi_2$ to a lift $\widetilde{\xi_2\ast\xi_1}$ of $c_{\gamma_1,\gamma_2}(\xi_1,\xi_2)=\xi_2\ast\xi_1$. A straightforward computation shows that its $j$-target is $(j'',\xi_2'\ast\xi_1')$. 
This shows that the diagram is commutative.
\end{proof}

\begin{comment}
\begin{remark}
It can be written as a pasting diagram:
\begin{equation*}
\alxydim{@=1.7cm}{\inf P_x \ar[r]^{F_{\gamma_1}} \ar[d]_{J_x} & \inf P_y \ar[r]^{F_{\gamma_2}} \ar@{=>}[dl]|*+{J_{\gamma_1}} \ar[d]^{J_y} & \inf P_z \ar@{=>}[ld]|*+{J_{\gamma_2}} \ar[d]^{J_z} \\ \inf P_x' \ar@/_4pc/[rr]_{F'_{\gamma_2\circ\gamma_1}}="1" \ar@{=>}[r];"1"|*{c'_{\gamma_1,\gamma_2}} \ar[r]_{F_{\gamma_1}'} & \inf P_y' \ar[r]_{F'_{\gamma_2}} & \inf P'_z}
=\alxydim{@=1.7cm}{\inf P_x \ar[r]^{F_{\gamma_1}}  \ar@/_4pc/[rr]_{F_{\gamma_2\circ\gamma_1}}="1" \ar[d]_{J_x} & \inf P_{y} \ar@{=>}"1"|{c_{\gamma_1,\gamma_2}} \ar[r]^{F_{\gamma_2}}& \inf P_z \ar@/^4pc/@{=>}[dll]|*+{J_{\gamma_2\circ \gamma_1}} \ar[d]^{J_z} \\ \inf P_x' \ar@/_4pc/[rr]_{F_{\gamma_2\circ\gamma_1}'} && \inf P_z'}
\end{equation*}
\end{remark}
\end{comment}
Further,  $J_{\gamma}$  is compatible with the   composition of bundle morphisms in the following sense:

\begin{proposition}
\label{lem:morphcompcomp}
Suppose $J:\inf P \to \inf P'$ and $K:\inf P'\to  \inf P''$ are 1-morphisms in $\zweibuncon\Gamma M$. Then, the following diagram commutes for each path $\gamma:x \to y$:
\begin{equation*}
\alxydim{@C=0em@R=\xyst}{K_y \circ J_y \circ F_{\gamma} \ar@{=>}[dr]_{\id \circ J_{\gamma}} \ar@{=>}[rr]^{(K \circ J)_{\gamma}} && F''_{\gamma} \circ K_x \circ J_x \ar@{<=}[dl]^{K_{\gamma} \circ \id}
\\&K_y \circ F'_{\gamma} \circ J_x &}
\end{equation*}
\end{proposition}

\begin{proof}
We start with an element $(\xi,j,k)$ in $F_{\gamma} \ttimes{\alpha_r}{\alpha_l} J_y \ttimes{\alpha_r}{\alpha_l} K_y$. Let $\tilde \xi=(\{\rho_i^{J}\},\{\tilde\gamma_i^{J}\})$ be a horizontal lift of $\xi$ to $J$, and let $(\xi',j')$ be its $j$-target. Let $\tilde\xi'=(\{\rho_i^{K}\},\{\tilde\gamma_i^{K}\})$ be a horizontal lift of $\xi'$ to $K$, and let $(k',\xi'')$ be its $k$-target. Then, going counter-clockwise results in $(j',k',\xi'')$. In order to compute the clock-wise direction, we notice that $(\{\rho_i^{J}\},\{(\tilde\gamma_i^J,\tilde\gamma_i^{K})\})$ is a horizontal lift of $\xi$ to $K \circ J$, using the definition of the pullback on the composition $K\circ J$ (see \refpullbackoncomposition).
A straightforward computation shows that its $(j,k)$-target is $(j',k',\xi'')$.
\begin{comment}
Indeed, we have $\alpha_r(\tilde\gamma_i^J,\tilde\gamma_i^{K})=\alpha_r(\tilde\gamma_i^{K})=\gamma_i''$. The defining identities for $1\leq i < n$ are
\begin{equation*}
 \tilde\gamma_i^{J}(t_i) \circ \rho_i^{K-1} =\rho^{J-1}_i\circ \tilde\gamma_{i+1}^{J}(t_i)
\quand
\rho_i ^{K}\circ \tilde\gamma_i^{K}(t_i) = \tilde\gamma_{i+1}^{K}(t_i)\circ \rho_i''\text{;}
\end{equation*}
they combine to
\begin{equation*}
\rho_i^{J}\circ( \tilde\gamma_i^{J}(t_i) , \tilde\gamma_i^{K}(t_i)) =( \tilde\gamma_{i+1}^{J}(t_i), \tilde\gamma_{i+1}^{K}(t_i))\circ \rho_i''\text{.}
\end{equation*}
At the end point, the defining identities are
\begin{equation*}
 \tilde\gamma_n^{J}(1)\circ \rho_n^{K-1}  =\rho_n^{J-1} \circ j\quand
\rho_n^{K} \circ \tilde\gamma_n^{K}(1) = k\circ \rho_n''\text{;}
\end{equation*}
they combine to
\begin{equation*}
\rho_n^{J} \circ(\tilde \gamma_n^{J}(1), \tilde\gamma_n^{K}(1)) = (j,k)\circ \rho_n''\text{.}
\end{equation*}
This shows that we obtain $\xi''$. Finally, we have (since $\rho_0^{K}=\id$)
\begin{equation*}
\rho_0^{J-1}\circ (\tilde\gamma_1^{J}(0) ,\tilde\gamma_1^{K}(0))=  (\rho_0^{J-1}\circ\tilde\gamma_1^{J}(0) ,\rho_0^{K-1}\tilde\gamma_1^{K}(0))=(j',k')\text{.}
\end{equation*}  
\end{comment}
\end{proof}

Finally, there is a compatibility condition with 2-morphisms, which is responsible for an axiom of a  modification in the 2-functor formalism of \cref{sec:orga}.

\begin{proposition}
\label{lem:modification}
Suppose $J,J':\inf P \to \inf P'$ are 1-morphisms in $\zweibuncon\Gamma M$, and $f: J \Rightarrow J'$ is a connection-preserving 2-morphism. Then, the following diagram commutes for each path $\gamma:x \to y$:
\begin{equation*}
\alxydim{@=\xyst}{J_y \circ F_{\gamma} \ar@{=>}[r]^{J_{\gamma}} \ar@{=>}[d]_{f_y \circ \id} & F'_{\gamma}\circ J_x \ar@{=>}[d]^{\id \circ f_x} \\ J'_y \circ F_{\gamma} \ar@{=>}[r]_{J'_{\gamma}}  & F'_{\gamma}\circ J'_x \text{.} }
\end{equation*}
\end{proposition}

\begin{proof}
Suppose $(\xi,j)\in F_{\gamma} \ttimes{\alpha_r}{\alpha_l} J_y$. Let $(\{\rho_i\},\{\tilde\gamma_i\})$ be a horizontal lift of $\xi$ to $J$, and let $(j',\xi')$ be its $j$-target. Then, $(\{\rho_i\},\{f(\tilde\gamma_i)\})$ is obviously a horizontal lift of $\xi$ to $J'$, and it is easy to show that its $f(j)$-target is $(f(j'),\xi')$. This shows commutativity.
\begin{comment}
Indeed, horizontality of $\varphi(\tilde\gamma_i)$ follows because $\varphi$ is connection-preserving. Since $\varphi$ is anchor-preserving, we have $\alpha_r(\varphi(\tilde\gamma_i))=\alpha_r(\tilde\gamma_i) $, and since it is action-preserving we obtain the same $\rho_i'$, and $\varphi(j')$.
\end{comment} 
\end{proof}

\subsection{Naturality with respect to pullback}

\label{sec:natpullbackpaths}

Suppose $\inf P$ is a principal $\Gamma$-2-bundle over $N$ with connection $\Omega$,  $f:M \to N$ is a smooth map, and $\gamma:[0,1] \to M$ is a path. We denote by $\inf P' := f^{*} \inf P$ the pullback bundle,  obtain a $\Gamma$-equivariant smooth functor
$\tilde f: \inf P' \to \inf P$,
and $\Omega' := \tilde f^{*}\Omega$ is a connection on $\inf P'$. We construct a transformation
\begin{equation*}
\alxydim{@=\xyst}{\inf P'_x \ar[d]_{\tilde f_x} \ar[r]^{F'_{\gamma}} &  \inf P_y \ar@{=>}[dl]|*+{\tilde f_{\gamma}} \ar[d]^{\tilde f_y} \\ \inf P_{f(x)} \ar[r]_{F_{f \circ \gamma}} & \inf P_{f(y)}}
\end{equation*}
We first recall that $\ob{\inf P'} = M \ttimes{f}{\pi} \ob{\inf P}$ and $\mor{\inf P'}=M \ttimes{f}{\tilde\pi}\mor{\inf P}$, using that $\pi:\ob{\inf P} \to M$ and $\tilde\pi: \mor{\inf P} \to M$ (defined as $\tilde\pi = \pi \circ t = \pi \circ s$) are submersions. We can hence canonically identify $\inf P'_x = \inf P_{f(x)}$, so that the functor $\tilde f_x: \inf P_x' \to \inf P_{f(x)}$ is just the identity. It remains to construct a transformation
\begin{equation*}
\tilde f_{\gamma}: F_{\gamma}' \Rightarrow F_{f \circ \gamma}\text{.}
\end{equation*}
Suppose $\xi\in F'_{\gamma}$, i.e. $\xi'=\rho_n \ast \gamma_n \ast ... \ast \rho_1\ast \gamma_1\ast \rho_0$, with $\gamma_i$ horizontal paths in $\inf P'$ such that $\pi'(\gamma_i(t))=\gamma(t)$. It is clear that $\tilde f \circ \gamma_i$ are horizontal paths in $\inf P$ with $\pi((\tilde f \circ \gamma_i)(t))= (f \circ \gamma)(t)$. Thus, $\tilde f_{\gamma}(\xi) := \tilde f(\rho_n) \ast (\tilde f \circ \gamma_n)\ast ... \ast \tilde f(\rho_0) \in F_{f \circ \gamma}$. It is straightforward to show that this definition indeed defines a $\Gamma$-equivariant transformation.

\section{Parallel transport along bigons}

\label{sec:ptbigons}

 A \emph{bigon} in a smooth manifold $M$ is a smooth, fixed-ends homotopy between two paths $\gamma$ and $\gamma'$ with common end-points $x$ and $y$. More precisely, a bigon is a smooth map $\Sigma: [0,1]^2 \to M$ such that $\Sigma(s,0)=x$ and $\Sigma(s,1)=y$ for all $s\in [0,1]$, and  $\gamma(t) = \Sigma(0,t)$ and $\gamma'(t):= \Sigma(1,t)$ for all $t\in[0,1]$.  We use the notation $\Sigma:\gamma\Rightarrow \gamma'$, and the instructive picture
\begin{equation*}
\alxydim{@C=\xyst}{x \ar@/^1.3pc/[r]^{\gamma}="1" \ar@/_1.3pc/[r]_{\gamma'}="2" \ar@{=>}"1";"2"|*+{\Sigma} & y\text{.}}
\end{equation*}
Bigons represent directed pieces of surfaces, along which we are going to define parallel transport.

Let $\inf P$ be a principal $\Gamma$-bundle over $M$  with a fake-flat connection $\Omega$. That parallel transport along surfaces can only be defined for \emph{fake-flat} connections is a well-known phenomenon \cite{schreiber2}.
In this section we define for each bigon $\Sigma:\gamma \Rightarrow \gamma'$ a $\Gamma$-equivariant transformation
\begin{equation*}
\varphi_{\Sigma} : F_{\gamma} \to F_{\gamma'}
\end{equation*}
between the parallel transports along $\gamma$ and $\gamma'$.
For this purpose, we first introduce in \cref{sec:horliftbigon} the notion of a horizontal lift of a bigon to the total space of $\inf P$. In \cref{sec:defpartransbigon} we give a complete definition of the transformation $\varphi_{\Sigma}$. In \cref{sec:compbigoncomp,sec:natbundlemorphbigon,sec:natpullbackbigon} we derive several properties of $\varphi_{\Sigma}$ with respect to the composition of bigons, 1-morphisms between principal 2-bundles, and pullback. 

\subsection{Horizontal lifts of bigons}

\label{sec:horliftbigon}

Let $\inf P$ be a principal $\Gamma$-2-bundle over $M$ together with a fake-flat connection $\Omega$.

\begin{definition}
\label{def:horliftbigon}
Let $\Sigma:\gamma \Rightarrow \gamma'$ be a bigon and $\xi\in F_{\gamma}$. 
A \emph{horizontal lift} of  $\Sigma$ with source $\xi$ is a tuple 
$(n, t, \{\Phi_i\}_{i=1}^{n}, \{\rho_i\}_{i=0}^{n},\{g_i\}_{i=1}^{n})$
consisting of $n\in \N$, a subdivision $t\in T_n$ and
smooth maps\begin{itemize}

\item
 $\Phi_i:[0,1] \times [t_{i-1},t_i] \to \ob{\inf P}$

\item
 $\rho_{i}: [0,1] \to \mor{\inf P}$ with $\rho_0$ and $\rho_n$ constant

\item
 $g_{i}: [0,1] \to G$  with $g_i(0)=1$

\end{itemize}
such that
the following conditions are satisfied:
\begin{enumerate}[(a)]

\item 
\label{lem:bigoncon:a}
$\Phi_i$ is a  lift of  $\Sigma$, i.e., 
$\pi\circ \Phi_i = \Sigma|_{[0,1] \times [t_{i-1},t_i]}$ for all for $1\leq i \leq n$

\item
\label{lem:bigoncon:b}
$t(\rho_{i}(s)) = \Phi_{i+1}(s,t_i)$ for all $0\leq i < n$ and $s(\rho_{i}(s))=R(\Phi_i(s,t_{i}),g_{i}(s))$ for all  $1\leq i \leq n$.

\item
\label{lem:bigoncon:d}
The paths $\gamma_i'(t) := \Phi_i(1,t)$, $\nu_i(s) := \Phi_{i}(s,t_{i-1})$ and $\rho_i$ are horizontal for all $1\leq i \leq n$.

\item
\label{lem:bigoncon:c}
$\xi=\rho_n \ast \gamma_n \ast ... \ast \gamma_1\ast \rho_0$ with $\gamma_i(t):=\Phi_i(0,t)$ and $\rho_i := \rho_i(0)$

\end{enumerate}
\end{definition}

We begin with \quot{small} bigons: a bigon $\Sigma:\gamma \Rightarrow \gamma'$ is called \emph{small}, if there exist $n\in \N$, $t\in T_n$ and sections $\sigma_i:U_i \to \ob{\inf P}$ defined on  open sets $U_i$ such that 
\begin{equation*}
\Sigma(\{(s,t) \sep t_{i-1}\leq t \leq t_i,0\leq s\leq 1\})\subset U_i\text{.}
\end{equation*} 

\begin{lemma}
\label{lem:existencehorizontallifts}
For every small bigon  $\Sigma:\gamma \Rightarrow \gamma'$ and every $\xi\in F_{\gamma}$ there exists a horizontal lift with source $\xi$.
\end{lemma}

\begin{proof}
We choose for $1\leq i \leq n$ sections $\sigma_i: U_i \to \ob{\inf P}$, and für $1\leq i < n$ transition spans $\sigma_{i,i+1}:U_i \cap U_j \to \mor{\inf P}$ along $(\sigma_{i},\sigma_{i+1})$ together with transition functions $g_{i,i+1}: U_i \cap U_{i+1} \to G$. 
\begin{comment}
That means, we have
\begin{equation*}
t(\sigma_{i,i+1}(x))=\sigma_i(x)
\quand 
s(\sigma_{i,i+1}(x))=R(\sigma_{i+1}(x),g_{i,i+1}(x))\text{.}
\end{equation*}
\end{comment}
These choices can successively be adjusted such that $g_{i,i+1}(\Sigma(0,t_i))=1$.
\begin{comment}
Indeed, suppose $g:=g_{1,2}(\Sigma(0,t_1))\neq 1$. Then we choose $\sigma_2':=R(\sigma_{2},g)$, $\sigma_{1,2}':=\sigma_{1,2}$ and $g_{1,2}':=g^{-1}g_{1,2}$ instead of $\sigma_2$, $\sigma_{1,2}$, and $g_{1,2}$. We have
\begin{align*}
t(\sigma_{1,2}'(x)) &=t(\sigma_{1,2}(x))=\sigma_1(x)=\sigma_1'(x)
\\
s(\sigma_{1,2}'(x)) &=s(\sigma_{1,2})=R(\sigma_{2}(x),g_{1,2}(x))=R(R(\sigma_{2}(x),g),g^{-1}g_{1,2}(x))=R(\sigma_{2}'(x),g_{1,2}'(x))
\end{align*}
and $g_{1,2}'(\Sigma(0,t_1))=1$. We proceed with $\sigma_3$, and so on. 
\end{comment}

We set, for $1\leq i \leq n$,
$\Phi_i := \sigma_i \circ \Sigma|_{[0,1] \times [t_{i-1},t_i]}$;
this satisfies \cref{lem:bigoncon:a*}. We also set, for $1\leq i < n$,
\begin{equation*}
\rho_i(s) := R( \sigma_{i,i+1}(\Sigma(s,t_i))^{-1},g_{i,i+1}(\Sigma(s,t_i))^{-1})
\quand g_i (s):= g_{i,i+1}(\Sigma(s,t_i))^{-1}\text{;}
\end{equation*}
above adjustment achieves $g_i(0)=1$. 
Further, we set $\rho_0(s) := \id_{\sigma_{1}(x)}$, $\rho_n(s) := \id_{\sigma_n(y)}$ and $g_n(s) := 1$. This satisfies \cref{lem:bigoncon:b*} by definition of a transition span. 
\begin{comment}
Indeed,
\begin{equation*}
t(\rho_{i}(s))=R(s( \sigma_{i,i+1}(\Sigma(s,t_i))),g_{i,i+1}(\Sigma(s,t_i))^{-1}) =\sigma_{i+1}(\Sigma(s,t_i))= \Phi_{i+1}(s,t_i)
\end{equation*} 
and
\begin{equation*}
s(\rho_i(s))=R(t( \sigma_{i,i+1}(\Sigma(s,t_i))),g_{i,i+1}(\Sigma(s,t_i))^{-1}) =R( \sigma_{i}(\Sigma(s,t_i)),g_i(s))=R(\Phi_{i}(s,t_i),g_i(s))\text{.} \end{equation*}
\end{comment}
Next we perform some modifications. Let $\xi=\rho'_n \ast \gamma_n' \ast ... \ast \gamma_1'\ast \rho_0'$ be a representative.
We choose for $1\leq i \leq  n$ paths  $\tilde \rho_i: [t_{i-1},t_i] \to \mor{\inf P}$ and $\tilde g_i: [t_{i-1},t_i]\to G$ such that $s(\tilde\rho_i(t))=\gamma'_i(t)$ and $t(\tilde\rho_i(t))=R(\Phi_{i}(0,t),\tilde g_i(t))$. 
By \cref{lem:makemorhor} there exist a unique path  $h_i: [t_{i-1},t_i] \to H$ with $h_i(t_{i-1})=1$ such that $\tilde\rho_i^{hor}:=R(\tilde\rho_i,(h_i,1))$ is horizontal. Successively, this data can be arranged such that $ t(h_i(t_i))^{-1}\tilde g_i(t_i)^{-1}\tilde g_{i+1}(t_i)=1$. 
We 
define:
\begin{align*}
\Phi'_i(s,t) &:= R(\Phi_i(s,t),\tilde g_i(t)t(h_i(t))) &&\text{ for }1\leq i \leq n
\\
\rho_i'(s) &:= R(\rho_i(s),\tilde g_{i+1}(t_i)) &&\text{ for }0\leq i <n
\\
\rho_n'(s) & := \rho_n(s)
\\
g_i'(s) &:= t(h_i(t_i))^{-1}\tilde g_i(t_i)^{-1}g_i(s)\tilde g_{i+1}(t_i) &&\text{ for }1\leq i <n 
\\
g_n'(s) &:= t(h_n(t_n))^{-1}\tilde g_n(t_n)^{-1}g_n(s)
\end{align*}
This modification does not affect \cref{lem:bigoncon:a*} and still satisfies \cref{lem:bigoncon:b*}.
\begin{comment}
First of all, $\rho_0$ and $\rho_n$ remain constant, and $g_i'(0)=g_i(0)=1$.  
Indeed, we have for $0\leq i < n$:
\begin{equation*}
t(\rho_i'(s))=R(t(\rho_i(s)),\tilde g_{i+1}(t_i))=R( \Phi_{i+1}(s,t_i),\tilde g_{i+1}(t_i))=\Phi'_{i+1}(s,t_i)
\end{equation*}
and
 for $1\leq i < n$
 \begin{equation*}
s(\rho_i'(s))=R(s(\rho_i(s)),\tilde g_{i+1}(t_i))=R(\Phi_i(s,t_i),g_i(s)\tilde g_{i+1}(t_i))= R(\Phi_{i}'(s,t_i),g_i'(s))
\end{equation*}
and finally
 \begin{equation*}
s(\rho_n'(s))=s(\rho_n(s))=R(\Phi_n(s,t_n),g_n(s))= R(\Phi_{n}'(s,t_n),g_n'(s))
\end{equation*}
\end{comment}
Since
\begin{equation*}
s(\tilde\rho_i^{hor}(t))=\gamma'_i(t)
\quand
t(\tilde\rho_i^{hor}(t))=\Phi_{i}'(0,t)\text{,}
\end{equation*}
we can apply the equivalence relation to the horizontal paths $\tilde\rho_i^{hor}$, so that $\xi=\rho_n \ast \gamma_n \ast ... \ast \gamma_1\ast \rho_0$ with $\gamma_i(t) := \Phi'_i(0,t)$ and some new $\rho_0,...,\rho_n\in\mor{\inf P}$. This makes up the first part of \cref{lem:bigoncon:c*}; we have not yet achieved that $\rho_i=\rho_i'(0)$. Note that by \cref{lem:bigoncon:b*} we have $t(\rho_i)=\gamma_{i+1}(0)=t(\rho_i'(0))$ for $0\leq i <n$ and $s(\rho_i)=\gamma_i(1)=s(\rho_i'(0))$ for $1\leq i \leq n$. By \refactionfullyfaithful\ there exist, for $1\leq i < n$ unique $h_i\in H$ with $t(h_i)=1$ such that $\rho_i=R(\rho_i'(0),(h,1))$. We set $\rho_i''(s) := R(\rho_i'(s),(h_i,1))$. At the endpoints we define $\rho_0''(s) := \rho_0$ and $\rho_n''(s) := \rho_n$; this still satisfies \cref{lem:bigoncon:b*}. 
\begin{comment}
Indeed, we have
\begin{equation*}
t(\rho_0''(s)) = t(\rho_0)=\gamma_1(0)=\Phi_1'(0,0)\eqcref{lem:bigoncon:b*}t(\rho_0'(0))\eqtext{$\rho_0$ is constant}t(\rho_0'(s))\eqcref{lem:bigoncon:b*}\Phi'_1(s,0)
\end{equation*}
and similarly
\begin{equation*}
s(\rho''_n(s))=s(\rho_n)=\gamma_n(1)=\Phi'_n(0,1)=s(\rho_n'(0))=s(\rho_n'(s))=\Phi_n'(s,0)\text{.}
\end{equation*}
\end{comment}
Now $\{\Phi_i',\rho_i'',g_i'\}$ satisfy \cref{lem:bigoncon:a*}, \cref{lem:bigoncon:b*} and \cref{lem:bigoncon:c*}.

Next we look for the first part of \cref{lem:bigoncon:d*}, horizontality of the $\gamma_i'$. By \cref{lem:obhorexists} there exist, for $1\leq i \leq n$ smooth maps $\tilde g_i: [t_{i-1},t_i] \to G$   with $\tilde g_i(t_{i-1})=1$ and $t\mapsto R(\Phi_i(1,t),\tilde g_i(t))$ horizontal.  We define $\varphi_i: [0,1] \times [t_{i-1},t_i] \to G$ by $\varphi_i(s,t) := \tilde g_i(s(t-t_{i-1}) + t_{i-1})$. We put 
\begin{equation*}
\Phi_i''(s,t):=R(\Phi'_i(s,t),\varphi_i(s,t))
\quomma
\rho''_i := \rho_i' 
\quand
g_i'' :=  \tilde g_i^{-1} \cdot g_i'\text{.}
\end{equation*}
Obviously, \cref{lem:bigoncon:a*} is not affected. For \cref{lem:bigoncon:b*} we check
\begin{equation*}
t(\rho_{i}''(s)) = t(\rho_{i}'(s))=\Phi_{i+1}'(s,t_i) =  \Phi_{i+1}''(s,t_i)
\end{equation*}
since $\varphi_{i+1}(s,t_i)=\tilde g_{i+1}(0)=1$, and
\begin{equation*}
s(\rho_{i}''(s)) =  s(\rho_{i}'(s)) =R(\Phi_i'(s,t_{i}),g_{i}'(s))
=R(\Phi'_i(s,t_i),\tilde g_i(s)^{-1}\varphi_i(s,t_i)g_i'(s) )  =  R(\Phi_i''(s,t_{i}),g_{i}''(s))\text{.}
\end{equation*}
Since $\varphi_i(0,t)=1$, \cref{lem:bigoncon:c*} is also not affected, and the new $\gamma_i'$ are horizontal.

Next we look for horizontality of the $\nu_i$. There exist paths $\tilde g_i: [0,1] \to G$ with $\tilde g_i(0)=1$ and $s \mapsto R(\nu_i(s),\tilde g_i(s))$ horizontal. We set $\rho_n' := \rho_n$ and
\begin{equation*}
\Phi'_i(s,t) := R(\Phi_i(s,t),\tilde g_i(s))
\quomma
\rho_i'(s) := R(\rho_i(s),\tilde g_{i+1}(s))
\quand
g_i'(s) := \tilde g_i(s)^{-1}g_i(s)\tilde g_{i+1}(s)\text{.}
\end{equation*}
Obviously, \cref{lem:bigoncon:a*} is not affected. For \cref{lem:bigoncon:b*} we check
\begin{align*}
t(\rho_i'(s)) &= R(t(\rho_i(s)),\tilde g_{i+1}(s))=R(\Phi_{i+1}(s,t_i),\tilde g_{i+1}(s)) = \Phi'_{i+1}(s,t_i)
\\
s(\rho_i'(s))&=R(s(\rho_i(s)),\tilde g_{i+1}(s))=R(\Phi_{i}(s,t_i),g_i(s)\tilde g_{i+1}(s))=R(\Phi'_{i}(s,t_i),g_i'(s))\text{.}
\end{align*}
Further, since $\nu_1$ is constant, we have $\tilde g_1=1$, meaning that $\rho_0'$ remains constant.
Since $\tilde g_i(0)=1$,  \cref{lem:bigoncon:c*} is  not affected, and since $\gamma_i'$ is shifted by constant $\tilde g_i(1)$, horizontality of $\gamma_i'$ is not spoiled. 

Finally, we look for horizontality of the $\rho_i$. By \cref{lem:makemorhor} there exist paths $\tilde h_i: [0,1] \to H$ with $\tilde h_i(0)=1$ and $s \mapsto R(\rho_i(s),(\tilde h_i(s),1))$ horizontal. By \cref{lem:hormor:c} also $s\mapsto R(\rho_i(s),(\tilde h_i(s),t(\tilde h_i(s)^{-1})))$ is horizontal. We set
\begin{equation*}
\Phi_i':=\Phi_i
\quomma
\rho'_i(s) :=R(\rho_i(s),(\tilde h_i(s),t(\tilde h_i(s)^{-1}))) 
\quand
g_i'(s) := g_i(s)t(\tilde h_i(s))^{-1}\text{.}
\end{equation*}
Obviously, \cref{lem:bigoncon:a*} is not affected. For \cref{lem:bigoncon:b*} we  check
\begin{align*}
t(\rho_i'(s))&=t(\rho_i(s))=\Phi_{i+1}(s,t_i)=\Phi_{i+1}'(s,t_i)
\\
s(\rho_i'(s)) &= R(s(\rho_i(s)),t(\tilde h_i(s))^{-1}) =R(\Phi_{i}(s,t_i),g_i(s)t(\tilde h_i(s))^{-1})=R(\Phi'_{i}(s,t_i),g_i'(s))\text{.}
\end{align*}
Since $\tilde h_i(0)=1$,  \cref{lem:bigoncon:c*} is  not affected, and since $\Phi_i$ is unchanged, horizontality of $\gamma_i'$ and $\mu_i'$ persists.
\end{proof}

Next we define the target of a horizontal lift of a small bigon. 
We set $\mu_i(s) := \Phi_{i}(s,t_{i})$; then we can reformulate \cref{lem:bigoncon:b*} as: 
\begin{equation*}
s(\rho_{i})=R(\mu_i,g_{i}) \text{ for all }1\leq i \leq n 
\quand
t(\rho_{i}) = \nu_{i+1}\text{ for all $0\leq i < n$.}
\end{equation*}
We consider a bigon-parameterization $\Sigma_i$ of $\Phi_i$, see \cref{rem:bigonpar}, and the associated surface-ordered exponential $h_i :=\SE_{\Omega}(\Sigma_i) \in H$ defined in \cref{sec:surfaceordered}. 
\begin{comment}
We have
\begin{equation*}
\Sigma_i: \mu_i \ast \gamma_i \Rightarrow \gamma_i' \ast \nu_{i}\text{,}
\end{equation*}
up to a thin homotopy on the boundary paths.
\end{comment}

\begin{lemma}
\label{lem:thi}
We have $t(h_i)=g_{i}(1)^{-1}$ for all $1\leq i \leq n$.
\end{lemma}

\begin{proof}
We have
$\Sigma_i: \tilde\mu_i \ast\tilde \gamma_i \Rightarrow \tilde\gamma_i' \ast\tilde \nu_{i}$, 
where $\tilde\mu_i$ is thin homotopic to $\mu_i$, and similar for the other paths. 
Since $\nu_i$, $\gamma_i$ and $\gamma_i'$ are horizontal, we have by \cref{lem:SE:b} $t(h_i)\PE_{\fa\Omega}(\mu_{i})=1$. Since $\rho_i$ and $t(\rho_i)$ are horizontal, 
\begin{comment}
($\rho_n$ is constant, so $t(\rho_n)$ is horizontal)
\end{comment}
$s(\rho_i)$ is horizontal by \cref{lem:hormor:e}. Together with \cref{lem:poeOmega:x} we obtain
\begin{equation*}
1=\PE_{\fa\Omega}(s(\rho_i))=\PE_{\fa\Omega}(R(\mu_i,g_{i})) = g_i(1)^{-1}\PE_{\fa\Omega}(\mu_i)\text{.} \eqendofproof
\end{equation*}
\end{proof}

We define  
\begin{equation}
\label{eq:def:targethorlift}
\rho_0' := \rho_0(1)
\quand
\rho_i' := R(\rho_i(1),(h_i^{-1},g_i(1)^{-1}))\text{.}
\end{equation}
It is straightforward to check that this gives an element $\xi' := \rho_n' \ast ... \ast \gamma_1' \ast \rho_0'$ in $F_{\gamma'}$, which we call the \emph{target} of the horizontal lift.  
\begin{comment}
First of all, the paths $\gamma_i''$ are horizontal by \cref{lem:bigoncon:d*} and \cref{lem:obhor}. 
For the well-definedness, we check
\begin{align*}
s(\rho_i')&= R(s(\rho_i(1)),g_i(1)^{-1})\eqcref{lem:bigoncon:b*} \mu_i(1)= \gamma_i'(t_i)
\end{align*}
and
\begin{align*}
t(\rho_i')&=R(t(\rho_i(1)),t(h_i)^{-1}g_i(1)^{-1})\eqcref{lem:bigoncon:b*} R(\nu_{i+1}(1),1)
= \gamma_{i+1}'(t_i)
\end{align*}
\end{comment}

\begin{lemma}
\label{lem:bigonmapdef}
If $\Sigma:\gamma \Rightarrow \gamma'$ is a small bigon and $\xi\in F_{\gamma}$, then the target of a horizontal lift of $\Sigma$ with source $\xi$ is independent of the choice of the horizontal lift.
\end{lemma}

\begin{proof}
Let $(n,t,\Phi_i,\rho_i,g_i)$ and $(\tilde n,\tilde t,\tilde \Phi_i,\tilde \rho_i,\tilde g_i)$ be horizontal lifts of $\Sigma$ with source $\xi$. First of all, we can assume that $\tilde n=n$ and $\tilde t=t$, since we can introduce new points $t_{i-1}<t'<t_{i}$ and then cut a horizontal lift at $t'$ (it is easy to see that one can arrange the new path $s\mapsto \Phi_i(s,t')$ to be horizontal by compensating with the map $g_i'$).

Next, we note that we have two sections $\Phi_i$ and $\tilde\Phi_i'$ into $\ob{\inf P}$ along $\Sigma|_{[0,1]\times[t_{i-1},t_i]}$. By \reftransitionspan\ they admit a transition span $\Psi_i$ with transition function $G_i$.
\begin{comment}
That is,
\begin{equation*}
s(\Psi_i) = R(\Phi_i,G_i)
\quand
t(\Psi_i)= \tilde\Phi_i\text{.}
\end{equation*}
\end{comment}
In the following we show that we can assume a couple of properties for $\Psi_i$ and $G_i$.

The condition that our horizontal lifts have the same source, $\xi$, means that there exist horizontal path $\eta_i:[t_{i-1},t_i] \to \mor{\inf P}$ with $s(\eta_i(t))= \Phi_i(0,t)$ and $t(\eta_i(t))=\tilde\Phi_i(0,t)$ and $\tilde\rho_{i}\circ \eta_{i}(t_{i})=\eta_{i+1}(t_{i})\circ \rho_{i}$. Comparing the transition spans $(\eta_i,1)$ with $(\Psi_i(0,-),G_i(0,-))$ we obtain by \reftransitionspan\ a smooth map $h_i:[t_{i-1},t_i] \to H$ such that $R(\Psi_i(0,t),(h_i(t),G_i(0,t)^{-1})=\eta_i(t)$ and $t(h_i(t))=G_i(0,t)$. We consider 
\begin{equation*}
\Psi'_i(s,t) :=R(\Psi_i(s,t),(h_i(t),G_i(0,t)^{-1})
\quand
G_i'(s,t) :=G_i(s,t)t(h_i(t))^{-1}\text{,}
\end{equation*}
which satisfy $t(\Psi'_i(s,t)) =\tilde\Phi_i(s,t)$ and $s(\Psi'_i(s,t)) = R(\Phi_i(s,t),G_i'(s,t))$.
\begin{comment}
Indeed,
\begin{align*}
t(\Psi'_i(s,t)) &= R(t(\Psi_i(s,t)),t(h_i(t))G_i(0,t)^{-1}) =\tilde\Phi_i(s,t)
\\
s(\Psi'_i(s,t)) &=R(s(\Psi_i(s,t)),G_i(0,t)^{-1})=R(\Phi_i(s,t),G_i(s,t)t(h_i(t))^{-1})=R(\Phi_i(s,t),G_i'(s,t))\text{.} \end{align*}
\end{comment}
This shows that  we can always choose our transition spans $\Psi_i$ such that $t \mapsto \Psi_i(0,t)$ is horizontal, $G_i(0,t)=1$ and 
\begin{equation}
\label{eq:initialspan}
\tilde\rho_{i}\circ \Psi_i(0,t_i)=\Psi_{i+1}(0,t_i)\circ \rho_{i}\text{.}
\end{equation}

Next we consider the path $s \mapsto \Psi_i(s,t_{i-1})$. By \cref{lem:makemorhor} there exist paths $h_i:[0,1] \to H$ with $h_i(0)=1$ such that $R(\Psi_i(-,t_{i-1}),(h_i,1))$ is horizontal. We consider 
\begin{equation*}
\Psi_i'(s,t) := R(\Psi_i(s,t),(h_i(s),t(h_i(s))^{-1}))
\quand
G_i'(s,t) :=G_i(s,t)t(h_i(s))^{-1}\text{,}
\end{equation*}
this gives by \cref{lem:hormor:c} a horizontal path,  satisfying $t(\Psi_i'(s,t))=\tilde\Phi_i(s,t)$ and $s(\Psi_i'(s,t))=R(\Phi_i(s,t),G_i'(s,t))$.
\begin{comment}
Indeed,
\begin{align*}
t(\Psi_i'(s,t)) &=t(\Psi_i(s,t))=\tilde\Phi_i(s,t) 
\\
s(\Psi_i'(s,t)) &=R(s(\Psi_i(s,t)),t(h_i(s))^{-1}) =R(\Phi_i(s,t),G_i(s,t)t(h_i(s))^{-1})=R(\Phi_i(s,t),G_i'(s,t))
\end{align*}  
\end{comment}
Since the quantities at $s=0$ are unchanged, we can add the horizontality of $s \mapsto \Psi_i(s,t_{i-1})$ to our assumptions.

Finally, we consider  the path $t \mapsto \Psi_i(1,t)$. By \cref{lem:makemorhor} there exists a path $h_i:[0,1] \to H$ with $h_i(0)=1$ such that $R(\Psi_i(1,-),(h_i,1))$ is horizontal. We consider 
\begin{align*}
\Psi_i'(s,t) &:=\textstyle R(\Psi_i(s,t),(h_i(s\frac{t-t_{i-1}}{t_{i}-t_{i-1}}),t(h_i(s\frac{t-t_{i-1}}{t_{i}-t_{i-1}}))^{-1}))
\\
G_i'(s,t) &:=\textstyle G_i(s,t)t(h_i(s\frac{t-t_{i-1}}{t_{i}-t_{i-1}}))^{-1}\text{,}
\end{align*}
satisfying $t(\Psi_i'(s,t)) =\tilde\Phi_i(s,t)$  and $s(\Psi_i'(s,t)) =R(\Phi_i(s,t),G_i'(s,t))$.
\begin{comment}
Indeed,
\begin{align*}
t(\Psi_i'(s,t)) &=t(\Psi_i(s,t))=\tilde\Phi_i(s,t) 
\\
s(\Psi_i'(s,t)) &=R(s(\Psi_i(s,t)),t(h_i(s\frac{t-t_{i-1}}{t_{i}-t_{i-1}}))^{-1}) \\&=R(\Phi_i(s,t),G_i(s,t)t(h_i(s\frac{t-t_{i-1}}{t_{i}-t_{i-1}}))^{-1})
\\&=R(\Phi_i(s,t),G_i'(s,t))
\end{align*}  
\end{comment}
Since the quantities at $s=0$ and $t=t_{i-1}$ are unchanged, we can add the horizontality of $t \mapsto \Psi_i(1,t)$ to our assumptions.

We continue with choices $\Psi_i$ and $G_i$ satisfying all assumptions collected above. We notice the following: since $\Psi_i(-,t_{i-1})$ is horizontal, and $t(\Psi_i(-,t_{i-1}))=\tilde\Phi_i(-,t_{i-1})=\tilde\nu_i$ is horizontal, we have by \cref{lem:hormor:e}  that $s(\Psi_i(-,t_{i-1}))=R(\nu_i,G_i(-,t_{i-1}))$ is horizontal, too. But since $\nu_i$ itself is horizontal, it follows that $s \mapsto G_i(s,t_{i-1})$ is constant, i.e. $G_i(s,t_{i-1})=1$. 
With the same argument, we have that $t\mapsto G_i(1,t)$  is constant, i.e. $G_i(1,t)=1$.

Next we consider bigon-parameterizations $\tilde\Psi_i: \beta_i \Rightarrow \beta_i'$ of $\Psi_i$ and $\tilde G_i: \mu_i \Rightarrow \mu_i'$ of $G_i$, see \cref{rem:bigonpar}.
\begin{comment}
We have, up to thin homotopies
\begin{align*}
\beta_i&:=\Psi_i(-,t_i)\circ \Psi_i(0,-)
&\quand
\beta_i'&:=\Psi_i(1,-)\circ \Psi_i(-,t_{i-1})
\\
\mu_i&:=G_i(-,t_i)\circ \id_1
&\quand
\mu_i'&:=\id_1
\end{align*}
\end{comment}
We have by \cref{prop:SE:gauge:z,co:hcalc:a}
\begin{align*}
h_{\Omega}(\beta_i) & = h_{\Omega}(\Psi_i(-,t_i))\cdot \alpha(\PE_{\fa\Omega}(s(\Psi_i(-,t_i))),h_{\Omega}(\Psi_i(0,-)))= h_{\Omega}(\Psi_i(-,t_i))
\\
h_{\Omega}(\beta_i') &= h_{\Omega}(\Psi_i(1,-))\cdot \alpha(\PE_{\fa\Omega}(s(\Psi_i(1,-))),h_{\Omega}(\Psi_i(-,t_{i-1})))=  1
\end{align*}
\begin{comment}
More explicitly,
\begin{align*}
h_{\Omega}(\beta_i) & \eqcref{prop:SE:gauge:z} h_{\Omega}(\Psi_i(-,t_i))\cdot \alpha(\PE_{\fa\Omega}(s(\Psi_i(-,t_i))),h_{\Omega}(\Psi_i(0,-)))\eqcref{co:hcalc:a} h_{\Omega}(\Psi_i(-,t_i))
\\
h_{\Omega}(\beta_i') &\eqcref{prop:SE:gauge:z} h_{\Omega}(\Psi_i(1,-))\cdot \alpha(\PE_{\fa\Omega}(s(\Psi_i(1,-))),h_{\Omega}(\Psi_i(-,t_{i-1})))\eqcref{co:hcalc:a}  1
\end{align*}
\end{comment}
Further, we have $\SE_{\Omega}(R(s(\tilde \Psi_i),\tilde G_i^{-1}))=h_i$ and $\SE_{\Omega}(t(\tilde \Psi_i))=\tilde h_i$. 
\begin{comment}
More explicitly,
\begin{align*}
\SE_{\Omega}(R(s(\tilde \Psi_i),\tilde G_i^{-1}))=\SE_{\Omega}(\Phi_i)=h_i
\\
\SE_{\Omega}(t(\tilde \Psi_i))=\SE_{\Omega}(\tilde \Phi_i)=\tilde h_i\text{,}
\end{align*}
\end{comment}
Now \cref{prop:SE:gauge:b} implies
\begin{equation}
\label{eq:hcomp}
h_i \cdot h_{\Omega}(\beta_i)^{-1}= \tilde h_i\text{.}
\end{equation}
\begin{comment}
We don't need this directly, but might want to compute $t(h_{g,\varphi}(\beta_i,\mu_i))$. 
We have from
\cref{lem:SE:gauge:a}
\begin{align*}
\PE_{\fa\Omega}(\chi_2(\beta_i)) &= t(h_{g,\varphi}(\beta_i)^{-1})\cdot G_i(1,t_i)\cdot \PE_{\fa\Omega}(\chi_1(\beta_i))
\\
\PE_{\fa\Omega}(\chi_2(\beta_i')) &= t(h_{g,\varphi}(\beta_i')^{-1})\cdot G_i(1,t_i)\cdot \PE_{\fa\Omega}(\chi_1(\beta_i'))
\end{align*}
Calculating
\begin{align*}
\PE_{\fa\Omega}(\chi_1(\beta_i)) &= \PE_{\fa\Omega}(\tilde\Phi_i(-,t_i))\cdot \PE_{\fa\Omega}(\tilde\Phi_i(0,-))=\PE_{\fa\Omega}(\tilde\mu_i)\eqcref{lem:thi} \tilde g_i(1) 
\\
\PE_{\fa\Omega}(\chi_1(\beta_i')) &= \PE_{\fa\Omega}(\tilde\Phi_i(1,-))\cdot \PE_{\fa\Omega}(\tilde\Phi_i(-,t_{i-1}))=1
\\
\PE_{\fa\Omega}(\chi_2(\beta_i)) &= \PE_{\fa\Omega}(\Phi_i(-,t_i))\cdot \PE_{\fa\Omega}(\Phi_i(0,-))=\PE_{\fa\Omega}(\mu_i)\eqcref{lem:thi} g_i(1)
\\
\PE_{\fa\Omega}(\chi_2(\beta_i')) &= \PE_{\fa\Omega}(\Phi_i(1,-))\cdot \PE_{\fa\Omega}(\Phi_i(-,t_{i-1}))=1
\end{align*}
we obtain
\begin{align*}
t(h_{g,\varphi}(\beta_i,\mu_i))= \tilde g_i(1)g_i(1)^{-1} 
\text{.}
\end{align*}
\end{comment}
We notice that
\begin{align*}
&\Psi_{i+1}(s,t_{i}) \circ \rho_i(s)
&&\quand
g_i(s)
\\
&\tilde\rho_i(s)\circ R(\Psi_i(s,t_i),\tilde g_i(s))
&&\quand
G_i(s,t_i) \tilde g_i(s)
\end{align*}
are two transition spans with transitions functions along $s \mapsto (\tilde\Phi(s,t_i),\Phi(s,t_i))$.
\begin{comment}
Indeed,
\begin{equation*}
t(\Psi_{i+1}(s,t_{i}))=\tilde\Phi_{i+1}(s,t_{i})=t(\tilde\rho_i(s))
\end{equation*}
and
\begin{align*}
R(s(\rho_i(s),g_i(s)^{-1})
&=\Phi_{i}(s,t_i)
\\& =R(s(\Psi_i(s,t_i)),G_i(s,t_i)^{-1})
\\&=R(s(R(\Psi_i(s,t_i),\tilde g_i(s))),\tilde g_i(s)^{-1}G_i(s,t_i)^{-1})
\end{align*}
\end{comment}
Hence, by \reftransitionspan\ there exists a unique path $\eta:[0,1] \to H$ with
\begin{align}
g_i(s)t(\eta(s)) &=G_i(s,t_i) \tilde g_i(s)
\nonumber
\\
\label{eq:etacomp}
R(\tilde\rho_i(s)\circ R(\Psi_i(s,t_i),\tilde g_i(s)),(\eta(s),t(\eta(s))^{-1})) &=\Psi_{i+1}(s,t_{i}) \circ \rho_i(s)
\end{align}
From \cref{eq:initialspan} we conclude that $\eta(0)=1$. \cref{eq:etacomp} is equivalent to: 
\begin{equation*}
 R(\Psi_i(s,t_i),G_i(s,t_i)^{-1}) = R(\tilde\rho_i(s)^{-1}\circ \Psi_{i+1}(s,t_{i}) \circ \rho_i(s),(\eta(s)^{-1},g_i(s)^{-1})\text{.}
\end{equation*}
In their dependence on $s$, this is an equality between two paths in $\mor{\inf P}$.
We compute $h_{\Omega}$ on both sides. On the left, we obtain $h_{\Omega}(\beta_i)$ via \cref{co:h:b}.
\begin{comment}
Indeed,
\begin{equation*}
h_{\Omega}(R(\Psi_i(-,t_i),G_i(-,t_i)^{-1}))\eqcref{co:h:b}  h_{\Omega}(\Psi_i(-,t_i))=h_{\Omega}(\beta_i)\text{.}
\end{equation*}
\end{comment}
On the right we compute:
\begin{align*}
&\mquad h_{\Omega}(R(\tilde\rho_i(s)^{-1}\circ \Psi_{i+1}(s,t_{i}) \circ \rho_i(s),(\eta(s)^{-1},g_i(s)^{-1}))
\\&=h_{\Omega}(R(R(\tilde\rho_i(s)^{-1}\circ \Psi_{i+1}(s,t_{i}) \circ \rho_i(s),(\eta(s)^{-1},1)),g_i(s)^{-1}))
\\&\eqcref{co:h:b} \alpha(g_i(1),h_{\Omega}(R(\tilde\rho_i(s)^{-1}\circ \Psi_{i+1}(s,t_{i}) \circ \rho_i(s),(\eta(s)^{-1},1)),1))
\\&\eqcref{co:hcalc:b} \alpha(g_i(1),\eta(1))\text{.}
\end{align*}
In the last step we have used that  $\tilde\rho_i(s)^{-1}\circ \Psi_{i+1}(s,t_{i}) \circ \rho_i(s)$ is horizontal (\cref{lem:hormor:d}) and has horizontal source $s(\rho_i)$ (since $\rho_i$ is horizontal and $t(\rho_i)=\nu_{i+1}$ is horizontal, see \cref{lem:hormor:e}). Summarizing, we have 
$\eta(1) = \alpha(g_i(1)^{-1},h_{\Omega}(\beta_i) )$.
Thus, we get from \cref{eq:etacomp}:
\begin{equation}
\label{eq:etacomp2}
\tilde\rho_i(1)\circ R(\Psi_i(1,t_i),(h_{\Omega}(\beta_i,\mu_i),g_i(1))) =\Psi_{i+1}(1,t_{i}) \circ \rho_i(1)\text{.}
\end{equation}

Now we consider the paths $t \mapsto \Psi_i'(t):= \Psi_{i}(1,t)$ in $\mor{\inf P}$ which are horizontal and have $t(\Psi_i')=\tilde\gamma_i'$ and $s(\Psi'_i)=\gamma_i'$.
\begin{comment}
Indeed, 
\begin{align*}
t(\Psi'_i(t))&= t(\Psi_{i}(1,t)) =  \tilde\Phi_i(1,t) = \tilde\gamma_i'(t)
\\
s(\Psi'_i(t))&=s(\Psi_{i}(1,t))=\Phi_i(1,t)=\gamma_i'(t)\text{.}
\end{align*} 
\end{comment}
We claim that they establish an equivalence between the targets of the two horizontal lifts. This is confirmed by the following calculation:
\begin{align*}
\tilde\rho_{i}'\circ \Psi'_{i}(t_{i})&= \tilde\rho_i(1)\circ R(\Psi_i(1,t_i),(\tilde h_{i}^{-1},1))
\\&\eqcref{eq:hcomp} \tilde\rho_i(1)\circ R(\Psi_i(1,t_i),(h_{\Omega}(\beta_i)h_{i}^{-1},1))
\\&= R(\tilde\rho_i(1)\circ R(\Psi_i(1,t_i),(h_{\Omega}(\beta_i),g_i(1))),(h_{i}^{-1},g_{i}(1)^{-1}))
\\&\eqcref{eq:etacomp2} R(\Psi_{i+1}(1,t_{i})\circ \rho_{i}(1),(h_{i}^{-1},g_{i}(1)^{-1}))
\\&=\Psi'_{i+1}(t_{i})\circ \rho'_{i}\text{.}
\\[-3.6em]
\end{align*}
\begin{comment}
An extended version of this calculation is:
\begin{align*}
\tilde\rho_{i}'\circ \Psi'_{i}(t_{i})&= R(\tilde\rho_{i}(1),(\tilde h_{i}^{-1},\tilde g_{i}(1)^{-1}))\circ \Psi_{i}(1,t_{i}) 
\\&= R(R(\tilde\rho_i(1),\tilde g_i(1)^{-1})\circ\Psi_i(1,t_i)),(\tilde h_{i}^{-1},1))
\\&= R(R(\tilde\rho_i(1),\tilde g_i(1)^{-1}),\tilde g_i(1))\circ R(\Psi_i(1,t_i),(\tilde h_{i}^{-1},1))
\\&= \tilde\rho_i(1)\circ R(\Psi_i(1,t_i),(\tilde h_{i}^{-1},1))
\\&\eqcref{eq:hcomp} \tilde\rho_i(1)\circ R(\Psi_i(1,t_i),(h_{\Omega}(\beta_i)h_{i}^{-1},1))
\\&= R(\tilde\rho_i(1)\circ R(\Psi_i(1,t_i),(h_{\Omega}(\beta_i),g_i(1))),(h_{i}^{-1},g_{i}(1)^{-1}))
\\&\eqcref{eq:etacomp2} R(\Psi_{i+1}(1,t_{i})\circ \rho_{i}(1),(h_{i}^{-1},g_{i}(1)^{-1}))
\\&= \Psi_{i+1}(1,t_{i})\circ R(\rho_{i}(1),(h_{i}^{-1},g_{i}(1)^{-1}))
\\&=\Psi'_{i+1}(t_{i})\circ \rho'_{i}
\end{align*}
\end{comment}
\end{proof}

\subsection{Definition of parallel transport along bigons}

\label{sec:defpartransbigon}

The results of the previous section show that choosing a horizontal lift of a small bigon and computing its target establishes a well-defined  map $\varphi_{\Sigma}^{small}:F_{\gamma} \to F_{\gamma'}$.
 Before we extend it to arbitrary bigons, we discuss some properties. 
\begin{lemma}
\label{lem:varphisigma}
The map $\varphi_{\Sigma}^{small}$ has the following properties:
\begin{enumerate}[(a)]

\item 
\label{lem:varphisigma:a}
It preserves the anchors $\alpha_l$ and $\alpha_r$.

\item
\label{lem:varphisigma:b}
It is equivariant with respect to the left $\inf P_{x}$-action and the right $\inf P_{y}$-action.

\item
\label{lem:varphisigma:c}
It is equivariant with respect to the $\mor\Gamma$-action. 
\item
\label{lem:varphisigma:e}
It is smooth.

\end{enumerate}
\end{lemma}

\begin{proof}
\cref{lem:varphisigma:a*} is straightforward to check using that $\rho_0$ and $\rho_n$ are constant. \begin{comment}
Indeed, for $\alpha_l$ we have
\begin{equation*}
\alpha_l(\xi')=s(\rho_0')=s(\rho_0(1))=s(\rho_0(0))=s(\rho_0)=\alpha_l(\xi)\text{.}
\end{equation*}
For $\alpha_r$ we have
\begin{equation*}
\alpha_r(\xi') = t(\rho_n') =R(t(\rho_n(1)),t(h_n)^{-1}g_n(1)^{-1})=t(\rho_n(0))=t(\rho_n)= \alpha_r(\xi)
\end{equation*}
\end{comment}
In \cref{lem:varphisigma:b*} is even more obvious. In \cref{lem:varphisigma:c*} we have
to prove coincidence between\begin{equation*}
\varphi_{\Sigma}^{small}(\rho_n \ast \gamma_n \ast ... \ast \gamma_1\ast \rho_0) \cdot (h,g)=R(\rho_n',g) \ast R(\gamma_n',g) \ast ... \ast R(\gamma_1',g)\ast R(\rho_0',(h^{-1},t(h)g))
\end{equation*}
\begin{comment}
In fact,
\begin{align*}
\varphi_{\Sigma}^{small}(\rho_n \ast \gamma_n \ast ... \ast \gamma_1\ast \rho_0) \cdot (h,g)&=(\rho_n' \ast \gamma_n' \ast ... \ast \gamma_1'\ast \rho_0')\cdot (h,g)
\\&=R(\rho_n',g) \ast R(\gamma_n',g) \ast ... \ast R(\gamma_1',g)\ast R(\rho_0',(h^{-1},t(h)g))
\end{align*}
\end{comment}
and
\begin{align*}
&\mquad\varphi_{\Sigma}^{small}((\rho_n \ast \gamma_n \ast ... \ast \gamma_1\ast \rho_0) \cdot (h,g))
\\&=\varphi_{\Sigma} ^{small}(R(\rho_n,g) \ast R(\gamma_n,g) \ast ... \ast R(\gamma_1,g)\ast R(\rho_0,(h^{-1},t(h)g)))\text{.} 
\end{align*}
Let a horizontal lift of $\rho_n \ast \gamma_n \ast ... \ast \gamma_1\ast \rho_0$ consist of  $\Phi_i$, $\rho_i$ and $g_i$. We consider for $1\leq i \leq n$ the data of $\tilde\Phi_i := R(\Phi_i,g)$, $\tilde\rho_i := R(\rho_i,g)$ and $\tilde g_i := g^{-1}g_ig$, as well as $\tilde\rho_0:= R(\rho_0, (h^{-1},t(h)g))$. It is straightforward to check that this is a horizontal lift of $(\rho_n \ast \gamma_n \ast ... \ast \gamma_1\ast \rho_0) \cdot (h,g)$.
\begin{comment}
Indeed, \cref{lem:bigoncon:a*} and \cref{lem:bigoncon:d*} are obvious. For \cref{lem:bigoncon:b*} we check
\begin{equation*}
t(\tilde\rho_0(s))=R(t(\rho_0), g)=\Phi_{1}(s,g)=\tilde\Phi_1(s,t_i)\text{,}
\end{equation*}
for $1\leq i < n$
\begin{equation*}
t(\tilde\rho_{i}(s)) =R(t(\rho_i(s)),g)=R(\Phi_{i+1}(s,t_i),g) = \tilde\Phi_{i+1}(s,t_i)\text{,}
\end{equation*}
and for $1\leq i \leq n$:
\begin{equation*}
s(\tilde \rho_{i}(s))=R(s( \rho_{i}(s)),g)=R(\Phi_i(s,t_i),gg^{-1}g_i(s)g)=R(\tilde \Phi_i(s,t_{i}),\tilde g_{i}(s))\text{.}
\end{equation*}
For \cref{lem:bigoncon:c*} we have to check:
\begin{align*}
R(\gamma_i(t),g) &=R(\Phi_i(0,t),g)=\tilde \Phi_i(0,t)
\\
R(\rho_i,g) &=R(\rho_i(0),g)= \tilde \rho_i(0)
\\
R(\rho_0,(h^{-1},t(h)g)) &=R(\rho_0(0), (h^{-1},t(h)g))=\tilde\rho_0(0)
\end{align*}
\end{comment}
The target of $(\Phi_i,\rho_i, g_i)$ consists by definition of the paths $\gamma_i'(t) := \Phi_i(1,t)$ and the morphisms $\rho_0' := \rho_0(1)
$ and $
\rho_i' := R(\rho_i(1),(h_i^{-1},g_i(1)^{-1}))$, where $h_i:=\SE_{\Omega}(\Sigma_i)$ and $\Sigma_i$ is a bigon-parameterization of $\Phi_i$. Now we compute the target of $(\tilde\Phi_i,\tilde\rho_i, \tilde g_i)$. We have $\tilde\gamma_i'(t)=\tilde\Phi_i(1,t)=R(\Phi_i(1,t),g)=R(\gamma_i'(t),g)$. We  use the bigon-parameterization  $\tilde\Sigma_i(s,t) := R(\Sigma_i(s,t),g)$ and $\tilde h_i :=\SE_{\Omega}(\tilde\Sigma_i)$. By \cref{lem:soe2bun} we get
$\tilde h_i = \alpha(g^{-1}, h_i)$.
Then, a short calculation shows that $\tilde\rho_0'=R(\rho_0', (h^{-1},t(h)g))$ and $\tilde\rho_i'=R(\rho_i',g)$.
\begin{comment}
Indeed, $\tilde\rho_0'=\tilde\rho_0(1)=R(\rho_0(1), (h^{-1},t(h)g))=R(\rho_0', (h^{-1},t(h)g))$ and
\begin{align*}
\tilde\rho'_i&=R(\tilde \rho_i(1),(\tilde h_i^{-1},\tilde g_i(1)^{-1}))
\\&=R(\rho_i(1),(1,g)(\tilde h_i^{-1},\tilde g_i(1)^{-1}))
\\&=R(\rho_i(1),(\alpha(g,\tilde h_i^{-1}),g\tilde g_i(1)^{-1}))
\\&=R(\rho_i(1),(h_i^{-1},g_i(1)^{-1})(1,g))
\\&=R(\rho_i',g)\text{.}
\end{align*}
\end{comment}
This shows the required coincidence. 
For \cref{lem:varphisigma:e*} we consider a chart
$\phi_{\xi_0,\rho,g}$
of $F_{\gamma}$, and a chart $\phi_{\xi_0',\rho,g}$ of $F_{\gamma'}$ with $\xi_0':=  \varphi_{\Sigma}^{small}(\xi_0)$. In these charts, $\varphi_{\Sigma}^{small}$ is the identity, as one can see using \cref{lem:varphisigma:b*} and \cref{lem:varphisigma:c*};  in particular, it is smooth. \begin{comment}
Indeed, the map
\begin{equation*}
\alxydim{}{U \ttimes{\alpha_r\circ \sigma_{\xi_0,\rho,g}}{t} \mor{\inf P_y} \ar[r] & F_{\gamma} \ar[r]^{\varphi^{small}_{\Sigma}} & F_{\gamma'} \ar[r] & U \ttimes{\alpha_r\circ \sigma_{\xi_0',\rho',g'}}{t} \mor{\inf P_y}}
\end{equation*}
is given by
\begin{align*}
(p,\tilde\rho) &\mapsto \sigma_{\xi_0,\rho,g}(p)\circ \tilde\rho =  ((\rho(p_0,p)^{-1}\circ \xi_0)\cdot (1,g(p_0,p)^{-1}))\circ \tilde\rho 
\\&\mapsto ((\rho(p_0,p)^{-1}\circ \xi_0')\cdot (1,g(p_0,p)^{-1}))\circ \tilde\rho =\sigma_{\xi_0',\rho,g}(p)\circ \tilde\rho 
\\&\mapsto (p,\tilde\rho)
\end{align*}
\end{comment}
\end{proof}

Next we extend $\varphi^{small}_{\Sigma}$ to arbitrary bigons. For an arbitrary bigon $\Sigma:\gamma \Rightarrow \gamma'$ there exists a subdivision $s\in T_n$ (i.e., $0=s_0<...<s_n=1$) such that the  pieces 
$\Sigma_i(s,t):= \Sigma((s_i-s_{i-1})s+s_{i-1},t)$
are small. We define 
\begin{equation*}
\varphi_{\Sigma}(s) := \varphi^{small}_{\Sigma_n} \circ ... \circ \varphi^{small}_{\Sigma_1}\text{.}
\end{equation*}

\begin{lemma}
\label{lem:indepsshow}
The map $\varphi_{\Sigma}(s)$ is independent of the choice of $s$.
\end{lemma}

\begin{proof}
It suffices to prove that, for a small bigon $\Sigma$, $\varphi_{\Sigma}(s)=\varphi_{\Sigma}(s')$, where $s\in T_1$ and $s'\in T_2$ with $0=s_0'<s_1'<s_2'=1$. Thus, we have to show that
\begin{equation}
\label{eq:lemindepsshow}
\varphi^{small}_{\Sigma}= \varphi^{small}_{\Sigma_2} \circ \varphi^{small}_{\Sigma_1}\text{,}
\end{equation}
where $\Sigma_1$ and $\Sigma_2$ are (reparameterizations of) $\Sigma|_{[0,s_1'] \times [0,1]}$ and $\Sigma|_{[s_1',1] \times [0,1]}$, respectively. 
We choose a horizontal lift $(n,t,\Phi_i,\rho_i,g_i)$ of $\Sigma$ with source $\xi\in F_{\gamma}$. By a slight modification of the arguments in the proof of \cref{lem:existencehorizontallifts} we can assume that the paths $\gamma_i'(t):= \Phi_i(s_1',t)$ are horizontal.  We consider the elements $h_i\in H$ and $\xi'' :=\rho_n'' \ast \gamma_n'' \ast ... \ast \gamma_1''\ast \rho_0''$ with $\gamma_i''(t) := \Phi_i(1,t)$, $\rho_0'' := \rho_0(1)$, and $\rho_i'' := R(\rho_i(1),(h_i^{-1},g_i(1)^{-1}))$. 
By restriction of all parameters $s$ to $0=s_0'\leq s \leq s_1'$ and reparameterization to $[0,1]$, we obtain a horizontal lift $(n,t,\Phi_1^{1},\rho_i^1,g_i^1)$ of $\Sigma_1$ with source $\xi$.
\begin{comment}
It is is important that reparameterizations of horizontal paths are horizontal, and it is important that the restriction goes to $s=0$ on the left -- because the condition $g_i(0)=1$ has to be satisfied. 
\end{comment} 
We consider the elements $h_i^1 \in H$ and the target $\xi'=\rho_n' \ast \gamma_n' \ast ... \ast \gamma_1'\ast \rho_0'$ with $\gamma_i'(t) := \Phi^1_i(1,t)$, $\rho_0' := \rho_0^1(1)$, and $\rho_i' := R(\rho_i^1(1),((h_i^1)^{-1},g_i^1(1)^{-1}))$. Let $(\Phi_i^2,\tilde \rho_i^2,\tilde g_i^2)$ denote the restriction of $(\Phi_i,\rho_i,g_i)$ to $s_1'\leq s \leq s_2'=1$. Define the following modification:
\begin{equation*}
\rho^2_0(s) := \tilde \rho^2_0(s)
\quomma
\rho^2_i(s) := R(\tilde \rho^2_i(s),((h_i^1)^{-1},g_i^1(1)^{-1}))
\quand
g^2_i(s) :=  \tilde g_i^2 g_i^1(1)^{-1}
\text{.}
\end{equation*}
Then, $(n ,t,\Phi_i^2,\rho_i^2,g_i^2)$ is a horizontal lift of $\Sigma_2$ with source $\xi'$. 
\begin{comment}
We have $g_i^2(0)=\tilde g_i^2(0)g_i^1(1)^{-1}=g_i(s_1')g_i(s_1')^{-1}=1$. Further,
\begin{equation*}
t(\rho_i^2(s))=t(\tilde\rho_i^2(s))=\Phi_{i+1}^2(s,t_i)
\end{equation*} 
and
\begin{equation*}
s(\rho_i^2(s))=R(s(\tilde \rho^2_i(s)),g_i^1(1)^{-1})=R(\Phi^2_i(s,t_i),\tilde g_i^2(s) g_i^1(1)^{-1})=R(\Phi^2_i(s,t_i),g^2_i(s))\text{.}
\end{equation*}
\end{comment}
We consider again the corresponding elements $h_i^2\in H$.  We have $g_i^2(1) = g_i(1) g_i^1(1)^{-1}$ and 
\begin{equation*}
h_i = h_i^2\alpha(\PE_{\fa\Omega}(\mu_i),h_i^1)= h_i^2\alpha(g_i^2(1),h_i^1)=h_i^2\alpha(t(h_i^2)^{-1},h_i^1)=h_i^1h_i^2\text{.}
\end{equation*}
Then we obtain
$\Phi_i^2(1,t)=\gamma_i''(t)$, $\rho_0^2(1)=\rho_0''$ and $R(\rho_i^2(1),((h_i^2)^{-1},g_i^2(1)^{-1}))= \rho_i''$.
\begin{comment}
Indeed,
\begin{align*}
\Phi_i^2(1,t) &= \Phi_i(1,t)=\gamma_i''(t)
\\
\rho_0^2(1) &= \tilde \rho_0^2(1)=\rho_0(1)=\rho_0''
\\
R(\rho_i^2(1),((h_i^2)^{-1},g_i^2(1)^{-1})) &=R(\tilde \rho_i^2(1),((h_i^1)^{-1},g_i^1(1)^{-1})((h_i^2)^{-1},g_i^2(1)^{-1}))
\\&=R( \rho_i(1),((h_i^1)^{-1}\alpha(g_i^1(1)^{-1},(h_i^2)^{-1}),g_i^1(1)^{-1}g_i^2(1)^{-1})) \\&=R( \rho_i(1),((h_i^1)^{-1}\alpha(t(h_i^1),(h_i^2)^{-1}),g_i(1)^{-1})) \\&= R(\rho_i(1),(h_i^{-1},g_i(1)^{-1}))
\\&= \rho_i'' 
\end{align*}
\end{comment}
This shows \cref{eq:lemindepsshow}. 
\end{proof}

By \cref{lem:indepsshow}, we simply write $\varphi_{\Sigma}$ for any of the maps $\varphi_{\Sigma}(s)$, and summarize the properties of \cref{lem:varphisigma} as follows.

\begin{proposition}
\label{prop:varphiSigma}
The  map $\varphi_{\Sigma}:F_{\gamma} \to F_{\gamma'}$ is a $\Gamma$-equivariant transformation.
\end{proposition}

\subsection{Compatibility with bigon composition}

\label{sec:compbigoncomp}

Bigons can be composed in two ways, vertically and horizontally, which can most easily be described by a picture:
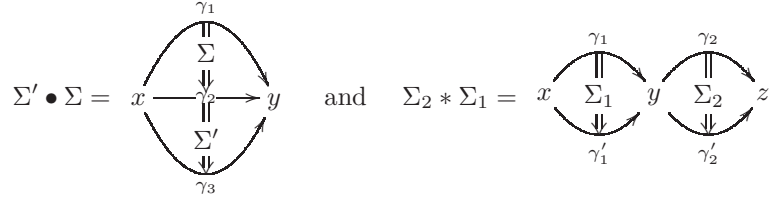
\begin{equation*}
\Sigma' \bullet \Sigma = \alxydim{@C=4em}{x \ar@/^2.4pc/[r]^{\gamma_1}="1" \ar[r]|{\gamma_2}="2"  \ar@/_2.4pc/[r]_{\gamma_3}="3" & y \ar@{=>}"1";"2"|*+{\Sigma}\ar@{=>}"2";"3"|--*+{\Sigma'}}
\quand
\Sigma_2 \ast \Sigma_1 = \alxydim{@C=\xyst}{x \ar@/^1.3pc/[r]^{\gamma_1}="1"\ar@/_1.3pc/[r]_{\gamma_1'}="2" & y \ar@/^1.3pc/[r]^{\gamma_2}="3"\ar@/_1.3pc/[r]_{\gamma_2'}="4" & z \ar@{=>}"1";"2"|*+{\Sigma_1} \ar@{=>}"3";"4"|*+{\Sigma_2}}
\end{equation*}
A more detailed description of bigon composition can be found in \cite[Section 2.1]{schreiber5}.
The content of the following two propositions is that parallel transport along bigons is compatible with these two compositions. In the transport 2-functor formalism described in \cref{sec:orga} they prove the functoriality of the 2-functor on the level of 2-morphisms. 

\begin{proposition}
\label{lem:F1}
Let $\inf P$ be a principal $\Gamma$-2-bundle with fake-flat connection.
Suppose $\Sigma:\gamma_1\Rightarrow \gamma_2$ and $\Sigma':\gamma_2\Rightarrow \gamma_3$ are vertically composable bigons. Then, 
\begin{equation*}
\varphi_{\Sigma_2} \bullet \varphi_{\Sigma_1} = \varphi_{\Sigma_2 \bullet \Sigma_1}
\quand
\varphi_{\id_{\gamma}}=\id_{F_{\gamma}}\text{.}
\end{equation*}
\end{proposition}

\begin{proof}
This follows immediately from the definition of $\varphi_{\Sigma}$. 
\end{proof}

\begin{proposition}
\label{lem:F2}
Let $\inf P$ be a principal $\Gamma$-2-bundle with fake-flat connection.
Suppose $\Sigma_1: \gamma_1\Rightarrow \gamma_1'$ and $\Sigma_2: \gamma_2\Rightarrow \gamma_2'$ are horizontally composable bigons. Then, the following diagram is commutative:
\begin{equation*}
\alxydim{@C=5em@R=\xyst}{F_{\gamma_2} \circ F_{\gamma_1}\ar@{=>}[r]^{\varphi_{\Sigma2}\circ \varphi_{\Sigma_1}} \ar@{=>}[d]_{c_{\gamma_1,\gamma_2}} & F_{\gamma_2'}\circ F_{\gamma_1'} \ar@{=>}[d]^{c_{\gamma_1',\gamma_2'}} \\ F_{\gamma_2\ast \gamma_1}  \ar@{=>}[r]_{\varphi_{\Sigma_2 \ast \Sigma_1}}   & F_{\gamma_2'\ast \gamma_1'}  \\ }
\end{equation*}
\end{proposition}

\begin{proof}
Given $(\xi_1,\xi_2)\in F_{\gamma_1} \ttimes{\alpha_r}{\alpha_l} F_{\gamma_2}$ we choose horizontal lifts $(\Phi^1_i,\rho^1_i,g^1_i)$ and $(\Phi_i^2,\rho_i^2,g_i^2)$ of $\Sigma^1$ and $\Sigma^2$ with sources $\xi_1$ and $\xi_2$, respectively. We have
$t(\rho_{0}^2) = \Phi^2_{1}(s,0)$ and $s(\rho_{n}^1)=R(\Phi_n^1(s,1),g_{n}^1(s))$, as well as $s(\rho_0^2)=\alpha_l(\xi_2)=\alpha_r(\xi_1)=t(\rho_n^1)$, by \cref{def:horliftbigon}. It is now obvious that under the usual reparameterizations of paths, the collection consisting of the families $(\Phi_1^1,...,\Phi_n^1,\Phi_1^2,...,\Phi_n^2)$, $(\rho_0^1,...,\rho^1_{n-1},\rho_0^2\circ \rho_n^1,\rho_1^2,...,\rho_n^2)$ and $(g_1^1,...,g_n^1,g_1^2,...,g_n^2)$ is a horizontal lift of $c_{\gamma_1,\gamma_2}(\xi_1,\xi_2)$.
Computing the separate targets, we get from \cref{eq:def:targethorlift} $\rho_0^{k\prime} := \rho_0^k$ and $\rho^{k\prime}_i := R(\rho_i^k(1),((h_i^{k})^{-1},g_i^k(1)^{-1}))$, giving us $\xi_k'=\rho^{k\prime}_n\ast \gamma_i^k\ast...\ast \rho_0^{k\prime}$. For the target of the combined lift, we obtain in the middle the morphism
\begin{equation*}
R(\rho_0^2\circ \rho_n^1,((h_n^{1})^{-1},g_n^1(1)^{-1}))=\rho_0^{2\prime}\circ \rho_n^{1\prime}\text{.}
\end{equation*}
This shows that $\varphi_{\Sigma_2\ast \Sigma_1}(c_{\gamma_1,\gamma_2}(\xi_1,\xi_2))=c_{\gamma_1',\gamma_2'}(\xi_1',\xi_2')$.
\end{proof}

\subsection{Naturality with respect to bundle morphisms}

\label{sec:natbundlemorphbigon}

In this section we compare the parallel transports along a bigon in two isomorphic principal $\Gamma$-2-bundles. In the 2-functor formalism of \cref{sec:orga}, this is one axiom for a pseudonatural transformation associated to $J$.

\begin{proposition}
\label{lem:natbundlemorph}
Suppose $J:\inf P \to \inf P'$ is a 1-morphism in $\zweibunconff\Gamma M$. Let $\Sigma: \gamma_1 \Rightarrow \gamma_2$ be a bigon between paths $\gamma_1,\gamma_2$ with $x:=\gamma_1(0)=\gamma_2(0)$ and $y=\gamma_1(1):=\gamma_2(1)$. Let $J_{\gamma_1}$ and $J_{\gamma_2}$ be the transformations associated to $J$ defined  in \cref{sec:natbundlemorph}.   Then, the following diagram is commutative:
\begin{equation*}
\alxydim{@=\xyst}{J_y \circ F_{\gamma_1} \ar@{=>}[d]_{J_{\gamma_1}} \ar@{=>}[r]^{\id \circ \varphi_{\Sigma}} & J_y \circ F_{\gamma_2} \ar@{=>}[d]^{J_{\gamma_2}} \\ F'_{\gamma_1} \circ J_x \ar@{=>}[r]_{\varphi'_{\Sigma} \circ \id} & F'_{\gamma_2} \circ J_x}
\end{equation*}
\end{proposition}

\begin{comment}
The statement of the lemma can be visualized as a\quot{tin can} identity
\begin{equation*}
\alxydim{@=1.7cm}{\inf P_x \ar@/_1.2pc/[r]_{F_{\gamma_2}}="2" \ar@/^1.2pc/[r]^{F_{\gamma_1}}="1" \ar@{=>}"1";"2"|*+{\varphi_{\Sigma}} \ar[d]_{J_x} & \inf P_y \ar@{=>}[dl]|>>>>>>>>*{J_{\gamma_2}} \ar[d]^{J_y} \\ \inf P_x' \ar@/_1.2pc/[r]_{F_{\gamma_2}'} & \inf P_y'}
=
\alxydim{@=1.7cm}{\inf P_x  \ar@/^1.2pc/[r]^{F_{\gamma_1}}  \ar[d]_{J_x} & \inf P_y \ar@{=>}[dl]|<<<<<<<<*{J_{\gamma_1}} \ar[d]^{J_y} \\ \inf P_x' \ar@/_1.2pc/[r]_{F'_{\gamma_2}}="2" \ar@/^1.2pc/[r]^{F'_{\gamma_1}}="1" \ar@{=>}"1";"2"|*+{\varphi'_{\Sigma}} & \inf P_y'}
\end{equation*}
\end{comment}

\begin{proof}
We start with $(\xi_1,j)\in F_{\gamma_1} \ttimes{\alpha_r}{\alpha_l} J_y$. Let  $( \{\Phi_i\}, \{\rho_i\},\{g_i\})$ be a horizontal lift of $\Sigma$ to $\inf P$ with source $\xi_1$, and let $\xi_2$ be its target.
\begin{comment}
Thus, $\xi_2 = \zeta_n \ast \beta_n \ast ... \ast \beta_1\ast \zeta_0$ with $\beta_i(t) = \Phi_i(1,t)$ as well as
\begin{equation*}
\zeta_0 := \rho_0(1)
\quand
\zeta_i := R(\rho_i(1),(h_i^{-1},g_i(1)^{-1}))\text{,}
\end{equation*}
where $h_i$ is the surface-ordered exponential of a bigon representative of $\Phi_i$. 
\end{comment}
Let $\tilde\xi_1=(\{\tilde \rho_i\},\{\tilde\gamma_i\})$ be a horizontal lift of $\xi_1$ to $J$, and let $(j',\xi_1')$ be its $j$-target. 
We can assume that the induced representatives for $\xi_1$ coincide, i.e. $\alpha_l(\tilde\gamma_i)=\Phi_i(0,-)$ and $\rho_i(0)=\tilde\rho_i$, and  $j':=\tilde \rho_0^{-1}\circ \tilde\gamma_1(0)  $. 
Finally, let  $( \{\Phi_i'\}, \{\rho_i'\},\{g_i'\})$ be a horizontal lift of $\Sigma$ to $\inf P'$ with source $\xi_1'$, and let $\xi_2'$ be its target.
We can assume that the induced representatives for $\xi_1'$ coincide, i.e. the path pieces of $\xi_1'$ are $\alpha_r(\tilde\gamma_i)=\Phi_i'(0,-)$, the the jumps $\rho_i'(0)$ satisfy $\rho_0'(0)=\id$ and  $\tilde \rho_i \circ \tilde\gamma_i(t_i) = \tilde\gamma_{i+1}(t_i)\circ \rho_i'(0)$ for $1\leq i < n$
 and $\tilde\rho_n \circ \tilde\gamma_n(1) = j\circ \rho_n'(0)$.
\begin{comment}
We also have $\xi_2' = \zeta_n' \ast \beta_n' \ast ... \ast \beta_1'\ast \zeta_0'$ with $\beta_i'(t) = \Phi_i'(1,t)$ as well as
\begin{equation*}
\zeta_0' := \rho_0'(1)
\quand
\zeta_i' := R(\rho_i'(1),(h_i^{\prime-1},g'_i(1)^{-1}))\text{,}
\end{equation*}
where $h_i'$ is the surface-ordered exponential of a bigon representative of $\Phi_i'$. 
\end{comment}
We have to prove that
$J_{\gamma_2}(\xi_2,j)=(j',\xi_2')$. For this purpose we provide a horizontal lift $\tilde\xi_2$ of $\xi_2$ to $J$ with $j$-target $(j',\xi_2')$.

Due to \cref{lem:F1} it suffices to discuss small bigons. We can even assume by \reftransitionspananafunctor\ that the image of $\Phi_i\times\Phi_i'$ is contained in an open subset $V_i \subset \ob{\inf P} \times_M \ob{\inf P'}$ that supports a transition spans $\tau_i:V_i \to J$, with  transition functions $p_i$. 
\begin{comment}
That is, 
\begin{equation*}
\alpha_l(\tau_i(x_1,x_2))=x_1
\quand
\alpha_r(\tau_i(x_1,x_2))=R(x_2,p_i(x_1,x_2))\text{.}
\end{equation*} 
\end{comment}
We define $\Psi_i: [0,1]\times [t_{i-1},t_i] \to J$ by $\Psi_i(s,t) := \tau_i(\Phi_i(s,t),\Phi_i'(s,t))$ and similarly $G_i(s,t):= p_i(\Phi_i(s,t),\Phi_i'(s,t))$; these satisfy $\alpha_l(\Psi_i(s,t))=\Phi_i(s,t)$ and $\alpha_r(\Psi_i(s,t))=R(\Phi_i'(s,t),G_i(s,t))$.  After performing several adjustments analogously to the ones of \cref{lem:existencehorizontallifts} we can assume that $\Psi_i(0,t)=\tilde\gamma_i(t)$; in particular, $t \mapsto \Psi_i(0,t)$ is horizontal, and we can assume that  $s\mapsto  \Psi_i(s,0)$ and $t \mapsto \Psi_i(1,t)$  are horizontal. 
Since the left anchors of these three paths are horizontal, their right anchors are also horizontal by \cref{lem:horF:a}. But since the corresponding three paths in $\Phi_i'(s,t)$ are horizontal, too, it follows from the uniqueness of \cref{lem:obhorexists} that $G_i(0,t)=G_i(1,t)=G_i(s,t_{i-1})=1$.

We use this in the following way. We write $\xi_2 = \zeta_n \ast \beta_n \ast ... \ast \beta_1\ast \zeta_0$, and obtain from the definition of $\xi_2$ as the target of the chosen horizontal lift $\beta_i(t) = \Phi_i(1,t)$, $\zeta_0 := \rho_0(1)$ and $\zeta_i :=  R(\rho_i(1),(h_i^{-1},g_i(1)^{-1}))$, where $h_i$ is the surface-ordered exponential of a bigon-parameterization of $\Phi_i$. We define  $\tilde\gamma_i'(t) := \Psi_i(1,t)$.  Then, $\tilde\xi_2 := (\{\zeta_i\},\{\tilde\gamma_i'\})$ is a horizontal lift of $\xi_2$ to $J$. 
\begin{comment}
Indeed, $\tilde\gamma_i'$ are  by construction horizontal, and since $\alpha_l(\tilde\gamma_i(t))=\Phi_i(1,t)=\beta_i(t)$, the jumps $\zeta_i$ fit. 
\end{comment}
It remains to prove that its  $j$-target is $(j',\xi_2')$. There are three parts: the path pieces of $\xi_2'$, the morphisms pieces, and the element $j'$. 
First, the paths are $\alpha_r(\tilde\gamma_i')=\alpha_r(\Psi_i(1,-))=R(\Phi_i'(1,t),G_i(1,t))=\Phi_i'(1,t)$; these are indeed the paths of $\xi_2'$. 

Second, the morphisms $\zeta_i'$ are characterized by $\zeta_0'=\id$ and $\zeta_i \circ \tilde\gamma_i'(t_i) = \tilde\gamma_{i+1}'(t_i)\circ \zeta_i'$ (for $1\leq i < n$) and $\zeta_n \circ \tilde\gamma_n'(1) = j\circ \zeta_n'$. We have to show that they coincide with the result of computing the target of the horizontal lift of $\Sigma$ to $\inf P'$, namely
\begin{equation*}
\zeta_0' := \rho_0'(1)
\quand
\zeta_i' := R(\rho_i'(1),(h_i^{\prime-1},g'_i(1)^{-1}))\text{,}
\end{equation*}
where $h_i'$ is the surface-ordered exponential of a bigon-parameterization of $\Phi_i'$. 
We have $\zeta_0'=\rho_0'(1)=\rho_0'(0)=\id$  at the beginning. For the pieces in the middle, let $\tau:[0,1] \to \mor{\inf P'}$ be the unique path such that
\begin{equation}
\label{eq:idtau}
R(\rho_i(s),g_i(s)^{-1}) \circ \Psi_i(s,t_i) \circ \tau(s) = (\Psi_{i+1}(s,t_i) \circ \rho_i'(s) )\cdot \id_{g_i(s)^{-1}}\text{.}
\end{equation}
\begin{comment}
This makes sense:
\begin{align*}
\alpha_l(R(\rho_i(s),g_i(s)^{-1}) \circ \Psi_i(s,t_i) )
&= t(R(\rho_i(s),g_i(s)^{-1})
\\&= R(t(\rho_i(s)),g_i(s)^{-1})
\\&=R(\Phi_{i+1}(s,t_i),g_i(s)^{-1})
\\&= R(\alpha_l(\Psi_{i+1}(s,t_i)), g_i(s)^{-1})
\\&= \alpha_l((\Psi_{i+1}(s,t_i) \circ \rho_i'(s) )\cdot \id_{g_i(s)^{-1}})
\end{align*}
\end{comment}
It is straightforward to check using \refactionfullyfaithful\ that \begin{equation*}
\tau(s)=R(\id_{\Phi'_i(s,t_i)},(\eta(s),g_i'(s)g_i(s)^{-1} ))
\end{equation*}
for a unique smooth map $\eta:[0,1] \to H$ with $t(\eta(s)) =G_i(s,t_i)g_i(s)g_i'(s)^{-1}$.
\begin{comment}
We have
\begin{align*}
s(\tau(s)) &= \alpha_r((\Psi_{i+1}(s,t_i) \circ \rho_i'(s) )\cdot \id_{g_i(s)^{-1}})
\\&=R(s(\rho_i'(s)),g_i(s)^{-1})
\\&=R(\Phi_i'(s,t_i),g_i'(s)g_i(s)^{-1}) \end{align*}
and
\begin{equation*}
t(\tau(s))= \alpha_r(\Psi_i(s,t_i)) =R(\Phi'_{i}(s,t_i),G_i(s,t_i))
\end{equation*}
We set $\rho_2 := \tau(s)$ and $\rho_1 := \id_{\Phi'_i(s,t_i)}$, as well as $g_1 :=g_i'(s)g_i(s)^{-1}$ and $g_2 :=G_i(s,t_i)$. These satisfy
\begin{align*}
s(\rho_2) &=R(\Phi_i'(s,t_i),g_i'(s)g_i(s)^{-1})=R(s(\rho_1),g_1)
\\
t(\rho_2) &=R(\Phi'_{i}(s,t_i),G_i(s,t_i))=R(t(\rho_1),g_2)
\end{align*}
Thus, by \refactionfullyfaithful\ there exists a unique $\eta:[0,1] \to H$ such that $R(\rho_1,(h,g_1))=\rho_2$ and $t(h)g_1=g_2$, i.e.
\begin{equation*}
R(\id_{\Phi'_i(s,t_i)},(\eta(s),g_i'(s)g_i(s)^{-1} ))=  \tau(s)
\quand
t(\eta(s))g_i'(s)g_i(s)^{-1} =G_i(s,t_i)\text{.}
\end{equation*}
\end{comment}
At $s=0$, \cref{eq:idtau} results by construction in $\tau(0)=\id$, i.e. $\eta(0)=1$.
\begin{comment}
At $s=0$ we have
\begin{equation*}
\rho_i(0) \circ \tilde\gamma_i(t_i) \circ \tau(s) = \tilde\gamma_{i+1}(t_i) \circ \rho_i'(0)\text{;}
\end{equation*}
this was assumed at the beginning. 
\end{comment}
We claim that 
\begin{equation}
\label{eq:claimtau}
\eta(1)=h_i^{-1}h_i'\text{.}
\end{equation}
Using this claim we get
\begin{align*}
\zeta_i \circ \tilde\gamma_i'(t_i) &=R(\rho_i(1),(h_i^{-1},g_i(1)^{-1})) \circ \Psi_i(1,t_i)
\\&\eqcref{eq:claimtau}(R(\rho_i(1),g_i(1)^{-1}) \circ \Psi_i(1,t_i)\circ \tau(1)) \cdot (\alpha(g_i(1),h_i'^{-1}),g_i(1)g_i'(1)^{-1})
\\&\eqcref{eq:idtau} (\Psi_{i+1}(1,t_i) \circ \rho_i'(1) )\cdot (1,g_i(1)^{-1}) \cdot (\alpha(g_i(1),h_i'^{-1}),g_i(1)g_i'(1)^{-1})
\\&=\Psi_{i+1}(1,t_i)\circ R(\rho_i'(1),(h_i^{\prime-1},g'_i(1)^{-1}))
\\&= \tilde\gamma_{i+1}'(t_i)\circ \zeta_i'\text{;}
\end{align*}
\begin{comment}
More explicitly, this calculation is:
\begin{align*}
\zeta_i \circ \tilde\gamma_i'(t_i) &=R(\rho_i(1),(h_i^{-1},g_i(1)^{-1})) \circ \Psi_i(1,t_i)
\\&= R(\rho_i(1),(1,g_i(1)^{-1})(h_i^{-1},1)) \circ \Psi_i(1,t_i)
\\&= (R(\rho_i(1),g_i(1)^{-1}) \circ \Psi_i(1,t_i)) \cdot (h_i^{-1},1)
\\&= (R(\rho_i(1),g_i(1)^{-1}) \circ \Psi_i(1,t_i)\circ \id) \cdot(h_i^{-1}h_i',g_i'(1)g_i(1)^{-1})\cdot (\alpha(g_i(1),h_i'^{-1}),g_i(1)g_i'(1)^{-1})
\\&\eqcref{eq:claimtau} (R(\rho_i(1),g_i(1)^{-1}) \circ \Psi_i(1,t_i)\circ \tau(1)) \cdot (\alpha(g_i(1),h_i'^{-1}),g_i(1)g_i'(1)^{-1})
\\&\eqcref{eq:idtau} (\Psi_{i+1}(1,t_i) \circ \rho_i'(1) )\cdot (1,g_i(1)^{-1}) \cdot (\alpha(g_i(1),h_i'^{-1}),g_i(1)g_i'(1)^{-1})
\\&=(\Psi_{i+1}(1,t_i)\circ \rho_i'(1))\cdot (h_i^{\prime-1},g'_i(1)^{-1})
\\&=\Psi_{i+1}(1,t_i)\circ R(\rho_i'(1),(h_i^{\prime-1},g'_i(1)^{-1}))
\\&= \tilde\gamma_{i+1}'(t_i)\circ \zeta_i'\text{;}
\end{align*}
\end{comment}
this proves the desired property of the $\zeta_i'$. In order to prove \cref{eq:claimtau} we write \cref{eq:idtau} in the equivalent form
\begin{equation*}
\lambda_{lhs}:=\Psi_i(s,t_i)\cdot G_i(s,t_i)^{-1}  g_i'(s)=(\rho_i(s)^{-1} \circ \Psi_{i+1}(s,t_i) \circ \rho_i'(s))\cdot (\alpha(g_i'(s)^{-1},\eta(s)^{-1}),1)=:\lambda_{rhs}
\end{equation*}
\begin{comment}
Indeed,  first we have
\begin{align*}
\Psi_i(s,t_i)\cdot G_i(s,t_i)^{-1}\cdot (\eta(s),g_i'(s)) 
&= \Psi_i(s,t_i) \cdot G_i(s,t_i)^{-1}\cdot (\eta(s),g_i'(s)g_i(s)^{-1} )\cdot \id_{g_i(s)} 
\\&= (\Psi_i(s,t_i) \circ R(\id_{\Phi'_i(s,t_i)},(1,G_i(s,t_i))\cdot (1,G_i(s,t_i)^{-1})\cdot (\eta(s),g_i'(s)g_i(s)^{-1} )))\cdot \id_{g_i(s)} 
\\&= (\Psi_i(s,t_i) \circ R(\id_{\Phi'_i(s,t_i)},(\eta(s),g_i'(s)g_i(s)^{-1} )))\cdot \id_{g_i(s)} 
\\&= \rho_i(s)^{-1} \circ \rho_i(s) \circ ((\Psi_i(s,t_i) \circ \tau(s))\cdot \id_{g_i(s)})
\\&=\rho_i(s)^{-1} \circ (R(\rho_i(s),g_i(s)^{-1}) \circ \Psi_i(s,t_i) \circ \tau(s))\cdot \id_{g_i(s)})
\\&\eqcref{eq:idtau} \rho_i(s)^{-1} \circ (((\Psi_{i+1}(s,t_i) \circ \rho_i'(s) )\cdot \id_{g_i(s)^{-1}})\cdot \id_{g_i(s)})
\\&= \rho_i(s)^{-1} \circ \Psi_{i+1}(s,t_i) \circ \rho_i'(s) 
\end{align*}
We have
\begin{equation*}
(\eta(s),g_i'(s))^{-1}=(\alpha(g_i'(s)^{-1},\eta(s)^{-1}),g_i'(s)^{-1})=(\alpha(g_i'(s)^{-1},\eta(s)^{-1}),1)\cdot \id_{g_i'(s)^{-1}}\text{;}
\end{equation*}
this shows the claim.
\end{comment}
This is an equality between two paths $\lambda_{lhs}$ and $\lambda_{rhs}$ in $J$. We compute the  path-ordered exponential $\PE_{\nu_0}$ separately on both sides. On the right hand side, the path $s\mapsto \rho_i(s)^{-1} \circ \Psi_{i+1}(s,t_i) \circ \rho_i'(s)$ is horizontal by \cref{lem:horF:c,lem:horF:d*}. Its right anchor is $s(\rho_i')$; since $\rho_i'$ and $t(\rho_i')$ are horizontal by \cref{def:horliftbigon}, this is horizontal by \cref{lem:hormor:e}. Hence, by \cref{lem:pathJ:a} $\PE_{\nu_0}(\lambda_{rhs})=\alpha(g_i'(1)^{-1},\eta(1))$.

On the left hand side, the right anchor of $\lambda_{lhs}$ is again $s(\rho_i')$ and thus horizontal.
Hence, by  \cref{lem:calchgvarphiJ} we have $\PE_{\nu_0}(\lambda_{lhs})=\alpha(g_i'(1),h_{\nu}(\lambda_{lhs}\cdot g_i'^{-1}))$. 
Let $\Sigma_i:\lambda_i \Rightarrow \lambda_i'$ be a bigon-parameterization for $\Psi_i$, where $\lambda_i=\Psi_i(-,t_i) \circ \Psi_i(0,-)$ and $\lambda_i'=\Psi_i(1,-)\circ\Psi_i(-,0)$ up to thin homotopy, and let $\Theta_i:\gamma_i \Rightarrow \gamma_i'$ be a  bigon-parameterization of $G_i$ with analogous  $\gamma_i$ and $\gamma_i'$. Then, $\alpha_l(\Sigma_i)$ is a bigon-parameterization for $\Phi_i$ and $R(\alpha_r(\Sigma_i),\Theta_i^{-1})$ is one for $\Phi_i'$.
Now \cref{lem:bigonF} implies
\begin{equation*}
h_i' \cdot h_{\nu}(\lambda_i\cdot\gamma_i^{-1})^{-1}= h_{\nu}(\lambda_i'\cdot\gamma_i'^{-1})^{-1} \cdot h_i\text{.}
\end{equation*}
Note that $\lambda_i'$ is horizontal with horizontal right anchor. By  \cref{lem:calchgvarphiJ} we get $h_{\nu}(\lambda'_i\cdot\gamma_i^{\prime-1})=\alpha(\gamma'_i(1),\PE_{\nu_0}(\lambda'_i))=1$. The same applies to the first half of the path $\lambda_i$; hence by \cref{lem:SE:gauge:z} we have
\begin{equation*}
h_{\nu}(\lambda_i \cdot \gamma_i^{-1}) = h_{\nu}(\Psi_i(-,t_i)G_i(-,t_i))=h_{\nu}(\lambda_{lhs}\cdot g_i'^{-1})\text{.}
\end{equation*} 
Summarizing collected identities, we obtain $\PE_{\nu_0}(\lambda_{lhs})=\alpha(g_i'(1),h_i^{-1}h_i')$.
\begin{comment}
Indeed,
\begin{align*}
\PE_{\nu_0}(\lambda_{lhs})&=\alpha(g_i'(1),h_{\nu}(\lambda_{lhs} \cdot g_i'^{-1}))
\\&=\alpha(g_i'(1),h_{\nu}(\Psi_i(-,t_i)G_i(-,t_i)))
\\&=\alpha(g_i'(1),h_{\nu}(\lambda_i \cdot \gamma_i^{-1}))
\\&=\alpha(g_i'(1),h_i^{-1}h_i')
\end{align*}
\end{comment}
Equating with the result of the right hand side yields the claim \cref{eq:claimtau}.

Third, for $i=n$, we obtain from \cref{lem:bigonF}, similarly as above, $h_n'=h_n$, and we have $g_n(1)=g_n'(1)=1$. Using this it is straightforward to show that $\zeta_n \circ \tilde\gamma_n'(1) = j\circ R(\rho_n'(1),(h_n^{\prime-1},1))$; this is the correct characterization for $\zeta_n'$.
\begin{comment}
Indeed,
\begin{align*}
\zeta_n \circ \tilde\gamma_n'(1) &=R(\rho_n(1),(h_n^{-1},1))\circ \Psi_n(1,1)
\\&= (\rho_n(0) \circ \Psi_n(0,1)) \cdot (h_n^{\prime-1},1)
\\&= (\tilde\rho_n \circ \tilde\gamma_n(1)) \cdot (h_n^{\prime-1},1)
\\&= (j\circ \rho_n'(0)) \cdot (h_n^{\prime-1},1)
\\&= (j\circ \rho_n'(1)) \cdot (h_n^{\prime-1},1)
\\&= j\circ R(\rho_n'(1),(h_n^{\prime-1},1))
\end{align*}
\end{comment}
Third, we show that the element $j'$ is reproduced:
\begin{equation*}
\zeta_0^{-1}\circ \tilde\gamma'_1(0)=\rho_0(1)^{-1}\circ \Psi_1(1,0)=\rho_0(0)^{-1}\circ \Psi_1(0,0)=\tilde \rho_0^{-1}\circ \tilde\gamma_1(0)=j'\text{.}
\end{equation*}
This completes the proof. 
\end{proof}

\subsection{Naturality with respect to pullback}

\label{sec:natpullbackbigon}

Suppose $\inf P$ is a principal $\Gamma$-2-bundle over $N$ with fake-flat connection $\Omega$, and  $f:M \to N$ is a smooth map. We denote by $\inf P' := f^{*} \inf P$ the pullback bundle,  obtain a $\Gamma$-equivariant smooth functor
$\tilde f: \inf P' \to \inf P$, and $\Omega' := \tilde f^{*}\Omega$ is a connection on $\inf P'$. We recall from \cref{sec:natpullbackpaths} that we have associated to each path $\gamma:x \to y$ a $\Gamma$-equivariant transformation $\tilde f_{\gamma}: \tilde f_y \circ F'_{\gamma} \Rightarrow  F_{f(\gamma)}\circ \tilde f_{x}$.

\begin{proposition}
\label{lem:natpullback}
Suppose $\inf P$ is a principal $\Gamma$-2-bundle over $N$ with fake-flat connection $\Omega$, and  $f:M \to N$ is a smooth map. Let $\Sigma: \gamma_1 \Rightarrow \gamma_2$ be a bigon in $M$ with $x:=\gamma_1(0)=\gamma_2(0)$ and $y:=\gamma_1(1)=\gamma_2(1)$. Let $F_{\gamma}'$ and $\varphi_{\Sigma}'$ denote the parallel transport in the pullback bundle $f^{*}\inf P$. Then, the following diagram is commutative:
\begin{equation*}
\alxydim{@C=4em@R=\xyst}{\tilde f_{y} \circ F_{\gamma_1}' \ar@{=>}[d]_{\tilde f_{\gamma_1}} \ar@{=>}[r]^{\id \circ \varphi'_{\Sigma}} & \tilde f_y \circ F'_{\gamma_2} \ar@{=>}[d]^{\tilde f_{\gamma_2}} \\ F_{f(\gamma_1)} \circ \tilde f_x \ar@{=>}[r]_{\varphi_{f(\Sigma)} \circ \id} & F_{f(\gamma_2)} \circ \tilde f_x}
\end{equation*}
\end{proposition}

\begin{comment}
The statement of the lemma can be visualized as a\quot{tin can} identity
\begin{equation*}
\alxydim{@=1.7cm}{\inf P_x' \ar@/_1.2pc/[r]_{F'_{\gamma_2}}="2" \ar@/^1.2pc/[r]^{F'_{\gamma_1}}="1" \ar@{=>}"1";"2"|*+{\varphi'_{\Sigma}} \ar[d]_{\tilde f_x} & \inf P_y' \ar@{=>}[dl]|>>>>>>>>*{\tilde f_{\gamma_2}} \ar[d]^{\tilde f_y} \\ \inf P_{f(x)} \ar@/_1.2pc/[r]_{F_{f \circ \gamma_2}} & \inf P_{f(y)}}
=
\alxydim{@=1.7cm}{\inf P_x'  \ar@/^1.2pc/[r]^{F_{\gamma_1}'}  \ar[d]_{\tilde f_x} & \inf P_y' \ar@{=>}[dl]|<<<<<<<<*{\tilde f_{\gamma_1}} \ar[d]^{\tilde f_y} \\ \inf P_{f(x)} \ar@/_1.2pc/[r]_{F_{f \circ \gamma_2}}="2" \ar@/^1.2pc/[r]^{F_{f \circ \gamma_1}}="1" \ar@{=>}"1";"2"|*+{\varphi_{f \circ \Sigma}} & \inf P_{f(y)}}
\end{equation*}
\end{comment}

\begin{proof}
Like in \cref{sec:natpullbackpaths} we identify canonically $\inf P'_x = \inf P_{f(x)}$ so that $\tilde f_x = \id$, and $\tilde f_{\gamma}: F_{\gamma}'\Rightarrow F_{f(\gamma)}$  is given by $\rho_n \ast \gamma_n \ast ... \ast \rho_0 \mapsto \tilde f(\rho_n) \ast \tilde f(\gamma_n) \ast ... \ast \tilde f(\rho_0)$. Suppose we have a horizontal lift $( \{\Phi_i\}, \{\rho_i\},\{g_i\})$  of $\Sigma$ to $\inf P'$ with source $\xi$. Let $\xi'$ be its target, so that $\xi' = \varphi'_{\Sigma}(\xi)$. Since $\Omega' = \tilde f^{*}\Omega$ and $\tilde f$ is $\Gamma$-equivariant it is clear that $( \{ \tilde f\circ \Phi_i\}, \{ \tilde f \circ \rho_i\},\{g_i\})$ is a horizontal lift of $f(\Sigma)$ to $\inf P$ with source $\tilde f_{\gamma}(\xi)$. Using the naturality of the surface ordered exponential under pullbacks (\cref{lem:SE:e}), its target is $\tilde f_{\gamma}(\xi')$; this shows that commutativity of the diagram. \end{proof}

\section{Backwards compatibility}

\label{sec:backcomp}

We exhibit our constructions of \cref{sec:ptpaths,sec:ptbigons} for two particular classes of principal $\Gamma$-2-bundles: trivial 2-bundles and 2-bundles induced from ordinary principal bundles. 

\subsection{Trivial principal 2-bundles}

\label{ex:ptpath:trivbun}

It is certainly important to see  what the parallel transport constructions of \cref{sec:ptpaths,sec:ptbigons} reduce to in case of the trivial bundle. Also, we will need the results of this section in the proofs in \cref{sec:orga,sec:main}. 

In the following remark we relate (in  a  functorial way) $\Gamma$-connections on $M$ to connections on the trivial principal $\Gamma$-2-bundle. The next remark identifies what that relation is over a one-point-manifold.

\begin{remark}
\label{rem:trivbun}
Let $\inf I := \idmorph{M} \times \Gamma$ be the trivial bundle. 
We summarize three constructions of \reftrivialbundle; also see \cref{def:gammacon} for the categorical structure of $\Gamma$-connections:
\begin{enumerate}[(a)]

\item 
\label{rem:trivbun:a}
Every (fake-flat) $\Gamma$-connection $(A,B)$ on $M$ defines a (fake-flat) connection $\Omega_{A,B}$ on $\inf I$; we denote by $\inf I_{A,B}$ the trivial bundle equipped with that connection. In more detail, we have
\begin{equation*}
\fa\Omega_{A,B}= \mathrm{Ad}_{g}^{-1}(p^{*}A) + g^{*}\theta
\!\!\!\quomma\!\!\!
\fb\Omega_{A,B} =(\alpha_{g^{-1}})_{*} ((\tilde\alpha_{h})_{*}(p^{*}A)+h^{*}\theta)
\!\!\!\quand\!\!\!
\fc\Omega_{A,B} =-(\alpha_{g^{-1}})_{*}(p^{*}B)\text{,}
\end{equation*}  
where $g$, $h$ and $p$ denote the projections to $G$, $H$, and $M$, respectively.

\item
\label{rem:trivbun:b}
Every gauge transformation $(g,\varphi)$ between $\Gamma$-connections $(A,B)$ and $(A',B')$ on $M$ defines a 1-morphism $J_{g,\varphi}:\inf I_{A,B} \to \inf I_{A',B'}$ in $\zweibuncon\Gamma M$. If the $\Gamma$-connections are fake-flat, this is a 1-morphism in $\zweibunconff\Gamma M$.

\item
\label{rem:trivbun:c}
Every gauge 2-transformation $a$ between gauge transformations $(g_1,\varphi_1)$  and $(g_2,\varphi_2)$ defines a 2-morphism $f_a:J_{g_1,\varphi_1}\Rightarrow J_{g_2,\varphi_2}$.

\item
\label{rem:trivbun:d}
By \reffunctorL,\  \cref{rem:trivbun:a*,rem:trivbun:b*,rem:trivbun:c*} yield a 2-functor
$L^{\ff}_M:\conff\Gamma M \to \zweibunconff\Gamma M$.

\end{enumerate}
\end{remark}

\begin{remark}
\label{sec:gammator}
We reduce the structure of \cref{rem:trivbun} to the one-point manifold $M=\ast$. It is easy to see that $\conff\Gamma \ast=\con\Gamma \ast = B\Gamma$, the delooping of $\Gamma$: this bigroupoid has a single object, whose Hom-groupoid is $\Gamma$. In order to identify $\zweibunconff\Gamma\ast$ we fix the following definition:  a \emph{$\Gamma$-torsor} is a Lie groupoid $\inf P$ together with a smooth right $\Gamma$-action $R$ of $\Gamma$ on $\inf P$ such that the functor
\begin{equation*}
\tau :=(\pr_1,R):\inf P \times \Gamma \to \inf P \times \inf P 
\end{equation*}
is a weak equivalence. The bicategory $\tor\Gamma$ is the full sub-bicategory of the bicategory of Lie groupoids with smooth $\Gamma$-action. Then we have $\zweibunconff\Gamma\ast=\zweibun\Gamma\ast=\tor\Gamma$.
\begin{enumerate}[(a)]

\item
\label{rem:i:obj}
The canonical $\Gamma$-torsor is $\inf P:=\Gamma$ with $R$ given by the 2-group structure. 
\begin{comment}
In this case the functor $\tau$ is invertible as a smooth functor, via  $\tau^{-1}=(\pr_1,m \circ (i \times \id))$.
\end{comment}

\item
\label{rem:i:morph}
For every $g\in G$ there is a 1-morphism $i_g:\Gamma\to\Gamma$ in $\tor\Gamma$, which can be given as a smooth functor: we set  $i_g(g'):=gg'$ and $i_g(h,g'):=(\alpha(g,h),gg')$. \begin{comment}
This is a functor: 
\begin{equation*}
s(i_g(h,g'))= s(\alpha(g,h),gg')=gg'=i_g(g')=i_g(s(h,g'))
\end{equation*}
and
\begin{equation*}
t(i_g(h,g'))= t(\alpha(g,h),gg')=t(\alpha(g,h))gg'=gt(h)g'=i_g(t(h)g')=i_g(t(h,g'))
\end{equation*}
The composition is respected:
if $g_2 = t(h_1)g_1$ then
\begin{multline*}
i_g((h_2,g_2)\circ (h_1,g_1)) =i_g(h_2h_1,g_1)=(\alpha(g,h_2h_1),gg_1)
\\
=(\alpha(g,h_2),gg_2)\circ (\alpha(g,h_1),gg_1)=i_g(h_2,g_2) \circ i_g(h_1,g_1)\text{.}
\end{multline*}
\end{comment}
This is a functor and strictly equivariant with respect  to  the $\Gamma$-action.
\begin{comment}
I.e. the diagram of functors
\begin{equation*}
\alxydim{}{\Gamma \times \Gamma \ar[d]_{i_g \times \id} \ar[r]^-{m} & \Gamma \ar[d]^{i_g} \\ \Gamma \times \Gamma \ar[r]_-{m} & \Gamma}
\end{equation*}  
is strictly commutative. 
\end{comment}
Hence it induces a 1-morphism in $\tor\Gamma$. 
\begin{comment}
Indeed, for objects we have
\begin{equation*}
i(g)(g'\cdot g'')=gg'g''=i(g)(g')\cdot g''\text{.}
\end{equation*}
For morphisms we have
\begin{multline*}
i(g)((h_1,g_1)\cdot(h_2,g_2)) =i(g)(h_1\alpha(g_1,h_2),g_1g_2)=(\alpha(g,h_1\alpha(g_1,h_2)),gg_1g_2)
\\
=(\alpha(g,h_1)\alpha(gg_1,h_2),gg_1g_2)=(\alpha(g,h_1),gg_1)\cdot (h_2,g_2)= i(g)(h_1,g_1)\cdot (h_2,g_2)\text{.}
\end{multline*}
\end{comment}

\item
\label{rem:i:twomorph}
For every $(h,g)\in \mor{\Gamma}$ we define a natural transformation $i_{(h,g)}: i_g \Rightarrow i_{t(h)g}$ whose component at $g'$ is $i_{(h,g)}(g'):=(h,gg')$. \begin{comment}
The naturality condition at a morphism $(h',g')$ in $\Gamma$ is equivalent to the commutativity of the diagram
\begin{equation*}
\alxydim{@C=2.5cm}{i_g(g') \ar[r]^-{i_g(h',g')}\ar[d]_{i_{(h,g)}(g')} & i_g(t(h')g') \ar[d]^{i_{(h,g)}(t(h')g')} \\ i_{t(h)g}(g') \ar[r]_-{i_{t(h)g}(h',g')} & i_{t(h)g}(t(h')g')\text{,} }
\end{equation*}
which is easily verified.
\end{comment}
\begin{comment}
We check this:
\begin{multline*}
i_{(h,g)}(t(h')g') \circ i_g(h',g')  = (h,gt(h')g')\circ (\alpha(g,h'),gg')=(h\alpha(g,h'),gg')\\=(\alpha(t(h)g,h'),t(h)gg') \circ (h,gg')=i_{t(h)g}(h',g') \circ i_{(h,g)}(g')\text{.}
 \end{multline*}
\end{comment}
The natural transformation $i_{(g,h)}$ is $\Gamma$-equivariant in the sense that
\begin{equation*}
i_{(h,g)}(g'g'')=i_{(h,g)}(g')\cdot \id_{g''}\text{.}
\end{equation*}
Hence we can regard it as a 2-morphism in $\tor\Gamma$.
\begin{comment}
Indeed,
\begin{equation*}
i_{(h,g)}(g'g'')=(h,gg'g'')=(h,gg'g'')=(h,gg')\cdot (1,g'')=i_{(h,g)}(g')\cdot \id_{g''}\text{.}
\end{equation*}
\end{comment} 

\item
\label{rem:Lresi}
It is straightforward to verify directly that \cref{rem:i:obj*,rem:i:morph*,rem:i:twomorph*} form a (strict) 2-functor $i: B\Gamma \to \tor\Gamma$.
Further it is easy to check that under the identification of $\inf I_x = \{x\} \times \Gamma \cong \Gamma$ the restriction of the 2-functor $L^{\ff}_M$ to $M=\ast$ is exactly $i$. Finally, one can show that $i$ is an equivalence of bicategories.

\end{enumerate}
\end{remark}

Now we start to identify the parallel transport along a path $\gamma:x \to y$ in the trivial principal $\Gamma$-2-bundle $\inf I_{A,B}$, where $(A,B)$ is a $\Gamma$-connection.
We show that the anafunctor $F_{\gamma}$ is canonically 2-isomorphic to (the anafunctor induced by) the functor $i_{\PE_{A}(\gamma)}$ of \cref{rem:i:morph}\text{,}
where $\PE_A(\gamma)\in G$ is the path-ordered exponential of $A$ along $\gamma$, see \cref{sec:pathordered}.
 For this purpose we define a $\Gamma$-equivariant transformation
\begin{equation}
\label{eq:transeta}
\eta_{\gamma}: i_{\PE_{A}(\gamma)} \Rightarrow
 F_{\gamma}\text{.}
 \end{equation}
 For simplicity we set  $g:= \PE_{A}(\gamma)$. We define $\eta_{\gamma}$ using \cref{rem:indtrans}; the underlying smooth map
$\tilde \eta_{\gamma}: G \to F_{\gamma}$ is defined as follows. Let $\kappa:[0,1]\to G$ be the solution of the initial value problem
\begin{equation}
\label{eq:initialkappa}
\dot\kappa(t)=-A(\dot\gamma(t))\kappa(t)
\quand
\kappa(0)=1\text{,}
\end{equation} 
so that $g=\kappa(1)$. Consider the path $(\gamma,\kappa)$ in $\ob{\inf P}=M \times G$. It is horizontal:
\begin{equation*}
\fa\Omega_{A,B}(\dot\gamma(t),\dot\kappa(t))=\mathrm{Ad}^{-1}_{\kappa(t)}(A(\dot\gamma(t)))+\theta(\dot\kappa(t))=0\text{.}
\end{equation*} 
For $g'\in G$ the path $(\gamma,\kappa g')$ is then horizontal, too, by \cref{lem:obhor}.
Thus, we obtain an element $\xi_{g'}:=\id_{(y,gg')} \ast (\gamma,\kappa g') \ast \id_{(x,g')} \in F_{\gamma}$. We set $\tilde \eta_{\gamma}(g'):=\xi_{g'}$.

\begin{lemma}
The map $\tilde\eta_{\gamma}$ satisfies \cref{T1*,T2*,T3*}.
\end{lemma}

\begin{proof}
 We have $\alpha_l(\xi_{g'})=(x,g')$ and $\alpha_r(\xi_{g'})=(y,gg')$, this is \cref{T1*}.  For  morphisms $\alpha\in \mor{\inf P_x}=\mor{\Gamma}$ (with $s(\alpha)=g'$) and $\beta\in \mor{\inf P_y}=\mor{\Gamma}$ (with $t(\beta)=gg'$) we have
\begin{equation*}
\alpha\circ \tilde \eta_{\gamma}(g')\circ \beta =\beta^{-1}\ast (\gamma,\kappa g')\ast\alpha^{-1}\sim(\beta^{-1}\circ i_g(\alpha)^{-1})\ast (\gamma,\kappa t(\alpha))\ast \id=\tilde \eta_{\gamma}(t(\alpha))\circ i_g(\alpha)\circ \beta\text{,}
\end{equation*}
where $\sim$ denotes one application of the equivalence relation on $F_{\gamma}$, performed as follows. 
Consider the path $\rho=(\gamma,\id_{\kappa}\cdot \alpha)$ in $\mor{\inf P}$  satisfying $s(\rho)=(\gamma,\kappa g')$ and $t(\rho)=(\gamma,\kappa t(\alpha))$. It is horizontal by \cref{lem:hormor:a,lem:hormor:c}.
\begin{comment}
Indeed, $\rho=\id_{\gamma,\kappa}\cdot \alpha$.
\end{comment}
Thus, $\id\ast (\gamma,\kappa g')\ast\alpha^{-1}$ is equivalent to 
\begin{align*}
\rho(1)^{-1}\ast (\gamma,\kappa t(\alpha))\ast (\rho(0) \circ \alpha^{-1}) 
&= (\id_{g}\cdot \alpha)^{-1}\ast (\gamma,\kappa t(\alpha))\ast ((\id_{1}\cdot \alpha) \circ \alpha^{-1})
\\&=  i_g(\alpha)^{-1}\ast (\gamma,\kappa t(\alpha))\ast \id\text{.}
\end{align*}
This shows \cref{T2*}.  Finally, we have  $\xi_{g'}\cdot \id_g=\xi_{g'g}$, this is \cref{T3*}.
\end{proof}

\begin{proposition}
\label{eq:T1:eta}
Let $(A,B)$ be a $\Gamma$-connection on $M$, and let $F_{\gamma}$ denote the parallel transport in the associated trivial principal $\Gamma$-2-bundle $\inf I_{A,B}$.  Then, the transformation $\eta_{\gamma}: i_{\PE_{A}(\gamma)} \Rightarrow
 F_{\gamma}$ is compatible with path composition: if $\gamma_2$ and $\gamma_1$ are composable paths, 
then we have
\begin{equation*}
\eta_{\gamma_2\ast \gamma_1}= c_{\gamma_1,\gamma_2} \bullet ( \eta_{\gamma_2}\circ \eta_{\gamma_1})\text{,}
\end{equation*}
where $c_{\gamma_1,\gamma_2}: F_{\gamma_2} \circ F_{\gamma_1} \Rightarrow F_{\gamma_2\ast \gamma_1}$ was defined in \cref{sec:compcomppaths}.
\end{proposition}

\begin{comment}
Now, this is an equality between transformations $i_{\PE_{A}(\gamma_2 \ast \gamma_1)} \Rightarrow F_{\gamma_2\circ\gamma_1}$. 
\end{comment}

\begin{proof}
On the level of the corresponding smooth maps $\tilde\eta_{\gamma}$, the claim becomes
\begin{equation*}
\tilde\eta_{\gamma_2\ast\gamma_1}(g')=c_{\gamma_1,\gamma_2}(\tilde\eta_{\gamma_1}(g'),\tilde\eta_{\gamma_2}(g_1g'))
\end{equation*}
for all $g'\in G$.
\begin{comment}
Note that $i_{g_2}\circ i_{g_1} = i_{g_2g_1}=i_{\PE_{A}(\gamma_2 \ast \gamma_1)}$ by \cref{lem:PE:b}, where $g_i := \PE_A(\gamma_i)$. On the level of associated anafunctors, this equality becomes an identification \begin{equation*}
((g_1',\alpha_1),(g_2',\alpha_2)) \mapsto (g_1',i_{g_2}(\alpha_1) \circ \alpha_2)\text{.}
\end{equation*}
The transformation $ \eta_{\gamma_2}\circ \eta_{\gamma_1}:i_{g_2}\circ i_{g_1} \Rightarrow F_{\gamma_2}\circ F_{\gamma_1}$ is given by 
\begin{equation*}
((g_1',\id_{g_1g_1'})),(g_2',\id_{g_2g_2'})) \mapsto (\eta_{\gamma_1}(g_1',\id_{g_1g_1'}),\eta_{\gamma_2}(g_2',\id_{g_2g_2'})) = (\tilde\eta_{\gamma_1}(g_1'),\tilde\eta_{\gamma_2}(g_2'))\text{.}
\end{equation*} 
Here, $g_2'=\alpha_l(g_2',\id_{g_2g_2'})=\alpha_r(g_1',\id_{g_1g_1'})=g_1g_1'$.
\end{comment}
Let $\kappa_1,\kappa_2:[0,1] \to G$ be the solutions to the initial value problems \cref{eq:initialkappa} corresponding to $\gamma_1$ and $\gamma_2$, respectively, so that $\tilde\eta_{\gamma_i}(g')=\id \ast (\gamma_i,\kappa_i g') \ast \id$. Then,  $\tilde \kappa := \kappa_2g_1 \ast \kappa_1$ (composition of paths in $G$) is the solution for $\gamma_2\ast\gamma_1$, i.e. $\tilde\eta_{\gamma_2\ast\gamma_1}(g')=\id \ast (\gamma_2\ast\gamma_1,\tilde \kappa g') \ast \id \in F_{\gamma_2\ast\gamma_1}$. In the direct limit definition of $F_{\gamma_2\ast\gamma_1}$, this is equivalent to $\id \ast (\gamma_2,\tilde \kappa_2g_1g')\ast \id \ast (\gamma_1,\kappa_1g') \ast \id$, which is precisely $c_{\gamma_1,\gamma_2}(\tilde\eta_{\gamma_1}(g'),\tilde\eta_{\gamma_2}(g_1g'))$. \end{proof}

\begin{remark}
\label{ex:ptpath:trivbarbun}
Let $\inf P$ be a principal $\Gamma$-bundle with a   connection $\Omega$, $(A,B)$ be a $\Gamma$-connection on $M$, and $J:\inf I_{A,B} \to \inf P$ be a 1-morphism in $\zweibuncon\Gamma{M}$. Such \quot{trivializations} always exist locally.  Combining the transformation $\eta_{\gamma}$ with the transformation $J_{\gamma}$ from \cref{sec:natbundlemorph} we obtain a transformation
\begin{equation*}
\alxydim{@=\xyst}{\Gamma \ar[r]^{i_{\PE_{A}(\gamma)}}\ar[d]_{J_x} & \Gamma \ar@{=>}[dl] \ar[d]^{J_y} \\ \inf P_x \ar[r]_{F_{\gamma}} & \inf P_y }
\end{equation*}
In this sense, parallel transport in any principal $\Gamma$-2-bundle is -- locally -- multiplication with the path-ordered exponential of a local connection 1-form $A$ along the path.  
\end{remark}

Suppose $(A,B)$ and $(A',B')$ are  $\Gamma$-connections on $M$ and $(g,\varphi)$ is a gauge transformation. By \cref{rem:trivbun:b} there is a 1-morphism $J:= J_{g,\varphi}:\inf I_{A,B} \to \inf I_{A',B'}$ in $\zweibuncon\Gamma M$. It is induced from a smooth functor $\phi_g$, whose restriction to a point $x$ is the functor $i_{g(x)}$ determined by the gauge transformation $g$ and
 \cref{rem:i:morph}. Thus, we have $J_x=i_{g(x)}$. According to \cref{sec:natbundlemorph}, $J$ determines a transformation $J_{\gamma}: J_y \circ F_{\gamma} \Rightarrow F'_{\gamma} \circ J_x$ for each path $\gamma:x \to y$.
The goal of the following proposition is to determine $J_{\gamma}$  in the present case of $J=J_{g,\varphi}$.

We consider  $h_{g,\varphi}(\gamma)\in H$ explained in \cref{sec:surfaceordered}. By \cref{lem:SE:gauge:a} it satisfies \begin{equation}
\label{eq:triv:gaugetrans:a}
\PE_{A'}(\gamma)\cdot g(x)=t(h_{g,\varphi}(\gamma))^{-1}\cdot g(y)\cdot \PE_{A}(\gamma)\text{.} 
\end{equation}
In other words, $\alpha_{g,\varphi}(\gamma):=(h_{g,\varphi}(\gamma)^{-1},g(y) \PE_{A}(\gamma))\in H \times G =\mor{\Gamma}$ is a morphism with source
$g(y) \PE_{A}(\gamma)$ and target $\PE_{A'}(\gamma) g(x)$. Associated to $a_{g,\varphi}(\gamma)$ is by \cref{rem:i:twomorph}  a natural transformation
\begin{equation*}
i_{\alpha_{g,\varphi}(\gamma)}:i_{g(y) \PE_{A}(\gamma)}\Rightarrow i_{\PE_{A'}(\gamma) g(x)}\text{.}
\end{equation*}
The following proposition shows that $i_{a_{g,\varphi}(\gamma)}$ corresponds to $J_{\gamma}$ under the transformation of \cref{eq:transeta}.

\begin{proposition}
\label{ex:gt}
Let $(g,\varphi):(A,B) \to (A',B')$ be a gauge transformation between $\Gamma$-connections, and let $J:\inf I_{A,B} \to \inf I_{A',B'}$ be the associated 1-morphism in $\zweibuncon\Gamma M$. For every path $\gamma:x \to y$ the diagram
\begin{equation*}
\alxydim{@=\xyst}{i_{g(y)\PE_A(\gamma)}\ar@{=}[d] \ar@{=>}[r]^-{i_{\alpha_{g,\varphi}(\gamma)}} &  i_{\PE_{A'}(\gamma)g(x)} \ar@{=}[d] \\i_{g(y)} \circ i_{\PE_A(\gamma)}  \ar@{=>}[d]_{\id \circ \eta_{\gamma}}  & i_{\PE_{A'}(\gamma)}\circ i_{g(x)} \ar@{=>}[d]^{\eta_{\gamma}'\circ \id} \\J_y \circ F_{\gamma} \ar@{=>}[r]_{J_{\gamma}} & F_{\gamma}' \circ J_x\text{,}}
\end{equation*}
is commutative, where $F_{\gamma}$, $\eta_{\gamma}$ and $F_{\gamma}'$, $\eta'_{\gamma}$ denote the parallel transports and the transformations of \cref{eq:transeta}  for $\inf I_{A,B}$ and $\inf I_{A',B'}$, respectively.
\end{proposition}

\begin{comment}
As a pasting diagram, this is:
\begin{equation*}
\alxydim{@=1.7cm}{\Gamma \ar@/_1.2pc/[r]_{F_{\gamma}}="2" \ar@/^1.2pc/[r]^{i_{F_{A,B}(\gamma)}}="1" \ar@{=>}"1";"2"|*+{\eta_{\gamma}} \ar[d]_{J_x} & \Gamma \ar@{=>}[dl]|>>>>>>>>*{J_{\gamma}} \ar[d]^{J_y} \\ \Gamma \ar@/_1.2pc/[r]_{F_{\gamma}'} & \Gamma}
=
\alxydim{@=1.7cm}{\Gamma  \ar@/^1.2pc/[r]^{i_{F_{A,B}(\gamma)}}  \ar[d]_{i_{\rho_{g,\varphi}(x)}} & \Gamma \ar@{=>}[dl]|<<<<<<<<*{i_{\rho_{g,\varphi}(\gamma)}} \ar[d]^{i_{\rho_{g,\varphi}(y)}} \\ \Gamma \ar@/_1.2pc/[r]_{F'_{\gamma}}="2" \ar@/^1.2pc/[r]^{i_{F_{A',B'}(\gamma)}}="1" \ar@{=>}"1";"2"|*+{\eta'_{\gamma}} & \Gamma}
\end{equation*}
\end{comment}

\begin{proof}
The diagram is an equality between two transformations from a smooth functor to an anafunctor. We express them under the correspondence of \cref{rem:indtrans}, getting counter-clockwise the smooth map
\begin{equation}
\label{eq:trivbun:11}
g'\mapsto J_{\gamma}(\eta_{\gamma}(g'),(y,\PE_A(\gamma)g',\id_{g(y)\PE_A(\gamma)g'}))
\end{equation}
and clockwise the smooth map
\begin{align}
\label{eq:trivbun:12}
g' \mapsto  ((x,g',\id_{g(x)g'}),\tilde\eta'_{\gamma}(g(x)g'))\circ i_{\alpha_{g,\varphi}(\gamma)}(g')\text{.}
\end{align}
We  show that both expressions coincide. In the clockwise direction, we employ the definition of $\eta_{\gamma'}$ and obtain after some straightforward manipulations
\begin{equation}
\label{eq:trivbun:13}
((x,g',\id_{g(x)g'}),((y,h_{g,\varphi}(\gamma),\PE_{A'}(\gamma) g(x)g') \ast (\gamma,\kappa' g(x)g') \ast \id_{(x,g(x)g')})\text{.}
\end{equation}
\begin{comment}
Indeed,
\begin{align*}
g' \mapsto  &((x,g',\id_{g(x)g'}),\tilde\eta'_{\gamma}(g(x)g'))\circ i_{\alpha_{g,\varphi}(\gamma)}(g')
\\&=((x,g',\id_{g(x)g'}),(\id \ast (\gamma,\kappa' g(x)g') \ast \id)\circ (h_{g,\varphi}(\gamma)^{-1},g(y) \PE_{A}(\gamma)g'))
\\&=((x,g',\id_{g(x)g'}),((h_{g,\varphi}(\gamma),t(h_{g,\varphi}(\gamma))^{-1}g(y) \PE_{A}(\gamma)g') \ast (\gamma,\kappa' g(x)g') \ast \id)
\\&=((x,g',\id_{g(x)g'}),((h_{g,\varphi}(\gamma),\PE_{A'}(\gamma) g(x)g') \ast (\gamma,\kappa' g(x)g') \ast \id)
\end{align*}
\end{comment}
Counter-clockwise, we write $\xi_{g'} := \eta_{\gamma}(g')=\id_{(y,\kappa(1)g')} \ast (\gamma,\kappa g') \ast \id_{(x,g')}$. The result of $J_{\gamma}$ will be computed using  \cref{re:Jfunctor} and the following facts about $J=J_{g,\varphi}$, which can be looked up in \reftrivialbundle. The first fact is that $J$ has an underlying functor $\phi_g$, and the second fact is that the canonical $\Omega_{A',B'}$-pullback on $J$ is shifted by a pair of forms $(\varphi_0,\varphi_1)$, with $\varphi_0 :=(\alpha_{\pr_G^{-1}\cdot g^{-1}})_{*}(\pr_M^{*}\varphi)$, see \refdefshiftedpullbacktrivialbundle.
Now,   \cref{re:Jfunctor}  implies

 , 
\begin{equation}
\label{eq:trivbun:14}
J_{\gamma}(\xi_{g'} ,(y,\PE_A(\gamma)g',\id_{g(y)\PE_A(\gamma)g'}))=((x,g',\id_{g(x)g'}), \phi_g^{\varphi}(\xi_{g'}))\text{,}
\end{equation}
where 
\begin{align*}
\phi_g^{\varphi}(\xi_{g'}) &= (y,(\alpha(g(y)\kappa(1)g',\tilde h(1)),g(y)\kappa(1)g't(\tilde h(1))^{-1})) \ast (\gamma,g(\gamma)\kappa g't(\tilde h)^{-1})\ast \id_{(x,g(x)g')}\text{,}
\end{align*}
and $\tilde h$ is the solution to the initial value problem
\begin{equation}
\label{eq:trivbun:17}
\partial_t \tilde h(t) = -\tilde h(t) \varphi_{0}(\dot\gamma(t),\dot\kappa(t)g') \quand \tilde   h(0)=1\text{.}
\end{equation}
\begin{comment}
Indeed,
\begin{align*}
\phi_g^{\varphi}(\xi_{g'}) &= R(\id_{\phi_g(y,\kappa(1)g')},(\tilde h(1),t(\tilde h(1))^{-1})) \ast R(\phi_g(\gamma,\kappa g'),t(\tilde h)^{-1})\ast \id
\\&= R((y,1,g(y)\kappa(1)g'),(\tilde h(1),t(\tilde h(1))^{-1})) \ast (\gamma,g(\gamma)\kappa g't(\tilde h)^{-1})\ast \id
\\&= (y,(\alpha(g(y)\kappa(1)g',\tilde h(1)),g(y)\kappa(1)g't(\tilde h(1))^{-1})) \ast (\gamma,g(\gamma)\kappa g't(\tilde h)^{-1})\ast \id
\end{align*}
\end{comment}
The key to the proof that \cref{eq:trivbun:13,eq:trivbun:14} coincide  is to understand the relation between $\kappa$, $\kappa$, and $\tilde h$. The relation between $\kappa$ and $\kappa'$ is established by the gauge transformation, which gives $\mathrm{Ad}_{g}(A) - g^{*}\bar \theta =A'+t_{*}(\varphi)$. From the proof of \cite[Lemma 2.18]{schreiber2} we have
\begin{equation}
\label{eq:triv:gaugetrans:b}
\kappa'(t)=t(h(t))^{-1} g(\gamma(t))\kappa(t)g(x)^{-1}
\end{equation} 
where $h:[0,1] \to H$ is a smooth map such that the pair $(h,\kappa')$ solves the initial value problem
\begin{equation*}
\partial_t (h(t),\kappa'(t))=-(\varphi(\partial_t\gamma(t)),A'(\partial_t\gamma(t)))\cdot (h(t),\kappa'(t))
\quand
h(0)=1\text{, }\kappa'(0)=1\text{.}
\end{equation*}
Splitting this into components, one obtains as an equivalent characterization that $h$ solves the initial value problem
\begin{equation}
\label{eq:trivbun:20}
h(t)^{-1}\partial h_t(t)=-\mathrm{Ad}_{h(t)}^{-1}(\varphi(\partial_t\gamma(t)))+ (\tilde\alpha_{h(t)})_{*}(\partial_t\kappa'(t)\kappa'(t)^{-1})
\quand h(0)=1\text{.}
\end{equation}
\begin{comment}
This means:
\begin{align*}
\partial_t (h(t),\kappa'(t))&=-(\varphi(\partial_t\gamma(t)),A'(\partial_t\gamma(t)))\cdot (h(t),\kappa'(t))
\\&=-(h(t)(\mathrm{Ad}_{h(t)}^{-1}(\varphi(\partial_t\gamma(t)))+ (\tilde\alpha_{h(t)})_{*}(A'(\partial_t\gamma(t)))),A'(\partial_t\gamma(t))\kappa'(t))
\end{align*}
and hence
\begin{equation*}
h(t)^{-1}\partial h_t(t)=-\mathrm{Ad}_{h(t)}^{-1}(\varphi(\partial_t\gamma(t)))+ (\tilde\alpha_{h(t)})_{*}(\partial_t\kappa'(t)\kappa'(t)^{-1})\text{.}
\end{equation*}
Here we have used
\begin{equation*}
(Y,X)\cdot (h,g) = \partial_t((y(t),x(t))\cdot (h,g))=\partial_t(y(t)\alpha(x(t),h),x(t)g)=(h(\mathrm{Ad}_h^{-1}(Y)+\tilde \alpha_h(X)),Xg)
\end{equation*}
\end{comment}
By construction, $\kappa(1)=\PE_A(\gamma)$, $\kappa'(1)=\PE_{A'}(\gamma)$, and $h(1)=h_{g,\varphi}(\gamma)$. Evaluating at $t=1$, \cref{eq:triv:gaugetrans:b} implies  \cref{eq:triv:gaugetrans:a}. We claim that
\begin{equation}
\label{eq:trivbun:15}
\tilde h=\alpha(g'^{-1}\kappa^{-1}  g(\gamma)^{-1},h)\text{.}
\end{equation} 
Given that claim, we have coincidence of \cref{eq:trivbun:13,eq:trivbun:14}, established by the two equalities
\begin{align*}
(\gamma,g(\gamma)\kappa g't(\tilde h)^{-1})&=(\gamma,\kappa' g(x)g')
\\
(\alpha(g(y)\kappa(1)g',\tilde h(1)),g(y)\kappa(1)g't(\tilde h(1))^{-1})&=(h_{g,\varphi}(\gamma),\PE_{A'}(\gamma) g(x)g')\text{,}
\end{align*}
which can easily be deduced from \cref{eq:trivbun:15}. 
\begin{comment}
Indeed,
we have\begin{equation*}
t(\tilde h)=g'^{-1}\kappa^{-1}  g(\gamma)^{-1}t(h)g(\gamma)\kappa g'
\end{equation*}
and then 
\begin{equation*}
g(\gamma)\kappa g't(\tilde h)^{-1}=t(h)^{-1}g(\gamma)\kappa g'\eqcref{eq:triv:gaugetrans:b} \kappa' g(x)g'\text{.}
\end{equation*}
For the second equality, we first note that
\begin{equation*}
g(y)\kappa(1)g't(\tilde h(1))^{-1}
=t(h(1))^{-1}g(y)\kappa(1) g' \eqcref{eq:triv:gaugetrans:b} \kappa'(1)g(x)g'=\PE_{A'}(\gamma) g(x)g'
\end{equation*}
and compute then
\begin{equation*}
\alpha(g(y)\kappa(1)g',\tilde h(1))=\alpha(g(y)\kappa(1)g',\alpha(g'^{-1}\kappa(1)^{-1}  g(y)^{-1},h(1)))=h(1)=h_{g,\varphi}(\gamma)\text{.}
\end{equation*}
\end{comment}
It remains to prove the claim, \cref{eq:trivbun:15}. For this purpose we prove that $\tilde h$ as defined in \cref{eq:trivbun:15} solves the initial value problem \cref{eq:trivbun:17}.
We have $\tilde  h(0)=1$ and obtain
\begin{equation*}
\partial_t \tilde h(t) = (\alpha_{h(t)})_{*}(g'^{-1}\partial_t\kappa(t)^{-1}  g(\gamma(t))^{-1}+g'^{-1}\kappa(t)^{-1}  \partial_tg(\gamma(t))^{-1})+(\alpha_{g'^{-1}\kappa(t)^{-1}  g(\gamma(t))^{-1}})_{*}(\partial_th(t))\text{.}
\end{equation*}
Taking derivative in the inverse of \cref{eq:triv:gaugetrans:b} gives
\begin{multline}
\label{eq:trivbun:18}
g'^{-1}g(x)^{-1}\partial_t\kappa'(t)^{-1}t(h(t))^{-1}=g'^{-1}\partial_t\kappa(t)^{-1}g(\gamma(t))^{-1}+g'^{-1}\kappa(t)^{-1}\partial_tg(\gamma(t))^{-1}
\\+g'^{-1}\kappa(t)^{-1}g(\gamma(t))^{-1}t_{*}(\partial_t h(t)h(t)^{-1})
\end{multline}
Using \cref{eq:trivbun:18,eq:trivbun:18} the differential equation of  \cref{eq:trivbun:17} follows.
\end{proof}

We continue our discussion of  parallel transport in the trivial principal $\Gamma$-bundle $\inf I_{A,B}$ with the parallel transport along bigons,  now assuming that $(A,B)$ is fake-flat. 

\begin{proposition}
\label{eq:T2:eta}
Let $(A,B)$ be a fake-flat $\Gamma$-connection on $M$. For a bigon $\Sigma: \gamma \Rightarrow \gamma'$ we let $\varphi_{\Sigma}:F_{\gamma}\Rightarrow F_{\gamma'}$ denote the parallel transport in the associated trivial principal $\Gamma$-2-bundle $\inf I_{A,B}$. We set $g := \PE_{A}(\gamma)$, $g' := \PE_{A}(\gamma')$ and $h :=\SE_{A,B}(\Sigma)$. Then, the  diagram
\begin{equation*}
\alxydim{@=\xyst}{i_{g} \ar@{=>}[d]_{i_{h,g}} \ar@{=>}[r]^-{\eta_{\gamma}} & F_{\gamma} \ar@{=>}[d]^{\varphi_{\Sigma}}  \\  i_{g'} \ar@{=>}[r]_-{\eta_{\gamma'}} & F_{\gamma'}}
\end{equation*}
is commutative,
where $i_g$ and $i_{g'}$ are the  functors of \cref{rem:i:morph} and $i_{g,h}$ is the natural transformation of \cref{rem:i:twomorph}. 
\end{proposition}

\begin{proof}
The diagram is an equality between transformations from a functor to an anafunctor. In terms of the corresponding smooth maps of \cref{rem:indtrans}  the commutativity means $\varphi_{\Sigma}(\tilde\eta_{\gamma}(\tilde g)) = \tilde\eta_{\gamma'}(\tilde g) \circ i_{h,g}(\tilde g)$.
We recall that in order to compute $\tilde\eta_{\gamma}$ and $\tilde\eta_{\gamma'}$ we have the paths $\kappa$ and $\kappa'$. In fact, since $\gamma$ and $\gamma'$ are homotopic via the bigon $\Sigma$, there is a family $\kappa_s$ of paths such that $\kappa_0=\kappa$ and $\kappa_1=\kappa'$.

We construct a horizontal lift of $\Sigma$ with source $\tilde\eta_{\gamma}(\tilde g)=\id_{(y,g\tilde g)} \ast (\gamma,\kappa\tilde  g)\ast \id_{(x,\tilde g)}$, in the sense of \cref{def:horliftbigon}. It is given by $\Phi_1:[0,1]^2 \to M \times G$ defined by $\Phi_1(s,t):=(\Sigma(s,t),\kappa_s(t)\tilde   g)$, $\rho_0=\id_{(x,\tilde g)}$, $\rho_1=\id_{(y,g\tilde g)}$ and $g_1(s):=\tilde g^{-1}\kappa_s(1)^{-1}g\tilde g$, and all other data trivial.
\begin{comment}
\begin{enumerate}[(a)]

\item
is obviously true

\item
the only non-trivial point is $s(\rho_{1}(s))  =(y,g\tilde g)=(y,\kappa_s(1)\tilde g g_1(s))=R(\Phi_1(s,1),g_{1}(s))$.

\item
we have $\nu_1(s)=\Phi_{1}(s,0)=(x,\tilde   g)$; this is horizontal, and 
\begin{equation*}
  \gamma_1'(t) = \Phi_1(1,t)=(\gamma'(t),\kappa'(t)\tilde   g)
\end{equation*}
this is horizontal by definition of $\kappa'$.

\item
is clear.

\end{enumerate}

\end{comment}
Its target is
$\varphi_{\Sigma}(\tilde\eta_{\gamma}(\tilde g))=\rho_1'\ast (\gamma',\kappa' \tilde g)\ast \id_{(x,\tilde g)}$
where $\rho_1'$ has to be determined.
We claim that $\rho_1'=(h,g\tilde g)^{-1}$; given this claim we have
\begin{equation*}
\varphi_{\Sigma}(\tilde\eta_{\gamma}(\tilde g))=\rho_1'\ast (\gamma',\kappa' \tilde g)\ast \id_{(x,\tilde g)} =(\id_{(y,g'\tilde g)} \ast (\gamma',\kappa' \tilde g)\ast \id_{(x,\tilde g)}) \circ (h,g\tilde g)=\tilde\eta_{\gamma'}(\tilde g)\circ i_{h,g}(\tilde g)\text{;}
\end{equation*}
this proves the commutativity of the diagram.

In order to prove the claim, we recall the definition of the target in \cref{eq:def:targethorlift}, resulting in $\rho_1':= R(\id_{(y,g\tilde g)},(h_1^{-1},g_1(1)^{-1}))$. In order to compute $h_1$ we have to choose a bigon-parameterization  $\Sigma_1$  of $\Phi_1$. 
\begin{comment}
We have
\begin{equation*}
\Sigma_1 : \mu_1\circ\gamma_1\Rightarrow \gamma_1'\circ \nu_1\text{,}
\end{equation*}
where $\mu_1(s) := \Phi_{1}(s,1)= (y,\kappa_s(1)\tilde   g)$. 
We have
\begin{equation*}
\PE_{\fa\Omega}(R(\mu_1,g_1))=\PE_{\fa\Omega}(y,g\tilde g)=1\text{;}
\end{equation*}
this is consistent with \cref{lem:thi}.
Going further through that lemma, we have
\begin{equation*}
1=\PE_{\fa\Omega}(R(\mu_1,g_1))=\PE_{R^{*}\fa\Omega}(\mu_1,g_1)=g_1(1)^{-1} \PE_A(\mu_1) g_1(0)=g_1(1)^{-1} \PE_A(\mu_1)\text{.}
\end{equation*}
\end{comment}
It will suffice to chose a bigon-parameterization $\Xi$ of  $(s,t) \mapsto \kappa_s(t)\tilde g$, so that $\Xi:g \tilde g g_1^{-1} \ast \kappa\tilde g \Rightarrow \kappa'\tilde g \ast \id_{\tilde g}$ is a bigon in $G$. Then we may choose $\Sigma_1:= (\Sigma,\Xi)=R((\Sigma,1),\Xi)$. The canonical section $s:x \mapsto (x,1)$ satisfies $s^{*}\fa\Omega=A$ and $s^{*}\fc\Omega=-B$; thus \cref{lem:SE:e} gives
\begin{equation}
\label{eq:se:section}
\SE_{\Omega}(\Sigma,1) =\SE_{\Omega}(s(\Sigma))=\SE_{A,B}(\Sigma)\text{.}
\end{equation}
Now we obtain
\begin{equation*}
h_1:=\SE_{\Omega}(\Sigma_1)\eqcref{lem:soe2bun}\alpha((g'\tilde g)^{-1},\SE_{\Omega}(\Sigma,1))\eqcref{eq:se:section}\alpha(\tilde g^{-1}g'^{-1},\SE_{A,B}(\Sigma))\eqcref{lem:SE:b}\alpha(\tilde g^{-1}g^{-1},h)\text{.}
\end{equation*}
\begin{comment}
The last step uses \cref{lem:SE:b}: $g'^{-1}t(h)=g^{-1}$, so that $\alpha(g^{-1},h)=\alpha(g'^{-1},h)$.
\end{comment}
\begin{comment}
This is consistent with
\begin{equation*}
g_1(1)^{-1}=\tilde g^{-1}g^{-1}g'\tilde g=\tilde g^{-1}g^{-1}t(h)g\tilde g=t(h_1)\text{.}
\end{equation*}
\end{comment}
Now a straightforward computation shows the claim. 
\begin{comment}
We can conclude with the computation
\begin{align*}
\rho_1'&=R(\id_{(y,g\tilde g)},(h_1^{-1},g_1(1)^{-1}))
\\&=(y,(1,g\tilde g)\cdot(h_1^{-1},g_1(1)^{-1}) )
\\&=(y,(\alpha(g\tilde g,h_1^{-1}),g\tilde gg_1(1)^{-1}) )
\\&=(\alpha(g\tilde g,h_1),g\tilde g)^{-1} 
\\&=(h,g\tilde g)^{-1} 
\end{align*}
\end{comment}
\end{proof}

\subsection{Ordinary principal bundles}

\label{ex:gbunred:paths}
Consider an ordinary principal $G$-bundle $P$ over $M$ with connection $\omega\in \Omega^1(P,\mathfrak{g})$. As discussed in \refactiontwobundle\ the action groupoid $\act PH$ for the right $H$-action on $P$ induced via $t: H \to G$ is a principal $\Gamma$-2-bundle over $M$, and it is equipped with a  connection $\Omega$ induced by $\omega$. For a point $x\in M$ we have $(\act PH)_x = \act{P_x}H$. For a path $\gamma:x \to y$ in $M$, we have the ordinary parallel transport map $\tau_{\gamma}: P_x \to P_y$. It is $G$-equivariant, hence $H$-equivariant, and thus induces a smooth functor 
\begin{equation*}
\phi_{\gamma}: \act {P_x}H \to \act {P_y}H
\end{equation*}
between action groupoids. 
\begin{comment}
We have $\phi_{\gamma}(p):=\tau_{\gamma}(p)$ and $\phi_{\gamma}(p,h) := (\tau_{\gamma}(p),h)$. 
\end{comment}
It is straightforward to check that it is $\Gamma$-equivariant.
\begin{comment}
Indeed, on the level of objects we have
\begin{equation*}
\phi_{\gamma}(R((p,g))=\phi_{\gamma}(pg)=\tau_{\gamma}(pg)=\tau_{\gamma}(p)g=R(\phi_{\gamma}(p),g)
\end{equation*}
and on the level of morphisms we have
\begin{multline*}
\phi_{\gamma}(R((p,h),(h',g)))=\phi_{\gamma}(pg,\alpha(g^{-1},hh'))\\=(\tau_{\gamma}(p)g,\alpha(g^{-1},hh'))=R((\tau_{\gamma}(p),h),(h',g))=R(\phi_{\gamma}(p,h),(h',g))\text{.}
\end{multline*}
\end{comment}
We claim that there exists a canonical $\Gamma$-equivariant transformation
\begin{equation*}
f_{\gamma}: J_{\phi_{\gamma}} \Rightarrow F_{\gamma}
\end{equation*}
between the anafunctor induced by $\phi_{\gamma}$ and $F_{\gamma}$. We construct $f_{\gamma}$ using \cref{rem:indtrans}; the underlying smooth map $\tilde f_{\gamma}: P_x \to F_{\gamma}$ is defined by $\tilde f_{\gamma}(p):= \id_{\phi_{\gamma}(p)} \ast \tilde\gamma_p \ast \id_p$, where $\tilde\gamma_p$ is the unique horizontal lift of $\gamma$ with initial point $p$.

\begin{lemma}
The map $\tilde f_{\gamma}$ satisfies \cref{T1*,T2*,T3*}.
\end{lemma}

\begin{proof}
\cref{T1*} is obvious. For \cref{T2*} we compute for $\alpha=(p,h)\in P \times H = \mor{\act {P_x}H}$:
\begin{align*}
\alpha\circ \tilde f_{\gamma}(p)\circ \beta &=  \beta^{-1}\ast \tilde\gamma_p \ast  \alpha^{-1}
\\&\sim (\beta^{-1}\circ(\tilde\gamma_p(1),h)^{-1} )\ast \tilde\gamma_{t(\alpha)}\ast (\alpha\circ \alpha^{-1}) 
\\&= (\beta^{-1}\circ \phi_{\gamma}(p,h)^{-1} ) \ast \tilde\gamma_{t(\alpha)} \ast \id_{t(\alpha)}
\\&= \tilde f_{\gamma}(t(\alpha))\circ \phi_{\gamma}(\alpha)\circ \beta
\end{align*}
Here we have applied the equivalence relation in $F_{\gamma}$ to the path $\rho(t):=(\tilde\gamma_p(t),h)\in P \times H$, which is horizontal: following \refactiontwobundle\ we have $\fb\Omega = (\tilde\alpha_{\pr_H})_{*}(\pr_P^{*}\omega)+\pr_H^{*}\theta$ and hence $\fb\Omega(\dot\rho(t))=0$. 
\begin{comment}
Further, $s(\rho)=\tilde\gamma_p$  and $t(\rho)=\tilde\gamma_p t(h)=\tilde\gamma_{t(\alpha)}$, because $\tilde\gamma_p t(h)$ is horizontal and has initial point $\tilde\gamma_p(0)t(h)=pt(h)=t(\alpha)$. \end{comment} 
Finally, \cref{T3*} is a straightforward calculation. 
\begin{comment}
Indeed,
\begin{equation*}
\tilde f_{\gamma}(pg)= \id_{\phi_{\gamma}(pg)} \ast \tilde\gamma_{pg} \ast \id_{pg}= (\id_{\phi_{\gamma}(p)} \ast \tilde\gamma_p \ast \id_p)\cdot \id_g=\tilde f_{\gamma}(p)\cdot \id_g
\end{equation*}
\end{comment}
\end{proof}

Summarizing, in the principal $\Gamma$-2-bundle $\act PH$, parallel transport along a path $\gamma$ is given, up to canonical isomorphism of $\Gamma$-equivariant anafunctors, by the smooth functor $\phi_{\gamma}$. It is obvious that this identification is compatible with pullbacks, bundle morphisms, and  path composition. 

\begin{proposition}
\label{ex:gbunred:bigons}
Let $P$ be a principal $G$-bundle with flat connection $\omega$. Let $\Sigma:\gamma_0 \Rightarrow \gamma_1$ be a bigon. The diagram 
\begin{equation*}
\alxydim{@=\xyst}{J_{\phi_{\gamma_0}} \ar@{=}[r] \ar@{=>}[d]_{f_{\gamma_0}} & J_{\phi_{\gamma_1}} \ar@{=>}[d]^{f_{\gamma_1}} \\F_{\gamma_0} \ar@{=>}[r]_{\varphi_{\Sigma}} & F_{\gamma_1} }
\end{equation*}
is commutative, where $\varphi_{\Sigma}:F_{\gamma_0}\Rightarrow F_{\gamma_1}$ denotes the parallel transport in the principal $\Gamma$-2-bundle $\act PH$.
In particular, $\varphi_{\Sigma}$ only depends on $\gamma_0$ and $\gamma_1$ but not on the bigon $\Sigma$. 
\end{proposition}

\begin{proof}
Since $\omega$ is flat, the induced connection $\Omega$ on $\act PH$ is fake-flat. Further, since the parallel transport of a flat connection only depends on the homotopy class of the path, we have  $\tau_{\gamma_0}=\tau_{\gamma_1}$, thus $\phi_{\gamma_0}=\phi_{\gamma_1}$, and in turn $J_{\phi_{\gamma_0}}=J_{\phi_{\gamma_1}}$.  In order to prove commutativity, we specify a horizontal lift of $\Sigma$ in the sense of \cref{def:horliftbigon}, with source $\tilde f_{\gamma_0}(p)$ for some $p\in P_x$. Let $\gamma_s:[0,1] \to M$ be defined by $\gamma_s(t):=\Sigma(s,t)$, and let $\tilde\gamma_{s,p}$ be the unique horizontal lift into $P$ of $\gamma_s$ with initial point $p$. Let $\Phi(s,t) := \tilde\gamma_{s,p}(t)$. Because $\omega$ is flat, we have $\Phi(s,1)=q$ for some constant point $q\in P_y$. Taking all other data trivial, $\Phi$ is indeed a horizontal lift of $\Sigma$ with source  $\tilde f_{\gamma_0}(p)$. Since $\fc\Omega=0$, we have  $\SE_{\Omega}(\Phi)=1$ by \cref{lem:soetriv}; hence, the target of this horizontal lift is $\tilde\gamma_{1,p}=\tilde f_{\gamma_1}(p)$. \end{proof}

\section{The parallel transport 2-functor}

\label{sec:pt2functor}

In this section we prove the main result of this article, namely that the parallel transport constructions of \cref{sec:ptpaths,sec:ptbigons} fit in the axiomatic framework of  transport 2-functors. This framework is formulated for \emph{thin homotopy classes} of paths and bigons. In \cref{sec:thinhomo} we provide a way to push our constructions into the setting of thin homotopy classes. In \cref{sec:orga} we show that the various properties we have proved in \cref{sec:ptpaths,sec:ptbigons} show that parallel transport is a 2-functor, and in \cref{sec:main} we show that this 2-functor is a \emph{transport} 2-functor.

\subsection{Thin homotopy invariance}

\label{sec:thinhomo}

We study the dependence of the parallel transports along paths and bigons under thin homotopies, i.e. smooth homotopies with non-maximal rank. Here, by rank of a smooth map we mean the supremum of the rank of its differential over all points. All kinds of reparameterizations are special cases of thin homotopies.

Two bigons $\Sigma,\Sigma':\gamma \Rightarrow \gamma'$ between paths $\gamma,\gamma':x \to y$ are called \emph{homotopic}, if there exists a smooth homotopy $h: [0,1]^3 \to M$ (i.e., $h(0,s,t)=\Sigma(s,t)$ and $h(1,s,t)=\Sigma'(s,t)$) that fixes all boundaries, i.e. $h(r,s,0)=x$, $\Sigma(r,s,1)=y$, $h(r,0,t)=\gamma(t)$ and $h(r,1,t)=\gamma'(t)$ for all $r,s,t\in[0,1]$. Two bigons are called \emph{thin homotopic}, if they are homotopic by a homotopy of rank \emph{less} than 3. 

\begin{proposition}
\label{lem:varphisigma:d}
\label{lem:varphisigma:d:a}
Let $\inf P$ be a principal $\Gamma$-2-bundle with fake-flat connection.
Then, the parallel transport along bigons depends only on the thin homotopy class of the bigon, i.e., if  $\Sigma,\Sigma':\gamma \Rightarrow \gamma'$ are thin homotopic bigons, then  $\varphi_{\Sigma}=\varphi_{\Sigma'}$.
\end{proposition}

\begin{proof}
We first note that every thin homotopy can be split into  finitely many small ones, so that it suffices to prove the claim for a small thin homotopy. By \quot{small} we mean that  there exist $n\in \N$, $t\in T_n$ and sections $\sigma_i:U_i \to \ob{\inf P}$ defined on  open sets $U_i$ such that 
\begin{equation*}
h(\{(r,s,t) \sep t_{i-1}\leq t \leq t_i,0\leq r,s\leq 1\})\subset U_i\text{.}
\end{equation*} 

We think of $\Sigma^{r} :=h(r,-,-)$ as a smooth family of small bigons. We claim  that we can consistently choose a smooth family of horizontal lifts with a common source $\xi\in F_{\gamma}$. This means, there exist $\Phi_i^{r}:[0,1] \times [t_{i-1},t_i] \to \ob{\inf P}$, $\rho_i^r: [0,1] \to \mor{\inf P}$, and $g_i^r:[0,1] \to G$ depending smoothly on $r$ (this means, for instance, that $[0,1] \times [0,1] \times [t_{i-1},t_i] \to\ \ob{\inf P}:(r,s,t) \mapsto \Phi_i^{r}(s,t)$ is smooth), such that $(\Phi_i^{r},\rho_i^{r},g_i^{r})$  is a horizontal lift of $\Sigma^{r}$ with source $\xi$, and $\Phi^{r}_i(s,t)$ has rank less than 3. Additionally, we can require $\Phi_i^{r}(0,t)$ and $\rho_i^{r}(0)$ are independent of $r$, and we require that there exist smooth maps $k_i:[0,1] \to G$, denoted $k_i^{r}$, such that $k_i^0=1$ and $\Phi^{r}_i(1,t)=R(\Phi_i^0(1,t),(k_i^{r})^{-1})$. This claim can be proved by repeating the proof of \cref{lem:existencehorizontallifts} in families.

We remark that $\rho_i^0(1)$ and $\rho_i^r(1)$ satisfy 
\begin{align*}
R(t(\rho_i^0(1)),(k_{i+1}^{r})^{-1})&=R(\Phi _{i+1}^0(1,t_{i-1}),(k_{i+1}^{r})^{-1})=\Phi^{r}_{i+1}(1,t_{i-1})=t(\rho_i^r(1))
\\
R(s(\rho_i^0(1)),g_i^0(1)^{-1}(k_{i}^{r})^{-1}g_i^r(1))&=R(\Phi _{i}^0(1,t_{i}),(k_{i}^{r})^{-1}g_i^r(1))=R(\Phi _{i}^r(1,t_{i}),g_i^r(1))=s(\rho_i^r(1))
\end{align*}
By \refactionfullyfaithful, there exist unique $\eta_i^{r}\in H$, smoothly depending on $r$, with $\eta_i^0=1$ such that 
\begin{equation}
\label{eq:thin:3}
\rho_i^r(1)=R(\rho_i^0(1),(\eta_i^{r},g_i^0(1)^{-1}(k_{i}^{r})^{-1}g_i^r(1)))
\;\text{ and }\;
t(\eta_i^{r})g_i^0(1)^{-1}(k_{i}^{r})^{-1}g_i^r(1)=(k_{i+1}^{r})^{-1}\text{.}
\end{equation} 
Next we perform the following pre-computations (recall the notion of a bigon-parameterization from \cref{rem:bigonpar}):
\begin{enumerate}[(a)]
\item 
Consider a bigon-parameterization $\Theta: \zeta_i \Rightarrow\ \zeta_i'$ in $G$ of $(r,t) \mapsto (k_i^r)^{-1}$ with $\zeta_i = (k_i^{-})^{-1} \ast \id_1$ and $\zeta_i' = \id_1 \ast (k_i^{-})^{-1}$. Then, $\Xi_i := R(\id_{\gamma_i^{\prime0}},\Theta)$ is a bigon-parameterization of $(r,t) \mapsto \Phi^{r}_i(1,t)$, going between  the paths $\kappa_i(r) := R(\gamma_i^{\prime 0}(t_i),(k_i^{r})^{-1})$ and $\kappa_i'(r) := R(\gamma_i^{\prime 0}(t_{i-1}),(k_i^{r})^{-1})$.   By \cref{lem:soe2bun} we have 
\begin{equation*}
\SE_{\Omega}(\Xi_i)=\SE_{\Omega}(R(\id_{\gamma_i^{\prime0}},\Theta))=\alpha(\zeta_i(1)^{-1},\SE_{\Omega}(\id_{\gamma_i^{\prime0}}))=1\text{.}
\end{equation*}

\item
We note that $r \mapsto \Phi^{r}_i(s,t)$ is a homotopy between bigons
\begin{equation*}
(\id_{\gamma_i^{\prime 1}}\circ \Psi_i') \bullet (\Xi_i\circ \id_{\nu_i^0}) \bullet (\id_{\kappa_i}\circ \Sigma_i^{0})
\quand
\Sigma_i^1 \bullet (\Psi_i \circ \id_{\gamma_i})\text{,}
\end{equation*}
both going from $\kappa_i \circ \mu_i^0 \circ \gamma_i$ to $\gamma_i^{\prime 1}\circ \nu_i^1$. Since $\Phi^{r}_i(s,t)$ is thin, the surface-ordered exponentials of both bigons coincide (\cref{lem:SE:a}).
Using \cref{lem:SE:b,lem:SE:c} we get
$\psi_i'\cdot \alpha(k_i^1,h_i^0)  =h_i^1\cdot \psi_i$. 
With \cref{eq:thin:1} we can rewrite this as
\begin{equation}
\label{eq:thin:2}
h_i^0\cdot  \psi_i'=h_i^1\cdot \psi_i
\end{equation}

\item
We consider bigon-parameterizations $\Upsilon_i:\rho_i^-(1) \ast \rho_i^0 \Rightarrow \rho_i^1$ in $\mor{\inf P}$ of $(r,s) \mapsto \rho_i^r(s)$ and $\Theta_i: g_i^-(1) \ast g_i^0 \Rightarrow g_i^1$ in $G$ of $(r,s) \mapsto g_i^r(s)$. Note that  $\Psi_{i+1} :=t(\Upsilon_i):\kappa_{i+1}\circ \mu_{i+1}^0 \Rightarrow \mu_{i+1}^{1}$  is a bigon-parameterization of $(r,s) \mapsto \Phi_{i+1}^{r}(s,t_{i+1})$, and that $\Psi_i ':= R(s(\Upsilon_i),\Theta_i^{-1}): \kappa_i'\circ \nu_i^0 \Rightarrow \nu_i^{1}$ is a bigon-parameterizations of $(r,s) \mapsto \Phi_i^{r}(s,t_{i-1})$. We set $\psi_i' := \SE_{\Omega}(\Psi_i')$ and  $\psi_i := \SE_{\Omega}(\Psi_i)$. 
We compute the quantity $h_{\Omega}$ of the source and target paths of $\Upsilon_i$. Since $\rho_i^{r}$ are horizontal with horizontal target $\nu_{i+1}^{r}$ (and hence also horizontal source), we have by \cref{co:hcalc:b} $h_{\Omega}(\rho_i^r)=1$. Further, we calculate \begin{align*}
h_{\Omega}(\rho_i^-(1))
&\eqcref{eq:thin:3}h_{\Omega}(R(\rho_i^0(1),(\eta_i^{r},g_i^0(1)^{-1}(k_{i}^{r})^{-1}g_i^r(1))))
\\&= h_{\Omega}(R(R(\rho_i^0(1),(\eta_i^{r},1)),(1,g_i^0(1)^{-1}(k_{i}^{r})^{-1}g_i^r(1)))))
\\&\eqcref{co:h:b}\alpha(g_i^1(1)^{-1}k_{i}^{1}g_i^0(1),h_{\Omega}(R(\rho_i^0(1),(\eta_i^{r},1)),1))
\\&\eqcref{co:hcalc:a}\alpha(g_i^1(1)^{-1}k_i^1g_i^0(1),(\eta^1_i)^{-1})
\end{align*}

Now, \cref{prop:SE:gauge:b} gives
\begin{equation}
\label{eq:thin:5}
\psi_i \cdot \alpha(k_i^1g_i^0(1),\eta^1_i)=\alpha(g_i^1(1),\psi'_{i+1})\text{.}
\end{equation}

\item
Since $\nu_i^0$ and $\nu_i^1$ are horizontal, we have from \cref{lem:poeOmega:x}
\begin{equation}
\label{eq:thin:1}
1=t(\psi_i') \PE_{\fa\Omega}(\kappa_i') = t(\psi_i')k_i^1\text{.} 
\end{equation}
\end{enumerate}
Summarizing our pre-calculations, we obtain:
\begin{align}
\nonumber
\eta_i^1h_i^0\psi_i'(h^{1}_i)^{-1}&\eqcref{eq:thin:2}h_i^1\psi_i\psi_i'^{-1}(h_i^0)^{-1}\eta^1_ih_i^0\psi_i'(h_i^1)^{-1}
\\&\eqcref{eq:thin:1}h_i^1\psi_i\alpha(k_i^1,(h_i^0)^{-1}\eta^1_ih_i^0)(h_i^1)^{-1\nonumber}
\\&\eqcref{lem:thi}\alpha(g_i^1(1)^{-1},\psi_i)\alpha(g_i^1(1)^{-1}k_i^1g_i^0(1),\eta^1_i)\nonumber
\\&\eqcref{eq:thin:5}\psi_{i+1}'
\label{eq:thin:4}
\end{align}
In order to show the statement of the proposition, we have to prove that the elements
\begin{equation*}
\rho_n^{\prime 0} \ast \gamma_n^{\prime 0} \ast ... \ast \gamma_1^{\prime 0} \ast \rho_0^{\prime 0} 
\quand
\rho_n^{\prime 1} \ast \gamma_n^{\prime 1} \ast ... \ast \gamma_1^{\prime 1} \ast \rho_0^{\prime 1} 
\end{equation*}
are equivalent, where $\rho_i^{\prime r} = R(\rho_i^{r}(1),((h^{r}_i)^{-1},g_i^{r}(1)^{-1}))$ according to the definition of the target of a horizontal lift. We consider the paths $\tilde\rho_i:[t_{i-1},t_i] \to \mor{\inf P}$ defined by $\tilde\rho_i(t) := R(\id_{\gamma_i'^0(t)},(\psi_i',1))$. They are horizontal by \cref{lem:hormor:a,lem:hormor:h}, and have $s(\tilde\rho_i)=\gamma_i'^0$ and $t(\tilde\rho_i)=\gamma_i^{\prime 1}$ using \cref{eq:thin:1}.
\begin{comment}
Indeed,
\begin{equation*}
t(\tilde\rho_i)=R(\gamma_i^{\prime 0},t(\psi_i'))\eqcref{eq:thin:1}R(\gamma_i^{\prime 0},(k_i^1)^{-1})=\gamma_i^{\prime 1}\text{.}
\end{equation*} 
\end{comment}
It remains to check that the paths $\tilde\rho_i$  convey the required equivalence: \begin{align*}
\rho_{i}'^1\circ \tilde\rho_{i}(t_{i})&\eqcref{eq:thin:3} R(R(\rho_i^0(1),(\eta_i^1,g_i^0(1)^{-1}(k_{i}^{1})^{-1}g_i^1(1))),((h^{1}_i)^{-1},g_i^{1}(1)^{-1}))\circ R(\id_{\gamma_i'^0(t_i)},(\psi_i',1))
\\&=R(\rho_i^0(1),(\eta_i^1\alpha(g_i^0(1)^{-1}(k_{i}^{1})^{-1},(h^{1}_i)^{-1})\alpha(g_i^0(1)^{-1},\psi_i'),g_i^0(1)^{-1}))
\\&\eqcref{eq:thin:1,lem:thi}R(\rho_i^0(1),(\eta_i^1\alpha(t(h_i^0)t(\psi_i'),(h^{1}_i)^{-1})\alpha(t(h_i^0),\psi_i'),g_i^0(1)^{-1}))
\\&= R(\rho_i^0(1),(\eta_i^1h_i^0\psi_i'(h^{1}_i)^{-1}(h_i^0)^{-1},g_i^0(1)^{-1}))
\\&\eqcref{eq:thin:4}R(\rho_i^{0}(1),(\psi_{i+1}'(h^{0}_i)^{-1},g_i^{0}(1)^{-1}))
\\&=\tilde\rho_{i+1}(t_{i})\circ \rho_{i}^{\prime 0}
\\[-3.8em]
\end{align*}
\begin{comment}
In full detail, this is
\begin{align*}
\rho_{i}'^1\circ \tilde\rho_{i}(t_{i})&=  R(\rho_i^{1}(1),((h^{1}_i)^{-1},g_i^{1}(1)^{-1}))\circ R(\id_{\gamma_i'^0(t_i)},(\psi_i',1))
\\&\eqcref{eq:thin:3} R(R(\rho_i^0(1),(\eta_i^1,g_i^0(1)^{-1}(k_{i}^{1})^{-1}g_i^1(1))),((h^{1}_i)^{-1},g_i^{1}(1)^{-1}))\circ R(\id_{\gamma_i'^0(t_i)},(\psi_i',1))
\\&=R(\rho_i^0(1),(\eta_i^1\alpha(g_i^0(1)^{-1}(k_{i}^{1})^{-1},(h^{1}_i)^{-1}),g_i^0(1)^{-1}(k_{i}^{1})^{-1}))\circ R(\id_{R(\gamma_i'^0(t_i),g_i^0(1))},(1,g_i^0(1)^{-1})(\psi_i',1))
\\&=R(\rho_i^0(1),(\eta_i^1\alpha(g_i^0(1)^{-1}(k_{i}^{1})^{-1},(h^{1}_i)^{-1}),g_i^0(1)^{-1}(k_{i}^{1})^{-1}))\circ R(\id_{R(\gamma_i'^0(t_i),g_i^0(1))},(\alpha(g_i^0(1)^{-1},\psi_i'),g_i^0(1)^{-1}))
\\&=R(\rho_i^0(1),(\eta_i^1\alpha(g_i^0(1)^{-1}(k_{i}^{1})^{-1},(h^{1}_i)^{-1})\alpha(g_i^0(1)^{-1},\psi_i'),g_i^0(1)^{-1}))
\\&\eqcref{eq:thin:1,lem:thi}R(\rho_i^0(1),(\eta_i^1\alpha(t(h_i^0)t(\psi_i'),(h^{1}_i)^{-1})\alpha(t(h_i^0),\psi_i'),g_i^0(1)^{-1}))
\\&= R(\rho_i^0(1),(\eta_i^1h_i^0\psi_i'(h^{1}_i)^{-1}(h_i^0)^{-1},g_i^0(1)^{-1}))
\\&\eqcref{eq:thin:4}R(\rho_i^{0}(1),(\psi_{i+1}'(h^{0}_i)^{-1},g_i^{0}(1)^{-1}))
\\&=R(\id_{\gamma_{i+1}'^0(t_i)},(\psi_{i+1}',1))\circ R(\rho_i^{0}(1),((h^{0}_i)^{-1},g_i^{0}(1)^{-1}))
\\&=\tilde\rho_{i+1}(t_{i})\circ \rho_{i}^{\prime 0}
\end{align*}
\end{comment}
\end{proof}

Next we come to the paths, where the situation is more complicated. First of all, two paths $\gamma,\gamma':x \to y$ are called \emph{thin homotopic}, if there exists a bigon $ \Sigma:\gamma \Rightarrow\ \gamma'$ of rank less than two. The complications arise because the anafunctors $F_{\gamma}$ and $F_{\gamma'}$ associated to thin homotopic paths are not equal. The following proposition shows that they are canonically 2-isomorphic, which is the best that we can expect. The 2-isomorphism is the $\Gamma$-equivariant transformation $\varphi_{\Sigma}$ associated to a thin homotopy $\Sigma:\gamma\Rightarrow \gamma'$. The important point is that this does not depend on the choice of the bigon. 

\begin{proposition}
\label{lem:canidthin}
Let $\inf P$ be a principal $\Gamma$-2-bundle with fake-flat connection. Then, the parallel transport along a thin bigon is independent of the bigon, i.e., 
if $\Sigma,\Sigma':\gamma_1 \Rightarrow \gamma_2$ are  bigons of rank less than two, then $\varphi_{\Sigma}=\varphi_{\Sigma'}$. 
\end{proposition}

\begin{proof}
We prove the equivalent statement that a thin homotopy $\Sigma: \gamma \Rightarrow \gamma$ induces the identity $\varphi_{\Sigma}=\id_{F_{\gamma}}$.  \cref{lem:varphisigma:d} allows us to change $\Sigma$ within its thin homotopy class, so that we can assume that $\Sigma$ has the following properties:
\begin{enumerate}[(a)]

\item 
There exists $\varepsilon>0$ such that $\Sigma(s,t)=x$  for $0\leq t <\varepsilon$ and $\Sigma(s,t)=y$ for  $1-\varepsilon<t\leq 1$.

\item 
There exists $\varepsilon>0$ such that $\Sigma(s,t)=\gamma(t)$ for all $0\leq s <\varepsilon$ and all $1-\varepsilon<s\leq 1$.  

\end{enumerate}
We define a smooth map $f:S^2 \to M$ by $f(\vartheta,\varphi) := \Sigma(\frac{\varphi}{2\pi},\frac{\vartheta}{\pi})$, where $0\leq \vartheta\leq \pi$ and $0\leq \varphi \leq 2\pi$ are spherical coordinates. This is well-defined due to (a) and  smooth due to (b). Obviously $f$ has rank one. By \refactiontwobundlereduction\ there exists a  principal $G$-bundle $P$ over $S^2$ with flat connection $\omega$, together with a 1-morphism $J: \act PH \to  f^{*}\inf P$ equipped with a fake-flat, connective, connection-preserving pullback. We define the map $\Xi:[0,1] \to S^2$ by $\Xi(s,t) := (2\pi t,\pi s)$. This is a bigon $\Xi:\mu \Rightarrow \mu$, where $\mu$ is the $\varphi=0$ meridian passing from north pole $N$ ($\vartheta=0$) to south pole ($\vartheta=1$). We have $\gamma=f \circ \mu$ and $\Sigma = f \circ \Xi$. Combining \cref{lem:natbundlemorph,lem:natpullback} we obtain the following the \quot{tin can} equation between transformations:
 \begin{equation*}
\alxydim{@=1.6cm}{\act{P_N}H \ar@/_1.2pc/[r]_{F''_{\mu}}="2" \ar@/^1.2pc/[r]^{F''_{\mu}}="1" \ar@{=>}"1";"2"|*+{\varphi''_{\Xi}}         \ar[d]_{J_N} & \act{P_S}H \ar@{=>}[dl]|>>>>>>>>*{J_{\mu}} \ar[d]^{J_S} \\ f^{*}\inf P_{N} \ar[d]_{\tilde f_{N}}\ar@/_1.2pc/[r]^{F'_{\mu}} & f^{*}\inf P_{S}\ar@{=>}[dl]|*+{\tilde f_{\mu}} \ar[d]^{\tilde f_{S}}\\ \inf P_x \ar@/_1.2pc/[r]_{F_{\gamma}} & \inf P_y}
=
\alxydim{@=1.6cm}{\act{P_N}H  \ar@/^1.2pc/[r]^{F''_{\mu}}  \ar[d]_{J_{N}} & \act{P_S}H \ar@{=>}[dl]|*{J_{\mu}} \ar[d]^{J_{S}} \\ f^{*}\inf P_{N} \ar@/^1.2pc/[r]_{F'_{\mu}}\ar[d]_{\tilde f_{N}} & f^{*}\inf P_{S}\ar@{=>}[dl]|<<<<<<<<<*+{\tilde f_{\mu}} \ar[d]^{\tilde f_{S}} \\ \inf P_x \ar@/_1.2pc/[r]_{F_{\gamma}}="2" \ar@/^1.2pc/[r]^{F_{\gamma}}="1" \ar@{=>}"1";"2"|*+{\varphi_{\Sigma}} & \inf P_y}
\end{equation*}
Here, $F'_{\mu}$ denotes the parallel transport along $\mu$ in $f^{*}\inf P$, and $F''_{\mu}$ denotes the parallel transport in $\act PH$.
From \cref{ex:gbunred:bigons} we conclude that $\varphi''_{\Xi}=\id_{F_{\mu}''}$. Thus,  $\varphi_{\Sigma}=\id_{F_{\gamma}}$.
\end{proof}

Let $[\gamma]$ be a thin homotopy class of paths. We define the set
\begin{equation*}
F_{[\gamma]} := \left (\bigsqcup_{\gamma\in [\gamma]} F_{\gamma} \right )/\sim\text{,}
\end{equation*}
where $\xi\in F_{\gamma}$ and $\xi'\in F_{\gamma'}$ are defined to be equivalent if $\xi'=\varphi_{\Sigma}(\xi)$ for some (and hence by \cref{lem:canidthin} all) thin homotopy $\Sigma:\gamma \Rightarrow \gamma'$. By \cref{lem:varphisigma:a,lem:varphisigma:b,lem:varphisigma:c} it is clear that anchors and actions are well-defined on the set $F_{[\gamma]}$.

\begin{lemma}
\label{prop:quotanathin}
There exists a unique smooth manifold structure on $F_{[\gamma]}$ such that $F_{[\gamma]}$ is a $\Gamma$-equivariant anafunctor, the projections $i_{\gamma}: F_{\gamma} \to F_{[\gamma]}$ are $\Gamma$-equivariant transformations, and the diagram
\begin{equation*}
\alxydim{@=\xyst}{F_{\gamma} \ar[rr]^{\varphi_{\Sigma}} \ar[dr]_{i_{\gamma}} && F_{\gamma'} \ar[dl]^{i_{\gamma'}} \\ & F_{[\gamma]} }
\end{equation*}
is commutative for all thin bigons $\Sigma:\gamma \Rightarrow \gamma'$
\end{lemma}

\begin{proof}
We consider the open cover $\{U_{\gamma}\}_{\gamma\in [\gamma]}$ of $\inf P_x$, where  $U_{\gamma}:= \ob{\inf P_x}$.  Over each open set $U_{\gamma}$ we have the principal $\inf P_y$-bundle $F_{\gamma}$. Over each double overlap ($U_{\gamma}\cap U_{\gamma'}=\ob{\inf P_x}$) we have a bundle isomorphism $\varphi_{\Sigma}: F_{\gamma} \to F_{\gamma'}$, for some choice of a thin homotopy $\Sigma$. Over each triple overlap, these satisfy the cocycle condition due to \cref{lem:canidthin}. Our definition of  $F_{[\gamma]}$ realizes the descend construction for the stack of principal $\inf P_y$-bundles over $\inf P_x$; hence $F_{[\gamma]}$ is a principal $\inf P_y$-bundle. Now, the remaining statements follow.
\end{proof}

Now we are in position to re-define parallel transport in a setting of thin homotopy classes of paths and bigons. To a thin homotopy class $[\gamma]$ of paths between $x$ and $y$ we associate the $\Gamma$-equivariant anafunctor
\begin{equation*}
F_{[\gamma]}: \inf P_x \to  \inf P_y\text{.}
\end{equation*}
Two bigons  $\Sigma:\gamma_0\Rightarrow\gamma_1$ and  $\Sigma':\gamma_0'\Rightarrow \gamma_1'$ will now be called thin homotopic, if there exists a homotopy $h$ between them of rank less than three, that fixes the endpoints $x$ and $y$, and restricts to homotopies $h_0:\gamma_0 \Rightarrow \gamma_0'$ and $h_1:\gamma_1\Rightarrow \gamma_1'$ of rank less than 2. This generalizes the relation introduced at the beginning of this section in that the bounding paths do not have to be equal but can be thin homotopic themselves. 
\begin{comment}
More precisely, this is
a homotopy $h:[0,1] \times [0,1]^2 \to M$ between $\Sigma$ and $\Sigma'$ with the following properties:
\begin{enumerate}[(a)]

\item 
$h(r,s,0)=x$ and $h(r,s,1)=y$ for all $r,s\in[0,1]$.

\item
$h_i:[0,1] \times [0,1] \to M$ defined by $h_i(r,t):= h(r,i,t)$ satisfies $\mathrm{rk}(\mathrm{d}h_i)\leq 1$

\item
$\mathrm{rk}(\mathrm{d}h)\leq 2$

\end{enumerate}
Conditions (a) and (b) imply that $h$ restricts to thin bigons $h_i:\gamma_i\Rightarrow \gamma_i'$. Condition (c) implies that the bigons $h_1^{-1}\bullet \Sigma' \bullet h_0$ and $\Sigma$ are thin homotopic. 
\end{comment}
For a thin homotopy class $[\Sigma]:[\gamma_0] \Rightarrow [\gamma_1]$ of bigons we define
\begin{equation*}
\varphi_{[\Sigma]} = i_{\gamma_1}\circ\varphi_{\Sigma}\circ  i_{\gamma_0}^{-1}\text{.}
\end{equation*}
It is straightforward to check using \cref{lem:varphisigma:d,prop:quotanathin} that this definition is independent of the choice of the representative $\Sigma$. \begin{comment}
Indeed,
\begin{align*}
\varphi_{[\Sigma]} &= i_{\gamma_1}\circ\varphi_{\Sigma}\circ  i_{\gamma_0}^{-1} 
 \\&\eqcref{lem:varphisigma:d}i_{\gamma_1}\circ\varphi_{h_1^{-1}\bullet \Sigma' \bullet h_0}\circ  i_{\gamma_0}^{-1}
\\&=i_{\gamma_1}\circ\varphi_{h_1^{-1}}\circ \varphi_{\Sigma'}\circ \varphi_{h_0}\circ  i_{\gamma_0}^{-1}
 \\&\eqcref{prop:quotanathin} i_{\gamma_1'}\circ \varphi_{\Sigma'}\circ  i_{\gamma_0'}^{-1}
\\&= \varphi_{[\Sigma']}\text{.}
\end{align*}
\end{comment}
Similarly, the transformations $u_x$ and $c_{\gamma_1,\gamma_2}$ of \cref{sec:compcomppaths} induce well-defined transformations
\begin{equation*}
u_x: F_{[\id_x]} \Rightarrow \id_{\inf P_x}
\quand
c_{[\gamma_1],[\gamma_2]}: F_{[\gamma_2]} \circ F_{[\gamma_1]} \Rightarrow F_{[\gamma_2]\circ [\gamma_1]}\text{.}
\end{equation*}
\begin{comment}
Indeed, the first is just $u_x \circ i_{\id_x}^{-1}$, and the second is defined by
\begin{equation*}
c_{[\gamma_1],[\gamma_2]} := i_{\gamma_2\circ\gamma_1} \bullet  c_{\gamma_1,\gamma_2} \bullet (i_{\gamma_2}^{-1} \circ i_{\gamma_1}^{-1})\text{.}
\end{equation*}
We have to verify that this is well-defined under going to thin homotopy classes of paths. If $\Sigma_1:\gamma_1\Rightarrow \gamma_1'$ and $\Sigma_2:\gamma_2\Rightarrow \gamma_2'$ are thin homotopies, then (under a reparameterization) they are horizontally composable, and we have\begin{align*}
c_{[\gamma_1],[\gamma_2]} &= i_{\gamma_2\circ\gamma_1} \bullet  c_{\gamma_1,\gamma_2} \bullet (i_{\gamma_2}^{-1} \circ i_{\gamma_1}^{-1})
\\&\eqcref{lem:F2} i_{\gamma_2\circ\gamma_1} \bullet \varphi^{-1}_{\Sigma_2 \circledast \Sigma_1} \bullet c_{\gamma_1',\gamma_2'}\bullet  (\varphi_{\Sigma_2} \circ \varphi_{\Sigma_1}) \bullet (i_{\gamma_2}^{-1} \circ i_{\gamma_1}^{-1})
\\&\eqcref{prop:quotanathin} i_{\gamma_2'\circ\gamma_1'} \bullet  c_{\gamma_1',\gamma_2'} \bullet (i_{\gamma_2'}^{-1} \circ i_{\gamma_1'}^{-1}) \\ &=c_{[\gamma_1'],[\gamma_2']} \text{.}
\end{align*}
\end{comment}
Finally, suppose $J: \inf P_1 \to \inf P_2$ is a 1-morphism in $\zweibunconff\Gamma M$.  Then we define a transformation
\begin{equation*}
J_{[\gamma]}:= (i_{\gamma} \circ \id_{J_x}) \bullet J_{\gamma}\bullet (\id_{J_y} \circ i^{-1}_{\gamma})  :J_y \circ F_{[\gamma]}  \Rightarrow  F'_{[\gamma]} \circ J_x\text{.}
\end{equation*}
Using \cref{lem:natbundlemorph} one can check that this definition is independent of the choice of the representative $\gamma$.
\begin{comment}
Indeed, if $\Sigma:\gamma_1\Rightarrow \gamma_2$ is thin, then we have
\begin{align*}
J_{[\gamma_1]} &= (i_{\gamma_1} \circ \id_{J_x}) \bullet J_{\gamma_1}\bullet (\id_{J_y} \circ i^{-1}_{\gamma_1})
\\&\eqcref{lem:natbundlemorph}(i_{\gamma_1} \circ \id_{J_x}) \bullet (\varphi'_{\Sigma} \circ \id)^{-1} \bullet  J_{\gamma_2} \bullet (\id \circ \varphi_{\Sigma})  \bullet (\id_{J_y} \circ i^{-1}_{\gamma_1})
\\&=(i_{\gamma} \circ \id_{J_x}) \bullet J_{\gamma}\bullet (\id_{J_y} \circ i^{-1}_{\gamma})
\\&= J_{[\gamma_2]}\text{.}
\end{align*} 
\end{comment} 

Summarizing, all our definitions of \cref{sec:ptpaths,sec:ptbigons} persist under the passage to thin homotopy classes. In the following section we will see the main advantage of this passage, namely that it allows an organization in  bicategories.   

\subsection{Organization in bicategories}

\label{sec:orga}

The path 2-groupoid of $M$ is the 2-groupoid $\mathcal{P}_2(M)$ whose objects are the points of $M$, 1-morphisms  are thin homotopy classes of  paths in $M$, and 2-morphisms  are  thin homotopy classes of bigons in $M$. A  detailed definition is in \cite[Section 2.1]{schreiber5}.  
In this subsection we assemble the parallel transports of the previous subsection into a 2-functor $\mathrm{tra}_{\inf P}: \mathcal{P}_2(M) \to \tor\Gamma$, where $\tor\Gamma$ is the bicategory of $\Gamma$-torsors, see \cref{sec:gammator}.
For the terminology of bicategories we refer to \cite[Appendix A]{schreiber6}.

\begin{proposition}
\label{prop:fun}
Let $\inf P$ be a principal $\Gamma$-2-bundle $\inf P$ with fake-flat connection $\Omega$. Then, the as\-sign\-ments
$x \mapsto \inf P_x$, $[\gamma]\mapsto F_{[\gamma]}$, and $[\Sigma] \mapsto \varphi_{[\Sigma]}$ 
form a 2-functor 
\begin{equation*}
\mathrm{tra}_{\inf P}:\mathcal{P}_2(M) \to \tor\Gamma
\end{equation*}
with unitors $u_x$ and compositors $c_{[\gamma_1],[\gamma_2]}$.
\end{proposition}

\begin{proof}
There are four axioms to check, see, e.g. \cite[Def. A.5]{schreiber6}. Axiom (F1) is functoriality with respect to vertical composition; this is \cref{lem:F1}. Axiom (F2) is the compatibility with the horizontal composition; this is \cref{lem:F2}. Axioms (F3) and  (F4) concern the coherence of compositors and unitors; these are \cref{lem:F3,lem:F4}.
\end{proof}

\begin{example}
\label{ex:trivbuntra}
Let $\inf I_{A,B}$ be the trivial principal $\Gamma$-2-bundle over $M$ whose connection is induced from a fake-flat $\Gamma$-connection $(A,B)$. We obtain from $(A,B)$ the smooth 2-functor $F_{A,B}: \mathcal{P}_2(M) \to B\Gamma$ (see \cref{sec:smoothfunctors}). Its composition $i(F_{A,B})$ with the 2-functor $i: B\Gamma \to \tor\Gamma$  of \cref{rem:Lresi} is a \quot{trivial} transport 2-functor only depending on the $\Gamma$-connection $(A,B)$. On the other hand, we have the 2-functor $\mathrm{tra}_{\mathcal{I}_{A,B}}$ of \cref{prop:fun}. The two 2-functors are equivalent via a pseudonatural transformation
\begin{equation*}
\eta_{A,B}: i (F_{A,B}) \to \mathrm{tra}_{\inf I_{A,B}}\text{,}
\end{equation*}
whose components are the assignments $x\mapsto\ \id_{\Gamma}$ and $\gamma \mapsto \eta_{\gamma}$ of \cref{ex:ptpath:trivbun}. There are two axioms to check \cite[Def. A.6]{schreiber6}; these are precisely \cref{eq:T1:eta,eq:T2:eta}. This \quot{computes} the parallel transport 2-functor of a connection on the trivial principal $\Gamma$-2-bundle. 
\end{example}

\begin{proposition}
\label{prop:pt}
Let $J: \inf P_1 \to \inf P_2$ be a 1-morphism in $\zweibunconff\Gamma M$. Then, the assignments $x \mapsto J_x$ and $[\gamma] \mapsto J_{[\gamma]}$ 
form a pseudonatural transformation
\begin{equation*}
\rho_J: \mathrm{tra}_{\inf P_1} \to \mathrm{tra}_{\inf P_2}\text{.}
\end{equation*} 
\end{proposition}

\begin{proof}
There are two axioms to check, see \cite[Def. A.6]{schreiber6}. Axiom (T1) is the compatibility with path composition; this is \cref{lem:T1}. Axiom (T2) is naturality with respect to 2-morphisms; this is \cref{lem:natbundlemorph}.
\end{proof}

\begin{example}
\label{ex:trivbunpt}
Suppose fake-flat $\Gamma$-connections $(A,B)$ and $(A',B')$ are related by a gauge transformation $(g,\varphi)$. On one side, we have a smooth pseudonatural transformation
\begin{equation*}
\rho_{g,\varphi}:F_{A,B} \to F_{A',B'}\text{,}
\end{equation*}
see \cref{sec:smoothfunctors}. On the other side, we have a 1-morphism $J=J_{g,\varphi}:\inf I_{A,B} \to \inf I_{A',B'}$ in $\zweibunconff\Gamma M$ (\cref{rem:trivbun:b}), to which  \cref{prop:pt} associates  a pseudonatural transformation $\rho_J: \mathrm{tra}_{\inf I_{A,B}} \to \mathrm{tra}_{\inf I_{A',B'}}$. 
We find a commutative diagram
\begin{equation*}
\alxydim{@=\xyst}{i(F_{A,B})   \ar[d]_{i ( \rho_{g,\varphi})} \ar[r]^-{\eta_{A,B}} & \mathrm{tra}_{\inf I_{A,B}} \ar[d]^{\rho_J} \\  i( F_{A',B'}) \ar[r]_-{\eta_{A',B'}} & \mathrm{tra}_{\inf I_{A',B'}}}
\end{equation*}
of pseudonatural transformations, where $\eta_{A,B}$ and $\eta_{A',B'}$ are the pseudonatural transformations of \cref{ex:trivbuntra}. 
Commutativity
means that clockwise and counter-clockwise compositions have the same assignments to points and paths.
Coincidence for points follows from the definition of  $J_{g,\varphi}$, which has an underlying smooth functor $\phi_g$. Coincidence for paths follows from \cref{ex:gt}.
\begin{comment}
This means that clockwise and counter-clockwise compositions have the same assignments to points and paths. To points, they associated 1-morphisms in $\tor\Gamma$ between $\mathrm{tra}_{\inf I_{\Omega}}(x)=\Gamma$ and $(i \circ F_{A',B'})(x)=\Gamma$. Note that $\eta(x)=\id$. We compute clock-wise. At a point $x$ we have
\begin{equation*}
\rho_J(x) \circ \eta(x)=J_x=\phi_g|_x=i_{g(x)}
\end{equation*}
We compute counter-clockwise. At a point $x$ we have
\begin{equation*}
\eta'_x\circ i_{\rho_{g,\varphi}(x)}=i_{g(x)}\text{.}
\end{equation*}
Thus we have coincidence. To path a path $\gamma:x \to y$, they assign  2-morphisms in $\tor\Gamma$ between $i(F_{A,B}(\gamma)) \circ i_{g(y)}$ and  $\mathrm{tra}_{\inf I_{\Omega'}}(\gamma) \circ i_{g(x)}$. Clockwise, this is
\begin{equation*}
\rho_J(\gamma) \bullet (\id_{J_y} \circ \eta_{\gamma})=J_{\gamma} \bullet (\id_{J_y} \circ \eta_{\gamma})
\text{.}
\end{equation*}
Counter-clockwise, it is
\begin{equation*}
(\id_{i_{g(x)}} \circ \eta'_{\gamma}) \bullet i_{\rho_{g,\varphi}(\gamma)}\text{.}
\end{equation*}
Coincidence is exactly \cref{ex:gt}.
\end{comment}
This \quot{computes} the transformation between parallel transport 2-functors of trivial bundles with gauge equivalent connections. \end{example}

\begin{proposition}
\label{prop:mod}
Let $f: J \Rightarrow J'$ be a 2-morphism in $\zweibunconff\Gamma M$. Then, the assignment
$x \mapsto f_x$
forms a modification 
\begin{equation*}
\mathcal{A}_{f}: \rho_{J} \Rightarrow \rho_{J'}\text{.}
\end{equation*}
\end{proposition}

\begin{proof}
There is only one axiom (\cite[Definition A.8]{schreiber6}); proved by \cref{lem:modification}.
\end{proof}

\begin{example}
\label{ex:trivbunmod}
Suppose we have two fake-flat $\Gamma$-connections $(A,B)$ and $(A',B')$, two gauge transformations $(g_1,\varphi_1)$ and $(g_2,\varphi_2)$, and a gauge 2-transformation $a$ between $(g_1,\varphi_1)$ and $(g_2,\varphi_2)$. Then we have a smooth modification
\begin{equation*}
\mathcal{A}_a: \rho_{g_1,\varphi_1}\Rightarrow \rho_{g_2,\varphi_2}\text{,}
\end{equation*}
see \cref{sec:smoothfunctors}. On the other side, we have a 2-morphism $f_a: J_{g_1,\varphi_1} \Rightarrow J_{g_2,\varphi_2}$ between the 1-morphisms associated to the gauge transformations, see \cref{rem:trivbun:c}. In turn, we obtain a modification $\mathcal{A}_{f_a}: \rho_{J_{g_1,\varphi_1}}\Rightarrow \rho_{J_{g_2,\varphi_2}}$.
Then we have a commutative diagram
\begin{equation*}
\alxydim{@=\xyst}{i(\rho_{g_{1},\varphi_{1}})  \ar@{=>}[d]_{i(\mathcal{A}_{a})} \ar@{=}[r] &  \eta_k^{-1} \circ \rho_{J_{g_1,\varphi_1} } \circ \eta_i \ar@{=>}[d]^{\id \circ \mathcal{A}_{f_a} \circ \id} \\ i(\rho_{g_{2},\varphi_{2}}) \ar@{=}[r] &  \eta_k^{-1} \circ \rho_{J_{g_2,\varphi_2}} \circ \eta_i}
\end{equation*}
of modifications between pseudonatural transformations between 2-functors from $\mathcal{P}_2(M)$ to $\tor\Gamma$. 
\begin{comment}
The equalities are due to \cref{ex:trivbunpt}.  
\end{comment}
Indeed, evaluating at a point $x$ gives a 2-morphism in $\tor\Gamma$, in fact between 1-morphisms that are smooth functors (not anafunctors). Thus, the diagram is, for each $x\in M$, an equality between natural transformations between functors from $\Gamma$ to $\Gamma$. We compare its components at an object $g$, only using the given definitions: \begin{equation*}
\mathcal{A}_{f_{a}}(x)(g)=f_{a}|_x(g)=f_{a}(x,g)
=(\id_x,(a(x),g_{1}(x)g)) = i_{a(x),g_{1}(x)}(g)=i(\mathcal{A}_{a})(g)\text{.}
\end{equation*}
This shows commutativity.
\end{example}

\begin{theorem}
\label{th:2funct}
\cref{prop:fun,prop:pt,prop:mod} furnish a (strict) 2-functor
\begin{equation*}
\mathrm{tra}:\zweibunconff\Gamma M \to \fun(\mathcal{P}_2(M),\tor\Gamma)\text{.}
\end{equation*}
\end{theorem}

\begin{proof}
 That the composition of 1-morphisms is respected is the content of \cref{lem:morphcompcomp}. On the level of 2-morphisms, the 2-functor is just restriction to points (see \cref{prop:mod}); this clearly preserves horizontal and vertical composition. 
\end{proof}

\begin{remark}
\cref{ex:trivbuntra,ex:trivbunpt,ex:trivbunmod} can be interpreted as follows. The constructions of \cref{rem:trivbun} relating $\Gamma$-connections to trivial 2-bundles form a 2-functor
\begin{equation*}
L^{\ff}:\conff\Gamma M \to \zweibunconff\Gamma M
\end{equation*}
relating $\Gamma$-connections to connections on trivial $\Gamma$-2-bundles.
The constructions of \cref{sec:smoothfunctors} form another 2-functor
\begin{equation*}
\mathcal{P}: \conff\Gamma M \to \fun(\mathcal{P}_2(M),B\Gamma)
\end{equation*}
relating $\Gamma$-connections to 2-functors on the path 2-groupoid. We have a pseudonatural equivalence \begin{equation*}
\mathrm{tra} \circ L^{\ff} \cong i\circ \mathcal{P}\text{,}
\end{equation*}
established by assigning $(A,B)\mapsto \eta_{A,B}$ and $(g,\varphi)\mapsto \id$.
It expresses the fact that trivial principal 2-bundles have  trivial (more precisely: canonically trivializable) parallel transport 2-functors.
\end{remark}

\subsection{The transport 2-functor formalism}

\label{sec:main}

The  transport 2-functor formalism \cite{schreiber2} axiomatically specifies a sub-bicategory 
\begin{equation*}
\mathrm{Trans}_{\Gamma}(M,\tor\Gamma) \subset \fun(\mathcal{P}_2(M),\tor\Gamma)
\end{equation*}
of 2-functors, pseudonatural transformations, and modifications that are supposed to implement higher-dimensional parallel transport. Essentially, the axioms require that a transport 2-functor can locally be described by path-ordered and surface-ordered exponentials of $\Gamma$-connections.
We will give more details in the proof of the following result.

\begin{theorem}
\label{th:pt2fun}
The image of the 2-functor $\mathrm{tra}$
of \cref{th:2funct} is contained in the sub-bicategory $\mathrm{Trans}_{\Gamma}(M,\tor\Gamma)$, and hence induces a 2-functor 
\begin{equation*}
\mathrm{tra}:\zweibunconff\Gamma M \to \mathrm{Trans}_{\Gamma}(M,\tor\Gamma)\text{.}
\end{equation*}
In other words, parallel transport in principal $\Gamma$-2-bundles fits into the axiomatic framework for higher-dimensional parallel transport. 
\end{theorem}

One nice consequence of \cref{th:pt2fun} is the following general result about transport 2-functors, see \cite[Proposition 3.3.6]{schreiber2}.

\begin{corollary}
\label{co:flat}
If $\Omega$ is a  flat connection on a principal $\Gamma$-2-bundle, then the parallel transport along bigons only depends on the homotopy class of the bigon.
\end{corollary}

As a further consequence, the discussion of surface holonomy given in \cite[Section 5]{schreiber2} applies to principal $\Gamma$-2-bundles. 
In the remainder of this subsection we prove \cref{th:pt2fun}, split into \cref{prop:transfun,prop:transpseudo,prop:transmodi}.

\begin{proposition}
\label{prop:transfun}
If $\inf P$ is a principal $\Gamma$-2-bundle over $M$ with fake-flat connection, then the 2-functor $\mathrm{tra}_{\inf P}$ of \cref{prop:fun} is a  transport 2-functor with $B\Gamma$-structure.
\end{proposition}

\begin{remark}
For the proof we extract and slightly reformulate the following  results of \reffunctorLproperties.
Let $\inf P$ be a principal $\Gamma$-2-bundle over $M$ with fake-flat connection. 
\begin{enumerate}[(a)]

\item 
\label{re:examples:obj:a}
Every point $x\in M$ has an open neighborhood $x\in U \subset M$ together with a $\Gamma$-connection $(A,B)$ on $U$ and a 1-morphism $T: \inf I_{A,B} \to \inf P$ in $\zweibunconff\Gamma M$.

\item
\label{re:examples:obj:b}
Suppose $U\subset M$ is a contractible open set. If $(A,B)$ and $(A',B')$ are $\Gamma$-connections on $U$, and $T: \inf I_{A,B} \to \inf P|_U$ and $T': \inf I_{A',B'} \to \inf P|_U$ are 1-morphisms in $\zweibunconff\Gamma U$, then there exists a gauge transformation $(g,\varphi):(A,B) \to (A',B')$ and a 2-morphism
$\tilde\sigma:T'^{-1}\circ T \Rightarrow J_{g,\varphi}$
in $\zweibunconff\Gamma M$.
\begin{comment}
We use that the Hom-Functor
\begin{equation*}
\hom_{\conff\Gamma U}((A,B),(A',B')) \to \hom_{\zweibunconff \Gamma U}(\inf I_{A,B},\inf I_{A',B'})
\end{equation*}
is essentially surjective, applied to the object $T'^{-1} \circ T:\inf I_{A,B} \to \inf I_{A',B'}$. 
\end{comment}

\item
\label{re:examples:obj:c}
Suppose $U\subset M$ is  an open set, and $I$ is some (index) set. If, for each $i\in I$,  $(A_i,B_i)$ are $\Gamma$-connections on $U$, $T_i:\inf I_{A_i,B_i} \to \inf P|_U$ are 1-morphisms in  $\zweibunconff\Gamma U$, $(g_{ij},\varphi_{ij}):(A_i,B_i) \to (A_j,B_j)$ are gauge transformations, and $\sigma_{ij}:T_j^{-1}\circ T_i \Rightarrow J_{g_{ij},\varphi_{ij}}$ are 2-morphisms in   $\zweibunconff\Gamma U$, then there exists a unique gauge 2-transformation
\begin{equation*} 
a_{ijk}: (g_{jk},\varphi_{jk}) \circ (g_{ij},\varphi_{ij}) \Rightarrow (g_{ik},\varphi_{ik})\text{,}
\end{equation*}
and a commutative diagram
\begin{equation}
\label{eq:re:examples:obj:c}
\alxydim{@R=\xyst@C=4em}{ T_k^{-1} \circ T_j \circ T_j^{-1} \circ T_i \ar@{=>}[r]^-{\sigma_{jk} \circ \sigma_{ij}} \ar@{=>}[d]_-{\id \circ d_{T_j} \circ \id} & J_{g_{jk},\varphi_{jk}} \circ J_{g_{ij},\varphi_{ij}}  \ar@{=>}[d]^-{f_{a_{ijk}}}  \\  T_k^{-1} \circ T_i\ar@{=>}[r]_{\sigma_{ik}} & J_{g_{ik},\varphi_{ik}}}
\end{equation}
where $d_{F}: J \circ J^{-1} \Rightarrow \id$ stands for the canonical \quot{death} transformation expressing the invertibility of an anafunctor $J$. 
\begin{comment}
We use that the Hom-functor
\begin{equation*}
\hom_{\conff\Gamma U}((A_1,B_1),(A_3,B_3)) \to \hom_{\zweibunconff \Gamma U}(\inf I_{A_1,B_1},\inf I_{A_3,B_3})
\end{equation*}
is full and faithful, applied to the 2-morphism
\begin{equation*}
\alxydim{@C=1.6cm}{J_{g_{jk},\varphi_{jk}} \circ J_{g_{ij},\varphi_{ij}} \ar@{=>}[r]^-{\sigma_{jk}^{-1} \circ \sigma_{ij}^{-1}} & T_k^{-1} \circ T_j \circ T_j^{-1} \circ T_i \ar@{=>}[r]^-{\id \circ d_{T_j} \circ \id} & T_k^{-1} \circ T_i \ar@{=>}[r]^-{\sigma_{ik}} & J_{g_{ik},\varphi_{ik}} }
\end{equation*}
\end{comment}

\end{enumerate}
\end{remark}

\begin{proof}[Proof of \cref{prop:transfun}]
The first step  is to specify local trivializations. 
Consider an open set $U \subset M$ as in \cref{re:examples:obj:a}, a $\Gamma$-connection $(A,B)$ and a 1-morphism $T:\inf I_{A,B} \to \inf P|_{U}$.  By \cref{ex:trivbuntra,prop:pt} we obtain a pseudonatural transformation
\begin{equation}
\label{eq:pseudotau}
\tau:= \rho_{T} \circ \eta : i( F_{A,B} )\to \mathrm{tra}_{\inf P}|_{U}\text{.} 
\end{equation}
We also have to fix a \quot{weak inverse} of $\tau$, and choose $\tau^{-1} :=\eta^{-1}\circ \rho_{T^{-1}}$. Here it is important that $\eta$ is \quot{strictly invertible} because $\eta_x=\id$. Finally, we need to fix
 modifications $\mathcal{D}_{\tau}: \tau \circ \tau^{-1} \to \id_{i(F_{A,B})}$ and $\mathcal{B}_{\tau}: \id_{\mathrm{tra}_{\inf P}|_U} \Rightarrow \tau^{-1}\circ \tau$, which can be induced from the canonical transformations $d_T: T \circ T^{-1} \Rightarrow \id_{\inf I_{\Psi}}$ and $b_T: \inf P|_U \Rightarrow T^{-1} \circ T$ via \cref{prop:mod}.
\begin{comment}
Indeed, $\mathcal{B}_{\tau}$ is defined by
\begin{equation*}
\alxydim{@C=2.5cm}{\id=\eta^{-1} \circ \eta =\eta^{-1} \circ \rho_{\id}\circ \eta^{-1} \ar@{=>}[r]^-{\id_{\eta^{-1}} \circ \mathcal{A}_{b_T} \circ \id_{\eta}} & \eta^{-1} \circ \rho_{T^{-1}\circ T}\circ \eta}\eqtext{Strictness in \cref{th:2funct}} \eta^{-1}\circ \rho_{T^{-1}}\circ \rho_T \circ \eta=\tau^{-1} \circ \tau
\end{equation*}
and $\mathcal{D}_{\tau}$ is defined by
\begin{equation*}
\tau \circ \tau^{-1} =\eta^{-1}\circ \rho_{T}\circ \rho_{T^{-1}} \circ \eta \eqtext{Strictness in \cref{th:2funct}} \eta^{-1} \circ \rho_{T\circ T^{-1}}\circ \eta \alxydim{@C=2.5cm}{ \ar@{=>}[r]^-{\id_{\eta^{-1}} \circ \mathcal{A}_{d_{T}} \circ \id_{\eta}} &\eta^{-1}\circ \rho_{\id}\circ\eta=\eta^{-1}\circ \eta=\id\text{.}} \end{equation*}
\end{comment} 

In the second step we form an open cover $\{U_i\}_{i\in M}$ of $M$ composed of open sets as above, with contractible double  intersections. Over each open set $U_i$ we choose a $\Gamma$-connection $(A_i,B_i)$ and a 1-morphism $T_i: \inf I_{A_i,B_i} \to \inf P|_{U_i}$,  and consider the induced  local trivialization $(\tau_i,\tau_{i}^{-1},\mathcal{D}_i,\mathcal{B}_i)$. Now we have to  extract descent data. The first descent datum are the 2-functors $F_{A_i,B_i}:\mathcal{P}_2(U_i) \to B\Gamma$. The second descent datum are the pseudonatural transformations
\begin{equation*}
\gamma_{ij} :=\tau_j^{-1}\circ \tau_i: i(F_{A_i,B_i}) \to i( F_{A_j,B_j})
\end{equation*}
between 2-functors $\mathcal{P}_2(U_i \cap U_j) \to \tor\Gamma$. The third descent datum consists of the modifications $\mathcal{B}_i : \id_{i(F_{A_i,B_i})} \Rightarrow \gamma_{ii}$  and  $\mathcal{F}_{ijk}: \gamma_{jk} \circ \gamma_{ij} \Rightarrow \gamma_{ik}$ defined by
\begin{equation}
\label{eq:def:f}
\alxydim{@C=2.5cm}{\gamma_{jk} \circ \gamma_{ij} =\tau_k^{-1}\circ \tau_j \circ \tau_j^{-1}\circ \tau_i \ar@{=>}[r]^-{\id \circ \mathcal{D}_j  \circ \id} & \tau_k^{-1}\circ \tau_i= \gamma_{ik}\text{.}} 
\end{equation}

The third step is to show that all this descent data is smooth in a certain sense. For the 2-functors $F_{A_i,B_i}$ this simply means that they have to be smooth, which is the case. For the pseudonatural transformation $\gamma_{ij}$ it suffices to show that it factors through a \emph{smooth} pseudonatural transformation $\tilde \gamma_{ij}:F_{A_i,B_i} \to  F_{A_j,B_j}$, i.e. $\gamma_{ij}\cong i(\tilde \gamma_{ij})$. We construct $\tilde \gamma_{ij}$ as follows. Since $U_i \cap U_j$ is contractible,  there exist  gauge transformations $(g_{ij},\varphi_{ij}):(A_i,B_i) \to (A_j,B_j)$ and  2-morphisms $\sigma_{ij}:T_j^{-1}\circ T_i \Rightarrow  J_{ij}$, where we write $J_{ij} := J_{g_{ij},\varphi_{ij}}$ for short; see \cref{re:examples:obj:b}. We let $\tilde \gamma_{ij} := \rho_{g_{ij},\varphi_{ij}}$ be the smooth pseudonatural transformation (\cref{sec:smoothfunctors}).
 We define a modification $\mathcal{A}_{ij}: \gamma_{ij} \Rightarrow i(\tilde \gamma_{ij})$  as follows:
\begin{equation}
\label{eq:def:alpha}
\alxydim{@C=1.7cm}{\gamma_{ij} =\eta_j^{-1}\circ \rho_{T_j}^{-1} \circ  \rho_{T_i} \circ \eta_i = \eta_j^{-1} \circ \rho_{T_j^{-1}\circ T_i} \circ \eta_i \ar@{=>}[r]^-{\id \circ \mathcal{A}_{\sigma_{ij}} \circ \id} &} \eta_j^{-1}\circ \rho_{J_{ij}} \circ \eta_i \eqcref{ex:trivbunpt} i (\tilde\gamma_{ij})\text{.}
\end{equation}

Finally, we have to verify the smoothness of the modifications $\mathcal{B}_i$ and $\mathcal{F}_{ijk}$. For this we have to show that there exist smooth modifications $\widetilde {\mathcal{B}}_i: \id_{F_{A_i,B_i}} \Rightarrow \tilde \gamma_{ii}$ and $\widetilde {\mathcal{F}}_{ijk}:\tilde \gamma_{jk} \circ \tilde \gamma_{ij} \Rightarrow \tilde \gamma_{ik}$ such that
\begin{equation}
\label{eq:tpf:show}
\mathcal{B}_i=\mathcal{A}_{ii}^{-1} \bullet i(\widetilde {\mathcal{B}}_i)
\quand
\mathcal{F}_{ijk} = \mathcal{A}_{ik}^{-1}\bullet i(\widetilde {\mathcal{F}}_{ijk})\bullet (\mathcal{A}_{jk} \circ \mathcal{A}_{ij})\text{.}
\end{equation}
Without loss of generality we can assume that $g_{ii}=1$, $\varphi_{ii}=0$, and $\sigma_{ii}=b_{T_i}^{-1}$, so that  $\tilde \gamma_{ii}=\id_{F_{A_i,B_i}}$ and $J_{ii}=\id_{\inf I_{A_i,B_i}}$. This shows $\mathcal{B}_i = \mathcal{A}_{ii}^{-1}$, i.e. $\widetilde {\mathcal{B}}_i:=\id$ does the job. On a triple overlap $U_i \cap U_j \cap U_k$, we obtain via \cref{re:examples:obj:c} a gauge 2-transformation
\begin{equation*}
a_{ijk}: (g_{jk},\varphi_{jk}) \circ (g_{ij},\varphi_{ij}) \Rightarrow (g_{ik},\varphi_{ik})\text{.}
\end{equation*} 
We let $\widetilde {\mathcal{F}}_{ijk} := \mathcal{A}_{a_{ijk}}$ be the smooth modification (\cref{sec:smoothfunctors}).
We have a diagram
\begin{equation*}
\alxydim{@C=1cm@R=\xyst}{ \gamma_{jk} \circ \gamma_{ij}\ar@{=}[dr] \ar@{=>}[ddd]_{\mathcal{F}_{ijk}} \ar@{=>}[rrrr]^{\mathcal{A}_{jk} \circ \mathcal{A}_{ij}} &&&& i(\tilde \gamma_{jk}) \circ i(\tilde \gamma_{ij}) \ar@{=}[dl]\ar@{=>}[ddd]^{i(\widetilde {\mathcal{F}}_{ijk})} \\ & \mquad\mquad\eta_k^{-1} \circ \rho_{T_k^{-1}}\circ \rho_{T_j} \circ \rho_{T_j^{-1}} \circ \rho_{T_i} \circ \eta_i \ar@{=>}[d]_{\id \circ \id \circ \mathcal{A}_{d_j} \circ \id \circ \id} \ar@{=>}[rr]^-{  \id \circ \mathcal{A}_{\tilde\sigma_{jk}} \circ \mathcal{A}_{\tilde\sigma_{ij}} \circ \id} &&  \eta_k^{-1} \circ \rho_{J_{jk}} \circ \rho_{J_{ij}} \circ \eta_i\mquad\mquad \ar@{=>}[d]^{\id \circ \mathcal{A}_{\eta_{a_{ijk}}} \circ \id}
\\ & \eta_k^{-1} \circ \rho_{T_k^{-1}}\circ \rho_{T_i} \circ \eta_i \ar@{=}[dl] \ar@{=>}[rr]_{\id \circ \mathcal{A}_{\tilde\sigma_{ik}} \circ \id} && \eta_k^{-1} \circ \rho_{J_{ik}} \circ \eta_i \ar@{=}[dr]
\\      \gamma_{ik} \ar@{=>}[rrrr]_{\mathcal{A}_{ik}} &&&& i(\tilde \gamma_{ik})}
\end{equation*}
The subdiagrams at the top and at the bottom are commutative due to the definition of $\mathcal{A}_{ij}$ (\cref{eq:def:alpha}). The subdiagram on the left  is the definition of $\mathcal{F}_{ijk}$ (\cref{eq:def:f}). The subdiagram on the right is the one of \cref{ex:trivbunmod}. The subdiagram in the middle is induced from  the diagram of \cref{eq:re:examples:obj:c} and hence commutative.
Thus, the whole diagram is commutative; this is the second equation in \cref{eq:tpf:show}.
\end{proof}

\begin{proposition}
\label{prop:transpseudo}
If $J:\inf P \to \inf P'$ is a 1-morphism in $\zweibunconff \Gamma M$, then the pseudonatural transformation $\rho_J:\mathrm{tra}_{\inf P}\to \mathrm{tra}_{\inf P'}$ of \cref{prop:pt} is a 1-morphism between  transport 2-functors.
\end{proposition}

\begin{remark}
For the proof we extract and slightly reformulate the following  results of   \reffunctorLproperties.
Suppose $J:\inf P \to \inf P'$ is a 1-morphism in $\zweibunconff \Gamma M$. \begin{enumerate}[(a)]

\item 
\label{re:examples:morph:a}
Every point $x\in M$ has an open neighborhood $x\in U \subset M$ such that there exist $\Gamma$-connections $(A,B)$ and $(A',B')$ on $U$, 1-morphisms $T:\inf I_{A,B} \to \inf P|_U$ and $T':\inf I_{A',B'} \to \inf P'|_U$, a gauge transformation $(h,\phi):(A,B) \to (A',B')$, and a 2-morphism
$\tau:  T'^{-1} \circ J \circ T\Rightarrow J_{h,\phi}$
in $\zweibunconff \Gamma M$.
\begin{comment}
The existence of $U$, $(A,B)$, $(A',B')$, $T$ and $T'$ follows from the essential surjectivity. If then $U$ is further refined to be contractible, essential surjectivity of the Hom-functor
\begin{equation*}
\hom_{\conff\Gamma U}((A,B),(A',B')) \to \hom_{\zweibunconff \Gamma U}(\inf I_{A,B},\inf I_{A',B'})
\end{equation*}
applied to the 1-morphism $T'^{-1} \circ J \circ T$ in $\zweibunconff \Gamma U$  yields the claim. 
\end{comment}

\item
\label{re:examples:morph:b}
Suppose $U \subset M$ is open, and we
 have a diagram of $\Gamma$-connections and gauge transformations
\begin{equation*}
\alxydim{@=\xyst}{(A_1,B_1) \ar[r]^-{(g,\varphi)} \ar[d]_{(h_1,\phi_1)} & (A_2,B_2)\ar[d]^{(h_2,\phi_2)} \\ (A_1',B_1') \ar[r]_-{(g',\varphi')} & (A_2',B_2')}
\end{equation*}
together with 1-morphisms  $T_i: \inf I_{A_i,B_i} \to \inf P|_U$ and $T_i': \inf I_{A'_i,B_i'} \to \inf P'|_U$, 2-morphisms  $\sigma:T_2^{-1}\circ T_1 \Rightarrow J_{g,\varphi}$, $\sigma':T_2^{\prime-1}\circ T'_1 \Rightarrow J_{g',\varphi'}$ and $\tau_i:  T_i^{\prime-1} \circ J \circ T_i\Rightarrow J_{h_i,\phi_i}$. Then, there exists a unique gauge 2-transformation 
\begin{equation*}
e: (h_2,\phi_2) \circ (g,\varphi) \Rightarrow (g',\varphi') \circ\ (h_1,\phi_1)\text{,}
\end{equation*}
and a commutative diagram
\begin{equation*}
\alxydim{@=\xyst}{T_j'^{-1}\circ J \circ T_j\circ T_j^{-1}\circ T_i \ar@{=>}[r]^-{\tau_j \circ \sigma}\ar@{=>}[d]_{\id \circ d_{T_i'}^{-1} \circ d_{T_j} \circ \id} & J_{h_j,\phi_j} \circ J_{g,\varphi} \ar@{=>}[d]^{\mathcal{A}_{e}} \\ T_j'^{-1} \circ T_i' \circ T_i'^{-1} \circ J \circ T_i\ar@{=>}[r]_-{\sigma' \circ \tau_i} & J_{g',\varphi'} \circ J_{h_i,\phi_i}\text{.} }
\end{equation*}
\begin{comment}
Indeed, this is the fully faithfulness of the Hom-functor
\begin{equation*}
\hom_{\conff\Gamma U}((A_1,B_1),(A'_2,B'_2)) \to \hom_{\zweibunconff \Gamma U}(\inf I_{A_1,B_1},\inf I_{A'_2,B'_2})
\end{equation*}
applied to the morphism
\begin{equation*}
\alxydim{@C=1.4cm}{J_{h_j,\phi_j} \circ J_{g,\varphi} \ar@{=>}[r]^-{\tau_j^{-1} \circ \sigma^{-1}} & T_j'^{-1}\circ J \circ T_j\circ T_j^{-1}\circ T_i \ar@{=>}[d]^-{\id \circ d_{T_i'}^{-1} \circ d_{T_j} \circ \id} \\ & T_j'^{-1} \circ T_i' \circ T_i'^{-1} \circ J \circ T_i \ar@{=>}[r]_-{\sigma' \circ \tau_i} & J_{g',\varphi'} \circ J_{h_i,\phi_i}}
\end{equation*}
in $\hom_{\zweibunconff \Gamma U}(\inf I_{A_1,B_1},\inf I_{A'_2,B'_2})$. \end{comment}

\end{enumerate}
\end{remark}

\begin{proof}[Proof of \cref{prop:transpseudo}]
We choose an open cover $\{ U_i \}_{i\in I}$ and over each open set the data of \cref{re:examples:morph:a}. We form the  pseudonatural transformations $\tau_i$ and $\tau'_i$  for $\inf P|_{U_i}$ and $\inf P'|_{U_i}$, respectively, as in \cref{eq:pseudotau}. We define the pseudonatural transformation 
\begin{equation*}
\lambda_i := \tau_i'^{-1}\circ  \rho_J \circ\tau_i:i ( F_{A_i,B_i})\to i( F_{A'_i,B'_i})\text{,}
\end{equation*}
which is the first descent datum. The second decent datum is over double intersections; it is the modification $\mathcal{E}_{ij}:\lambda_j \circ \gamma_{ij} \Rightarrow  \gamma'_{ij}\circ \lambda_i$ defined by
\begin{equation*}
\alxydim{@C=2.3cm}{\tau_j^{\prime-1}\circ \rho_J \circ \tau_j \circ \tau_j^{-1}\circ \tau_i \ar@{=>}[r]^-{\id \circ \id \circ \mathcal{D}_j \circ \id} & \tau_j^{\prime-1}\circ \rho_J \circ \tau_i \ar@{=>}[r]^-{\id \circ  \mathcal{D}_i^{-1} \circ\id \circ \id} &  \tau_j^{\prime-1}\circ \tau_i' \circ\tau^{\prime-1}_i\circ  \rho_J \circ\tau_i}\text{.}
\end{equation*}
For the first smoothness condition it suffices to show that $\lambda_i$ factors through a \emph{smooth} pseudonatural transformation $\tilde \lambda_{i}:F_{A_i,B_i} \to  F_{A'_i,B'_i}$, i.e. $\lambda_{i}\cong i(\tilde \lambda_{i})$. We construct $\tilde \lambda_{i}$ as follows. Using the gauge transformations $(h_i,\phi_i)$ of \cref{re:examples:morph:a} we let $\tilde \lambda_{i} := \rho_{h_{i},\phi_{i}}$ be the smooth pseudonatural transformation associated to $(h_i,\phi_i)$, see \cref{sec:smoothfunctors}.
We obtain a modification $\mathcal{L}_{i}: \lambda_{i} \Rightarrow i(\tilde \lambda_{i})$ defined as:
\begin{equation}
\label{eq:modiL}
\alxydim{@C=1.7cm}{\lambda_{i} = \eta_i'^{-1}\circ \rho_{T_i'^{-1}}\circ  \rho_J \circ \rho_{T_i} \circ \eta_i= \eta_i'^{-1} \circ \rho_{T_i'^{-1}\circ J\circ T_i} \circ \eta_i \ar@{=>}[r]^-{\id \circ \mathcal{A}_{\tilde\tau_{ij}} \circ \id} &} \eta_i'^{-1}\circ \rho_{J_{h_i,\phi_i}} \circ \eta_i \eqcref{ex:trivbunpt} i (\tilde\lambda_i)\text{.}
\end{equation}
For the second smoothness condition we have to show that there exists a smooth modification $\widetilde{\mathcal{E}}_{ij}:\tilde \lambda_j \circ \tilde \gamma_{ij} \Rightarrow  \tilde \gamma'_{ij}\circ \tilde \lambda_i$ such that
\begin{equation}
\label{eq:prooftrans}
i(\widetilde{\mathcal{E}}_{ij}) \bullet (\mathcal{L}_j  \circ \mathcal{A}_{ij})=(\mathcal{A}'_{ij} \circ \mathcal{L}_i)\bullet \mathcal{E}_{ij}\text{.}
\end{equation}
Indeed, over double intersections we find by \cref{re:examples:morph:b} a gauge-2-transformation
\begin{equation*}
e_{ij}: (h_j,\phi_j) \bullet (g_{ij},\varphi_{ij}) \Rightarrow (g'_{ij},\varphi'_{ij}) \bullet (h_i,\phi_i)\text{,}
\end{equation*}
from which we induce $\widetilde{\mathcal{E}}_{ij} := \mathcal{A}_{a_{ij}}$ via \cref{sec:smoothfunctors}. We have a diagram:\begin{equation*}
\alxydim{@C=2.2cm@R=\xyst}{\lambda_j \circ \gamma_{ij}\ar@{=}[dr] \ar@{=>}[ddd]_{\mathcal{E}_{ij}} \ar@{=>}[rrr]^{\mathcal{L}_j \circ \mathcal{A}_{ij}} &&& i(\tilde \lambda_j) \circ i(\tilde \gamma_{ij}) \ar@{=>}[ddd]^{i(\widetilde{\mathcal{E}_{ij}})} \\&\mquad\mquad\mquad \eta_j^{\prime-1} \circ \rho_{T_j^{\prime-1}\circ J\circ T_j} \circ \rho_{T_j^{-1}\circ T_i} \circ \eta_i \ar@{=>}[d]_{\mathcal{A}_{d_i^{-1}} \circ \mathcal{A}_{d_j}} \ar@{=>}[r]^-{\id \circ \mathcal{A}_{\tilde\tau_j} \circ \mathcal{A}_{\tilde \sigma_{ij}}\circ \id} &\eta_j'^{-1} \circ \rho_{K_j} \circ \rho_{F_{ij}} \circ \eta_i \ar@{=>}[d]^{\id \circ \mathcal{A}_{a_{ij}}\circ \id}\ar@{=}[ur] \mquad\mquad &\\&\mquad\mquad\mquad\eta_j^{\prime-1}\circ \rho_{T_j^{\prime-1}\circ T'_i} \circ \rho_{T_i'^{-1}\circ J\circ T_i} \circ \eta_i \ar@{=>}[r]_-{\id \circ \mathcal{A}_{\tilde\sigma_{ij}'}\circ \mathcal{A}_{\tilde\tau_i} \circ \id}&\eta_j^{\prime-1}\circ \rho_{F_{ij}'} \circ \rho_{K_i} \circ \eta_i \ar@{=}[dr]\mquad\mquad\mquad&\\\gamma'_{ij} \circ \lambda_i \ar@{=}[ur]\ar@{=>}[rrr]_{\mathcal{A}'_{ij} \circ \mathcal{L}_i} &&& i(\tilde \gamma'_{ij})\circ i(\tilde \lambda_i) }
\end{equation*}
The subdiagrams at the bottom and at the top are the definitions of $\mathcal{A}_{ij}$, $\mathcal{A}'_{ij}$, and $\mathcal{L}_i$. The subdiagram on the left is the definition of $\mathcal{E}_{ij}$, and the subdiagram on the right is commutative due to \cref{ex:trivbunmod}. The subdiagram in the middle is commutative by \cref{re:examples:morph:b}. Hence, the whole diagram is commutative; this is \cref{eq:prooftrans}. \end{proof}

\begin{proposition}
\label{prop:transmodi}
If $f:J_1 \Rightarrow J_2$ is a 2-morphism in $\zweibunconff\Gamma M$, then the modification $\mathcal{A}_{f}:\rho_{J_1}\Rightarrow \rho_{J_2}$ of \cref{prop:mod} is a 2-morphism of transport 2-functors.
\end{proposition}

\begin{remark}
\label{re:examples:2morph}
For the proof we extract and slightly reformulate the following  results of   \reffunctorLproperties. 
Suppose $J_1,J_2:\inf P \to \inf P'$ are 1-morphisms in $\zweibunconff \Gamma M$ and $f:J_1 \Rightarrow J_2$ is a 2-morphism. Suppose $U \subset M$ is an open set with $\Gamma$-connections $(A,B)$ and $(A',B')$, 1-morphisms $T:\inf I_{A,B} \to \inf P$ and $T':\inf I_{A',B'} \to \inf P'$, for $i=1,2$ gauge transformations $(h_i,\phi_i)$ with 2-morphism $\tilde\tau_i: T'^{-1}\circ J_i \circ T \Rightarrow J_{h_i,\phi_i}$. In other words, we have for $J_1$ and $J_2$ the structure of \cref{re:examples:morph:a}. Then, there exists a unique  gauge 2-transformation $a:(h_1,\phi_1) \Rightarrow (h_2,\phi_2)$ such that the diagram
\begin{equation*}
\alxydim{@=\xyst}{T'^{-1}\circ J_1 \circ T \ar@{=>}[r]^-{\tilde\tau_1} \ar@{=>}[d]_{\id \circ f \circ \id} & J_{h_1,\phi_1} \ar@{=>}[d]^{f_a} \\ T'^{-1}\circ J_2 \circ T \ar@{=>}[r]_-{\tilde\tau_2} & J_{h_2,\phi_2}}
\end{equation*}
is commutative.
\end{remark}

\begin{proof}[Proof of \cref{prop:transmodi}]
Let $U \subset M$ be an open set over which we have the pseudonatural transformations $\tau$ and $\tau'$ of \cref{eq:pseudotau} for the two principal $\Gamma$-2-bundles $\inf P$ and $\inf P'$, respectively. 
We form the modification $\mathcal{F} : \lambda_1 \Rightarrow \lambda_2$ by
\begin{equation*}
\alxydim{@C=2cm}{ \lambda_1:=  \tau'^{-1}\circ  \rho_{J_1} \circ\tau\ar@{=>}[r]^-{\id \circ \mathcal{A}_{f} \circ \id} & \tau'^{-1}\circ  \rho_{J_2} \circ\tau =:\lambda_2\text{.}} \end{equation*}
We let $\mathcal{L}_1$ and $\mathcal{L}_2$ be the modifications \cref{eq:modiL} associated to $J_1$ and $J_2$. The smoothness condition we have to check is that there exists a smooth modification $\widetilde{\mathcal{F}}:\tilde\lambda_1\Rightarrow \tilde\lambda_2$ such that $\mathcal{F}= \mathcal{L}_2^{-1}\bullet i(\widetilde{\mathcal{F}}) \bullet \mathcal{L}_1$.
Let $a:(h_1,\phi_1)\Rightarrow (h_2,\phi_2)$ be the gauge 2-transformation of \cref{re:examples:2morph}, and let $\widetilde{\mathcal{F}} := \mathcal{A}_{a}$ using \cref{sec:smoothfunctors}, which gives a smooth modification between $\tilde\lambda_1$ and $\tilde\lambda_2$. We have a diagram: 
\begin{equation*}
\alxydim{@C=1.8cm@R=\xyst}{\lambda_1 \ar@{=}[dr] \ar@{=>}[rrr]^-{\mathcal{L}_1}\ar@{=>}[ddd]_-{\mathcal{F}} &&& i(\tilde\lambda_1)\ar@{=}[dl] \ar@{=>}[ddd]^-{i(\widetilde{\mathcal{F}})} \\ &\eta \circ \rho_{T'^{-1}}\circ  \rho_{J_1} \circ \rho_{T} \circ \eta^{-1} \ar@{=>}[r]^-{\id \circ \mathcal{A}_{\tilde\tau_1} \circ\id} \ar@{=>}[d]_{\id \circ\id\circ \mathcal{A}_f \circ\id\circ \id}&\eta \circ   \rho_{K_1} \circ \eta^{-1} \ar@{=>}[d]^{\id \circ \mathcal{A}_{f_a} \circ \id}&\\&\eta \circ \rho_{T'^{-1}}\circ  \rho_{J_2} \circ \rho_{T} \circ \eta^{-1}\ar@{=>}[r]_-{\id \circ \mathcal{A}_{\tilde\tau_2} \circ\id}&\eta \circ   \rho_{K_2} \circ \eta^{-1}&\\ \lambda_2 \ar@{=}[ur] \ar@{=>}[rrr]_-{\mathcal{L}_2} &&& i(\tilde\lambda_2)\ar@{=}[ul]}
\end{equation*}
The subdiagram in the middle is induced from the commutative diagram of \cref{re:examples:2morph}, and the subdiagram on the right hand side commutes by \cref{ex:trivbunmod}. The other subdiagrams commute by definition. Hence, the whole diagram is commutative; this is what we had to show. \end{proof}

\begin{appendix}

\section{Appendix}

\subsection{2-group connections and gauge transformations}

\label{def:gammacon}
\label{def8}

We summarize the bicategory of $\Gamma$-connections following \cite{schreiber5}.
Let $X$ be a smooth manifold and $\Gamma$ be a Lie 2-group, given by a crossed module $(G,H,t,\alpha)$.
A  \emph{$\Gamma$-connection on $X$ }is a pair $(A,B)$ of a 1-form $A\in\Omega^1(X,\mathfrak{g})$
and a 2-form $B \in \Omega^2(X,\mathfrak{h})$. The 2-form   
\begin{equation*}
\mathrm{fcurv}(A,B) := \mathrm{d}A +\frac{1}{2} [A
\wedge A] - t_{*} ( B)\in\Omega^2(X,\mathfrak{g})
\end{equation*}
is called the \emph{fake-curvature}, and the 3-form
\begin{equation*}
\mathrm{curv}(A,B) := \mathrm{d}B + \alpha_{*}( A\wedge B)\in\Omega^3
(X,\mathfrak{h})
\end{equation*}
is called the \emph{curvature}. 
A connection $(A,B)$ is called \emph{fake-flat}, if $\mathrm{fcurv}(A,B)=0$, and it is called \emph{flat}, if it is fake-flat and $\mathrm{curv}(A,B)=0$. 
Let $(A,B)$ and $(A',B')$ be $\Gamma$-connections on $X$.
A \emph{gauge transformation}
\begin{equation*}
(g,\varphi): (A,B) \to (A',B')
\end{equation*}
is a smooth map $g:X \to G$ and a 1-form $\varphi\in \Omega^1(X,\mathfrak{h})$
such that:
\begin{align}
A' + t_{*} (\varphi) &= \mathrm{Ad}_g(A) - g^{*}\bar\theta   
\label{gt1}
\\
B' + \alpha_{*}(A' \wedge \varphi) + \mathrm{d}\varphi + \frac{1}{2}[\varphi \wedge \varphi]&= (\alpha_g)_{*} (B)  \text{.}
\label{gt2}
\end{align}
Here, $\bar\theta$ is the right-invariant Maurer-Cartan form.
The identity gauge transformation is given by $g=1$ and $\varphi=0$.
The composition of gauge transformations
\begin{equation*}
\alxydim{@C=1.5cm}{(A,B) \ar[r]^-{(g_1,\varphi_1)} & (A',B') \ar[r]^-{g_2,\varphi_2} & (A'',B'')}
\end{equation*}
is given by the map $g_2g_1:X \to G$ and the 1-form $\varphi_2+(\alpha_{g_{2}})_{*} (\varphi_1)$. 
A \emph{gauge 2-transformation} $a:(g_1,\varphi_1) \Rightarrow (g_2,\varphi_2)$ is a smooth map $a:X \to H$ such that
\begin{equation*}
g_2 = (t \circ a) \cdot g_1
\quad\text{ and }\quad
\varphi_2 +(r_{a}^{-1} \circ \alpha_{a})_{*}(A') =  \mathrm{Ad}_a (\varphi_1) -a^{*}\bar\theta\text{.}
\end{equation*}
The vertical composition 
\begin{equation*}
\alxydim{}{(g,\varphi) \ar@{=>}[r]^-{a_1} & (g',\varphi') \ar@{=>}[r]^{a_2} & (g'',\varphi'')}
\end{equation*}
is given by $a_2a_1$. The horizontal composition is
\begin{equation*}
\alxydim{}{(A,B) \ar@/^2pc/[r]^{(g_1,\varphi_1)}="1" \ar@/_2pc/[r]_{(g_1',\varphi_1')}="2" \ar@{=>}"1";"2"|{a_1} & (A',B') \ar@/^2pc/[r]^{(g_2,\varphi_2)}="1" \ar@/_2pc/[r]_{(g_2',\varphi_2')}="2" \ar@{=>}"1";"2"|{a_2} & (A'',B'')}
=
\alxydim{@C=2.3cm}{(A,B) \ar@/^2pc/[r]^{(g_2g_1, (\alpha_{g_2})_{*}(\varphi_1) + \varphi_2)}="1" \ar@/_2pc/[r]_{(g_2'g_1', (\alpha_{g_2'})_{*}(\varphi'_1) +\varphi'_2)}="2" \ar@{=>}"1";"2"|{a_2\alpha(g_2,a_1)} & (A'',B'')\text{,}}
\end{equation*}
and the identity gauge 2-transformation is given by $a=1$.
$\Gamma$-connections on $X$, gauge transformations, and gauge 2-transformations form a strict bicategory $\con\Gamma X$. The restriction to fake-flat $\Gamma$-connections forms a full sub-bicategory $\conff\Gamma X$.

\subsection{Path-ordered exponentials}

\label{sec:pathordered}

For a 1-form $\omega\in \Omega^1(X,\mathfrak{g})$ with values in the Lie algebra of a Lie group $G$ and a path $\gamma:[0,1] \to X$ we denote by $\PE_{\omega}(\gamma) \in G$ the path ordered exponential of $\omega$ along $\gamma$. That is, we let $g:[0,1] \to G$ be the  unique solution of the initial value problem
\begin{equation*}
\dot g(\tau) = - \omega(\dot\gamma(\tau)) g(\tau)
\quith g(0) = 1\text{,}
\end{equation*}
and put $\PE_{\omega}(\gamma):= g(1)$. 
\begin{comment}
The  conventions are so that if $g:[0,1] \to G$ is a path in $G$ with $g(0)=1$ we get $\PE_{\theta}(g)=g(1)^{-1}$. 
\end{comment}
We need the following well-known general properties of the path ordered exponential.

\begin{lemma}
Let $\omega\in \Omega^1(X,\mathfrak{g})$.
\begin{enumerate}[(a)]

\item
\label{lem:PE:a}
It depends only on the thin homotopy class of the path: if there is a fixed-ends homotopy between $\gamma,\gamma':x \to y$ whose rank is less than two, then $\PE_{\omega}(\gamma)=\PE_{\omega}(\gamma')$. 

\item
\label{lem:PE:b}
It is compatible with path composition: if $\gamma:x \to y$ and $\gamma':y \to z$ are composable paths, then $\PE_{\omega}(\gamma'\ast \gamma) = \PE_{\omega}(\gamma')\cdot \PE_{\omega}(\gamma)$.
\begin{comment}
If $g,g':[0,1] \to G$ are the two solutions, then $\tilde g := g'g(1) \circ g$ is smooth and a solution for the initial value problem for $\gamma'\circ\gamma$: we have $\tilde g(0)=g(0)=1$ and check
\begin{equation*}
\frac{\mathrm{d}}{\mathrm{d}\tau} \tilde g(\tau) = - \omega_{(\gamma'\circ\gamma)(\tau)}(\frac{\mathrm{d}}{\mathrm{d}\tau}(\gamma'\circ\gamma)(\tau)) \tilde g(\tau)
\end{equation*}
separately for $0\leq\tau\leq \frac{1}{2}$, where it becomes the equation for $g$,
\begin{equation*}
\frac{\mathrm{d}}{\mathrm{d}\tau}  g(2\tau) =  - \omega_{\gamma(2\tau)}(\frac{\mathrm{d}}{\mathrm{d}\tau}\gamma(2\tau))  g(2\tau)
\end{equation*}
and for $\frac{1}{2}\leq \tau \leq 1$,
\begin{equation*}
\frac{\mathrm{d}}{\mathrm{d}\tau}  g'(2\tau-1)g(1)= - \omega_{\gamma'(2\tau-1)}(\frac{\mathrm{d}}{\mathrm{d}\tau}\gamma'(2\tau-1))  g'(2\tau-1)g(1)\text{.}
\end{equation*}
\end{comment}

\item
\label{lem:PE:c}
It is natural under the pullback of differential forms:
if $f:W \to X$ is a smooth map, then
$\PE_{f^{*}\omega}(\gamma) = \PE_{\omega}(f(\gamma))$.
\begin{comment}
Indeed, let $g$ be a solution for $f^{*}\omega$, i.e. 
\begin{equation*}
\dot g(\tau) = - \omega_{f(\gamma(\tau))}(T_{\gamma(\tau)}f(\dot\gamma(\tau))) g(\tau)
\quith g(0) = 1\text{.}
\end{equation*}
The very same initial problem is also the one for $\omega$ and $f \circ \gamma$.
\end{comment}

\item 
\label{lem:PE:d}
It is natural under Lie group homomorphisms:
if $\varphi:G \to\ G'$ is a Lie group homomorphism, then $\varphi(\PE_{\omega}(\gamma)) = \PE_{\varphi_{*}(\omega)}(\gamma)$.
\begin{comment}
Indeed, let $g:[0,1] \to G$ be as above. Let $g' := \varphi \circ g$. We have
\begin{equation*}
\dot g'(\tau) = T_{g(\tau)}\varphi(\dot g(\tau))=-T_{g(\tau)}\varphi( \omega(\dot\beta(\tau)) g(\tau))=-\varphi_{*}( \omega(\dot\beta(\tau))\varphi( g(\tau)))\text{;}
\end{equation*}
thus, $E_{\varphi_{*}(\omega)}(\gamma)=g'(1)$. 
\end{comment}

\item
\label{lem:PE:e}
It is compatible with gauge transformations: if  $g:X \to G$ is a smooth map and $\omega' := \mathrm{Ad}^{-1}_{g}(\omega)+g^{*}\theta$, then $\PE_\omega(\gamma) \cdot g(\gamma(0)) = g(\gamma(1)) \cdot \PE_{\omega'}(\gamma)$.

\end{enumerate}
\end{lemma}

The following propositions discuss special properties of path-ordered exponentials in the total space of principal 2-bundles and 1-morphisms between those; in combination with the notion of horizontality defined in \cref{sec:horizontality}. 

\begin{proposition}
Let $\inf P$ be  a principal $\Gamma$-2-bundle equipped with a connection $\Omega$. 
\begin{enumerate}[(a)]

\item
\label{lem:poeOmega:x}
If $\beta$ is a path in $\ob{\inf P}$ and $\gamma$  a path in $G$ with $\gamma(0)=1$, then $\PE_{\fa\Omega}(R(\beta,\gamma))=\gamma(1)^{-1}\cdot\PE_{\fa\Omega}(\beta)$.

\item 
\label{lem:poeOmega:a}
$\PE_{\fb\Omega}(\id_{\beta})=1$ for every path $\beta$ in $\ob{\inf P}$.

\item
\label{lem:poeOmega:b}
Let $\rho$ be a horizontal path in $\mor{\inf P}$ such that  $s(\rho)$ is horizontal, and let $h$ be a path in $H$ with $h(0)=1$. Then, $\PE_{\fb\Omega}(R(\rho,(h,1)))=h(1)^{-1}$. 

\end{enumerate}
\end{proposition}

\begin{proof}
For \cref{lem:poeOmega:x*}  we consider $X:=\ob{\inf P}\times G$, so that $(\beta,\gamma)$ is a path in $X$. It is easy to check using \cref{eq:conform:a} that the map $g:X \to G:(p,g')\mapsto g'^{-1}$ is a gauge transformation between $R^{*}\fa\Omega$ and $\pr_1^{*}\fa\Omega$.
\begin{comment}
Indeed,
\begin{equation*}
\pr_1^{*}\fa\Omega =\mathrm{Ad}_{g'}(\mathrm{Ad}_{g'}^{-1}(p^{*}\fa\Omega) + \theta_{g'})+\theta_{g'^{-1}}\eqcref{eq:conform:a} \mathrm{Ad}_{g'}(R^{*}\fa\Omega|_{p,g'})+\theta_{g'^{-1}}
\end{equation*}
\end{comment}
Thus, by \cref{lem:PE:c,lem:PE:e} we have 
\begin{equation*}
\PE_{\fa\Omega}(R(\beta,\gamma))  =\PE_{R^{*}\fa\Omega}(\beta,g) \cdot g(\beta(0),\gamma(0)) =g(\beta(1),\gamma(1)) \cdot \PE_{\pr_1^{*}\fa\Omega}(\beta,g)=\gamma(1)^{-1} \cdot \PE_{\fa\Omega}(\beta)
\end{equation*}
For \cref{lem:poeOmega:a*}, we apply \cref{lem:PE:c} to $\id: \ob{\inf P} \to \mor{\inf P}$ and use $\id^{*}\fb\Omega=0$. For \cref{lem:poeOmega:b*}
we note that
\cref{eq:conform:b} and the assumptions on $\rho$ imply $R^{*}\fb\Omega(\dot\rho,\dot h,0) =\eta^{-1}\dot\eta$. 
\begin{comment}
Indeed,
\begin{align*}
R^{*}\fb\Omega(\dot\rho(\tau),\dot h(\tau),0) = \mathrm{Ad}_{ h(\tau)}^{-1}(\fb\Omega(\dot\rho(\tau)))+(\tilde \alpha_{ h(\tau)})_{*}(\fa\Omega(s_{*}(\dot\rho))) + h(\tau)^{-1}\dot h(\tau) 
=  h(\tau)^{-1}\dot h(\tau)\text{.} 
\end{align*}
\end{comment}
The corresponding initial value problem is then solved by $\eta^{-1}$; this shows the claim.
\begin{comment}
Thus, $\kappa :=  h^{-1}$ solves
$\dot\kappa = - h^{-1}\dot h h^{-1}=-R^{*}\fb\Omega(\dot\rho,\dot h,0)\kappa$ and $\kappa(0)=1$;
this shows the claim.
\end{comment}
\end{proof}

\begin{proposition}
\label{lem:pathJ:a}
Suppose $J:\inf P_1 \to \inf P_2$ is a 1-morphism in $\zweibuncon\Gamma M$. Let $\lambda:[0,1] \to J$ be a horizontal path such that $\alpha_r(\lambda)$ is horizontal, and let $h:[0,1] \to H$ be a path with $h(0)=1$. Then, $\PE_{\nu_0}(\lambda\cdot(h,1))=h(1)^{-1}$. \end{proposition}

\begin{proof}
Let $\nu=(\nu_0,\nu_1)$ be the connective, connection-preserving pullback on $J$. Connectivity together with our assumptions on $\lambda$ imply $\nu_0(\partial_t(\lambda \cdot (h,1)))=h^{-1}\dot h$. The corresponding initial value problem is then solved by $\eta^{-1}$; this shows the claim.
\begin{comment}
If $\rho: J \times H \to J$ denotes the $\mor{\Gamma}$-action then we have 
\begin{equation*}
\rho^{*}\nu_0 = \mathrm{Ad}_{h}^{-1}(\pr_F^{*}\nu_0)+(\tilde\alpha_{h})_{*}(\pr_F^{*}\alpha_r^{*}\fa{\Omega}_2) +h^{*}\theta \text{.}
\end{equation*}
The differential equation is
\begin{equation*}
\dot \kappa = -\rho^{*}\nu_0(\lambda,h) \kappa=-h^{-1}\dot h\kappa\text{;}
\end{equation*}
it is solved by $\kappa:=h^{-1}$. 
\end{comment}
\end{proof}

Next we discuss an important application of the path-ordered exponential related to a gauge transformation 
$(g,\varphi):(A,B) \to (A',B')$ between $\Gamma$-connections on $X$, see \cref{def:gammacon}. We note that $(\varphi,A')$ is a 1-form on $X$ with values in $\mathfrak{h} \ltimes \mathfrak{g}$, so that $\PE_{\varphi,A'}(\gamma)\in H \ltimes G$ for any path $\gamma$ in $X$. Since $G$ acts on $H$ the $G$-component of $\PE_{\varphi,A'}(\gamma)$ is just  $\PE_{A'}(\gamma)$. 
The $H$-component, however, is an independent  quantity; we denote it by $h_{g,\varphi}(\gamma) \in H$. In the following we study some of its properties.

\begin{proposition}
Let $(g,\varphi):(A,B) \to (A',B')$ be a gauge transformation between $\Gamma$-con\-nec\-tions on $X$. 
\begin{enumerate}[(a)]

\item 
\label{lem:SE:gauge:z}
If $\gamma:x \to y$ and $\gamma':y \to z$ are composable paths, then
\begin{equation*}
h_{g,\varphi}(\gamma'\circ \gamma) = h_{g,\varphi}(\gamma')\cdot \alpha(\PE_{A'}(\gamma'),h_{g,\varphi}(\gamma))
\end{equation*}

\item 
\label{lem:SE:gauge:a}
For all paths $\gamma:x \to y$, we have
\begin{equation*}
\PE_{A'}(\gamma)\cdot g(x)=t(h_{g,\varphi}(\gamma)^{-1})\cdot g(y)\cdot \PE_{A}(\gamma)\text{.} 
\end{equation*}

\item
\label{lem:SE:gauge:m}
If $(g',\varphi'):(A',B') \to (A'',B'')$ is a second gauge transformation, and $\gamma:x\to y$ is a path, then
\begin{equation*}
h_{(g',\varphi') \circ (g,\varphi)}(\gamma) = \alpha(g'(\gamma(y)),h_{g,\varphi}(\gamma))\cdot h_{g',\varphi'}(\gamma)\text{.}
\end{equation*}

\end{enumerate}
\end{proposition}

\begin{proof}
We have $\PE_{\varphi,A'}(\gamma'\circ \gamma)=\PE_{\varphi,A'}(\gamma')\cdot \PE_{\varphi,A'}(\gamma)$ in $H \ltimes G$. Since the projection  of $\PE_{\varphi,A'}(\gamma)$ to $G$ is $\PE_{A'}(\gamma)$  we have  \cref{lem:SE:gauge:z*}. \cref{lem:SE:gauge:a*} is \cite[Lemma 2.18]{schreiber5}. \cref{lem:SE:gauge:m*} is the functoriality proved in \cite[Section 2.3.4]{schreiber5}.
\end{proof}

Gauge transformations can be produced from a fake-flat connection $\Omega$ on a principal $\Gamma$-2-bundle $\inf P$ in the following way. First we note that the pair $(\fa\Omega,-\fc\Omega)$ is a $\Gamma$-connection on $\ob{\inf P}$, with the sign chosen such that it is fake-flat in the sense of \cref{def:gammacon}. Consider the smooth manifold $X := \mor{\inf P} \times G$ equipped with the maps $\chi_1,\chi_2:X \to \ob{\inf P}$ defined by $\chi_1(\rho,g) := t(\rho)$ and $\chi_2(\rho,g) := R(s(\rho),g^{-1})$.
Now we have the fake-flat $\Gamma$-connections $(A,B) := \chi_1^{*}(\fa\Omega,-\fc\Omega)$ and $(A',B') := \chi_2^{*}(\fa\Omega,-\fc\Omega)$ over $X$. We define $g := \pr_2:X \to G$ and $\varphi := (\alpha_{g})_{*}(\pr_1^{*}\fb\Omega) \in \Omega^1(X,\mathfrak{h})$. 

\begin{lemma}
\label{lem:gaugetrafoX}
$(g,\varphi)$ is a gauge transformation between $(A,B)$ and $(A',B')$.  
\end{lemma}

\begin{proof}
Identity \cref{gt1} is proved by a direct calculation using only $t_{*}(\fb\Omega)=\Delta(\fa\Omega)$ and the transformation rule for $\fa\Omega$,  \cite[Eq. (5.1.1)]{Waldorf2016}.
\begin{comment}
Indeed,
\begin{eqnarray*}
t_{*}(\varphi_{\rho,g'}) &=& t_{*}((\alpha_{g'})_{*}(\fb\Omega_{\rho}))
\\&=& \mathrm{Ad}_{g'}(t_{*}(\fb\Omega_{\rho}))
\\&=&  \mathrm{Ad}_{g'}(\fa\Omega_{t(\rho)}) -\mathrm{Ad}_{g'}(\fa\Omega_{s(\rho)})+\bar\theta_{g'}- \bar\theta_{g'}
\\&=&  \mathrm{Ad}_{g'}(\fa\Omega_{t(\rho)}) -\fa\Omega_{R(s(\rho),g^{-1})}- \bar\theta_{g'}
\\&=&  \mathrm{Ad}_{g(\rho,g')}(A_{\chi_1(\rho,g')}) - A _{\chi_2(\rho,g')}- \bar\theta_{g(\rho,g')}
\end{eqnarray*}
\end{comment}
Identity \cref{gt1} is proved similarly using additionally the transformation rule for $\fc\Omega$ (\cref{eq:conform:c})  and the fake-flatness of $\Omega$.
\begin{comment}
Indeed,
\begin{eqnarray*}
&&\hspace{-4em}\mathrm{d}\varphi_{\rho,g} + \frac{1}{2}[\varphi_{\rho,g}\wedge\varphi_{\rho,g}] +\alpha_{*}(A_{\pr_2(\chi(\rho,g))} \wedge \varphi_{\rho,g} ) 
\\&=& \mathrm{d}(\alpha_g)_{*}(\fb\Omega_\rho) + \frac{1}{2}[(\alpha_g)_{*}(\fb\Omega_\rho)\wedge(\alpha_g)_{*}(\fb\Omega_\rho)] \\&&+\alpha_{*}(\fa\Omega_{R(s(\rho),g^{-1})} \wedge (\alpha_g)_{*}(\fb\Omega_\rho) )
\\&=& (\alpha_{g})_{*}(\mathrm{d}\fb\Omega_\rho+ \frac{1}{2}[\fb\Omega_\rho\wedge\fb\Omega_\rho])+\alpha_{*}(g^{*}\bar\theta \wedge (\alpha_{g})_{*}(\fb\Omega_\rho))  \\&&+\alpha_{*}((\mathrm{Ad}_{g}(\fa\Omega_{s(\rho)}) - \bar\theta_g) \wedge (\alpha_g)_{*}(\fb\Omega_\rho) )
\\&=& (\alpha_{g})_{*}(\mathrm{d}\fb\Omega_\rho+ \frac{1}{2}[\fb\Omega_\rho\wedge\fb\Omega_\rho]+\alpha_{*}(\fa\Omega_{s(\rho)}  \wedge \fb\Omega_\rho))
\\&=& (\alpha_{g})_{*}(-\Delta\fc\Omega_{\rho})
\\&=& -(\alpha_{g})_{*}(\fc\Omega_{t(\rho)})+\fc\Omega_{R(s(\rho),g^{-1})}
\\&=& (\alpha_{g})_{*}(B_{t(\rho)})-B_{R(s(\rho),g^{-1})}
\\&=& (\alpha_{\phi(\rho,g)})_{*}(B_{\pr_1(\chi(\rho,g))})-B_{\pr_2(\chi(\rho,g))}
\end{eqnarray*}
Here we have used
\begin{equation*}
\fa\Omega_{R(p,g)} = \mathrm{Ad}_{g}^{-1}(\fa\Omega_{p}) + \theta_g
\quand
\fc\Omega_{R(p,g)} = (\alpha_{g^{-1}})_{*}(\fc\Omega_{p})
\end{equation*}
We have further used the formula
\begin{equation*}
\mathrm{d}(\alpha_{g})_{*}(\varphi) = (\alpha_{g})_{*}(\mathrm{d}\varphi)+\alpha_{*}(g^{*}\bar\theta \wedge (\alpha_{g})_{*}(\varphi))\text{.}
\end{equation*}
\end{comment}
\end{proof}

Correspondingly, we have the quantity $h_{g,\varphi}(\rho,\gamma)\in H$ associated to any pair of paths $\rho$ in $\mor{\inf P}$ and $\gamma$ in $G$. The following two lemmas list its relevant  properties.

\begin{lemma}
\begin{enumerate}[(a)]

\item
\label{lem:gomega:c}
$h_{g,\varphi}(\id_{\beta},\gamma)=1$ for all paths $\beta$ in $\ob{\inf P}$ and $\gamma$ in $G$.

\item 
\label{lem:h:b}
$h_{g,\varphi}(\rho_{1} \circ R(\rho_{2},\gamma_1) ,\gamma_{2}\gamma_{1})=\alpha(\gamma_2(1),h_{g,\varphi}(\rho_1,\gamma_1))\cdot h_{g,\varphi}(\rho_2,\gamma_2)$
for all paths $\rho_1,\rho_2$ in $\mor{\inf P}$ and $\gamma_1,\gamma_2$ in $G$ such that  $t(\rho_2)=R(s(\rho_1),\gamma_1^{-1})$.

\end{enumerate}
\end{lemma}

\begin{proof}
For \cref{lem:gomega:c*} we consider the map $i:\ob{\inf P} \times G \to \mor{\inf P} \times G: (p,g) \mapsto (\id_p,g)$, under which  $i^{*}\varphi=0$ (because $\id^{*}\fb\Omega=0$).
\begin{comment}
We have 
\begin{equation*}
i^{*}\varphi|_{p,g} = \varphi_{\id_p,g}=(\alpha_g)_{*}(\fb\Omega_{\id_p})=0
\end{equation*}
\end{comment}
From \cref{lem:PE:c} we obtain
\begin{equation*}
\PE_{\varphi,A'}(\id_{\beta},\gamma)=\PE_{0,i^{*}A'}(\beta,\gamma)=(1,\PE_{i^{*}A'}(\gamma))\text{;}
\end{equation*}
this implies the claim.
For \cref{lem:h:b*} 
 we consider $\tilde X:=X \ttimes{\chi_1}{\chi_2}X$, where $X=\mor{\inf P} \times G$, so that $((\rho_2,\gamma_2),(\rho_1,\gamma_1))$ is a path in $\tilde X$. On $\tilde X$ we have the three $\Gamma$-connections 
\begin{align*}
(A,B)&:=\pr_2^{*}\chi_1^{*}(\fa\Omega,-\fc\Omega)
\\
(A',B')&:=\pr_2^{*}\chi_2^{*}(\fa\Omega,-\fc\Omega)=\pr_1^{*}\chi_1^{*}(\fa\Omega,-\fc\Omega)
\\
(A'',B'')&:=\pr_1^{*}\chi_2^{*}(\fa\Omega,-\fc\Omega)\text{.}
\end{align*}
and by \cref{lem:gaugetrafoX} two gauge transformations
\begin{align*}
\pr_2^{*}(g,\varphi): (A,B) \to (A',B')
\quand
\pr_1^{*}(g,\varphi): (A',B') \to (A'',B'')\text{.}
\end{align*}
We claim that the map $\mu: \tilde X \to X$ defined by $\mu((\rho_2,g_2),(\rho_1,g_1)) := (\rho_{1} \circ R(\rho_{2},g_1) ,g_{2}g_{1})$
satisfies
\begin{equation*}
\mu^{*}(g,\varphi) = \pr_1^{*}(g,\varphi)\circ \pr_2^{*}(g,\varphi)\text{,}
\end{equation*}
where $\circ$ denotes the composition of gauge transformations. The only non-trivial part is to show the required identity for the $\mathfrak{h}$-valued differential forms,   $\mu^{*}\varphi= \pr_1^{*}\varphi + (\alpha_{g \circ \pr_1})_{*}(\pr_2^{*}\varphi)$, which follows from the definition of $\varphi$ and the identity $\Delta\fb\Omega=0$.  
\begin{comment}
Indeed,
\begin{eqnarray*}
(\mu^{*}\varphi)_{(\rho_{23},g_{23}),(\rho_{12},g_{12})}&=& \varphi_{(\rho_{12} \circ R(\rho_{23},\id_{g_{12}}) ,g_{23}g_{12})}
\\&=& (\alpha_{g_{23}g_{12}})_{*}(\fb\Omega_{\rho_{12} \circ R(\rho_{23},\id_{g_{12}})}) 
\\&=& (\alpha_{g_{23}g_{12}})_{*}(\fb\Omega_{\rho_{12}  }+\fb\Omega_{R(\rho_{23},1,g_{12})}) \\&=& (\alpha_{g_{23}g_{12}})_{*}(\fb\Omega_{\rho_{12}  }+(\alpha_{g_{12}^{-1}})_{*} (\fb\Omega_{\rho_{23}}) \\&=& (\alpha_{g_{23}})_{*}(\fb\Omega_{\rho_{23}}) + (\alpha_{g_{23}})_{*}((\alpha_{g_{12}})_{*}(\fb\Omega_{\rho_{12}}))
\\&=& \varphi|_{(\rho_{23},g_{23})} + (\alpha_{g_{23}})_{*}(\varphi|_{(\rho_{12},g_{12})})\text{.}
\end{eqnarray*} 
\end{comment}
By \cref{lem:SE:gauge:m,lem:PE:c} we obtain
\begin{equation*}
h_{g,\varphi}(\rho_{1} \circ R(\rho_{2},\gamma_1) ,\gamma_{2}\gamma_{1})=h_{\mu^{*}(g,\varphi)}((\rho_2,\gamma_2),(\rho_1,\gamma_1))= \alpha(\gamma_2(1), h_{g,\varphi}((\rho_1,\gamma_1))\cdot h_{g,\varphi}(\rho_2,\gamma_2))\text{,}
\end{equation*}
this is the claim.
\begin{comment}
Indeed,
\begin{align*}
&\mquad h_{\pr_G,\varphi}(\rho_{1} \circ R(\rho_{2},\gamma_1) ,\gamma_{2}\gamma_{1}) 
\\&=h_{\mu^{*}(\pr_G,\varphi)}((\rho_2,\gamma_2),(\rho_1,\gamma_1))
\\&=h_{\pr_1^{*}(\pr_G,\varphi)\circ \pr_2^{*}(\pr_G,\varphi)}((\rho_2,\gamma_2),(\rho_1,\gamma_1))
\\&= \alpha(\pr_1^{*}\pr_G((\rho_2(1),\gamma_2(1)),(\rho_1(1),\gamma_1(1))),h_{\pr_2^{*}(\pr_G,\varphi)}((\rho_2,\gamma_2),(\rho_1,\gamma_1)))\cdot h_{\pr_1^{*}(\pr_G,\varphi)}((\rho_2,\gamma_2),(\rho_1,\gamma_1))
\\&= \alpha(\gamma_2(1), h_{\pr_G,\varphi}((\rho_1,\gamma_1))\cdot h_{\pr_G,\varphi}(\rho_2,\gamma_2))
\end{align*}
\end{comment}
\end{proof}

\begin{lemma}
\label{co:h:a}
$h_{g,\varphi}(\rho  ,\gamma)=\alpha(\gamma(1),h_{g,\varphi}(\rho,1))$.
\end{lemma}

\begin{proof}
Put $\rho_2=\id$ and $\gamma_1=1$ in \cref{lem:h:b} and then use \cref{lem:calchgvarphi}.
\begin{comment}
Indeed, we get
\begin{equation*}
h_{\Omega}(\rho_{1}  ,\gamma_{2})=\alpha(\gamma_2(1),h_{\Omega}(\rho_1,1))\cdot h_{\Omega}(\id,\gamma_2)\text{.}
\end{equation*}
\end{comment}
\end{proof}

We can thus restrict ourselves to the case of constant paths $\gamma=1$, and remain with a quantity $h_{\Omega}(\rho) := h_{g,\varphi}(\rho,1)\in H$ associated to any path $\rho$ in $\mor{\inf P}$. It has the following properties:

\begin{proposition}
Let $\inf P$ be a principal $\Gamma$-2-bundle over $M$ with fake-flat connection $\Omega$.
\label{co:h}
\begin{enumerate}[(a)]

\item
\label{co:h:b}
$h_{\Omega}( R(\rho,\gamma))=\alpha(\gamma(1)^{-1},h_{\Omega}(\rho))$.

\item
\label{co:h:c}
$h_{\Omega}(\rho_2)^{-1}=h_{\Omega}(\rho_{2}^{-1})$. 

\item
\label{co:h:d}
$h_{\Omega}(\rho_{1} \circ \rho_{2} )=h_{\Omega}(\rho_1)\cdot h_{\Omega}(\rho_2)$ whenever $\rho_1$ and $\rho_2$ are pointwise composable.

\item
\label{prop:SE:gauge:z}
$h_{\Omega}(\rho' \ast \rho) = h_{\Omega}(\rho')\cdot \alpha(\PE_{\fa\Omega}(s(\rho')),h_{\Omega}(\rho))$ whenever $\rho'$ and $\rho$ are composable paths.

\item
\label{lem:calchgvarphi}
$h_{\Omega}(\rho)=\PE_{\fb\Omega}(\rho)$ if $s(\rho)$ is horizontal.

\item
\label{co:hcalc:a}
$h_{\Omega}(R(\rho, (h,1)))=h(1)^{-1}$ if $\rho$ and $s(\rho)$ are horizontal, and $h$ is a path in $H$ with $h(0)=1$.  

\item
\label{co:hcalc:b}
$h_{\Omega}(\rho)=1$ if $\rho$ and $s(\rho)$ are horizontal.

\end{enumerate}
\end{proposition}

\begin{proof}
 \cref{co:h:b*} follows in the same way by putting $\rho_1 = \id$ and $\gamma_2=1$  in \cref{lem:h:b}, and then using \cref{co:h:a}.
\begin{comment}
Indeed, we first get 
\begin{equation*}
h_{\Omega}( R(\rho_{2},\gamma_1) ,\gamma_{1})=h_{\Omega}(\rho_2,1)\text{.}
\end{equation*} 
Using \cref{co:h:a} we have
\begin{equation*}
h_{\Omega}( R(\rho_{2},\gamma_1) ,\gamma_{1})=\alpha(\gamma(1),h_{\Omega}( R(\rho_{2},\gamma_1),1))\text{.}
\end{equation*} 
\end{comment}
\cref{co:h:c*} follows similarly with $\rho_1=\rho_2^{-1}$ and $\gamma_1=\gamma_2=1$, and \cref{co:h:d*} follows with $\gamma_1=\gamma_2=1$. In \cref{prop:SE:gauge:z*}, we have over $X$ an identity $A'=\mathrm{Ad}_{\pr_2}(\pr_1^{*}\rho^{*}\fa\Omega)-\pr_2^{*}\bar\theta$, i.e. the map $(\rho,g) \mapsto g^{-1}$ is a gauge transformation between $s^{*}\fa\Omega$ and $A'$ in the sense of \cref{lem:PE:e}. 
\begin{comment}
Indeed,
\begin{equation*}
A'|_{\rho,g}=\chi_2^{*}\fa\Omega|_{\rho,g}=\fa\Omega_{R(s(\rho),g^{-1})}\eqcref{eq:conform:a} \mathrm{Ad}_{g}(\fa\Omega|_{s(\rho)}) - g^{*}\bar\theta
\end{equation*}
\end{comment}
Thus, \cref{lem:PE:e} implies
\begin{equation*}
\PE_{\fa\Omega}(s(\rho')) \cdot \gamma'(0)^{-1} = \gamma'(1)^{-1} \cdot \PE_{A'}(\rho',\gamma')
\end{equation*}
and hence
\begin{multline*}
h_{g,\varphi}(\rho' \circ \rho,1) = h_{g,\varphi}(\rho',1)\cdot \alpha(\PE_{A'}(\rho',1),h_{g,\varphi}(\rho,1))=h_{g,\varphi}(\rho',1)\cdot \alpha(\PE_{\fa\Omega}(s(\rho')),h_{g,\varphi}(\rho,1))\text{.}
\end{multline*}
\cref{lem:calchgvarphi*} is proved  by a direct calculation of $h_{g,\varphi}(\rho,1)$. Let $(\eta,\kappa)$ be a path in $H \times G$ that solves the initial value problem for $(\varphi,A')$, i.e. $h_{g,\varphi}(\rho,1)=\eta(1)$.
Employing the definitions of $\varphi$ and $A'$ the differential equation splits into two components
\begin{align*}
\dot\kappa(\tau)&=-\fa\Omega(s_{*}(\dot\rho(\tau))\kappa(\tau) 
\\
\dot\eta(\tau) &=-\fb\Omega(\dot\rho(\tau))\eta(\tau)-(\alpha_{\eta(\tau)})_{*}(\fa\Omega(s_{*}(\dot\rho(\tau))) )
\end{align*}
\begin{comment}
Indeed,
\begin{align*}
(\dot\eta(\tau),\dot\kappa(\tau))&=-(\varphi,A')(\dot\rho(\tau),0)(\eta(\tau),\kappa(\tau))
\\&=-(\varphi(\dot\rho(\tau),0),A'(\dot\rho(\tau),0))(\eta(\tau),\kappa(\tau))
\\&=-(\fb\Omega(\dot\rho(\tau)),\fa\Omega(s_{*}(\dot\rho(\tau))))(\eta(\tau),\kappa(\tau))
\\&=-(\fb\Omega(\dot\rho(\tau))\eta(\tau)+(\alpha_{\eta(\tau)})_{*}(\fa\Omega(s_{*}(\dot\rho(\tau))) ), \fa\Omega(s_{*}(\dot\rho(\tau))\kappa(\tau) )
\end{align*}
Here we have used 
\begin{equation*}
(Y,X)(h,g)=\partial_t(\eta(t),\chi(t))(h,g)=\partial_t(\eta(t)\alpha(\chi(t),h),\chi(t)g)=(Yh+\alpha_h(X),Xg)\text{.}
\end{equation*}
\end{comment}
Since $s(\rho)$ is horizontal, we have $\kappa=1$, and we see that $\eta(1)=\PE_{\fb\Omega}(\rho)$.
This shows the claim.
\cref{co:hcalc:a*,co:hcalc:b*} follow from \cref{lem:calchgvarphi*} in combination with \cref{lem:poeOmega:b}.
\begin{comment}
Indeed, we note that $s(R(\rho,(h,1)))=s(\rho)$  is horizontal, so that 
\begin{equation*}
h_{\Omega}(R(\rho, (h,1)),\gamma) \eqcref{lem:calchgvarphi} \alpha(\gamma(1),\PE_{\fb\Omega}(R(\rho, (h,1)))) \eqcref{lem:poeOmega:b} \alpha(\gamma(1),h(1)^{-1})\text{.}
\end{equation*}

\end{comment}
\end{proof}

Another situation where a gauge transformation appears are 1-morphisms.
Suppose $J:\inf P \to \inf P'$ is a 1-morphism in $\zweibunconff\Gamma M$ between principal $\Gamma$-2-bundles with fake-flat connections $\Omega$ and $\Omega'$, respectively. Let $\nu=(\nu_0,\nu_1)$ be its connective, connection-preserving, and fake-flat $\Omega'$-pullback. We consider the smooth manifold $Q := J \times G$ equipped with the maps $\chi: Q \to \ob{\inf P}$ and $\chi': Q \to \ob{\inf P'}$ defined by $\chi(j,g) := \alpha_l(j)$ and $\chi'(j,g):= R(\alpha_r(j),g^{-1})$. Then we have the $\Gamma$-connections $(A,B):= \chi^{*}(\fa\Omega,-\fc\Omega)$ and $(A',B') := \chi'^{*}(\fa{\Omega'},-\fc{\Omega'})$.
We define the map $g := \pr_G:Q \to G$ and the  1-form $\varphi\in\Omega^1(Q,\mathfrak{h})$ by $\varphi := (\alpha_{g})_{*}( \pr_J^{*}\nu_0)$.

\begin{lemma}
\label{lem:1morphgauge}
$(g,\varphi)$ is a gauge transformation between $(A,B)$ and $(A',B')$.
\end{lemma}

\begin{proof}
\cref{gt1} is a straightforward computation using \cref{eq:conform:a} and the condition $t_{*}(\nu_0) = \alpha_l^{*}\fa\Omega-\alpha_r^{*}\fa{\Omega'}$ which is part of the relation $J_{\nu}^{*}\Omega'=\Omega$, see \refpullbackform. 
\begin{comment}
Indeed,
\begin{align*}
t_{*}(\varphi_{f,g}) &= t_{*}((\alpha_{g})_{*}(\nu_f))
\\&= \mathrm{Ad}_g(t_{*}(\nu_f))
\\&= \mathrm{Ad}_g(\fa\Omega_{\alpha_l(f)}-\fa{\Omega'}_{\alpha_r(f)})
\\&= \mathrm{Ad}_g(\fa\Omega_{\alpha_l(f)})-(\mathrm{Ad}_g(\fa{\Omega'}_{\alpha_r(f)}) - g^{*}\bar\theta)- g^{*}\bar\theta
\\&= \mathrm{Ad}_g(\fa\Omega_{\alpha_l(f)})-\fa{\Omega'}_{R(\alpha_r(f),g^{-1})}- g^{*}\bar\theta
\\&= \mathrm{Ad}_{g}(A_{f,g}) - A' _{f,g}- g^{*}\bar\theta
\end{align*}
\end{comment}
For condition \cref{gt2} we first compute using  \cref{eq:conform:a} that
\begin{equation*}
\mathrm{d}\varphi + \frac{1}{2}[\varphi \wedge \varphi]+\alpha_{*}(A'\wedge \varphi)=(\alpha_{g})_{*}(\mathrm{d}\nu_0)+ \frac{1}{2}(\alpha_g)_{*}[\nu_0\wedge \nu_0]+ (\alpha_g)_{*}\alpha_{*}(\alpha_r^{*}\fa{\Omega'} \wedge \nu_0)\text{.}
\end{equation*}
\begin{comment}
Indeed,
\begin{eqnarray*}
&&\mquad \mathrm{d}\varphi _{f,g}+ \frac{1}{2}[\varphi_{f,g}\wedge\varphi_{f,g}] + \alpha_{*}(A'_{f,g}\wedge \varphi_{f,g})
\\&=& \mathrm{d}(\alpha_g)_{*}(\nu_0|_f)+ \frac{1}{2}[(\alpha_g)_{*}(\nu_0|_f)\wedge(\alpha_g)_{*}(\nu_0|_f)] + \alpha_{*}(\fa{\Omega'}_{R(\alpha_r(f),g^{-1})}\wedge (\alpha_g)_{*}(\nu_0|_f))
\\&=& (\alpha_{g})_{*}(\mathrm{d}\nu_0|_f)+\alpha_{*}(\bar\theta_g \wedge (\alpha_{g})_{*}(\nu_0|_f))+ \frac{1}{2}(\alpha_g)_{*}[\nu_0|_f \wedge \nu_0|_f] \\&&\qquad+ \alpha_{*}((\mathrm{Ad}_{g}(\fa{\Omega'}_{\alpha_r(f)}) -\bar\theta_g )\wedge (\alpha_g)_{*}(\nu_0|_f))
\\&=& (\alpha_{g})_{*}(\mathrm{d}\nu_0|_f)+ \frac{1}{2}(\alpha_g)_{*}[\nu_0|_f \wedge \nu_0|_f] + (\alpha_g)_{*}\alpha_{*}(\fa{(\Omega')}_{\alpha_r(f)} \wedge \nu_0|_f)
\end{eqnarray*}
\end{comment}
Using the fake-flatness of $\nu$, this is equal to $-(\alpha_g)_{*}(\nu_1)$.
Another part of the relation $J_{\nu}^{*}\Omega'=\Omega$ is $\nu_1 = \alpha_l^{*}\fc\Omega-\alpha_r^{*}\fc{\Omega'}$; using this and  \cref{eq:conform:c} it is easy to show that
$-(\alpha_g)_{*}(\nu_1)=   (\alpha_{g})_{*}(B)-B'$.
This shows  \cref{gt2}. 
\begin{comment}
Indeed,
\begin{eqnarray*}
- (\alpha_g)_{*}(\nu_1|_f)
&=&  -(\alpha_{g})_{*}(\fc\Omega_{\alpha_l(f)}-\fc{\Omega'}_{\alpha_r(f)})
\\&=&  -(\alpha_{g})_{*}(\fc\Omega_{\alpha_l(f)})+\fc{(\Omega')}_{R(\alpha_r(f),g^{-1})}
\\&=&  (\alpha_{g})_{*}(B_{f,g})-B'_{f,g}
\end{eqnarray*}
Here we have used
\begin{equation*}
\fc\Omega'_{R(\alpha_r(f),g^{-1})} &=& (\alpha_{g})_{*}(\Omega'_{\alpha_r(f)})\text{.}
\end{equation*}
\end{comment}
\end{proof}

Thus, we have the quantity $ h_{g,\varphi}(\lambda,\gamma)$ associated to any pair of  paths $\lambda$ in $J$ and $\gamma$ in $G$. Our first goal is to understand the dependence on $\gamma$.

\begin{lemma}
\label{lem:hJextend}
$h_{g,\varphi}(\lambda,\gamma)=h_{g,\varphi}(\lambda\cdot \gamma^{-1},1)$.
\end{lemma}

\begin{proof}
We consider $\rho:Q \times G \to Q$ defined by $\rho(j,g,g')=(j\cdot g',gg')$. It is easy to check that $\rho^{*}\varphi=\pr_Q^{*}\varphi$ and $\rho^{*}(\chi'^{*}\fa{\Omega'})=\pr_Q^{*}(\chi'^{*}\fa{\Omega'})$.
Now, the definition of $h_{g,\varphi}$ and \cref{lem:PE:c} give the claim.
\begin{comment}
We have
\begin{equation*}
\rho^{*}\varphi|_{f,g,g'}=\varphi|_{fg',gg'}=(\alpha_{gg'})_{*}(\nu_0|_{fg'})=(\alpha_{g})_{*}(\nu_0|_f)=\pr_Q^{*}\varphi\text{.}
\end{equation*}
Further, 
\begin{equation*}
\rho^{*}A'|_{f,g,g'}=\fa{\Omega'}|_{R(\alpha_r(fg'),g'^{-1}g^{-1})}=\fa{\Omega'}|_{R(\alpha_r(f),g^{-1})}=\pr_Q^{*}A'|_{f,g,g'}\text{.}
\end{equation*}
\end{comment}
\end{proof}

By the lemma, it suffices to consider the quantity $h_{\nu}(\lambda)\in H$ associated to each path $\lambda$ in $J$.

\begin{proposition}
\label{lem:calchgvarphiJ}
Let $J:\inf P \to \inf P'$ be a 1-morphism in $\zweibunconff\Gamma M$, and $\nu$ be its pullback. For every be a path $\lambda$ in $J$ such that $\alpha_r(\lambda)$ is horizontal, and every path $\gamma$ in $G$ we have $h_{\nu}(\lambda \cdot \gamma)=\alpha(\gamma(1)^{-1},\PE_{\nu_0}(\lambda))$. \end{proposition}

\begin{proof}
Similar to the proof of \cref{lem:calchgvarphi}. 
\end{proof}

\subsection{Surface-ordered exponentials}

\label{sec:surfaceordered}

If $\Gamma$ is a Lie 2-group and $(A,B)$ is a fake-flat  $\Gamma$-connection on a smooth manifold $X$, then there exists a \emph{surface-ordered exponential} $\SE_{A,B}(\Sigma) \in H$ associated to any bigon $\Sigma:\gamma \Rightarrow  \gamma'$ in $X$. It is defined by a two-fold iteration of path-ordered exponentials in \cite[Section 2.3.1]{schreiber5}. We summarize the properties of the surface-ordered exponential in the following four propositions.

\begin{proposition}
Let $(A,B)$ be a fake-flat $\Gamma$-connection, and $\Sigma:\gamma \Rightarrow \gamma'$ be a bigon.
\begin{enumerate}[(a)]

\item
\label{lem:SE:a}
$\SE_{A,B}(\Sigma)$ only depends on the thin homotopy class of $\Sigma$. 

\item 
\label{lem:SE:b}
It satisfies the target-source-matching condition $t(\SE_{A,B}(\Sigma)) \cdot \PE_{A}(\gamma) = \PE_{A}(\gamma')$.

\item
\label{lem:SE:c}
If $\Sigma': \gamma' \Rightarrow \gamma''$ is vertically composable to a bigon $\Sigma' \bullet \Sigma: \gamma \Rightarrow \gamma''$, then 
\begin{equation*}
\SE_{A,B}(\Sigma' \bullet \Sigma) = \SE_{A,B}(\Sigma') \cdot \SE_{A,B}(\Sigma)\text{.}
\end{equation*} 

\item
\label{lem:SE:d}
If $\tilde \Sigma: \tilde\gamma \Rightarrow \tilde\gamma'$ is horizontally composable to a bigon $\tilde \Sigma \circ \Sigma: \tilde\gamma \circ \gamma \Rightarrow \tilde\gamma'\circ \gamma'$, then 
\begin{equation*}
\SE_{A,B}(\tilde\Sigma\circ \Sigma)=\SE_{A,B}(\tilde\Sigma) \cdot \alpha(\PE_A(\tilde\gamma) , \SE_{A,B}(\Sigma) )\text{.}
\end{equation*}

\item
\label{lem:SE:e}
If $f:X \to Y$ is a smooth map, then $\SE_{f^{*}(A,B)}(\Sigma) = \SE_{A,B}(\Sigma \circ f)$.

\item
\label{lem:soetriv}
If $B=0$, then $\SE_{A,B}(\Sigma)=1$.

\end{enumerate}
\end{proposition}

\begin{proof}
\cref{lem:SE:a*,lem:SE:b*,lem:SE:c*,lem:SE:d*} are a reformulation of \cite[Proposition 2.17]{schreiber5}. \cref{lem:SE:e*} follows from \cref{lem:PE:c}. Only for \cref{lem:soetriv*} we have to look into the details of the definition of the surface ordered exponential in \cite[Section 2.3.1]{schreiber5}.
Since $B=0$, we have $\mathcal{A}_{\Sigma}=0$ for the 1-form $\mathcal{A}_{\Sigma}$ of Eq. (2.26) in that reference. Then, the function $f_{\Sigma}$ vanishes, and so does the map $k_{A,0}$ which defines $\SE_{A,0}(\Sigma)$.
\end{proof}

\begin{remark}
\label{rem:bigonpar}
Suppose $f:[0,1]^2 \to X$ is a smooth map, of which we can think of as a piece of surface in $X$. In order to compute the surface ordered exponential of $f$, we need the following terminology. A \emph{bigon-parameterization} of $f$ is a bigon $\Sigma: \gamma_r \ast \gamma_t \Rightarrow \gamma_b \ast \gamma_l$ in $X$ such that there exists a homotopy between $f$ and $\Sigma$ of rank less than three, which induces homotopies of rank less than two between the following pairs of paths:
$\gamma_t$ and the top edge $f(0,-)$, $\gamma_b$ and $f(1,0)$ the bottom edge, $\gamma_r$ and the right edge $f(-,0)$ and $\gamma_r$ and the left edge $f(-,1)$. 
It follows immediately that two bigon-parameterizations $\Sigma$ and $\Sigma'$ of $f$ are thin homotopic. In particular, the surface-ordered exponential of $f$ is well-defined. To see the existence of a bigon-parameterization, one can compose $f$ with a standard bigon in $\R^2$, see \cite[Eq. 2.5]{schreiber5}. 
\end{remark}

Let $\inf P$ be a principal $\Gamma$-2-bundle with fake-flat connection $\Omega$. We recall that $(\fa\Omega,-\fc\Omega)$ is a fake-flat $\Gamma$-connection on $\ob{\inf P}$. Hence, we have a surface-ordered exponential
\begin{equation*}
h_{\Omega}(\Sigma) := h_{\fa\Omega,-\fc\Omega}(\Sigma)\in H
\end{equation*}
associated to every bigon $\Sigma$ in $\ob{\inf P}$. 
Next we study the surface-ordered exponential under gauge transformations. We start with the following.

\begin{lemma}
\label{lem:SE:gauge:b}
Let $(g,\varphi):(A,B) \to (A',B')$ be a gauge transformation between fake-flat $\Gamma$-connections,  let  $\Sigma: \gamma \Rightarrow \gamma'$ be a bigon, and let  $y:=\gamma(1)=\gamma'(1)$. Then, 
\begin{equation*}
\SE_{A',B'}(\Sigma) \cdot h_{g,\varphi}(\gamma)^{-1}= h_{g,\varphi}(\gamma')^{-1} \cdot \alpha(g(y),\SE_{A,B}(\Sigma))\text{.}
\end{equation*}
\end{lemma}  

\begin{proof}
\cite[Lemma 2.19]{schreiber5}.
\end{proof}

Now recall from \cref{lem:gaugetrafoX}  that every principal $\Gamma$-2-bundle $\inf P$ with fake-flat connection $\Omega$ induces a gauge transformation on the smooth manifold $X := \mor{\inf P} \times G$. \cref{lem:SE:gauge:b} gives the following result.

\begin{proposition}
\label{prop:SE:gauge:b}
Let $\inf P$  be a principal $\Gamma$-2-bundle with fake-flat connection $\Omega$. Let $\Psi:\rho \Rightarrow \rho'$ be a bigon in $\mor{\inf P}$, and $\Theta:\gamma \Rightarrow \gamma'$ be a bigon in $G$. Then, 
\begin{equation*}
\alpha(\gamma(1)^{-1},\SE_{\Omega}(R(s(\Psi),\Theta^{-1}))) \cdot h_{\Omega}(\rho)^{-1}= h_{\Omega}(\rho')^{-1} \cdot \SE_{\Omega}(t(\Psi))\text{,}
\end{equation*}
where $\Theta^{-1}$ denotes the point-wise inversion in $G$. 
\end{proposition}

\begin{comment}
\begin{proof}
We consider $(\Psi,\Theta):(\rho,\gamma) \Rightarrow (\rho',\gamma')$ as a bigon in $X$. \cref{lem:SE:gauge:b} gives
\begin{equation*}
\SE_{A',B'}(\Psi,\Theta) \cdot h_{g,\varphi}(\rho,\gamma)^{-1}= h_{g,\varphi}(\rho',\gamma')^{-1} \cdot \alpha(g(\rho(1),\gamma(1)),\SE_{A,B}(\Psi,\Theta))\text{.}
\end{equation*}
Using \cref{lem:SE:e,co:h:a} this becomes
\begin{equation*}
\SE_{\Omega}(\chi_2(\Psi,\Theta)) \cdot\alpha(\gamma(1),h_{\Omega}(\rho)^{-1})= \alpha(\gamma##'(1),h_{\Omega}(\rho')^{-1} \cdot \alpha(\gamma(1),\SE_{\Omega}(\chi_1(\Psi,\Theta)))\text{.}
\end{equation*}
Using the definitions of $\chi_1$ and $\chi_2$ and $g$, we obtain the claim.
\end{proof}
\end{comment}

\begin{corollary}
\label{lem:soe2bun}
If $\Sigma$ is a bigon in $\ob{\inf P}$ and $\Theta:\gamma \Rightarrow \gamma'$ is a bigon in $G$, then
\begin{equation*}
\SE_{\Omega}(R(\Sigma,\Theta))=\alpha(\gamma(1)^{-1},\SE_{\Omega}(\Sigma))\text{.}
\end{equation*}
\end{corollary}

\begin{proof}
We use \cref{prop:SE:gauge:b} with $\Psi=\id(\Sigma)$, and then  \cref{lem:gomega:c}. \end{proof}

Finally, we recall from \cref{lem:1morphgauge} that every 1-morphism $J:\inf P \to \inf P'$ in $\zweibunconff\Gamma M$ induces a gauge transformation on the smooth manifold $Q := J \times G$. \cref{lem:SE:gauge:b} gives the following result. 

\begin{proposition}
\label{lem:bigonF}
Let $J: \inf P \to \inf P'$ be a 1-morphism in $\zweibunconff \Gamma M$. Let $\Omega$ and $\Omega'$ denote the connections on $\inf P$ and $\inf P'$, respectively, and let $\nu$ denote the $\Omega'$-pullback on $J$. Let $\Sigma: \lambda \Rightarrow\lambda'$ be a bigon in $J$, let $\Theta: \gamma \Rightarrow \gamma'$ be a bigon in $G$. Then,
\begin{equation*}
\SE_{\Omega'}(R(\alpha_r(\Sigma),\Theta^{-1})) \cdot h_{\nu}(\lambda\cdot \gamma^{-1})^{-1}= h_{\nu}(\lambda'\cdot \gamma'^{-1})^{-1} \cdot \alpha(\gamma(1),\SE_{\Omega}(\alpha_l(\Sigma)))\text{.}
\end{equation*}
\end{proposition}

\begin{comment}
\begin{proof}
\cref{lem:SE:gauge:b} implies for $l:=\lambda(1)=\lambda'(1)$ and $g':=\gamma(1)=\gamma'(1)$: 
\begin{equation*}
\SE_{\chi'^{*}\Omega'}(\Sigma,\Gamma) \cdot h_{g,\varphi}(\lambda,\gamma)^{-1}= h_{g,\varphi}(\lambda',\gamma')^{-1} \cdot \alpha(g(l,g'),\SE_{\chi^{*}\Omega}(\Sigma,\Gamma))\text{.}
\end{equation*}
The definitions of $\chi$, $\chi'$, and $g$ yield the claim.
Now \cref{lem:hJextend} gives
\begin{equation*}
\SE_{\chi'^{*}\Omega'}(\Sigma,\Gamma) \cdot h_{\nu}(\lambda\cdot \gamma^{-1})^{-1}= h_{\nu}(\lambda'\cdot \gamma'^{-1})^{-1} \cdot \alpha(g',\SE_{\chi^{*}\Omega}(\Sigma,\Gamma))\text{.}
\end{equation*}
\end{proof}
\end{comment}

\subsection{Smooth 2-functors on path 2-groupoids}

\label{sec:smoothfunctors}

The path 2-groupoid $\mathcal{P}_2(M)$ of a smooth manifold $M$ is defined in the following way:
\begin{enumerate}[(a)]

\item 
Its objects are the points $x$ of $M$.

\item
Its 1-morphisms are thin homotopy classes $[\gamma]:x \to y$ of paths in $M$.

\item
Its 2-morphisms are thin homotopy classes $[\Sigma]:[\gamma] \Rightarrow [\gamma']$ of bigons in $M$.

\end{enumerate}
Using thin homotopy classes is one way to turn this structure into a strict bigroupoid, with the usual composition of paths, and the obvious vertical and horizontal composition of bigons sketched at the beginning of \cref{sec:compbigoncomp}. A  detailed definition is in \cite[Section 2.1]{schreiber5}.

\noindent
We recall the following constructions from \cite[Section 2.3]{schreiber5}:
\begin{enumerate}[(a)]

\item 
If $(A,B)$ is a fake-flat $\Gamma$-connection on $M$, then  we obtain a 2-functor \begin{equation*}
F_{A,B}:\mathcal{P}_2(M) \to B\Gamma\text{,}
\end{equation*}
given by the following assignments:
\begin{equation*}
x \mapsto \ast
\quomma
[\gamma] \mapsto \PE_{A}(\gamma)
\quand
[\Sigma] \mapsto (\SE_{A,B}(\Sigma),\PE_{A}(\gamma))\text{,}
\end{equation*}
where $[\Sigma]:[\gamma]\Rightarrow [\gamma']$. The well-definedness under thin homotopies was already mentioned in \cref{lem:PE:a,lem:SE:a}. 

\item
If $(g,\varphi):(A,B) \to (A',B')$ is a gauge transformation between fake-flat $\Gamma$-connections, then we have a  pseudonatural transformation 
\begin{equation*}
\rho_{g,\varphi}: F_{A,B} \to F_{A',B'}\text{,}
\end{equation*}
given by the assignments
$x \mapsto g(x)$ and $[\gamma]\mapsto (h_{g,\varphi}(\gamma)^{-1},g(y)\PE_{A}(\gamma))$ for $[\gamma]:x \to y$.

\item
If $a: (g_1,\varphi_1) \Rightarrow (g_2,\varphi_2)$ is a a gauge 2-transformation, then we have a modification 
\begin{equation*}
\mathcal{A}_{a}:\rho_{g_1,\varphi_1} \Rightarrow \rho_{g_2,\varphi_2}\text{,}
\end{equation*}
given by the assignment $x\mapsto (a(x),g(x))$. 
\end{enumerate}
These three constructions define a 2-functor
\begin{equation*}
\mathcal{P}:\conff\Gamma M \to \fun(\mathcal{P}_2(M),B\Gamma)\text{,}
\end{equation*}
see \cite[Section 2.3.4]{schreiber5}. 
Besides of being a strict bigroupoid, the path 2-groupoid is naturally  enriched in the category of diffeological spaces. Hence, there is a sub-bicategory
\begin{equation*}
\fun^{\infty}(\mathcal{P}_2(M),B\Gamma) \subset \fun(\mathcal{P}_2(M),B\Gamma)
\end{equation*} 
consisting of \emph{smooth} 2-functors. The main result of \cite{schreiber5} is that $\mathcal{P}$ induces an equivalence of bicategories, $\conff\Gamma M \cong \fun^{\infty}(\mathcal{P}_2(M),B\Gamma)$.

\end{appendix}


\bibliographystyle{kobib}
\bibliography{kobib}

\begin{thebibliography}{LGSX09}

\bibitem[ACJ05]{aschieri}
P.~Aschieri, L.~Cantini, and B.~Jurco, \quot{Nonabelian bundle gerbes, their
  differential geometry and gauge theory}.
\newblock {\em Commun. Math. Phys.}, 254:367--400, 2005.
\newblock \kobiburl{http://arxiv.org/abs/hep-th/0312154}
\bibitem[BM05]{breen1}
L.~Breen and W.~Messing, \quot{Differential geometry of gerbes}.
\newblock {\em Adv. Math.}, 198(2):732--846, 2005.
\newblock \kobiburl{http://arxiv.org/abs/math.AG/0106083}
\bibitem[BS07]{baez2}
J.~C. Baez and U.~Schreiber, \quot{Higher gauge theory}.
\newblock In A.~Davydov, M.~Batanin, M.~Johnson, S.~Lack, and A.~Neeman,
  editors, {\em Categories in Algebra, Geometry and Mathematical Physics},
  Proc. Contemp. Math. AMS, Providence, Rhode Island, 2007.
\newblock \kobiburl{http://arxiv.org/abs/math/0511710}
\bibitem[LGSX09]{Laurent-Gengoux}
C.~Laurent-Gengoux, M.~Stiénon, and P.~Xu, \quot{Non-abelian differentiable
  gerbes}.
\newblock {\em Adv. Math.}, 220(5):1357--1427, 2009.
\newblock \kobiburl{http://arxiv.org/abs/math/0511696v5}
\bibitem[Mur96]{murray}
M.~K. Murray, \quot{Bundle gerbes}.
\newblock {\em J. Lond. Math. Soc.}, 54:403--416, 1996.
\newblock \kobiburl{http://arxiv.org/abs/dg-ga/9407015}
\bibitem[NW13]{Nikolaus}
T.~Nikolaus and K.~Waldorf, \quot{Four equivalent versions of non-abelian
  gerbes}.
\newblock {\em Pacific J. Math.}, 264(2):355--420, 2013.
\newblock \kobiburl{http://arxiv.org/abs/1103.4815}
\bibitem[Pro96]{pronk}
D.~A. Pronk, \quot{Entendues and stacks as bicategories of fractions}.
\newblock {\em Compos. Math.}, 102:243--303, 1996.
\bibitem[SP11]{pries2}
C.~Schommer-Pries, \quot{Central extensions of smooth 2-groups and a
  finite-dimensional string 2-group}.
\newblock {\em Geom. Topol.}, 15:609--676, 2011.
\newblock \kobiburl{http://arxiv.org/abs/0911.2483}
\bibitem[SW11]{schreiber5}
U.~Schreiber and K.~Waldorf, \quot{Smooth functors vs. differential forms}.
\newblock {\em Homology, Homotopy Appl.}, 13(1):143--203, 2011.
\newblock \kobiburl{http://arxiv.org/abs/0802.0663}
\bibitem[SW13]{schreiber2}
U.~Schreiber and K.~Waldorf, \quot{Connections on non-abelian gerbes and their
  holonomy}.
\newblock {\em Theory Appl. Categ.}, 28(17):476--540, 2013.
\newblock \kobiburl{http://arxiv.org/abs/0808.1923}
\bibitem[SW16]{schreiber6}
U.~Schreiber and K.~Waldorf, \quot{Local theory for 2-functors on path
  2-groupoids}.
\newblock {\em J. Homotopy Relat. Struct.}, pages 1--42, 2016.
\newblock \kobiburl{http://arxiv.org/abs/1303.4663}
\bibitem[Wal]{Waldorf2016}
K.~Waldorf, \quot{A global perspective to connections on principal 2-bundles}.
\newblock Preprint.
\newblock \kobiburl{http://arxiv.org/abs/1608.00401}
\bibitem[Woc11]{wockel1}
C.~Wockel, \quot{Principal 2-bundles and their gauge 2-groups}.
\newblock {\em Forum Math.}, 23:565--610, 2011.
\newblock \kobiburl{http://arxiv.org/abs/0803.3692}
\end{thebibliography}

\end{document}